\documentclass[11pt]{article}
\usepackage{amsmath,amssymb,amscd,amsthm}
\usepackage{epstopdf,graphicx,enumitem}

\makeatletter\@addtoreset {equation}{section}\makeatother

\newtheorem{theorem}{Theorem}[section]
\newtheorem{lemma}{Lemma}[section]
\newtheorem{remark}{Remark}[section]

\newtheorem{corollary}{Corollary}[section]
\newtheorem{definition}{Definition}[section]
\newtheorem{proposition}{Proposition}[section]
\def\R{\mathbb{R}}
\def\C{\mathbb{C}}
\def\N{\mathbb{N}}

\makeatletter\@addtoreset {figure}{section}\makeatother

\newcommand{\Ascr} {{\cal A}}
\newcommand{\Bscr} {{\cal B}}
\newcommand{\Cscr} {{\cal C}}

\newcommand{\Iscr} {{\cal I}}
\newcommand{\Jscr} {{\cal J}}
\newcommand{\Mscr} {{\cal M}}
\newcommand{\Tscr} {T}
\newcommand{\m} {{\hbox{\hskip 1pt}}}
\newcommand{\FORALL} {{\hbox{$\hskip 11mm \forall \;$}}}
\newcommand{\e}{\varepsilon}

\newcommand{\Gscr} {{\cal G}}

\newcommand{\Lscr} {{\cal L}}
\newcommand{\Nscr} {{\cal N}}

\newcommand{\Uscr} {{\cal U}}
\newcommand{\Sscr} {{\cal S}}

\newcommand{\FFF} {{\bf F}}
\newcommand{\GGG} {{\bf G}}
\newcommand{\HHH} {{\bf H}}

\newcommand{\dd}  {{\rm d}\hbox{\hskip 0.5pt}}

\begin{document}

\title{\bf The global bifurcation picture for ground states in nonlinear
Schr\"{o}dinger equations}

\author{E. Kirr$^1$, and V. Natarajan$^2$  \\
{\small $^{1}$ Department of Mathematics, University of Illinois,
Urbana--Champaign, Urbana, Illinois 61801} \\
{\small $^{2}$ Systems and Control Engineering Group, Indian Institute
of Technology Bombay, Mumbai, India 400076} }

\maketitle

\begin{abstract}
In this paper, we propose a method of finding all coherent structures supported by a given nonlinear wave equation. It relies on enhancing the recent global bifurcation theory as developed by Dancer, Toland, Buffoni and others, by determining all the limit points of the coherent structure manifolds at the boundary of the Fredholm domain. Local bifurcation theory is then used to trace back these manifolds from their limit points into the interior of the Fredholm domain identifying the singularities along them. This way all coherent structure manifold are discovered except may be the ones which form loops and hence never reach the boundary. The method is then applied to the Schr\"odinger equation with a power nonlinearity for which all ground states are identified.
\end{abstract}

\section{Introduction}
The long standing conjecture in dispersive wave equations is the asymptotic decomposition of any initial data into coherent structures. The conjecture is true in the linear case, in nonlinear integrable systems and also certain particular cases of small amplitude regimes, see for example \cite{SS:nls,swc,ty3,gnt,kz2,km:as3d}. However, the set of all coherent structures is not completely, which is a big obstacle in even starting to approach the conjecture in large amplitude regimes for non-integrable systems. In this paper we propose an original method for finding all coherent structures of a given equation, see Section \ref{sse:conc}. We apply it to the case of nonlinear Schr\"odinger equation with potential, which is sometimes called Gross-Pitaevski equation:
\begin{equation}\label{GP}
i u_t = -\Delta u + V(x) u + \sigma |u|^{2p} u,
\end{equation}
where $u(x,t) : \R^n \times \R \to \mathbb{C}$ is the wave function, $\Delta$ is the Laplacian,
$0<p<2/(n-2)$ is the nonlinearity power, $\sigma \in \R$ determines the
defocusing (repulsive), respectively focusing (attractive),
character of the nonlinearity when $\sigma >0,$ respectively
$\sigma<0,$ and $V(x) : \R^n\mapsto\mathbb{R}$ is an external,
real-valued potential satisfying:
\begin{itemize}
\item[(H1)] $V(x)\in L^{\infty}(\R^n),$
\item[(H2)] $\lim_{|x|\rightarrow\infty}V(x)=0,$
\end{itemize}
Hypothesis (H1) implies that $-\Delta+V(x)$ is a self-adjoint
operator on $L^2(\R^n)$ with domain $H^2(\R^n),$ while (H1)-(H2) make $V$ a relatively compact perturbation of $-\Delta,$ see \cite{RS:bk4}.  At the expense of complicating some of the proofs without adding new ideas we can recover all our result by just assuming $V$ is a relative compact perturbation of the Laplacian, for example $V\in L^q(\R^n),\ q\geq 1,\ q>n/2,$ see the discussion in \cite{Kir:dcs}. In certain instances we will also make the following spectral assumption:
\begin{itemize}
\item[(H3)] $L_0=-\Delta+V(x)$ has the lowest eigenvalue $-E_0<0.$
\end{itemize}
It is well known from second order elliptic theory that the lowest
eigenvalue of $L_0$ is simple, the corresponding eigenfunction,
$\psi_0,$ can be chosen to be real valued and normalized in $L^2:$
\begin{equation}\label{linear-mode}
-\Delta\psi_0(x) + V(x) \psi_0(x) + E_0 \psi_0(x) = 0, \quad
\psi_0(x)>0,\ x \in \R^n,\ \|\psi_0\|_{L^2}=1.
\end{equation}

The equation is an ubiquitous model for wave propagation in optics, fluids and statistical quantum physics. Its most striking features are the coherent structures (bound states) which propagate without changing shape leading to important application such as propagating modes in optical wave guides, solitary waves at the interfaces between fluids, Bose-Einstein condensates. This paper focuses on the bound states of \eqref{GP} which are solutions of the form
$$u(t,x)=e^{iEt}\psi(x).$$
By plugging in we get
\begin{equation}\label{stationary}
 -\Delta\psi(x) + V(x)\psi(x) + \sigma |\psi(x)|^{2p} \psi(x)  + E \psi(x) = 0, \quad x \in \R^n,
\end{equation}
The latter is usually studied in the weak (distributional) sense i.e., a solution $\psi\in H^1(\R^n,\C)$ satisfies:
\begin{equation} \label{eq:weak}
 \langle \nabla\psi, \nabla\phi\rangle + \langle (V+E)\psi+ \sigma|\psi|^{2p}\psi,\phi\rangle=0 \FORALL \phi\in H^1(\R^n).
\end{equation}
\begin{remark}\label{rmk:h1toh2}
A standard elliptic regularity argument, see the Appendix, shows that any $H^1$ solution of \eqref{eq:weak} is a strong, $H^2,$ solution of \eqref{stationary} i.e., the latter holds for almost all $x\in\R^n.$
\end{remark}

This paper is inspired by the fairly recent theory of global bifurcations, see for example \cite{Rab:gbt,BT:gbt}, which states that, in this particular case, the solutions of \eqref{stationary} bifurcating from the trivial solution at $E_0$ form a $C^1$ manifold which can be continued until $E=0$ or $E=\infty$ or $H^2$ norm of the solutions blows up at finite $E>0$.
We develop new techniques to study the limit points of this manifold at the above boundary sets. Our results show that in the attractive case, $\sigma<0$, there are no limit points at $E=0$ or at $0<E<\infty, \|\psi\|_{H^2}=\infty,$ hence this manifold can be extended for all $E>E_0.$ Moreover, we show that the limit point at the $E=\infty$ boundary must be a subset of a rather precise set defined with the help of solutions of a limiting equation
\begin{equation} \label{uinf2}
 -\Delta\Psi + \Psi + |\Psi|^{2p}\Psi = 0.
\end{equation}
In the case of the ground states, see Definition \ref{def:gs}, we actually fully determine their limit points, see Remark \ref{rmk:limitp}. Since the limit points depend on the critical points of the potential, we analyze a few well known examples such as one-well and double-well potentials, where we determine the rather complicated bifurcation diagram of the ground states, see Figure \ref{fig:2well2}.

Our point is that the techniques we have developed can be extended to other equations in order to determine, as in the above mentioned examples, all coherent structures supported by them, see Section \ref{sse:conc}. There are no such results that we know of in the existing literature. However, the existence of unique ground state branch in one space dimension $(n=1)$ in the case of an even and strictly monotonic potential (for $x>0$) has been proven in \cite{js:ubs} using monotonicity techniques for ordinary differential equations. Arguments using the behavior of a branch of ground states approaching the $E=\infty$ boundary have again been used in one space dimension in \cite{KKP:sb}, where, along the branch starting at $E_0$  a symmetry-breaking bifurcation point is found and analyzed. We actually generalize that result to any space dimension and more complicated symmetries of the potential in Section \ref{se:sb}. The result in this section is related to the one in \cite{AFGST:02}, where the authors studied Hartree equation using variational methods. We think that our techniques can be extended to the latter equation in which case they prove not only that the shape of ground states change from small amplitudes to large amplitudes, but also that there is a bifurcation point that can be analyzed.

One of the first results showing that ground states for our problem \eqref{stationary} with attractive nonlinearity $(\sigma<0)$ exists for all $E>E_0$ is \cite{RW:88}. In it the authors used variational methods and conjectured that the profile of the ground state will converge to the ground state of the limiting equation \eqref{uinf2}. In this paper, we not only rigorously prove their conjecture but we also show that depending on the number of critical points of the potential there are many other ``ground states'' in the limit $E\to\infty$, with complicated bifurcation diagram, see for example Figure \ref{fig:2well2}, where by ground states we no longer mean the global minimizer of a certain functional, but also other local minimizers or saddle points.

There are similarities between our equation \eqref{stationary} as $E\to\infty$ and the semi-classical limit for nonlinear Schr\"odinger equation studied in \cite{FW} and \cite{Oh:scl}. Their approach which combines Lyapunov-Schmidt decomposition with degree theory for the reduced finite-dimensional system leaves open the question of uniqueness of the solutions they find. However, our arguments in Theorem \ref{th:bfelarge} can be adapted to their problem to show not only that the single profile ground states in \cite{FW} and the multiprofile ground states in \cite{Oh:scl} exist but also that they are unique at each value of the parameter $h$ (the Planck constant).

Our the paper is organized as follows. In Section \ref{se:small}, we recall a few results on small amplitudes nonlinear bound states of \eqref{stationary} which, it turns out, bifurcate from the zero solutions at eigenvalues of the linear part $-\Delta+V$ and organize themselves in $C^1$ branches. Then we study how far these branches can be extended without passing through singular points i.e., we determine their maximal extension. We also generalize these results to arbitrary i.e., not necessarily small amplitude, branches of bound states. In Section \ref{se:bif} we study the finite end points of these maximal branches and via a compactness result we show that they are actually bound states themselves. In particular, these allows us to determine all the ground states in the defocusing case $(\sigma>0)$, see Theorem \ref{th:defocusing}. In doing this we generalize global bifurcation result due to Jeanjean, Lucia and Stuart, see \cite{jls:gsd} and Remark \ref{rmk:defocusing}. In the focusing case $(\sigma<0)$, we prove for the first time to the best of our knowledge, the compactness result given by Theorem \ref{th:comp}. It is a delicate result which involves concentration compactness and spectral properties of Schr\"odinger type linear operators with potentials separated by a large distance, but it allows us not only to identify the limit points of the maximal branches, see Corollary \ref{cor:comp}, but also to apply global bifurcation techniques to continue these branches until they reach the $E=\infty$ boundary. In Section \ref{se:large}, we study the behavior of bound state branches in the limit $E\to\infty$. In particular, we show that there exists a unique renormalization under which each of these branches converges to a superposition of solutions (profiles) of the limiting equation \eqref{uinf2}, see Theorem \ref{correct_scaling}. Since the ground state of the latter is unique up to translations, we fully determine all the ground state branches at this limit, see Theorems \ref{th:brelarge} and \ref{th:bfelarge}. Similar to the compactness result at finite $E,$ Theorem \ref{th:brelarge} uses an original argument based on local bifurcation tools to prevent the drifting of profiles to $x=\infty$.

In Section \ref{se:sb}, we combine all the results in the previous sections to generalize the symmetry-breaking bifurcation result in \cite{KKP:sb} to any dimensions and more complicated symmetries. Moreover, in Section \ref{se:examples} we show how our results can be combined with global bifurcation tools to determine all ground states of our problem. Certain implications about bound states are also discussed. This section concludes by showing that the algorithm for identifying all coherent structures of a large class of dispersive wave equations, proposed by one of the authors in \cite{Kir:dcs}, is successful for the nonlinear Schr\"{o}dinger equation with a power nonlinearity and by discussing possible generalizations.

In the Appendix, we recall or prove slight generalizations of known results about regularity and exponential decay of solutions of elliptic equations and about compactness in Sobolev spaces. We end this section by proving the existence of a natural decomposition of branches of solutions approaching a singularity or the limit $E=\infty$, based on invariant subspaces of the linearized operator. In comparison to similar decompositions, for example the one in \cite{Oh:scl}, this one enables us to refine the analysis of the limit point and establish the uniqueness of the multi profile solution.

\section{Maximal branches of solutions}\label{se:small}

In this section we show that the set of solutions of
\eqref{stationary} can be organized in $C^1$ manifolds and we
identify some of their properties near singular points. In
particular we construct the manifolds of nontrivial solutions
bifurcating from zero at the isolated eigenvalues of $-\Delta + V.$

We rewrite \eqref{stationary} as the problem of finding the zeroes of the
functional
\begin{align}
F&:H^2(\R^n,\C)\times \R\mapsto L^2(\R^n,\C)\nonumber\\
F(\psi,E)&=(-\Delta+V+E)\psi+\sigma|\psi|^{2p}\psi \label{def:F}
\end{align}
where $L^2(\R^n,\C)$ is the Hilbert space of the complex valued square integrable functions on $\R^n$ and  $H^2(\R^n,\C)\hookrightarrow L^2(\R^n,\C),$ is the Hilbert space of the complex valued functions on $\R^n$ which are square integrable together with their first and second weak derivatives. Note that the functional $F$ is not Fr\'echet differentiable over the natural complex structure of the Hilbert spaces $H^2(\R^n,\C),\ L^2(\R^n,\C)$ due to the absolute value term. It is however differentiable over the reals, hence, in what follows, the above Hilbert spaces will be considered over the real field. In particular the scalar product in $L^2(\R^n,\C)$  is: 
$$
\langle\phi,\psi\rangle=\Re\int_{\mathbb{R}^n}\overline\phi(x)\psi(x)dx=\int_{\mathbb{R}^n}\phi_1(x)\psi_1(x)dx+\int_{\mathbb{R}^n}\phi_2(x)\psi_2(x)dx,\ \phi=\phi_1+i\phi_2,\ \psi=\psi_1+i\psi_2,
$$
and any orthogonality statement will refer to the above scalar
product unless specified otherwise. In some situations it is also useful to separate the real and imaginary parts of the functions involved $\psi=\psi_1+i\psi_2=[\psi_1\ \psi_2]^\top$ and identify the following isomorphic \emph{real} Banach spaces:
$$
L^2=L^2(\R^n,\C)\cong L^2(\R^n,\R)\times L^2(\R^n,\R), \qquad
 H^2=H^2(\R^n,\C) \cong H^2(\R^n,\R)\times H^2(\R^n,\R). $$
Then \eqref{stationary} is equivalent to:
$$ (-\Delta+V+E)\psi + \sigma |\psi|^{2p} \psi = 0\quad \Leftrightarrow\quad F(\psi,E)=0, $$
where now
$$ F(\psi=\psi_1+i\psi_2,E)=[F_1\ F_2]^\top=\left[\begin{array}{c}(-\Delta+
 V+E)\psi_1+\sigma(|\psi_1|^2+|\psi_2|^2)^{p}\psi_1 \\ (-\Delta+
 V+E)\psi_2+\sigma(|\psi_1|^2+|\psi_2|^2)^{p}\psi_2\end{array}\right] $$
\begin{remark}\label{symmetries} $F$ is equivariant with respect to rotations in the complex plane and complex conjugation:
\begin{eqnarray}
F(e^{i\theta}\psi,E)&=&e^{i\theta}F(\psi,E),\qquad {\rm for}\ 0\leq\theta<2\pi\label{rotation}\\
F(\overline \psi,E)&=&\overline{F(\psi,E)}. \nonumber
\end{eqnarray}
\end{remark}

Note that $F$ is $C^1$ i.e., continuous with continuous Frechet
derivative. For real valued $\psi\in H^2$ we have
\begin{equation}\label{deriv}
D_\psi F(\psi,E)=L_+(\psi,E)+ iL_-(\psi,E)
 \equiv \left[\begin{array}{lr} L_+(\psi,E) & 0 \\ 0 & L_-(\psi,E)\end{array}\right]
\end{equation}
where,
\begin{eqnarray}
L_+(\psi,E)&=&-\Delta+ V+E+(2p+1)\sigma |\psi|^{2p}\label{defL+}\\
L_-(\psi,E)&=&-\Delta+ V+E+\sigma |\psi|^{2p}\label{defL-}
\end{eqnarray}
\begin{remark}\label{spL}
For $\psi\in H^2$ (actually $H^1$ suffices) and $0<2p<4/(n-2)$ we have $|\psi|^{2p}\in L^q$ for some $q>\max\{1,n/2\}.$ Hence the effective potentials $V+(2p+1)\sigma |\psi|^{2p},$ respectively $V+\sigma |\psi|^{2p},$ are relatively compact, symmetric perturbations of $-\Delta+E.$ Therefore, $L_\pm$ remain self adjoint operators on $L^2$ with domain $H^2$ and, Via Weyl's Theorem, their essential spectrum is $[E,\infty),$  see \cite{RS:bk4}. Moreover, by differentiating the symmetry relation \eqref{rotation} with respect to $\theta$ at $\theta=0,$ we get
$$ D_\psi F(\psi,E)[i\psi]=iF(\psi,E).$$
Therefore, at zeroes $(\psi_E,E)$ of $F$ i.e., solutions of \eqref{stationary}, zero is an e-value of the linearized operator with corresponding eigenfunction $i\psi_E$. In the case $\psi_E$ is real valued, zero becomes an e-value of $L_-(\psi_E,E)$ with corresponding eigenfunction $\psi_E.$
\end{remark}

Since $(0,E),\ E\in\mathbb{R}$ is always a zero of $F$ and the
Fr\'echet derivative in the first variable at these solutions is:
$$D_\psi F(0,E)=\left[\begin{array}{lr} L_+(0,E) & 0 \\ 0 & L_-(0,E)\end{array}\right]=\left[\begin{array}{lr} (-\Delta+
 V+E) & 0 \\ 0 & (-\Delta+
 V+E)\end{array}\right],$$
the implicit function theorem states:

\begin{proposition}\label{pr:nosol} If $-E_*$ is not in the spectrum of the self adjoint operator $-\Delta+V$ then there exist $\epsilon,\ \delta>0$ such that for $|E-E_*|<\delta$ and $\|\psi\|_{H^2}<\epsilon,$ \eqref{stationary} has only the trivial solution.
\end{proposition}

Next subsection discusses small, nontrivial zeroes of $F$ i.e., solutions of equation \eqref{stationary} as they bifurcate from the trivial ones at eigenvalues of $-\Delta + V.$

\subsection{Bifurcations from zero}

The following local bifurcation result describes the existence of
solutions of the stationary problem \eqref{stationary} near $E_0$
defined by \eqref{linear-mode} or any other negative, simple
eigenvalue of $-\Delta+V.$
\begin{proposition}\label{th:ex} Let $-E_0<0$ be a simple eigenvalue of
$-\Delta+ V$ and let $\psi_0$ with $\|\psi_0\|_{L^2}=1$ be a corresponding real valued eigenfunction. Then there exists $\epsilon ,\delta>0$ such that:
 \begin{itemize}
 \item[(i)] If $\sigma >0$, for each $E \in {\cal I} \equiv (E_0-\delta,E_0)$, the stationary problem \eqref{stationary}
 has exactly two, nonzero, real valued solutions $\pm\psi_E$ satisfying $\|\pm\psi_E\|_{H^2}<\epsilon,$ and we choose $\psi_E$ to be the one with positive $L^2$ projection on $\psi_0.$  If $\sigma <0$, for each $E \in {\cal I} \equiv (E_0,E_0+\delta)$, the stationary problem \eqref{stationary} has exactly two, nonzero, real valued solutions $\pm\psi_E$ satisfying $\|\psi_E\|_{H^2}<\epsilon,$ and $\psi_E$ is chosen as above.
 \item[(ii)] Any solution $(\psi,E)$ of \eqref{stationary}
 satisfying $|E-E_0|<\delta$ and $\|\psi\|_{H^2}<\epsilon$ is of
 the form $(\psi,E)=(e^{i\theta}\psi_E,E)$ where $E$ and $\psi_E$
 are given in (i) and $0\le\theta<2\pi.$
 \end{itemize}
Moreover, the function $E\mapsto\psi_E$ is $C^1$ from ${\cal I}$ to
$H^2,$ and $\lim_{E\rightarrow E_0}\|\psi_E\|_{H^2}=0.$
\end{proposition}

\begin{proof} Variants of this theorem have already appeared in Pillet-Wayne \cite{PW},
as well as in many recent publications. We sketch the main steps for
the proof of the result in this form.

Let $\psi_0$ be the eigenfunction of $-\Delta+V$ corresponding to
the simple eigenvalue $-E_0.$ Because of invariance of $(-\Delta+
V)\psi_0=-E_0\psi_0$ under complex conjugation we can choose
$\psi_0$ to be real valued and of $L^2$ norm 1. Then $D_\psi
F(0,E_0)$ is Fredholm of index zero with
kernel=span$_{\mathbb{R}}\{\psi_{0}\equiv[\psi_0\ 0]^\top, i\psi_{0}\equiv [0\ \psi_0]^\top\}$
and range equal to the orthogonal complement in
$L^2$ of its kernel. Let
 $$P_\parallel\phi=\langle\psi_{0},\phi\rangle\psi_{0}+\langle i\psi_{0},\phi\rangle i\psi_{0}=\langle \psi_0,\phi\rangle_\C \psi_0,\qquad
 P_\perp=I-P_\parallel$$
be the two projections associated to this Lyapunov-Schmidt decomposition where, depending on the context, $I$ is the identity on $H^2$ respectively on $L^2.$ Then
 $$F(\psi,E)=0$$ is equivalent to
 \begin{eqnarray}
 P_\perp F(P_\parallel\psi+P_\perp\psi,E)&=&0,\label{eq:perp}\\
 P_\parallel F(P_\parallel\psi+P_\perp\psi,E)&=&0.\label{eq:parallel}
 \end{eqnarray}
But the implicit function theorem can now be applied to
\eqref{eq:perp} which will have a unique solution:
 $$P_\perp\psi=h(P_\parallel\psi,E)$$
for $|E-E_0|$ and $P_\parallel\psi$ sufficiently small. Hence the
system \eqref{eq:perp}-\eqref{eq:parallel} is now equivalent with:
 $$
 P_\parallel F(P_\parallel\psi+h(P_\parallel\psi,E),E)=0.
 $$
Identifying now the range of $P_\parallel$ with $\R^2=\mathbb{C}$
via its basis $\{\psi_{0},i\psi_{02}\},$ the above equation has the
form $G(a,E)=0\in\mathbb{C},\ a\in\mathbb{C},\ E\in\R$ Moreover,
from the implicit function theorem
$h:\mathbb{C}\times\mathbb{R}\mapsto H^2$ is a $C^1$ map which
inherits the equivariance properties of $F,$ see Remark
\ref{symmetries}, because the two projections are also equivariant
with respect to these symmetries. Consequently $G$ is $C^1$ and
$G(\cdot, E)$ is equivariant with respect to rotations in complex
plane and complex conjugation. In particular if $(a,E)$ is a zero
for $G$ so is $(e^{i\theta} a,E),\ 0\leq\theta<2\pi.$ Hence we can
first find solutions with $a$ real, nonzero, and obtain all the
others by rotations. Now:
$$\tilde G(a,E)=G(a,E)/a=(E-E_0)+\sigma\langle\psi_0,\frac{1}{a}(a\psi_0+h(a,E))^{2p+1}\rangle=0,\ a\in\R\backslash{0},\ E\in\R ,$$
can be extended to a $C^1$ function at $a=0$ and
$$\tilde G(0,E_0)=0\qquad \partial_E\tilde G(0,E_0)=1.$$
Via implicit function theorem the above equation has a unique
solution $E=E(a)$ for $a\in\R$ sufficiently close to zero and
$E\in\R$ sufficiently close to $E_0.$ From it we get the $C^1$ curve
of solutions:
$$E=E(a),\qquad \psi_E=a\psi_0+h(a,E(a))\equiv a\psi_0+h(a),$$
where $\psi_E$ is real valued, since $\psi_0$ is and
$\overline{h(a)}=h(\overline{a})=h(a),\ a\in\R.$ All other solutions
are obtained by rotating $a,$ i.e. are of the form
$(e^{i\theta}\psi_E,E).$
\end{proof}

The above Proposition combined with regular perturbation theory allows us to determine the spectrum of the linearization $D_\psi F(\psi_E,E)$ along the nontrivial branch bifurcating from $-E_0.$ \begin{remark}\label{rm:lpme0} On the $C^1$ curve of bound states $E\mapsto \psi_E$ given by {\em Proposition \ref{th:ex}} the $L_\pm$ operators depend continuously on $E,$ hence their eigenvalues depend continuously on $E.$ Consequently $0$ is a simple eigenvalue of $L_-(\psi_E,E)$ at least on a subinterval containing zero: $E\in{\cal I}_{\epsilon}\subseteq {\cal I}$. Moreover, from $L_+-L_- = 2p \sigma |\psi_E|^{2p},$ we obtain that $L_+<L_-$ in the focusing case $\sigma <0.$ Via eigenvalue comparison principle, $L_+$ has its eigenvalues smaller than the
corresponding eigenvalues of $L_-,$ in particular $0$ is not an
eigenvalue of $L_+$ which has the number of strictly negative
eigenvalues increased by one compared to $-\Delta+V+E_0.$  If $-E_0$
is the lowest eigenvalue of $-\Delta+V$ then $0$ is the lowest
eigenvalue of $L_-,$ and $L_+$ has exactly one negative eigenvalue.
The situation is reversed in the defocusing case $\sigma >0.$ More
precisely $L_+>L_-,$ $L_+$ has the number of strictly positive
eigenvalues increased by one compared to $-\Delta+V+E_0,$ and if
$-E_0$ is the lowest eigenvalue of $-\Delta+V$ then the lowest
eigenvalue of $L_+$ is strictly positive. In both cases $L_\pm$ are
relatively compact perturbations of $-\Delta+E$ hence via Weyl's
Theorem their essential spectrum is the $[E,\infty)$ interval.
\end{remark}

Of particular interest are the ground states of \eqref{stationary}:
\begin{definition} \label{def:gs}
Solutions $(\psi,E)$ of \eqref{stationary} which can written as $\psi=e^{i\theta}\tilde\psi$ with $\tilde\psi>0$ and some $\theta\in[0,2\pi)$ will be called ground state.
\end{definition}
Linearization at ground states has the additional spectral property:
\begin{remark}\label{rmk:spgs}Standard properties of second order, self adjoint elliptic operators show that 0 must be the lowest e-value of $L_-(\tilde\psi,E)$ as it corresponds, see {\em Remark \ref{spL}}, to the positive eigenfunction $\tilde\psi.$ Moreover, it must be a simple e-value.
\end{remark}
In fact small amplitude ground states emerge from the lowest e-value of the linearization at zero  $-\Delta+V:$
\begin{corollary}\label{cor:sgs}
The $C^1$ manifold of zeroes of $F$ i.e., solutions of \eqref{stationary}, bifurcating from $(0,E_0)$ where $-E_0$ is the lowest e-value of $-\Delta+V$ is made of ground states.
\end{corollary}
\begin{proof}
Standard properties of the second order, self adjoint, elliptic operator $-\Delta+V$ on $L^2$ implies that its lowest e-value is simple with corresponding real eigenfunction $\psi_0$ which does not change sign. The existence of zeroes of $F$ in a neighborhood of $(0,E_0)$ follows from Proposition \ref{th:ex} and we know that they form a unique, two dimensional, $C^1$ manifold obtained via rotations of a curve of real valued solutions. We first restrict our analysis to this curve given by part (i) in Proposition \ref{th:ex}. Note that for $L_-(0,E_0)=-\Delta+V+E_0$ zero is the lowest e-value. By Remark \ref{rm:lpme0} it will remain the lowest at least on a small neighborhood of $(0,E_0)$ in this curve.  If a second eigenvalue of $L_-$ crosses zero on a point on this curve then $L_-$ will have its lowest e-value at least of multiplicity two in contradiction with  $L_-$ being a second order, elliptic, self adjoint operator. Hence, zero remains the lowest eigenvalue of $L_-(\psi_E,E)$ at all points $(\psi_E,E)$ on this curve, and, consequently, its real valued eigenfunction $\psi_E$ {\em cannot change sign}. Therefore, at each point $(\psi_E,E)$ on this curve we have either $\psi_E>0$ or $e^{i\pi}\psi_E=-\psi_E>0$ Using now part (ii) of Proposition \ref{th:ex} we get that any solution in a neighborhood of $(0,E_0)$ can be rotated into a positive one hence it is a ground state.
\end{proof}

Even though the kernel of $D_\psi F(\psi_E,E)$ is nontrivial along manifolds of solutions, see Remark \ref{spL}, the next two theorems
show that, by moding out rotations, the curves of real valued solutions
can be uniquely continued until $L_+$ does not have a continuous
inverse. Moreover, all solutions are rotations of this curve until the
kernel of $L_-$ is at least two dimensional.

\subsection{Arbitrary branches of solutions and their maximal extension}

\begin{theorem}\label{th:cont}
If $(\psi_{E_1},E_1)$ is a nonzero, real valued, $H^2$ solution of
\eqref{stationary} with the properties that $L_+$ is invertible with
continuous inverse at this solution, then there exist $\epsilon,\
\delta>0$ such that:
 \begin{itemize}
 \item[(i)] for each $E\in (E_1-\delta,E_1+\delta),$ \eqref{stationary}
 has a unique real valued solution $\psi_E$ satisfying
 $\|\psi_E-\psi_{E_1}\|_{H^2}<\epsilon ;$
 \item[(ii)] if in addition $L_-$ at $(\psi_{E_1},E_1)$ has
kernel=span$\{\psi_{E_1}\}$ then any solution $(\psi,E)$ of
\eqref{stationary}
 satisfying $|E-E_1|<\delta$ and $\min_{0\le\theta<2\pi}\|\psi-e^{i\theta}\psi_{E_1}\|_{H^2}<\epsilon$ is of
 the form $(\psi,E)=(e^{i\theta}\psi_E,E)$ where $E$ and $\psi_E$
 are given in (i) and $0\le\theta<2\pi.$
 \end{itemize}
Moreover the function $E\mapsto\psi_E$ is $C^1$ from
$(E_1-\delta,E_1+\delta)$ to $H^2.$
\end{theorem}
\begin{proof} By the symmetries of $F,$ see remark \ref{symmetries}, we have
that $F(\psi,E)$ is real valued if $\psi$ is real valued. We can now
apply the implicit function theorem for
$$F(\psi,E)=0$$
at $(\psi_{E_1},E_1)$ where $F$ is restricted to real valued
functions: $F:H^2(\R^n,\mathbb R)\times\R\mapsto L^2(\R^n,\mathbb R).$ Note that under this restriction:
 $$D_\psi F(\psi_{E_1},E_1)=L_+,$$
see \eqref{deriv}, hence, by hypothesis, it is an isomorphism.
Implicit function theorem now implies (i).

For (ii) we remark that if $(\psi,E)$ is a solution of $F(\psi,E)=0$
then so are $(e^{-i\theta}\psi,E),\ 0\le\theta <2\pi.$ If we choose
$\theta_0$ such that it minimizes
$\|\psi-e^{i\theta}\psi_{E_1}\|_{H^2}^2,$ then
$$\tilde\psi=e^{-i\theta_0}\psi-\psi_{E_1}$$
is orthogonal to $i\psi_{E_1}$ in $H^2$ and satisfies:
$$F(\psi_{E_1}+\tilde\psi,E)=0$$ in particular $\tilde\psi$ solves:
$$P_\perp F(\psi_{E_1}+h,E)=0,\qquad h\in\{i\psi_{E_1}\}^{\perp_{H^2}}.$$
where $P_\perp$ is the projection onto the space orthogonal to $i\psi_{E_1}$ in $L^2.$ But,
via implicit function theorem, the latter has a unique branch of
solutions near $(\psi_{E_1},E_1),$ and, by part (i)
$(h=\psi_E-\psi_{E_1},E)$ is such a branch. Uniqueness of the branch
implies $\tilde\psi=\psi_E-\psi_{E_1}$ or, equivalently,
$\psi=e^{i\theta_0}\psi_E$.
\end{proof}

\begin{remark}\label{rm:lpme1} As in Remark \ref{rm:lpme0} the linearized operators $L_\pm $ along the curve $E\mapsto\psi_E$ given by Theorem \ref{th:cont},see \eqref{deriv}, depend continuously on $E$ and are relatively compact perturbations of $-\Delta+E.$ Hence their discrete spectrum depends continuously on $E$ and the essential spectrum is $[E,\infty).$ In particular, along this curve $0$ is not in the spectrum of $L_+$ and, in case (ii), $0$ remains a simple eigenvalue of $L_-.$
\end{remark}


We can actually continue the curve given by the previous theorem to
a maximal one:

\begin{theorem}\label{th:max} (i) If $\sigma <0,$
then the curve $(\psi_E,E)$ of real valued solutions of
\eqref{stationary} given in Theorem \ref{th:cont} can be uniquely
continued to a maximal interval $(E_-,E_+)$ such that either:
 \begin{itemize}
 \item[(a+)] $E_+=\infty;$
 \item[]or
 \item[(b+)] $E_+ < \infty$ and there exists a sequence $\{E_m\}_{m\in\mathbb N}\subset (E_1,
E_+)$ such that $\lim_{m\rightarrow\infty} E_m = E_+$ and
$L_+(\psi_{E_m},E_m)$ has an eigenvalue $\lambda_m$ satisfying
$\lim_{m\rightarrow\infty} \lambda_m = 0$;
 \item[]and
 \item[(a-)] $E_-=0;$
 \item[]or
 \item[(b-)] $E_- >0$ and there exists a sequence $\{E_m\}_{m\in\mathbb N}\subset (E_-,
E_1)$ such that $\lim_{m\rightarrow\infty} E_m = E_-$ and
$L_+(\psi_{E_m},E_m)$ has an eigenvalue $\lambda_m$ satisfying
$\lim_{m\rightarrow\infty} \lambda_m = 0$;

 \end{itemize}

(ii) If $\sigma >0,$ then the curve $(\psi_E,E)$ of real valued
solutions of \eqref{stationary} given in Theorem \ref{th:cont} can
be uniquely continued to a maximal interval $(E_-,E_+)$ such that
either:
 \begin{itemize}
 \item[(a-)] $E_-=0;$
 \item[]or
 \item[(b-)] $0<E_-,$ and there exists a sequence $\{E_m\}_{m\in\mathbb N}\subset (E_-,
E_1)$ such that $\lim_{m\rightarrow\infty} E_m = E_-$ and
$L_+(\psi_{E_m},E_m)$ has an eigenvalue $\lambda_m$ satisfying
$\lim_{m\rightarrow\infty} \lambda_m = 0$;
\item[]and
 \item[(b+)] $0<E_+,$ and there exists a sequence $\{E_m\}_{m\in\mathbb N}\subset (E_1,
E_+)$ such that $\lim_{m\rightarrow\infty} E_m = E_+$ and
$L_+(\psi_{E_m},E_m)$ has an eigenvalue $\lambda_n$ satisfying
$\lim_{m\rightarrow\infty} \lambda_m = 0$.
 \end{itemize}

(iii) In both cases $\sigma<0,$ or $\sigma>0,$ if $(\psi_E,E)$ is a
real valued solution on such a maximal curve with $\ker L_-={\rm
span}\{\psi_E\}$ then there exist $\epsilon,\ \delta >0$ such that
any solution $(\psi,E_*)$ of \eqref{stationary}
 satisfying $|E_*-E|<\delta$ and $\min_{0\leq\theta<2\pi}\|\psi-e^{i\theta}\psi_{E}\|_{H^2}<\epsilon$ is of
 the form $(\psi,E_*)=(e^{i\theta}\psi_{E_*},E_*)$ where $(\psi_{E_*},E_*)$ is on the same maximal curve as $(\psi_E,E)$ and $0\le\theta<2\pi.$
\end{theorem}

\begin{proof} Consider first the case $\sigma <0.$ Define:
 \begin{eqnarray}
 E_+&=&\sup\{\tilde E : \quad \tilde E>E_1,\ E\mapsto\psi_E\ \mbox{is a}\ C^1\
 \mbox{extension on}\ [E_1,\tilde E)\ \mbox{ of the curve in Theorem \ref{th:cont}},\nonumber\\
   &&\mbox{for which}\ 0\ \mbox{is not in the spectrum of } L_+ \}\nonumber
 \end{eqnarray}
Theorem \ref{th:cont} and the Remark following it guarantees that
the set above is not empty. Assume neither (a+) nor (b+) hold for
$E_+>E_1$ defined above. Then for all $E\in [E_1,E_+),$ the spectrum of
$L_+(\psi_E,E)$ (which is real valued since $L_+$ is self-adjoint)
has no points in the interval $[-d,d]$ for some $0 < d < E_1$.
Indeed, as discussed in Remark \ref{rm:lpme1} the essential spectrum
of $L_+$ at $E$ is $[E,\infty),$ and if no $d > 0$ exists, there
must be a sequence of eigenvalues $\lambda_m$ for $L_+$ at $E_m\in
[E_1,E_+)$ such that $\lim_{m\rightarrow\infty}\lambda_m=0.$ But
$[E_1,E_+]$ is compact since (a+) does not hold, hence there exists
a subsequence $E_{m_k}$ of $E_m$ converging to $E_2\in [E_1,E_+]$.
However, $E_2\not=E_+$ because we assumed that (b+) does not hold,
so $E_2\in [E_1,E_+)$ and by continuous dependence of the
eigenvalues of $L_+$ on $E\in [E_1,E_+)$ we get that $0$ is an
eigenvalue of $L_+$ at $E_2<E_+$ which contradicts the choice of
$E_+.$

Consequently $L_+^{-1}:L^2(\R^n)\mapsto L^2(\R^n)$ is bounded with
uniform bound $K=1/d$ on $[E_1,E_+).$ Moreover, by differentiating
\eqref{stationary} with respect to $E$ we have:
\begin{equation}\label{eq:dpsie}
L_+ \partial_E\psi_E = -\psi_E \quad \Rightarrow \quad
\partial_E\psi_E= - L_+^{-1}\psi_E, \quad E \in (E_0,E_+),
\end{equation}
hence
\begin{equation}\label{estdpsie}
\| \partial_E \psi_E \|_{L^2} \le K \|\psi_E\|_{L^2}, \quad E \in
[E_1,E_+).
\end{equation}
and, by Cauchy-Schwarz inequality:
$$
\frac{d}{dE}\|\psi_E\|^2_{L^2}=2\langle \partial_E \psi_E
,\psi_E\rangle\leq 2K\|\psi_E\|^2_{L^2}, \quad E\in [E_1,E_+).
$$
The latter implies
\begin{equation}
\label{another-bound} \|\psi_E\|^2_{L^2}\leq
\|\psi_{E_1}\|^2_{L^2}e^{2K(E-E_1)}, \quad E\in [E_1,E_+),
\end{equation}
which combined with $E_+<\infty$ and bound \eqref{estdpsie} gives
that both $\partial_E \psi_E$ and $\psi_E$ have uniformly bounded
$L^2$ norms on $[E_1,E_+)$. By the Mean Value Theorem there exists
$\psi_{E_+}(x) \in L^2(\R)$ such that
\begin{equation}\label{eq:l2conv}
\lim_{E\nearrow E_+}\|\psi_E-\psi_{E_+}\|_{L^2}=0.
\end{equation}

We claim that $\psi_{E_+}(\R^n) \in H^1(\R^n)$ is a weak solution of the
stationary equation \eqref{stationary} with $E=E_+.$ Indeed,
consider the energy functional
\begin{equation}\label{def:energy}
{\cal E}(E)=\int_{\R^n}| \nabla\psi_E(x)|^2dx+\int_{\R^n} V(x)|\psi_E(x)|^2dx+\frac{\sigma}{p+1}\int_{\R^n}|\psi_E (x)|^{2p+2}dx.
\end{equation}
Note that because, $\psi_E$ is a weak solution of the stationary
equation \eqref{stationary} for any $E\in [E_1,E_+)$ we have
\begin{equation}\label{eq:denergy}
\frac{d{\cal E}}{dE}=-2E \langle\psi_E,\partial_E \psi_E
\rangle_{L^2}=-E\frac{d\|\psi_E\|_{L^2}^2}{dE}
\end{equation}
which integrated from $E_1$ to $E<E_+$ and using one integration by
parts gives:
\begin{equation}\label{eq:intenergy}
{\cal E}(E)-{\cal
E}(E_1)=E_1\|\psi_{E_1}\|_{L^2}^2-E\|\psi_E\|_{L^2}^2+\int_{E_1}^E\|\psi_E\|_{L^2}^2dE.
\end{equation}
Hence, from the uniform bounds for $E\in [E_1, E_+)$ and
$\|\psi_E\|_{L^2},\ E\in [E_1, E_+)$ see (\ref{another-bound}), the
energy ${\cal E}(E)$, is uniformly bounded on $[E_1,E_+)$. On the
other hand, from the weak formulation of solutions of
\eqref{stationary}, we get
\begin{equation}\label{eq:spstat}
\| \nabla\psi_E \|_{L^2}^2+ \int_{\R^n} V(x)|\psi_E(x)|^2dx +\sigma\|\psi_E\|_{L^{2p+2}}^{2p+2}+ E\|\psi_E\|_{L^2}^2=0
\end{equation}
Subtracting the latter from $(p+1){\cal E}(E)$ we get that there exists
an $M>0$ such that
\begin{equation}\label{eq:h1bounds}
\left|\ p\| \nabla\psi_E \|_{L^2}^2+p \int_{\R^n}V(x)|\psi_E(x)|^2dx\ \right|\le M,\qquad {\rm for\ all}\ E\in [E_1,E_+).
\end{equation}
From H\" older inequality we obtain:
$$
\left|\int_{\R^n}V(x)|\psi_E(x)|^2dx\right|\le\|V\|_{L^\infty}
\|\psi_E\|^2_{L^{2}}.
$$
Using the inequality in \eqref{eq:h1bounds} we deduce that $\|\nabla
\psi_E \|_{L^2}$ has to be uniformly bounded. Consequently, there
exists $M>0$ such that
\begin{equation}\label{eq:h1b}
\|\psi_E\|_{H^1}\le M,\qquad {\rm for\ all}\ E\in [E_1,E_+).
\end{equation}
Because of the embedding of $H^1(\R^n)$ into $L^{q}(\R^n),\ 2\leq
q\leq 2n/(n-2),$ and the Riesz-Thorin interpolation
$$
\| f \|_{L^q} \leq \| f \|^{1+n/q-n/2}_{L^2} \| f \|^{n/2-n/q}
_{L^{2n/(n-2)}}, \quad 2\leq q \leq 2n/(n-2),
$$
bound \eqref{eq:h1b} together with convergence \eqref{eq:l2conv}
imply that as $E\nearrow E_+$ we have:
$$
\left. \begin{array}{l} \psi_E \rightarrow \psi_{E_+},\quad {\rm in}\ L^2(\R^n)\\
\psi_E\rightharpoonup\psi_{E_+},\quad {\rm in}\ H^1(\R^n)
\end{array} \right\} \quad \Rightarrow \quad
\psi_E\rightarrow\psi_{E_+},\quad {\rm in}\ L^q(\R^n),\ 2\leq q<
2n/(n-2).
$$

Now, by passing to the limit in the weak formulation of the
stationary equation \eqref{stationary}, we conclude that
$\psi_{E_+}(x) \in H^1(\R^n)$ is a weak solution and hence also a strong $H^2$ solution, see Remark \ref{rmk:h1toh2}. Moreover, the
linearized operator $L_+$ depends continuously on $E$ on the
interval $[E_1,E_+]$. By the standard perturbation theory, the
discrete spectrum of $L_+$ depends continuously on $E$. Since $0$ is
not in the spectrum of $L_+$ for $E \in [E_1,E_+)$ and we assumed
that (b+) does not hold we deduce that $0$ is not an eigenvalue of
$L_+$ at $E_+.$ Moreover, since the essential spectrum of $L_+$ is
$[E,\infty)$, $0$ is not in the spectrum of $L_+$ at $E = E_+$.
Applying now Theorem \ref{th:cont} at $(\psi_{E_+},\ E_+)$ we can
continue the $C^1$ branch $(\psi_E,E)$ past $E=E_+$ which
contradicts the choice of $E_+.$

The dichotomy (a+) - (b+) is now proven. The above argument holds
for extensions to the left of $E_1$ with the only modification that
$E_-$ is defined as:
 \begin{eqnarray}
 E_-&=&\inf\{\tilde E : \quad \tilde E<E_1,\ E\mapsto\psi_E\ \mbox{is a}\ C^1\
 \mbox{extension on}\ (\tilde E,E_1]\ \mbox{ of the curve in Theorem \ref{th:cont}},\nonumber\\
   &&\mbox{for which}\ 0\ \mbox{is not in the spectrum of } L_+ \}\nonumber
 \end{eqnarray}
the uniform estimates are obtained on an interval of the form $(E_-,
E_1]$ and we study convergence as $E\searrow E_-.$ Obviously
$E_-\geq 0$ because the essential spectrum of $L_+$ at $(\psi_E,E)$
is $[E,\infty)$ and $0$ would be in the spectrum of $L_+$ for $E\leq 0.$

(ii) The proof for $\sigma >0$ is analogous to the one for $\sigma
<0.$ Part (iii) is a direct consequence of part (ii) in Theorem
\ref{th:cont}. The theorem is now completely proven. \end{proof}

Since the spectral assumption in part (iii) of Theorem \ref{th:max} is satisfied for ground states, we have the following corollary:
\begin{corollary}\label{cor:max} (i) If $\sigma <0,$
then the curve $(\psi_E,E)$ of solutions of \eqref{stationary} given
in Theorem \ref{th:max}, which bifurcates from the lowest eigenvalue
$-E_0$ of $-\Delta+V,$ can be uniquely continued to a maximal
interval $(E_0,E_+)$ such that either:
 \begin{itemize}
 \item[(a)] $E_+=\infty;$
 \item[]or
 \item[(b)] $E_+ < \infty$ and there exists a sequence $\{E_m\}_{m\in\N}\subset (E_0,E_+)$ such that $\lim_{m\rightarrow\infty} E_m = E_+$ and
$L_+(\psi_{E_m},E_m)$ has an eigenvalue $\lambda_m$ satisfying
$\lim_{m\rightarrow\infty} \lambda_m = 0$.
 \end{itemize}

(ii) If $\sigma >0,$ then the curve $(\psi_E,E)$ of solutions of
\eqref{stationary} given in Theorem \ref{th:max} which bifurcates
from the lowest eigenvalue $-E_0$ of $-\Delta +V,$ can be uniquely
continued to a maximal interval $(E_-,E_0)$ such that either:
 \begin{itemize}
 \item[(a)] $E_-=0;$
 \item[]or
 \item[(b)] $0<E_- < E_0$ and there exists a sequence $\{E_m\}_{m\in\mathbb N}\subset (E_-,E_0)$ such that $\lim_{m\rightarrow\infty} E_m = E_-$ and $L_+(\psi_{E_m},E_m)$ has an eigenvalue $\lambda_m$ satisfying
$\lim_{m\rightarrow\infty} \lambda_m = 0$.
 \end{itemize}

(iii) In both cases $\sigma<0,$ or $\sigma>0,$ if $(\psi_E,E)$ is a
real valued solution on such a maximal curve then there exists
$\epsilon,\ \delta >0$ such that any solution $(\psi,E_*)$ of
\eqref{stationary}
 satisfying $|E_*-E|<\delta$ and $\min_{0\leq\theta<2\pi}\|\psi-e^{i\theta}\psi_{E}\|_{H^2}<\epsilon$ is of
 the form $(\psi,E_*)=(e^{i\theta}\psi_{E_*},E_*)$ where $(\psi_{E_*},E_*)$ is on the same maximal curve as $(\psi_E,E)$ and $0\le\theta<2\pi.$
\end{corollary}

\begin{proof} Fix $E_1\in {\cal I}$ given by Proposition \ref{th:ex}. Parts (i) and (ii) follow directly from Theorem \ref{th:max} applied to the curve of solutions of \eqref{stationary} starting from $(\psi_{E_1},E_1).$ Part (iii) uses the fact that the lowest eigenvalue of $L_-$ must be simple according to the general theory of self adjoint, second order elliptic operators. Note that $0$ is always eigenvalue of $L_-$ along the maximal curve of solutions, and it is the lowest one on $(E_0,E_1],$ see Remark \ref{rm:lpme0}. If $0$ is not the lowest eigenvalue of $L_-$ at some $E_2,\ E_1<E_2<E_+$ then, from continuous dependence of the discrete spectrum of $L_-$ on $E,$ we deduce that there is $E_3,\ E_1<E_3<E_2,$ such that $0$ is the lowest eigenvalue of $L_-$ at $E_3$ and has multiplicity at least two, contradiction. So $0$ is a simple eigenvalue of $L_-$ on the entire maximal curve, and part (iii) follows from part (iii) in Theorem \ref{th:max}.
\end{proof}

\begin{remark}\label{rm:es1d}
Corollary \ref{cor:max} is also valid for the (excited state)
manifolds bifurcating from the second or higher negative eigenvalues
of $-\Delta+V$ provided we are in one space dimension $n=1.$ The
above proof needs no changes because Sturm-Liouville Theory implies
that all eigenvalues of $-\Delta+V$ and $L_\pm$ are simple. In
particular Theorem \ref{th:max} applies to all negative eigenvalues
of $-\Delta+V$ and the kernel of $L_-$ is always one dimensional.
\end{remark}

Note that along this maximal curves of solutions the number of negative eigenvalues of $L_+$ does not change. Using this information and the following proposition based on comparison principles for the linearized operators,
the estimates for the end points of the maximal curves of solutions
can be improved:

\begin{proposition}\label{pr:endp} (i) If $\sigma <0$ and $(\psi_E,E),\ E>0,$ is a nontrivial, real valued solution of \eqref{stationary}, then, at this solution, $L_+$  has $k\geq 1$ negative eigenvalues counting multiplicity and $E\geq E_{k-1}$, where $-E_{k-1}<0$ is the $k^{th}$ discrete eigenvalue of $-\Delta+V$ counting multiplicity. If the latter does not exist $-E_{k-1}$ is replaced with zero, the edge of the essential spectrum.

(ii) If $\sigma >0$ and $(\psi_E,E),\ E>0,$ is a nontrivial, real valued solution of \eqref{stationary} then the number $k\geq 0$ of negative eigenvalues of $L_+$  counting multiplicity is strictly less than the number of negative eigenvalues of $-\Delta+V$ counting multiplicity and $E\leq E_{k}$, where $-E_{k}<0$ is the $(k+1)^{th}$ discrete eigenvalue of $-\Delta+V.$
\end{proposition}

The proposition immediately implies:

\begin{corollary}\label{cor:endp} (i) If $\sigma <0$ and $(\psi_E,E),\ E\in (E_-,E_+),$ is a maximal curve of real valued solutions of \eqref{stationary}, then, along it, $L_+$  has $k\geq 1$ negative eigenvalues counting multiplicity and $E_-\geq E_{k-1}$, where $-E_{k-1}<0$ is the $k^{th}$ discrete eigenvalue of $-\Delta+V$ counting multiplicity. If the latter does not exist $-E_{k-1}$ is replaced with zero, the edge of the essential spectrum.

(ii) If $\sigma >0$ and $(\psi_E,E),\ E\in (E_-,E_+),$ is a maximal curve of real valued solutions of \eqref{stationary} then the number $k\geq 0$ of negative eigenvalues of $L_+$  counting multiplicity is less than the number of negative eigenvalues of $-\Delta+V$ counting multiplicity and $E_+\leq E_{k}$, where $-E_{k}<0$ is the $(k+1)^{th}$ discrete eigenvalue of $-\Delta+V.$
\end{corollary}

\begin{proof} Indeed, for each  $(\psi_E,E),\ E\in (E_-,E_+),$ on the curve we have by Proposition \ref{pr:endp} and the remark above it that in case (i) $E_{k-1}\leq E,$ while in case (ii) $E_{k}\geq E.$ Consequently $E_{k-1}\leq E_-$ in case (i), while in case (ii) $E_k\geq E_+.$
\end{proof}

We now prove Proposition \ref{pr:endp}

\begin{proof}  In case (i) we first have
$$\langle L_+\psi_E,\psi_E\rangle=\langle L_-\psi_E,\psi_E\rangle+2p\sigma\|\psi_E\|_{L^{2p+2}}^{2p+2}=2p\sigma\|\psi_E\|_{L^{2p+2}}^{2p+2}<0$$
hence the lowest eigenvalue of $L_+:$
$$\lambda_0=\inf_{\|\phi\|_{L^2}=1}\langle L_+\phi,\phi\rangle <0.$$
Then we use use $L_+\leq L_-$ to deduce that zero is at most the $k^{th}$ eigenvalue of $L_-$ counting multiplicity. Let $\phi$ be the linear combination of the first $k$ eigenvectors of $-\Delta+V$ which is orthogonal on each of the first $k-1$ eigenvectors of $L_-.$ Then we have:
$$0\leq \langle L_-\phi,\phi\rangle\leq \langle (-\Delta+V+E)\phi,\phi\rangle\leq (-E_{k-1}+E)\|\phi\|_{L^2}^2$$
which implies $E_{k-1}\leq E$.

For part (ii) we use $L_+\geq L_-$ to deduce that zero is at least the
$(k+1)^{th}$ eigenvalue of $L_- \geq -\Delta+V+E.$ In particular
$-\Delta+V$ has at least $k+1$ negative eigenvalue since $E>0.$
Moreover, if $\phi$ is the linear combination of the first $k+1$
eigenvectors of $L_-$ which is orthogonal on each of the first $k$
eigenvectors of $-\Delta+V,$ then we have:
$$0\geq \langle L_-\phi,\phi\rangle\geq \langle (-\Delta+V+E)\phi,\phi\rangle\geq (-E_{k}+E)\|\phi\|_{L^2}^2$$
which implies $E_{k}\geq E$.
\end{proof}

\begin{remark}\label{rm:endgb} Since along ground state branches zero is the lowest eigenvalue of $L_-,$ the above argument shows that, in the focusing case $\sigma <0,$ their left endpoint $E_-$ always satisfies $E_-\geq E_0,$ where $E_0$ is the lowest eigenvalue of $-\Delta+V,$ while in the defocusing case $\sigma >0$ their right endpoint $E_+$ always satisfies $E_+\leq E_0.$
\end{remark}

\section{Bifurcations at the endpoints of maximal branches}\label{se:bif}

\ \ \ In this section we first employ differential inequalities involving the terms in the energy to show that if $(\psi_E,E)$ is a $C^1$ branch of solutions of \eqref{stationary} then $\|\psi_E\|_{H^2}$ is bounded as long as $E$ does not approach infinity or zero. In other words, such branches have no limit points at the $\{\|\cdot\|_{H^2}=\infty,E_*>0\}$ boundary of the Fredholm domain. This result already has important consequences for the maximal branches of bound states in the defocusing $\sigma>0$ case which we discuss in Theorem \ref{th:defocusing} and Remark \ref{rmk:defocusing}. For the focusing case, we employ concentration compactness and spectral properties of the linearized operator $L_+$ along the manifold of solutions to show, for the first time to our knowledge, that the ground state maximal branches are compact. In other words, as $E$ approaches one of the endpoints, say $E_*,$  $\psi_E$ has at least one limit point $\psi_{E_*}$ in $H^2$ and $(\psi_{E_*},E_*)$ is also a solution of \eqref{stationary}. This is essential in ``continuing" the branch past the bifurcation point $(\psi_{E_*},E_*).$ For example, if the nonlinearity is real analytic ($p$ is an integer) then the limit point is unique and the branch has a unique analytic continuation, see \cite[Corollary 7.5.3]{BT:gbt}. For non-analytic nonlinearities we can obtain similar results when we have certain information on the eigenvalue(s) crossing zero, see the examples in Section \ref{se:examples}.
Eventually, we will combine the results of this section with
the ones analyzing branches with one end point at infinity, see next section,
to obtain information on the global
bifurcation properties of the ground state branches.

\begin{theorem}\label{th:boundedness} Consider a $C^1$ curve $(\psi_E,E),\ E\in (E_-,E_+)$, of solutions of \eqref{stationary} (for example the ones given by Theorem \ref{th:max}). Assume that $0<E_-<E_+<\infty,$ then:
$\|\psi_E\|_{L^2},\ \|\psi_E\|_{H^1}$ and $\|\psi_E\|_{H^2}$ are uniformly bounded on $(E_-,E_+)$.
\end{theorem}

\begin{proof} Since $E\mapsto \psi_E$ is a $C^1$ map from an interval to the normed space $H^2,$ we get that $\|\psi_E\|_{H^2}$ is bounded on compact subintervals of $(E_-,E_+).$ It remains to get bounds as $E$ approaches one of the endpoints $E_-,\ E_+.$

Consider the energy \eqref{def:energy}:
$${\cal E}(E) = \int_{\R^n}|\nabla\psi_E(x)|^2\,\mathrm{d}x + \int_{\R^n}
  V(x)|\psi_E(x)|^2\,\mathrm{d}x + \frac{\sigma}{p+1}\int_{\R^n}|\psi_E(x)|^
  {2p+2}\,\mathrm{d}x,$$
and mass functional:
\begin{equation} \label{def:mass}
 {\cal N}(E)=\|\psi_E\|_{L^2}^2,
\end{equation}
along this curve. Then from \eqref{eq:denergy} we have
\begin{equation}\label{eq:diffen}\frac{\mathrm{d}{\cal E}}{\mathrm{d}E} = -E\frac{\mathrm{d}{\cal N}}
  {\mathrm{d}E}\end{equation}
and from \eqref{eq:spstat} we have
\begin{equation}\label{eq:wpsie}
\|\nabla\psi_E\|_{L^2}^2 + \int_{\R^n}V(x)|\psi_E(x)|^2 \mathrm{d}x +
  \sigma\|\psi_E\|_{L^{2p+2}}^{2p+2} + E\|\psi_E\|_{L^2}^2 = 0,\end{equation}
which can be rewritten as
$${\cal E}(E)+\frac{\sigma p}{p+1}\|\psi_E\|_{L^{2p+2}}^{2p+2}=-E{\cal N}(E).$$
Differentiating the above equation with respect to $E$ gives
\begin{equation}\label{eq:dnormp1}
 \frac{dQ}{dE}(E)=\frac{\mathrm{d}}{\mathrm{d}E}\|\psi_E\|_{L^{2p+2}}^{2p+2} = \frac{p+1}
  {-\sigma p}{\cal N}(E),
\end{equation}
where
$$Q(E)=\|\psi_E\|_{L^{2p+2}}^{2p+2}.$$

We first show that ${\cal N}(E)=\|\psi_E\|_{L^2}^2$ is uniformly bounded in $(E_-,E_+).$

For focussing $\sigma<0$ nonlinearity, it follows from the above differential equation that $Q(E)$ i.e., the $L^{2p+2}$ norm of the solutions, is increasing with $E$ and obviously bounded below by zero. Consequently it is uniformly bounded on intervals $(E_-,\tilde E],\ \tilde E< E_+.$ If we assume:
$$\limsup_{E\rightarrow E_-}{\cal N}(E)=\infty,$$
we have at least on a sequence $E_m\searrow E_-$ as $m\rightarrow\infty$ that
$$\lim_{m\rightarrow\infty}\frac{Q(E_m)}{{\cal N}(E_m)}=0.$$
Then for
$$\Psi_m=\frac{\psi_{E_m}}{\|\psi_{E_m}\|_{L^2}}$$
we have by plugging in \eqref{eq:wpsie} and dividing by $\|\psi_{E_m}\|_{L^2}^2={\cal N}(E_m):$
$$\lim_{m\rightarrow\infty}\left(\|\nabla\Psi_m\|_{L^2}^2+\int_{\R^n}V(x)|\Psi_m(x)|^2dx\right)=-E_-$$
Since $\|\Psi_m\|_{L^{2p+2}}\rightarrow 0,$ it can be shown that $\lim_{m\rightarrow\infty}\int_{\R^n}V(x)|\Psi_m(x)|^2dx=0$ Therefore:
$$\lim_{m\rightarrow\infty}\|\nabla\Psi_m\|_{L^2}^2=-E_-<0$$
contradiction. So ${\cal N}(E)=\|\psi_E\|_{L^2}^2$ is uniformly bounded in a neighborhood of $E_-.$ At the $E_+$ endpoint we first assume
$$\liminf_{E\rightarrow E_+}\frac{Q(E)}{{\cal N}(E)}>0\quad\Leftrightarrow\quad \exists E_- < \tilde E <E_+,\ C>0:\ Q(E)>C{\cal N}(E)\ \forall E\in [\tilde E,E_+).$$
Plugging the inequality in \eqref{eq:dnormp1} and integrating we get:
$$Q(E)\leq Q(\tilde E)e^{\frac{p+1}{-\sigma p C}(E-\tilde E)},\qquad\forall E\in [\tilde E,E_+),$$
which again shows $Q(E)$ is uniformly bounded in a neighborhood of $E_+$ since $[\tilde E,E_+)$ has finite length. As above this will imply ${\cal N}(E)=\|\psi_E\|_{L^2}^2$ is uniformly bounded in a neighborhood of $E_+.$ Now, if we assume
$$\liminf_{E\rightarrow E_+}\frac{Q(E)}{{\cal N}(E)}=0\quad\Leftrightarrow \exists E_m\nearrow E_+:\ \lim_{m\rightarrow\infty}\frac{Q(E_m)}{{\cal N}(E_m)}=0,$$
we have, in the case $\lim_{m\rightarrow\infty}{\cal N}(E_m)=0,$ that $\lim_{m\rightarrow\infty}Q(E_m)=0$ which contradicts the fact that $Q(E)$ is an increasing function of $E$. In the case $\limsup_{m\rightarrow\infty}{\cal N}(E_m)>0$ we get by possibly passing to a subsequence that for $\Psi_m=\frac{\psi_{E_m}}{\|\psi_{E_m}\|_{L^2}}$ we have $\|\Psi_m\|_{L^{2p+2}}\rightarrow 0,$ hence $\lim_{m\rightarrow\infty}\int_{\R^n}V(x)|\Psi_m(x)|^2dx=0$ and by plugging in \eqref{eq:wpsie} and dividing by $\|\psi_{E_m}\|_{L^2}^2={\cal N}(E_m)$ we get the contradiction
 $$\lim_{m\rightarrow\infty}\|\nabla\Psi_m\|_{L^2}^2=-E_+.$$
All in all, for $\sigma<0,$ ${\cal N}(E)=\|\psi_E\|_{L^2}^2$ is uniformly bounded in $(E_-,E_+).$

For defocussing $\sigma>0$ nonlinearity, it follows from the differential equation \eqref{eq:dnormp1} that $Q(E)$ i.e., the $L^{2p+2}$ norm of the solutions, is decreasing with $E$ and obviously bounded below by zero. Consequently it is uniformly bounded on intervals $[\tilde E, E_+),\ \tilde E> E_-.$ Uniform bounds for ${\cal N}(E)=\|\psi_E\|_{L^2}^2$ at this endpoint follow as above. At the $E_-$ endpoint we again split the analysis based on whether $\liminf_{E\rightarrow E_-}\frac{Q(E)}{{\cal N}(E)}$ is zero or not and with obvious modifications of the above argument we get uniform bounds on the $L^2$ norm.

Now, by integrating \eqref{eq:diffen} and doing an integration by parts on the right hand side we get that uniform bounds for ${\cal N}(E),\ E\in (E_-,E_+)$ and the finite length of $(E_-,E_+)$ imply uniform bounds for ${\cal E}(E),\ E\in (E_-,E_+).$ This together with \eqref{eq:wpsie} implies $H^1$ uniform bounds, see \eqref{eq:intenergy}-\eqref{eq:h1b}. Finally, standard regularity theory and $H^1$ uniform bounds implies $H^2$ uniform bounds, see Remark \ref{rmk:h1toh2b} in the Appendix.

The theorem is now completely proven.
 \end{proof}

We first apply the above result to solutions $(\psi_E,E)$ of \eqref{stationary} with $\sigma>0$ (the defocusing case). From Corollary \ref{cor:endp} we get  $0<E<E_0,$ and the above result shows that the maximal branches remain bounded as long as $E$ stays away from zero. Therefore, these branches can either be extended to $E_-=0$ or they undergo a bifurcation at their left end $0<E_-<E_0.$ The statement can be made much more precise for the ground state branch:

\begin{theorem} \label{th:defocusing} If $\sigma>0,$ then the ground states $(\psi_E,E)$ of \eqref{stationary} form a unique $C^1$ manifold, namely the one bifurcating from $(\psi=0,E=E_0),$ see Proposition \ref{th:ex}, where $-E_0$ is the lowest e-value of $-\Delta+V.$ Moreover, this manifold can be uniquely extended to the maximal interval $E\in(0,E_0).$
\end{theorem}

\begin{proof}
Consider a nontrivial ground state solution $(\psi_E,E),\ E>0$ of \eqref{stationary} and, if necessary, rotate it such that $\psi_E>0,$ see Definition \ref{def:gs}. Consequently, zero is the lowest e-value of $L_-(\psi_E,E)$ and it is simple with corresponding eigenvector $\psi_E.$ As for the spectrum of $L_+(\psi_E,E)$ we already know that the essential part is $[E,\infty),$ see Remark \ref{spL}, and, for its lowest e-value $\lambda_0$ with corresponding real valued, $L^2$ normalized eigenfunction $\psi_0$ we can apply the following comparison principle:
$$\lambda_0=\langle\psi_0,L_+(\psi_E,E)\psi_0\rangle=\langle\psi_0,2p\sigma|\psi_E|^{2p} \rangle+\langle\psi_0,L_-(\psi_E,E)\psi_0\rangle\geq 2p\sigma\int_{\R ^n}|\psi_E|^{2p}|\psi_0|^2dx>0,$$
where we used
$$ L_+(\psi_E,E)-L_-(\psi_E,E)=2p\sigma|\psi_E|^{2p}>0, $$
see \eqref{defL+}-\eqref{defL-}, and $\langle\psi_0,L_-(\psi_E,E)\psi_0\rangle\geq 0$ due to the $L^2$ spectrum of $L_-(\psi_E,E)$ being included in $[0,\infty).$
Consequently, zero is not in the spectrum of $L_+(\psi_E,E)$ so this is an invertible operator with continuous inverse. We can then apply Theorems \ref{th:cont}-\ref{th:max} to construct a maximal $C^1$ curve of ground states $E\mapsto\psi_E,\ E\in (E_-,E_+)$ passing through our initial ground state. Note that due to Theorem \ref{th:max} and Corollary \ref{cor:endp} we have
\begin{equation}\label{Epmineq} 0\leq E_-<E_+\leq E_0,\end{equation}
where $-E_0$ is the lowest e-value of $-\Delta+V,$ and that zero remains the lowest e-value of $L_-(\psi_E,E),\ E\in(E_-,E_+).$

We claim that $E_-=0.$  Indeed, if we assume contrary we get, by Theorem \ref{th:max} part (b-), that on any sequence $(\psi_{E_n},E_n)$ on the above curve such that $\lim_{n\rightarrow\infty}E_n=E_-$ we have \begin{equation}\label{L+sing}
\lim_{n\rightarrow\infty}\lambda_n=0\mbox{ where }\lambda_n\mbox{ is the lowest e-value of }L_+(\psi_{E_n},E_n).
 \end{equation}
 By Theorem \ref{th:boundedness} the sequence $(\psi_{E_n})_{n\in\N}$ is bounded in $H^2,$ hence, using the compactness result in the Appendix Corollary \ref{cor:defocuscomp}, there is a subsequence re-denoted by $(\psi_{E_n})_{n\in N}$ and $\psi_{E_-}\in H^2$ such that
$$\lim_{n\rightarrow\infty}\|\psi_{E_n}-\psi_{E_-}\|_{H^2}=0,\qquad F(\psi_{E_-},E_-)=0.$$
Moreover, by continuity, zero will be the lowest e-value of $L_-(\psi_{E_-},E_-)$ with corresponding eigenfunction $\psi_{E_-}$ provided the latter is not the zero function. In this case $\psi_{E_-}$ does not change sign and, by repeating the argument at the beginning of this proof, the lowest e-value of $L_+(\psi_{E_-},E_-)$ must be strictly positive contradicting its continuity with respect to $E,$ see \eqref{L+sing}. Therefore, $\psi_{E_-}=0$ and $L_-(\psi_{E_-},E_-)=-\Delta+V+E_-$ must have zero as its lowest e-value i.e., $E_-=E_0,$ which contradicts \eqref{Epmineq}.

It remains to show that $E_+=E_0$ and
$$\lim_{E\rightarrow E_+}\|\psi_E\|_{H^2}=0,$$ hence the maximal curve through our initial, arbitrary, ground state actually bifurcates from $(0,E_0).$ The fact that $\lim_{n\rightarrow \infty}\|\psi_{E_n}\|_{H^2}=0$ on any sequence $(\psi_{E_n},E_n)$ on the above curve with $\lim_{n\rightarrow\infty}E_n=E_+$ follows exactly as in the previous paragraph by simply changing any minus index to a plus. Since the sequence is arbitrary we get
$$\lim_{E\rightarrow E_+}\|\psi_E\|_{H^2}=0$$
But now, instead of a contradiction we have:
 $L_-(0,E_+)=-\Delta+V+E_+$ must have zero as its lowest e-value i.e., $E_+=E_0.$ This proves that the curve bifurcates from $(0,E_0).$

The theorem is now completely proven.
\end{proof}

\begin{remark}\label{rmk:defocusing} Theorem \ref{th:defocusing} completely describes all ground states of \eqref{stationary} in the defocusing case $\sigma>0.$ It generalizes the result in \cite{jls:gsd} for this particular nonlinearity. We are currently extending Theorem \ref{th:boundedness} to general nonlinearities, see \cite{KS:gbs}, in which case Theorem \ref{th:defocusing} will be valid for general, repelling type nonlinearities. For a discussion on the excited states of this problem, see Section \ref{sse:exdef}.
\end{remark}

In what follows we focus on the attractive nonlinearity $\sigma<0.$ As in the proof of the previous theorem the existence of limit points at the end of the maximal branches plays an important role in our analysis. This compactness type result is not at all trivial, and we prove it for the first time to our knowledge in the following theorem. As opposed to the defocusing case, uniform decay as $|x|\rightarrow\infty$ for solutions $\psi_E(x)$ as $E$ approaches the endpoints of the maximal branches is no longer guaranteed. We use concentration compactness and further assumptions on the behavior of the potential as $|x|\rightarrow\infty$ to prevent splitting of solutions into profiles that drift to infinity.

We will need the following lemma through out this paper. Recall the map $F(U,E):H^2(\R^n)\times \R \mapsto L^2(\R^n)$ defined as
$$ F(U,E) \m=\m (-\Delta+V+E)U+\sigma|U|^{2p}U$$
and restricted to real valued $U$. For any $U,v\in H^2(\R^n)$, we have
$$ F(U+v, E) \m=\m F(U,E)+D_UF(U,E)v+N(U,v)\m,$$
where $D_UF(U,E)$ is the Fr\'echet derivative of $F$ with respect to $U$ at $(U,E)$ and the nonlinear function $N$ has the following properties.

\begin{lemma}\label{lm:nuv}
Consider the nonlinear operator $N(U,v):H^2(\R^n)\times H^2(\R^n)\mapsto L^2(\R^n)$ defined as
\begin{equation}\label{bound_N}
  N(U,v) = \sigma|U+v|^{2p}(U+v) - \sigma|U|^{2p}U - \sigma(2p+1)|U|^{2p}v,
\end{equation}
where $\sigma\in\R$ and $p\in(0,2n/(n-2))$. Then there exists $C>0$ such that
\begin{align}
 \|N(U,v_1)-N(U,v_2)\|_{L^2} &\leq C \|v_1-v_2\|^{2p+1}_{H^2}\quad \mbox{if}\quad p\leq1/2, \label{psmall}\\
 \|N(U,v_1)-N(U,v_2)\|_{L^2} &\leq C\|U\|^{2p-1}_{H^2}\|v_1-v_2\|^2_{H^2} + C \|v_1-v_2\|^{2p+1}_{H^2} \quad \mbox{if}\quad p>1/2. \label{plarge}
\end{align}

Furthermore, for any bounded sequences $\{U_k\}_{k\in\N},\  \{v_k\}_{k\in\N}\subset H^2(\R^n)$ with $\lim_{k\to\infty}\|v_k\|_{L^q}=0$  for some $2< q<2n/(n-2)$ ($2< q<\infty\textrm{ if } n=2$, $2< q \leq\infty\textrm{ if } n=1$), we have
\begin{equation}  \label{Nl2}
 \lim_{k\to\infty}\|N(U_k,v_k)\|_{L^2} \m=\m0 \m.
\end{equation}
\end{lemma}
\begin{proof}
For each $x\in\R^n$, using the mean value theorem, we get
$$ N(U(x),v_1(x))-N(U(x),v_2(x)) = \sigma(2p+1)|U+\tilde v|^{2p}(v_1(x)-v_2(x)) -\sigma(2p+1) |U|^{2p}(v_1(x)-v_2(x))$$
with $|\tilde v(x)|\leq|v_1(x)-v_2(x)|$.

Suppose that $p\leq1/2$. Then we have
$$ |U(x)+\tilde v(x)| \leq (|\tilde v(x)|^{2p}+|U(x)|^{2p})^{1/2p} $$
and therefore for some $C>0$ and all $x\in\R^n$,
$$ |N(U(x),v_1(x))-N(U(x),v_2(x))| \leq C |v_1(x)-v_2(x)|^{2p+1}. $$
Therefore
\begin{equation} \label{interpsmall}
 \|N(U,v_1)-N(U,v_2)\|_{L^2} \leq C \|v_1-v_2\|_{L^{4p+2}}^{2p+1}\m,
\end{equation}
which, using Sobolev embedding and the constraint $p\in(0,2n/(n-2))$, gives \eqref{psmall}. Now let $(U_k)_{k\in\N}$ and $(v_k)_{k\in\N}$ be as in the lemma. Then $\lim_{k\to\infty}\|v_k\|_{L^q}=0$ for some $2< q<2n/(n-4)$ ($2<q<\infty$ if $n<5$) and $\sup_{k\in\N} \|v_k\|_{L^{q'}} <\infty$ for every $2< q'<2n/(n-4)$ ($2<q'<\infty$ if $n<5$). Since $p\in(0,2n/(n-2))$ implies that $2<4p+2<2n/(n-4)$ ($2<4p+2<\infty$ if $n<5$), we can apply the Riesz-Thorin theorem with the appropriate $q'$ to interpolate at the intermediate space $L^{4p+2}$ to get $ \|v_k\|_{L^{4p+2}} \leq \|v_k\|_{L^q}^{\theta} \|v_k\|_{q'}^{1-\theta}$ for some $\theta\in(0,1)$. Hence $\lim_{k\to\infty} \|v_k\|_{L^{4p+2}}=0$. It now follows from \eqref{interpsmall} (with $U=U_k$, $v_1=v_k$ and $v_2=0$) that \eqref{Nl2} holds.

Suppose that $p>1/2$. In this case, using the Taylor's formula with second order remainder for \eqref{bound_N}, we have
$$ N(U(x),v_1(x))-N(U(x),v_2(x)) = \sigma p(2p+1)|U+\tilde v|^{2p-1}(v_1(x)-v_2(x))^2$$
with $|\tilde v(x)|\leq|v_1(x)-v_2(x)|$. In the region $|U(x)|\geq |v_1(x)-v_2(x)|$ we have
$$ |N(U(x),v_1(x))-N(U(x),v_2(x))| \leq 2^{2p-1} |\sigma|p(2p+1)|U(x)|^{2p-1} |v_1(x)-v_2(x)|^2$$
and in the region $|U(x)|<|v_1(x)-v_2(x)|$ we have
$$ |N(U(x),v_1(x))-N(U(x),v_2(x))| \leq 2^{2p-1}|\sigma|p(2p+1) |v_1(x)-v_2(x)|^{2p+1}. $$
From this we get that for all $x\in\R^n$
$$ |N(U(x),v_1(x))-N(U(x),v_2(x))| \leq C \left[|U(x)|^{2p-1} |v_1(x)-v_2(x)|^2 +|v_1(x)-v_2(x)|^{2p+1}\right] $$
and therefore
\begin{equation} \label{rcity7}
 \|N(U,v_1)-N(U,v_2)\|_{L^2} \leq C \left(\int_{\R^n} |U(x)|^{4p-2} |v_1(x) -v_2(x)|^4\dd x\right)^\frac{1}{2} +C \|v_1-v_2\|_{L^{4p+2}}^{2p+1}.
\end{equation}
Since $p\in(0,2n/(n-2))$, $p>1/2$ implies that $n\leq5$. In this case, there exist $\alpha,\alpha'>1$ with $1/\alpha+1/\alpha'=1$ such that $2\leq (4p-2)\alpha <2n/(n-4)$ ($2\leq (4p-2)\alpha <\infty$ if $n<5$), $2< 4\alpha'< 2n/(n-4)$ ($2<4\alpha'<\infty$ if $n<5$). Indeed if $n=1,2,3,4$ let $\alpha=\max\{2,1/(2p-1)\}$ and if $n=5$ let $\max\{5/3,1/(2p-1)\}<\alpha <5/(2p-1)$. (Note that if $n=5$, $p<2/3$.) We then get using H\"older's inequality that
\begin{equation} \label{rcity8}
 \int_{\R^n} |U(x)|^{4p-2} |v_1(x)-v_2(x)|^4\dd x \leq \|U\|^{4p-2}_{L^{(4p-2)\alpha}} \|v_1-v_2\|^4_{L^{4\alpha'}}
\end{equation}
which, along with \eqref{rcity7}, implies using Sobolev embedding that \eqref{plarge} holds i.e.,
$$  \|N(U,v_1)-N(U,v_2)\|_{L^2} \leq C \|U\|_{H^2}^{2p-1}\|v_1-v_2\|^2_{H^2} + C \|v_1-v_2\|_{H^2}^{2p+1}. $$
Let $(U_k)_{k\in\N}$ and $(v_k)_{k\in\N}$ be as in the lemma. Then like in the case of $p\leq1/2$, we can apply the Riesz-Thorin theorem to interpolate at the intermediate spaces $L^{4p+2}$ and $L^{4\alpha'}$ to conclude that $\lim_{k\to\infty} \|v_k\|_{L^{4p+2}}=0$ and $\lim_{k\to\infty} \|v_k\|_{L^{4\alpha'}}=0$. It now follows from \eqref{rcity7} and \eqref{rcity8} (with $U=U_k$, $v_1=v_k$ and $v_2=0$) and our choice of $\alpha$ that \eqref{Nl2} holds.
\end{proof}

For $p\geq1/2$, it follows from the above lemma that for each $U,v_1,v_2 \in H^2(\R^n)$ satisfying $\|U\|_{H^2}\leq L$, $\|v_1\|_{H^2}\leq L$ and $\|v_2\|_{H^2}\leq L$ there exists $C(L)>0$, independent of $U, v_1$ and $v_2$, such that
\begin{equation}\label{est_nl}
  \|N(U,v_1)-N(U,v_2)\|_{L^2}\leq C(L) \|v_1-v_2\|^2_{H^2}.
\end{equation}

Now we present our compactness result.

\begin{theorem}\label{th:comp} Let $\sigma < 0,\ p\geq 1/2$. Assume that the potential $V\in W^{1,\infty}(\R^n)$ satisfies Hypothesis (H2). Furthermore, let the radial derivative $\partial_r V(x)$ of $V$ and a set of its tangential derivatives $\{\partial_{s_i}V(x),i=1,2,\ldots n-1\}$, where $s_1,s_2, \ldots s_{n-1}$ are a set of mutually orthogonal directions which are also orthogonal to the radial direction, satisfy the following hypotheses: There exists a set $\cal{M}$ of measure zero such that
\begin{equation} \label{hypo_smallE1}
 \lim_{x\in \R^n\setminus{\cal M},\ |x|\to\infty} \frac{V^2(x)}{|\partial_r V(x)|} =0
\end{equation}
and there exist constants $\Cscr,\beta,\rho>0$ such that for any $x\in\R^n\setminus\Mscr$ with $|x|>\rho$
\begin{align}
  \Cscr|\partial_r V(x)| &> \frac{1}{|x|^\beta}, \label{hypo_smallE2}\\
  \Cscr|\partial_r V(x)| &> \Big(\sum\limits_{i=1}^{n-1}| \partial_{s_i}V(x)|^2 \Big)^\frac{1}{2}. \label{hypo_smallE3}
\end{align}

Under the above assumptions, consider a set $\Gscr$ of real-valued, positive ground states of \eqref{stationary}, see Definition \ref{def:gs}, bounded in $H^2(\R^n) \times(0,\infty)$ and with $E$ bounded away from zero i.e., there exist $E_-, E_+, M>0,$ such that
$$ \Gscr \subseteq \{(\psi,E)\in H^2\times(0,\infty)\m\big| \|\psi\|_{H^2}<M, E\in[E_-,E_+],  F(\psi,E)=0, \psi>0\}.$$
Then the set $\Gscr$ is relatively compact in $H^2(\R^n)\times(0,\infty)$. Therefore for any sequence $(\psi_m, E_m)_{ m=1}^\infty$ in $\Gscr$, there exists a subsequence $(\psi_{m_k}, E_{m_k} )_{k=1}^\infty$, an $E_*\in[E_-,E_+]$ and a function $\psi_{E_*}\in H^2(\R^n)$ such that
$$ \lim_{k\to\infty} E_{m_k}=E_*, \qquad\lim_{k\rightarrow \infty}\|\psi_{E_{m_k}}-\psi_{E_*}\|_{H^2}=0.$$
Furthermore $(\psi_{E_*},E_*)$ solves \eqref{stationary}.
\end{theorem}

Before presenting a proof of Theorem \ref{th:comp}, we will state and prove a corollary to this theorem.

\begin{corollary} \label{cor:comp}
Under the hypothesis on the potential $V$ in Theorem \ref{th:comp}, consider a maximal $C^1$ curve $(\psi_E,E),\ E\in(E_-,E_+)$ of real-valued, positive ground states of \eqref{stationary}, see Definition \ref{def:gs}, with $0<E_-<E_+<\infty$. Then for any sequence $(E_m)_{m=1}^\infty$ in $(E_-,E_+)$ with $\lim_{m\to\infty} E_m=E_*$, there exists a subsequence $(E_{m_k})_{k=1}^\infty$ and a function $\psi_{E_*}\in H^2(\R^n)$ such that  $(\psi_{E_*},E_*)$ solves \eqref{stationary} and
$$\lim_{k\rightarrow \infty}\|\psi_{E_{m_k}}-\psi_{E_*}\|_{H^2}=0.$$
\end{corollary}
\begin{proof}
From Theorem \ref{th:boundedness} we have $M=\sup_{E\in(E_-,E_+)} \|\psi_{E} \|_{H^2}<\infty$. Hence for any sequence $(E_m)_{m=1}^\infty$ in $(E_-,E_+)$ with $\lim_{m\to\infty} E_m=E_*$, applying  Theorem \ref{th:comp} to the set
$$ \Gscr \m=\m \{(\psi_{E_m},E_m)\m\big| m\in\N\}\subset H^2(\R^n,\C)\times(0,\infty),$$
it follows that there exists a subsequence $(E_{m_k})_{k=1}^\infty$ and a function $\psi_{E_*}\in H^2(\R^n)$ such that  $(\psi_{E_*},E_*)$ solves \eqref{stationary} and $\lim_{k\rightarrow \infty}\|\psi_{E_{m_k}}- \psi_{E_*}\|_{H^2}=0.$
\end{proof}

We now present the proof of Theorem \ref{th:comp}.
\begin{proof}
Consider a sequence $(\psi_{E_m}, E_m)_{m=1}^\infty$ in $\Gscr$. Since $(\psi_{E_m})_{m=1}^\infty$ is bounded in $H^2(\R^n)$ and $(E_m)_{m=1}^\infty$ is bounded in $\R$, we can construct a subsequence, denoted again by $(\psi_m, E_m)_{ m=1}^\infty$, such that $(\psi_{E_m})_{m=1}^\infty$ is weakly convergent in $H^1(\R^n)$ to a $\psi_*\in H^1(\R^n)$ and $\lim_{m\to\infty}E_m=E_*$.
Proposition \ref{lm:lqtoh2} in the Appendix completes the proof of the theorem provided we show that at least on a subsequence of $(\psi_{E_m})_{m\in\N}$ we have
\begin{equation}
  \psi_{E_m} \stackrel{L^q}{\rightarrow} \tilde\psi\ \ \ \mbox{for some} \
  \ \ 2\leq q<\frac{2n}{n-2}\ (2\leq q<\infty\ \ \mbox{if}\ \ n=2,\ 2 \leq q\leq\infty  \ \ \mbox{if}\ \ n=1). \ \ \ \label{lqconvb}
\end{equation}


If $\liminf_{m\to\infty}\|\psi_{E_m}\|_{L^2}=0$, then there is a subsequence of $(\psi_{E_m})_{m\in\N}$ converging to 0 in $L^2$ and hence \eqref{lqconvb} is satisfied with $q=2$. If $\liminf_{m\to\infty}\|\psi_{E_m}\|_{L^2}=a>0$, then consider a subsequence of $(E_m)_{m\in\N}$ such that $\lim_{m\to\infty}\|\psi_{E_m}\|_{L^2}=a$ and
\begin{equation}\label{def:Psi}
  \Psi_m = \frac{\psi_{E_m}}{\|\psi_{E_m}\|_{L^2(\R^n)}}\ ,
\end{equation}
which is a sequence normalized in $L^2$ and bounded in $H^1$.

Next we will apply concentration compactness theory, see for example \cite[Section 1.7]{Caz:nls} and our Appendix \ref{se:app}, to $(\Psi_m)_{m\in\N}$. Consider the concentration function $\rho(\phi,r)$ defined on $L^2(\R^n)\times [0,\infty)$ as
$$ \rho(\phi,r)=\sup_{y\in\R^n}\int_{|x-y|<r} |\phi(x)|^2 \mathrm{d}x$$
and let
$$\mu=\lim_{r\rightarrow\infty}\liminf_{m\rightarrow\infty}\rho(\Psi_m,r).$$
Out of three possible cases in the concentration compactness theory we will show that vanishing a splitting cannot occur and that compactness implies \eqref{lqconvb} for a subsequence of $\Psi_m$, hence for $\psi_{E_m}$.

\noindent
{\bf Vanishing ($\mu=0$):} In this case there is a subsequence $\Psi_
{m_k}$ of $\Psi_m$ convergent to zero in $L^q(\R^n)$ for each $2<q<2n/(n-2)$ ($2<q\leq\infty$ if $n=1$). Since
$$ \psi_{E_{m_k}}=\|\psi_{E_{m_k}}\|_{L^2}\Psi_{m_k}$$
and $\|\psi_{E_{m_k}}\|_{L^2}$ is bounded, see Theorem \ref{th:boundedness},
we get that
$$ \psi_{E_{m_k}} \stackrel{L^q}{\rightarrow} 0. $$
Hence, by Lemma \ref{lm:lqtoh2}, we deduce that $\psi_{E_{m_k}}$ converges to 0 in $H^1(\R^n)$ in contradiction with
$$ \lim_{k\to\infty}\|\psi_{E_{m_k}} \|_{L^2}=a>0.$$
Hence vanishing cannot occur.

\noindent
{\bf Compactness ($\mu=1$):} Since $ \psi_{E_m}=\|\psi_{E_m}\|_{L^2}\Psi_m$ and $\lim_{m\to\infty}\|\psi_{E_m}\|_{L^2}=a>0$, it follows from concentration compactness theory that if $\mu=1$, then there exists a sequence $(y_m)_{m\in \N}\subset \R^n$ and a $\hat\psi_*\in H^1(\R^n)$ such that on a subsequence of $E_m$ (again denoted as $E_m$) we can decompose $\psi_{E_m}$ as
\begin{equation}\label{compact_soln}
  \psi_{E_m}=\hat\psi_m(\cdot-y_m) + \hat v_m,
\end{equation}
such that $\hat\psi_m, \hat v_m \in H^1(\R^n)$ and as $E_m \to E_*$ we have
\begin{itemize}
  \item[(i)] $\hat\psi_m\stackrel{H^1}{\rightharpoonup} \hat\psi_*$\ \ \
  and\ \ \ $\hat\psi_m\stackrel{L^q}{\rightarrow} \hat\psi_*$\ \ \ for \
  \ \ $2\leq q<2n/(n-2)$\ ($2\leq q<\infty\textrm{ if } n=2$,\ \  $2\leq q \leq\infty\textrm{ if } n=1$),
  \item[(ii)] $\hat v_m\stackrel{H^1}{\rightharpoonup} 0$\ \ \ and\ \ \
  $\hat v_m\stackrel{L^q}{\rightarrow}0$\ \ \ for\ \ \ $2\leq q<2n/(n-2)$
  \ ($2\leq q<\infty\textrm{ if } n=2$,\ \  $2\leq q \leq\infty\textrm{ if } n=1$).
\end{itemize}
We will show that $(y_m)_{m\in\N}$ has a subsequence converging to some $y_*\in\R^n$. Therefore we can conclude that $\psi_{E_m}$ converges weakly in $H^1$ to a $\psi_*$ and satisfies \eqref{lqconvb} (at least on a subsequence) with $\psi_*=\hat\psi_*(\cdot-y_*)$. Now Proposition \ref{lm:lqtoh2} finishes the proof of the theorem.

By contradiction assume that
\begin{equation}\label{unb_y}
\lim_{m\to\infty}|y_m| = \infty.
\end{equation}
Using the translation operator $T_y u(x)=u(x-y)$, the equation \eqref{stationary} satisfied by the pair $(\psi_{E_m},E_m)$ can be translated to
$$ \left(-\Delta + T_{-y_m}V + E_m\right) T_{-y_m}\psi_{E_m}+ \sigma T_{-y_m}|\psi_{E_m}|^{2p}\psi_{E_m}=0.$$
Taking the $L^2$ scalar product of the above equation with $\phi \in C_0^
{\infty}(\R^n)$ we get
$$ \left<T_{-y_m}\nabla\psi_{E_m}, \nabla\phi\right> + \left<(T_{-y_m}V + E_m)T_{-y_m}\psi_{E_m}, \phi\right> + \sigma\left<T_{-y_m}|\psi_{E_m}|^{2p} \psi_{E_m}, \phi\right>=0.$$
Substituting for $\psi_{E_m}$ from \eqref{compact_soln} into the last equation
and taking the limit as $E_m \to E_*$, we get using \eqref{unb_y} that
$$ \big<\nabla\hat\psi_*, \nabla\phi\big> + \big< E_*\hat\psi_*, \phi
   \big> + \sigma\big<|\hat\psi_*|^{2p}\hat\psi_*,\phi\big>=0.$$
Therefore $\hat\psi_*$ is a weak solution (and hence a strong solution) to the equation
\begin{equation} \label{uEinf}
 (-\Delta+E)\psi + \sigma|\psi|^{2p}\psi = 0
\end{equation}
at $E=E_*$. Since $\psi_{E_m}>0$ for all $m$, it follows using \eqref{compact_soln} that for any nonnegative $\phi\in C_0^{\infty}(\R^n)$
$$ 0 \leq \liminf_{m\to\infty} \langle \psi_{E_m}, T_{y_m} \phi \rangle \m=\m \langle \hat\psi_*,\phi\rangle + \liminf_{m\to\infty}\langle \hat v_m, T_{y_m}\phi\rangle \m=\m \langle \hat \psi_*,\phi\rangle\m,$$
which implies that $\hat \psi_*$ is a nonnegative function. Since every non-negative solution of \eqref{uEinf} for any $E>0$ is a translation of the unique non-negative radially symmetric solution $U_E$ to \eqref{uEinf} \cite{kwong}, and so $U_E$ satisfies
\begin{equation}\label{schr_no_pot}
  (-\Delta+E)U_E + \sigma|U_E|^{2p} U_E = 0, \qquad U_E>0,
\end{equation}
it follows that $\hat\psi_*=T_{y_*}U_{E_*}$ for some $y_*\in\R^n$. By redefining $\hat v_m$ to be $\hat v_m-T_{y_m+y_*} U_{E_m} +T_{y_m}\hat\psi_m$ and then redefining $y_m$ to be $y_m+y_*$ we obtain the following different decomposition for $\psi_{E_m}$ in place of \eqref{compact_soln}:
\begin{equation}\label{new_compact_soln}
  \psi_{E_m} = T_{y_m}U_m + \hat v_m, \qquad U_m=U_{E_m}.
\end{equation}
Clearly, the redefined $y_m$ satisfies \eqref{unb_y} and the limits in (ii) continue to hold for the redefined $\hat v_m$.

We need the following properties of $U_E$ in \eqref{schr_no_pot} in the sequel. From Proposition \ref{prop:exdecay} in the Appendix we get that for every $\gamma\in(0,E)$, there exists a $C(\gamma,E)>0$ such that
\begin{equation} \label{UEdecayest}
 |U_E(x)|\leq C(\gamma,E)e^{-\sqrt{E-\gamma}|x|}, \qquad
 |\nabla U_E(x)|\leq C(\gamma,E)e^{-\sqrt{E-\gamma}|x|} \qquad \forall x\in\R^n.
\end{equation}
Furthermore,
\begin{equation} \label{diwali}
  U_E(x) \m=\m \left(E\right)^{\frac{1}{2p}} U_1(\sqrt{E}\m x) \FORALL x\in\R^n\m, \FORALL E>0\m,
\end{equation}
where $U_1$ is the unique positive radially symmetric solution to \eqref{schr_no_pot} with $E=1$.

Since $(\psi_{E_m},E_m)$ solves \eqref{stationary}, by substituting
\eqref{new_compact_soln} we obtain
\begin{equation}\label{comp_stat_ym}
  (-\Delta+V+E_m)(T_{y_m}U_m+\hat v_m) + \sigma\left|T_{y_m}U_m+\hat v_m\right|
  ^{2p}(T_{y_m}U_m+\hat v_m)=0.
\end{equation}
Define the linear operator $L_+(U,E):H^2(\R^n)\times\R\mapsto L^2(\R^n)$ as follows:
\begin{equation}\label{def:L+}
  L_+(U,E)[v] = (-\Delta +V+E)v + \sigma(2p+1)|U|^{2p}v.
\end{equation}
Recall the nonlinear operator $N(U,v):H^2(\R^n)\times H^2(\R^n)\mapsto L^2(\R^n)$ from \eqref{bound_N}:
$$ N(U,v) = \sigma|U+v|^{2p}(U+v) - \sigma|U|^{2p}U - \sigma(2p+1)|U|^{2p}v. $$
We can rewrite \eqref{comp_stat_ym} as
$$  L_+(T_{y_m}U_m,E_m)[\hat v_m]+V T_{y_m} U_m+N(T_{y_m} U_m,\hat v_m)=0.$$
From \eqref{comp_stat_ym} we get that
\begin{equation} \label{orlov0}
 \hat v_m=-(-\Delta+E_*)^{-1}\left[V T_{y_m}U_m+(E_m-E_*+V+\sigma(2p+1) |T_{y_m}U_m|^{2p})\hat v_m+N(T_{y_m}U_m,\hat v_m)\right].
\end{equation}
We will now show that
\begin{equation} \label{noconfmotion}
 \lim_{m\to\infty}\|\hat v_m\|_{H^2} \m=\m0.
\end{equation}
Recall the limits $\hat v_m\stackrel{L^q} {\rightarrow}0$ for $2\leq q<2n/(n-2)$ ($2\leq q<\infty\textrm{ if } n=2$, $2\leq q \leq\infty\textrm{ if } n=1$) as $m\to\infty$ from (ii). Note that from \eqref{UEdecayest}, \eqref{diwali} and $\lim_{m\to\infty} E_m=E_*$ we get that $U_m\in H^2(\R^n)\cap L^\infty(\R^n)$ and $\|U_m\|_{H^2}$ and $\|U_m\|_{L^\infty}$ can be bounded by a constant independent of $m$. This, along with the limit $\lim_{m\to\infty}\|\hat v_m\|_{L^2}=0$ in (ii), the limit $\lim_{m\to\infty} E_m=E_*$ and $V\in L^\infty(\R^n)$, implies that
\begin{equation} \label{orlov1}
 \lim_{m\to\infty} \|(E_m-E_*+ V + \sigma(2p+1) |T_{y_m} U_m|^{2p})\hat v_m\|_{L^2} \m=\m0.
\end{equation}
Also, using $\lim_{|x|\to\infty}V(x)=0$, \eqref{unb_y}, \eqref{UEdecayest}, \eqref{diwali} and $\lim_{m\to\infty} E_m=E_*$ we get that
\begin{equation} \label{orlov2}
 \lim_{m\to\infty} \|V T_{y_m} U_m\|_{L^2} \leq \lim_{m\to\infty} \|V T_{y_m}(U_m - U_{E_*})\|_{L^2} + \lim_{m\to\infty} \|V T_{y_m} U_{E_*}\|_{L^2} =0.
\end{equation}
From Lemma \ref{lm:nuv} we get that
\begin{equation} \label{orlov3}
 \lim_{m\to\infty} \|N(T_{y_m}U_m,\hat v_m)\|_{L^2} \m=\m0,
\end{equation}
The limit in \eqref{noconfmotion} follows from \eqref{orlov0} and the limits in \eqref{orlov1}, \eqref{orlov2} and \eqref{orlov3}.

Note that the decomposition
$$\psi_{E_m}=U_{E_m}(\cdot -y_m)+\hat v_m$$
together with $\hat v_m\stackrel{H^2}{\rightarrow}0$ and the assumptions $|y_m|\rightarrow\infty,\ E_m\rightarrow E_*$ imply that our solution bifurcates from $U_{E_*}$ at $|x|=\infty,$ see Figure \ref{fig:drift}.
\begin{figure}[h]
\begin{center}
\includegraphics[scale=0.2]{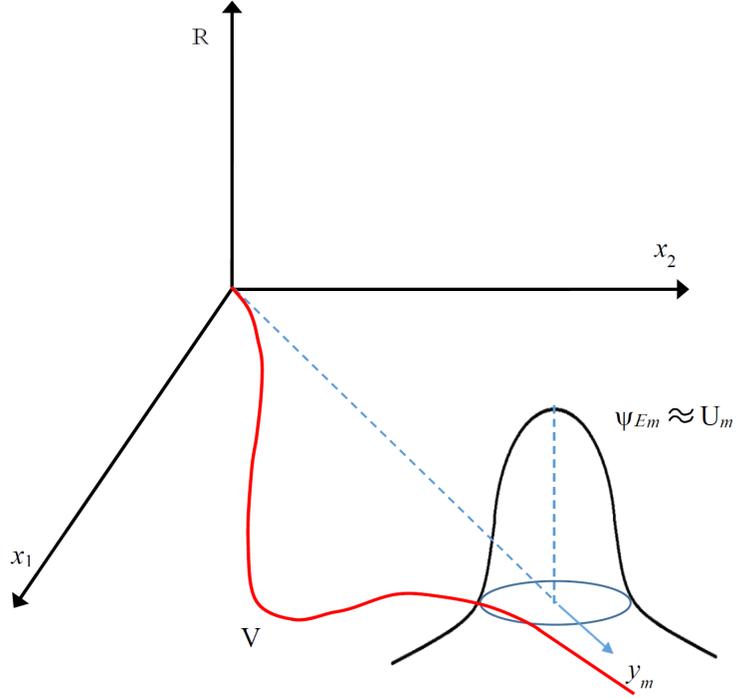}
\end{center}
\caption{A schematic representation showing that a ``drifting" $\psi_{E_m}$ experiences only the effect of the tails of the potential as it bifurcates from $U_{E_*}$ shifted to infinity.}
\label{fig:drift}
\end{figure}
While this bifurcation point is inaccessible in the Banach spaces we are working with, we will use a related Lyapunov-Schmidt like decomposition. In fact $L_+(T_{y_m}U_m,E_m)$ is the linearization at the profiles $T_{y_m}U_m=U_{E_m}(\cdot-y_m).$ Each can be viewed as a Schr\"odinger operator with potential given by a sum of two potentials $V$ and $(2p+1)\sigma|U_m(\cdot-y_m)|^{2p}.$ Since $\lim_{m\rightarrow\infty}|y_m|=\infty$ the distance between the two potentials eventually becomes very large. From the spectral theory for operators with potentials separated by a large distance \cite{MS:slsp}, as $|y_m| \to \infty$ every eigenvalue of $L_+(T_{y_m}U_m,E_m)$ tends to either an eigenvalue of $(-\Delta+V+E_m)$ or an eigenvalue of $-\Delta+E_m+(2p+1)\sigma |U_m|^{2p}$. The eigenvalues of the former operator are bounded away from zero by $E_m-E_0>0$ and those of the latter operator are also bounded away from zero if we remove the projection onto its kernel.

Therefore we now move to a decomposition with respect to the invariant subspaces of $-\Delta+E_m+\sigma(2p+1)T_{y_m}|U_m|^{2p}$. Recall that its kernel is
${\rm span}\left\{T_{y_m}\partial_{x_k}U_m| k=1,2,  \ldots n\right\}$, which are infinitesimal generators of translations of $U_m$ in $\R^n$. Hence the projection onto this subspace can be absorbed into translations of the profile $U_m$. This is established in the following lemma.

\begin{lemma}\label{lm_dec_finite_c}

There exists a constant $\e>0$ and a function $\Sscr$ defined on the set of all $y\in\R^n$, all $E\in(E_0,E_*)$ and all $\psi\in H^1(\R^n)$ which satisfy the condition
$$ \|\psi- T_y U_E\|_{L^2} \m<\m\e\m, $$
such that using $s=\Sscr(\psi,E,y)\in\R^n$, $\psi$ can be decomposed as follows:
\begin{equation}\label{decom}
  \psi \m=\m T_{y+s}U_E+v\m,
\end{equation}
where $v$ satisfies
$$ \langle v, T_{y+s} \partial_{x_k} U_E\rangle \m=0\m \FORALL
   k\in\{1,2,\ldots n\}\m. $$
In addition,
\begin{equation} \label{nsttest_sp}
 \lim_{\|\psi-T_yU_E\|_{L^2}\to0}\Sscr(\psi,E,y) \m=\m 0\m,
\end{equation}
uniformly in $y$ and $E$.
\end{lemma}

\begin{proof}
Define the map $\FFF:L^2(\R^n)\times\R\times\R^n\times\R^n\mapsto \R^n$ as follows:
$$ \FFF(\psi,E,y,s) \m=\m [\m F_1(\psi,E,y,s), F_2(\psi,E,y,s),\ldots F_n(\psi,E,y,s) \m]^\top\m,$$
where
$$ F_k(\psi,E,y,s) \m=\m \bigg\langle\frac{T_{y+s}\partial_{x_k}U_E}
  {\|\partial_{x_k}U_E\|^2_{L^2}}, \psi- T_{y+s} U_E\bigg \rangle_{L^2} \FORALL k\in\{1,2,\ldots n\}\m. $$
We will establish the proposition by solving
\begin{equation}\label{inf_eqdec_sp}
  \FFF(\psi,E,y,s) \m=\m 0
\end{equation}
for $s\in\R^n$, given $\psi$, $E$ and $y$. Note that an $s$ that solves \eqref{inf_eqdec_sp} for a given $\psi$, $y$ and $E$ implies the decomposition in \eqref{decom}. The function $\FFF$ is $C^1$ on $L^2(\R^n)\times\R\times\R^n\times\R^n$. Its $Fr\acute echet$ derivative with respect to $s$, denoted as $D_s\FFF(\psi,E,y,s)$, is a $\R^n\times\R^n$ matrix whose $(k,i)$ entry is
$$ \frac{\partial F_k}{\partial s_i} \m=\m \bigg\langle\frac{T_{y+s}\partial_{x_k} U_E}{\|\partial_{x_k}U_E\|^2_{L^2}}, T_{y+s}\partial_{x_i}U_E\bigg\rangle -\bigg\langle \frac{T_{y+s}\partial^2_{x_k x_i}U_E}{\|\partial_{x_k}U_E\|^2_{L^2}}, \psi-T_{y+s}U_E \bigg\rangle. $$
Here $s_i$ is the $i^{th}$ entry of the vector $s$. Using the above expression and the radial symmetry of $U_E$ it follows that if $\psi$, $E$ and $y$ satisfy the conditions in the lemma for some $\e$, then
$$ \left|\frac{\partial F_k}{\partial s_i}\right| \m\leq\m \delta_{ik}+ \big(\e+ \|T_{y+s}U_E-T_y U_E\|_{L^2}\big) \frac{\|U_E\|_{H^2}}{\|\partial_{x_k}U_E\|^2_{L^2}}. $$
Choose a $\delta>0$. Fix $\e>0$ and $\gamma\in(0,\delta)$ such that the following estimates hold: for any $E\in(E_0,E_*)$, $y\in\R^n$, $\psi$ that satisfies the condition in the lemma and  $s$ that satisfies $|s|<\gamma$,
\begin{equation}
 |\FFF(\psi,E,y,0)|\m<\m\frac{\gamma}{4}\m, \qquad \|[D_s\FFF(T_y U_E,
 E,y,0)]^{-1}\|<2\m, \label{inf_ineq1_sp}
\end{equation}
\begin{equation}
 \|D_s\FFF(T_y U_E,E,y,0)-D_s\FFF(\psi,E, y, s)\| \m<\m \frac{1}{4}\m. \label{inf_ineq2_sp}
\end{equation}
The existence of the constants $\e$ and $\gamma$ follows from the definition of $\FFF$ and the estimates for the elements of $D_s\FFF$. Given $\psi$, $E$ and $y$ that satisfy the conditions in the lemma, we now apply the implicit function theorem to \eqref{inf_eqdec_sp} at its solution $(T_yU_E,E,y,0)$ and conclude, using the estimates in \eqref{inf_ineq1_sp} and \eqref{inf_ineq2_sp}, that there exists a unique $s$ satisfying $|s|<\gamma$ such that $\FFF(\psi,E,y,s)=0$. We define $\Sscr(\psi,E,y)=s$. The estimate in \eqref{nsttest_sp} follows from the fact that we can choose $\delta>0$ to be arbitrarily small, in which case the corresponding $\e$ must be sufficiently small.
\end{proof}

\begin{remark}
In the proof of Lemma \ref{lm_dec_finite_c}, for some fixed $E\in(E_0,E_*)$ and $y$, we could have applied the implicit function theorem at $(T_yU_E,E,y,0)$ to conclude the existence of $\e(E,y)>0$ and $\gamma(E,y)>0$ such that if $\psi$ satisfies the condition in the statement of the lemma with $\e=\e_y$, then there is a unique $s$ with $|s|<\gamma(E,y)$ for which $\FFF(\psi,E,y,s)=0$. But to establish the existence of $\e$ and $\gamma$ independent of $E\in(E_0,E_*)$ and $y$, we have appealed to the estimates in \eqref{inf_ineq1_sp} and \eqref{inf_ineq2_sp} which are uniform in $E$ and $y$. These estimates can be used (in some form) in the proof of the implicit function theorem to construct closed balls for applying the contraction mapping principle.
\end{remark}

Applying Lemma \ref{lm_dec_finite_c} to \eqref{new_compact_soln} we obtain the following decomposition for $\psi_{E_m}$:
$$  \psi_{E_m} = T_{y_m+s_m}U_m + v_m = T_{z_m} U_m+v_m=U_m(\cdot-z_m)+v_m, $$
where $z_m=y_m+s_m$ and $v_m$ satisfies
$$ \langle v_m, T_{y_m+s_m} \partial_{x_k} U_m\rangle \m=0\m \FORALL
   k\in\{1,2,\ldots n\}\m. $$
Moreover as $E_m \to E_*$ we have $|s_m|\to 0$ and $|z_m|\to \infty$. Furthermore $v_m\stackrel {H^2} {\to} 0$ as $m\to\infty$ because $v_m=\hat v_m+T_{y_m}U_m-T_{y_m+s_m}U_m$ and both $\hat v_m\stackrel {H^2} {\to} 0$ (see \eqref{noconfmotion}) and, due to $s_m\to0$,  $T_{y_m}U_m-T_{y_m+s_m}U_m \stackrel {H^2} {\to} 0$.

Equation \eqref{stationary} can now be rewritten as
\begin{equation}\label{comp_stat}
  (-\Delta+V+E_m)(T_{z_m}U_m+ v_m) + \sigma\left|T_{z_m}U_m+ v_m\right|
  ^{2p}(T_{z_m}U_m+ v_m)=0
\end{equation}
or, equivalently:
\begin{equation}\label{eq:bif}
  L_+(T_{z_m}U_m,E_m)[v_m]+V T_{z_m} U_m+N(T_{z_m} U_m,v_m)=0,
\end{equation}
for which the Lyapunov-Schmidt procedure can be applied as follows.
Define the operator $P^\perp_m$ on $L^2(\R^n)$ as
\begin{equation} \label{ragaspeech1}
 P^\perp_m \phi = \phi - \sum_{k=1}^n \langle \phi,T_{z_m}\partial_{x_k} U_m\rangle \frac{T_{z_m}\partial_{x_k} U_m}{\|T_{z_m}\partial_{x_k} U_m\|^2_{L^2}} \qquad \forall \phi \in L^2(\R^n).
\end{equation}
Then \eqref{eq:bif} is equivalent to the set of equations
\begin{equation}\label{perp_eq}
  P^\perp_m L_+(T_{z_m}U_m,E_m)[v_m] + P^\perp_m V T_{z_m}U_m + P^\perp_m N(T_{z_m}U_m,v_m)=0,
\end{equation}
which is an infinite dimensional equation, and
\begin{equation}\label{parallel_eq}
  \langle \partial_{x_k} T_{z_m}U_m,V T_{z_m} U_m + V v_m + N(T_{z_m}U_m,v_m) \rangle =0, \quad k=1 \ldots n,
\end{equation}
which is a set of finite dimensional equations.

We can write $L_+(T_{z_m}U_m,E_m)=\widetilde L_m+\widetilde W_m$, where
$$ \widetilde L_m=-\Delta+ V+ E_* +\sigma(2p+1)\big|T_{z_m} U_{E_*}\big|^{2p}, \qquad  \widetilde W_m= E_m-E_*+\sigma(2p+1)\Big(\big|T_{z_m} U_m\big|^{2p}-\big|T_{z_m} U_{E_*}\big|^{2p}\Big).$$
Since $\lim_{m\to\infty}E_m=E_*$, we have $\|\widetilde W_m\|_{H^2\mapsto L^2}\to0$ as $m\to\infty$. Clearly, $\widetilde L_m$ is a linear operator with two potentials $V_1=V$ and $V_2=\sigma(2p+1)|U_{E_*}|^{2p}$ separated by a large distance. Let $\widetilde P_m$ be the projection operator in $L^2(\R^n)$ onto the set $\{\textrm{ker}\m(-\Delta+V_1+E_*)\} \cup \{\textrm{ker}\m(-\Delta+V_2+E_*)\}$. Note that $\textrm{ker}\m(-\Delta+V_1+E_*)$ is an empty set since the eigenvalues of $(-\Delta+V_1+E_*)$ are bounded away from zero by $E_*-E_0>0$ and $\textrm{ker}\m(-\Delta+V_2+E_*)=\{\partial_{x_k} U_{E_*}\m\big|\m k=1,2,\ldots n\}$. Hence
\begin{equation} \label{ragaspeech2}
 \widetilde P^\perp_m \phi = \phi - \sum_{k=1}^n \langle \phi,T_{z_m}\partial_{x_k} U_{E_*}\rangle \frac{T_{z_m}\partial_{x_k} U_{E_*}}{\|T_{z_m}\partial_{x_k} U_{E_*}\|^2_{L^2}} \qquad \forall \phi \in L^2(\R^n).
\end{equation}
Let $R_m=1/E_m$. We can apply Proposition \ref{prop:MS} in the Appendix to $\widetilde L_m$,  with $R=R_m$, $L_{R_m}=\widetilde L_m$, $E(R_m)=E_*$, $s_1(R_m)=0$, $s_2(R_m)=z_m$ and $V_1$ and $V_2$ as defined above, and then apply the spectral perturbation theory to $\widetilde L_m+\widetilde W_m$, by regarding $\widetilde W_m$ as a perturbation, to conclude that the operator $\widetilde P_{m}^{\perp}L_+(T_{z_m}U_m,E_m)\widetilde P_{m}^{\perp}:H^2(\R^n)\mapsto L^2(\R^n)$ has a bounded inverse and the norm of the inverse operator can be bounded uniformly in $m$, for large $m$. The same conclusion holds for the operator $ P_{m}^{\perp} L_+(T_{z_m}U_m,E_m) P_{m}^{\perp}$, since \eqref{ragaspeech1}, \eqref{ragaspeech2} and  $\lim_{m\to\infty}E_m=E_*$ imply that $\|P_m^\perp-\widetilde P_m^\perp\|_{L^2\mapsto L^2}\to0$ and $\|P_m^\perp-\widetilde P_m^\perp\|_{H^2\mapsto H^2}\to0$ as $m\to\infty$. Hence for all $m$ large and some $K>0$,
\begin{equation} \label{shankvisit}
 \|[P_m^{\perp}L_+(T_{z_m}U_m, E_m) P_m^{\perp}]^{-1}\|_{L^2\mapsto H^2} \leq K.
\end{equation}

We estimate $v_m$ from \eqref{perp_eq} which can be rewritten as
\begin{equation}\label{perp_eq_cont}
  v_m = [P^\perp_m L_+(T_{z_m}U_m, E_m) P^\perp_m]^{-1}P^\perp_m (VT_{z_m}U_m +
        N(T_{z_m}U_m,v_m)),
\end{equation}
where we have used $P^\perp_m v_m = v_m$. From the estimate on the nonlinear term \eqref{est_nl} and the convergence of $v_m$ to zero in $H^2(\R^n)$, it follows that there exist $L,\,\delta>0$ and $M\in\N$ such that for all $m>M$, \eqref{shankvisit} holds and
$$ \|U_m\|_{H^2}\leq L, \quad \|v_m\|_{H^2}\leq \delta \quad {\rm and }  \quad KC(L)2\delta\leq \frac {1}{2},$$
where $C(L)$ is defined as in \eqref{est_nl}. Then applying the contraction principle to \eqref{perp_eq_cont} it follows that for all $m>M$ there exists an unique solution $v_m$ to \eqref{perp_eq_cont} with the bound
\begin{equation}\label{bound_vk}
  \| v_m \|_{H^2}\leq 2K \| VU_{m}(\cdot-z_m) \|_{L^2}.
\end{equation}

Our assumption \eqref{unb_y}:
$$ \lim_{m\to\infty}|y_m|=\infty \equiv \lim_{m\to\infty}|z_m| = \infty$$
will now be contradicted when we show that no such sequence $z_m$ can solve the remaining, finite-dimensional part \eqref{parallel_eq} of \eqref{eq:bif}.

Let $x_m$ denote the radial direction at $z_m$. For each $m\in\N$ consider
the following finite dimensional equation obtained as a linear combination of
the equations in \eqref{parallel_eq}:
\begin{equation} \label{parallel_contr}
  \langle T_{z_m} \partial_{x_m} U_m, VT_{z_m}U_m + Vv_m + N(T_{z_m}U_m,
  v_m)\rangle=0.
\end{equation}
We claim that, under our hypothesis on the behavior of the potential $V(x)$ as $|x|\to\infty$, the dominant term in \eqref{parallel_contr} as $z_m\to\infty$ is  $\langle T_{z_m} \partial_{x_m} U_m, VT_{z_m}U_m \rangle=0$ and by dividing the right hand side of \eqref{parallel_contr} with it and taking the limit as $m\to\infty$, we get the contradiction that $1=0$.

\begin{figure}[t]
\begin{center}
\includegraphics[scale=0.15]{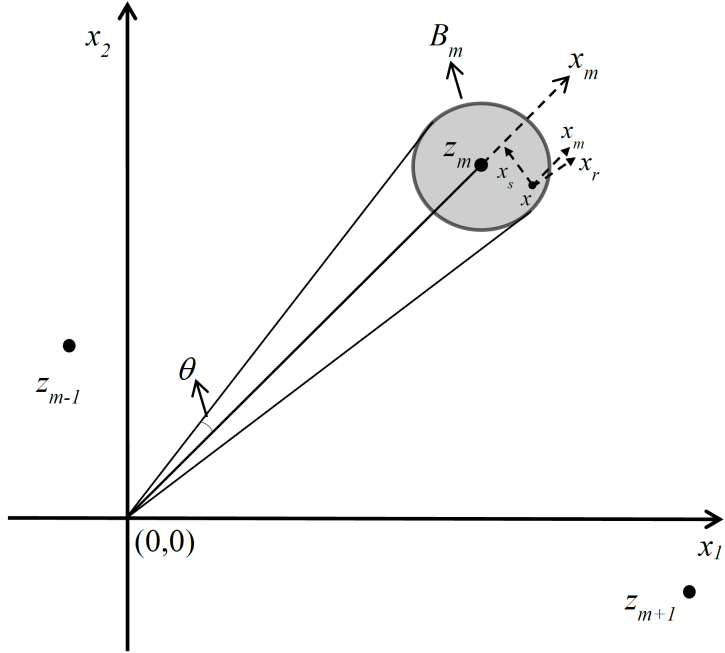}
\end{center}
\caption{This picture depicts the ideas used while computing the integrals
  in the compactness case when the space dimension $n=2$. The radially symmetric function $U_m$ is centered at $z_m$ and $x_m$ denotes the radial direction at $z_m$. As $m\to\infty$, $|z_m|\to\infty$. The $z_m$'s can be in any quadrant (see for example $z_{m-1}$ and $z_{m+1}$). When $m$ is large, since $U_m$ decays exponentially (uniformly in $m$), most of the value of the integrals involving $U_m$ can be obtained by integrating over the ball $B_m$ of radius $z_m\sin\theta$ centered at $z_m$. The ball $B_m$ lies within a cone of
  half-angle $\theta$. At each $x\in B_m$, the radial direction $x_r$ and the tangential direction $x_s$ are at an angle of $\pi/2-\theta_1$ and $\theta_1$, respectively, to the direction $x_m$ with $\theta_1<\theta$. Hence for each  $x\in B_m$ and sufficiently small $\theta$, under the hypothesis in \eqref{hypo_smallE3}, $\partial_{x_m}V(x)$ behaves like $\partial_r V(x)$ on which we have imposed the hypothesis in \eqref{hypo_smallE1} and \eqref{hypo_smallE2}.}
\label{fig:comp}
\end{figure}

Fix $0<\theta<\pi/2\ $ such that $2\Cscr\tan\theta<1$ where $\Cscr$ is the constant in \eqref{hypo_smallE3}. Let $B_m$ be closed ball of radius $|z_m|\sin\theta$ in $\R^n$ centered at $z_m$, i.e. $B_m$ is the closure of $B(z_m,|z_m|\sin\theta),$ see Figure \ref{fig:comp}. Clearly
\begin{align*}
  -2\left\langle \partial_{x_m}U_{m}(x-z_m), V(x)U_{m}(x-z_m)\right\rangle
  &= \left\langle {U^2_{m}}(x-z_m), \partial_{x_m}V(x) \right\rangle\\
  &= \int\limits_{B_m \cup \R^n\setminus B_m} U^2_{m}(x-z_m) \partial_{x_m} V(x) \,  \mathrm{d}x.
\end{align*}
Since $E_m\to E_*$ as $m\to\infty$, there exist $C,\epsilon>0$ and $M\in\N$ such that for all $m>M$ and each $x\in\R^n$ we have $\left|U_m(x)\right|<Ce^{-\sqrt{E_*-\epsilon}|x|}$ and
$\left|\partial_{x_m}U_m\right|<Ce^{-\sqrt{E_*-\epsilon}|x|}$, see \eqref{UEdecayest}. Using this we get for $m>M$
\begin{align}
  \left|\int\limits_{\R^n\setminus B_m} U^2_{m}(x-z_m) \partial_{x_m}V(x) \,\mathrm{d}x
  \right| &\leq C\|\nabla V\|_{L_{\infty}}\int\limits_{\R^n\setminus B_m} e^{-2\sqrt{E_*-\epsilon}|x-z_m|} \,\mathrm{d}x \nonumber\\
  &\leq C\|\nabla V\|_{L_{\infty}}\int\limits_{r>|z_m|\sin\theta} e^{-2\sqrt{E_*-
  \epsilon}\, r} r^{n-1}\,\mathrm{d}r\nonumber\\
  &\leq Ce^{-2\sqrt{E_*-2\epsilon} |z_m| \sin \theta} \label{eq:finv'estimate}.
\end{align}
For each $x\in B_m$, let $\{\theta_i(x)\big| i=1,2,\ldots n\}$ with  $|\theta_i(x)|<\theta$ be the set of angles using which the derivative of the potential along the direction $x_m$ can be written as a linear combination of the derivatives along the radial and tangential directions at $x$, i.e.
$$\partial_{x_m} V(x) = \partial_{r}V(x) \cos\theta_n(x)+ \sum\limits_{i=1}^{n-1}  \partial_{s_i} V(x)\sin\theta_i(x),$$
where $\partial_r V$ and $\partial_{s_i} V$ are the radial and tangential derivatives introduced in the theorem. The function $\partial_{x_m} V(x)$ has the same sign a.e. in $B_m$ for large $m$ since for each $x\in B_m$, using \eqref{hypo_smallE2} and \eqref{hypo_smallE3}, we have
\begin{align}
  &\left|\partial_{x_m} V(x)\right| = \Big|\partial_{r} V(x)\cos\theta_n(x)+
  \sum\limits_{i=1}^{n-1}\partial_{s_i} V(x)\sin\theta_i (x)\Big| \nonumber\\
  \implies&\ \left|\partial_{x_m}V(x)\right|\geq\cos \theta\left|\partial_r
  V(x)\right| - \Big(\sum\limits_{i=1}^{n-1}\left|\partial_{s_i}V(x)\right|
  ^2\Big)^{\frac{1}{2}}\Big(\sum\limits_{i=1}^{n-1}\sin^2\theta_i(x) \Big)^{\frac{1}{2}}   \nonumber\\
  \implies&\ \left|\partial_{x_m} V(x)\right|\geq\cos\theta\left|\partial_r
  V(x)\right|-\Cscr\sin\theta\left|\partial_{r}V(x)\right|\label{est_dvx}\\
  \implies&\ \left|\partial_{x_m} V(x)\right| > \frac{C}{|z_m|^\beta}\ .
  \nonumber
\end{align}
Therefore it follows that for $m$ large
$$\left|\int\limits_{B_m} U^2_m(x-z_m) \partial_{x_m}V(x)\,\mathrm{d}x\right|
  \geq\frac{C}{|z_m|^\beta}\int\limits_{B_m} U^2_m(x-z_m) \,\mathrm{d}x \geq
  \frac {C} {|z_m|^\beta}$$
which along with \eqref{eq:finv'estimate} gives
\begin{equation}\label{lower_bound_poly}
  \lim_{m\to\infty} \frac{\left|\int\limits_{\R^n\setminus B_m}U^2_m(x-z_m)\partial_{x_m}
  V(x)\,\mathrm{d}x\right|}{\left|\int\limits_{B_m}U^2_m(x-z_m)\partial_{x_m}
  V(x)\,\mathrm{d}x\right|} = 0.
\end{equation}
For the remaining terms in \eqref{parallel_contr} we have the following estimates using \eqref{est_nl} and \eqref{bound_vk}:
\begin{align*}
  &\hspace{-10mm}2\left|\left\langle T_{z_m} \partial_{x_m}U_m, V v_m + N(T_{z_m}U_m,v_m) \right\rangle\right|\\
  &\leq\ C\left\|V T_{z_m}\partial_{x_m} U_m \right\|_{L^2}\left\|v_m
  \right\|_{L^2} + C\left\|\partial_{x_m} U_m\right\|_{L^2}\left\|N(T_{z_m} U_m,v_m)\right\|_{L^2} \\
  &\leq\ C\left(\left\|V T_{z_m}\partial_{x_m} U_m \right\|_{L^2}\left\|
  VT_{z_m}U_m\right\|_{L^2}+\left\|V T_{z_m} U_m \right\|^2_{L^2}\right).
\end{align*}
Now as in \eqref{eq:finv'estimate}, using the decay estimates in \eqref{UEdecayest} for $|U_m(x)|$ and $|\partial_{x_i}U_m(x)|$, we obtain
\begin{align}
  \left|\int\limits_{\R^n\setminus B_m} U^2_m(x-z_m)V^2(x)\,\mathrm{d}x \right| &\leq C
  e^{-2\sqrt{E_*-2\epsilon}|z_m|\sin\theta} \label{eq:finvsqr_est},\\
  \left|\int\limits_{\R^n\setminus B_m} (\partial_{x_m} U_m(x-z_m))^2 V^2(x)\,\mathrm{d}x
  \right| &\leq Ce^{-2\sqrt{E_*-2\epsilon}|z_m|\sin\theta}
  \label{eq:finvsqr_der_est}.
\end{align}
Since $\partial_{x_m} V(x)$ does not change sign in $B_m$, using \eqref{hypo_smallE1} along with \eqref{est_dvx} gives us the following
estimates:
\begin{equation}\label{est1}
  \lim_{m \to \infty}\frac{\left|\int\limits_{B_m} U^2_m(x-z_m) V^2(x)\,
  \mathrm{d}x\right|}{\left|\int\limits_{B_m} U^2_m(x-z_m) \partial_{x_m}V(x)
  \,\mathrm{d}x\right|} = \lim_{m \to \infty} \frac{\left|\int\limits_{B_m}
  U^2_m(x-z_m) \partial_{x_m}V(x) \frac{V^2(x)}{\partial_{x_m} V(x)}\,
  \mathrm{d}x\right|}{\left|\int\limits_{B_m} U^2_m(x-z_m)\partial_{x_m}V(x)
  \,\mathrm{d}x\right|} =0
\end{equation}
and
\begin{align}
  &\lim_{m\to\infty}\frac{\left|\int\limits_{B_m}(\partial_{x_m}U_m(x-z_m))^2
  V^2(x)\,\mathrm{d}x\right|}{\left|\int\limits_{B_m}U^2_m(x-z_m)\partial_
  {x_m}V(x)\,\mathrm{d}x\right|} \nonumber\\
  =&\ \lim_{m \to \infty}\frac{\left|\int\limits_{B_m}U^2_m(x-z_m)\partial_
  {x_m}V(x)\frac{(\partial_{x_m}U_m(x-z_m))^2}{U^2_m(x-z_m)}\frac{V^2(x)}
  {\partial_{x_m}V(x)}\,\mathrm{d}x\right|}{\left|\int\limits_{B_m} U^2_m
  (x-z_m) \partial_{x_m} V(x)\,\mathrm{d}x\right|}= 0. \label{est2}
\end{align}
In deriving the last limit we have used the estimate
$$\sup_{x\in\R^n}\left|\frac{\partial_{x_m}U_m(x)}{U_m(x)}\right|<C$$
for some $C>0$ and each $m \in \N$. This follows directly from the relation
$$ U_E(x) \m=\m \left(E\right)^{\frac{1}{2p}} u_\infty(\sqrt{E}\m x) \FORALL x\in\R^n\m, \FORALL E>0\m, $$
where $U_E$ is the unique positive radially symmetric solution to \eqref{schr_no_pot} and $u_\infty=U_1$, and the estimate
$$ \sup_{x\in\R^n}\left|\frac{\partial_{x_m} u_\infty(x)} {u_\infty(x)} \right|<C \m.$$
for some $C>0$. While this estimate appears to be known, see \cite{NY:mps} who state that this estimate is established in \cite{kwong}, we could not locate its proof in the literature. Hence in the next lemma we present a short proof. We will denote the derivatives of functions on $\R$ using {\em primes}.

\begin{lemma}\label{lm:uE_ratio}
Define the function $v:[0,\infty)\to\R$ as follows: $v(|x|)=u_\infty(x)$ for each $x\in\R^n$. Then there exists $C>0$ such that for each $r\in[0,\infty)$
$$\left|\frac{v'(r)}{v(r)}\right|<C.$$
\end{lemma}

\begin{proof}
The radially symmetric function $u_\infty$ satisfies \eqref{schr_no_pot} with $E=1$. It follows from the properties of solutions to \eqref{schr_no_pot} that $v$ is a $C^2$ function on the interval $(0,\infty)$ which satisfies
\begin{equation}\label{uE}
  -v''-\frac{n-1}{r} v' +  v + \sigma|v|^{2p}v=0
\end{equation}
and there exist constants $\gamma,C_1(\gamma),C_2(\gamma)>0$ such that $\sqrt{1+\gamma}-\sqrt{1-\gamma}<1$ and for each $r\in[0,\infty)$
\begin{equation}\label{est_gen}
  v(r)< C_1(\gamma) e^{-\sqrt{1-\gamma}\,r}, \quad v(r)> C_2(\gamma) e^{-\sqrt{1+\gamma}\ r}\quad\textrm{and}\quad |v'(r)|
  <C_1(\gamma) e^{-\sqrt{1-\gamma}\ r}.
\end{equation}
From these estimates we can infer that there exists a $\widetilde L>0$ such that $v'(\widetilde L)\leq 0$ and $v(r)+\sigma|v|^{2p}v(r)>0$ for each $r>\widetilde L$. It now follows from \eqref{uE} that $v$ cannot have a local maxima on the interval $(\widetilde L,\infty)$ which in turn implies that, on the same interval, $v'(r)\leq 0$.

Since $v$ and $v'$ are continuous functions on $[0,\infty)$ and $v>0$, we only need to show that $|v'(r)/v(r)|$ is uniformly bounded for $r$ large. Define the function $f(r)\geq 0$ on $(\widetilde L,\infty)$ as
$$ f(r)=-\frac{v'(r)}{v(r)}.$$
Then it is easy to verify using \eqref{uE} that
\begin{equation}  \label{case 1E}
  f'(r) \m=\m -\frac{n-1}{r}f(r)-1-\sigma|v|^{2p}(r)+f^2(r).
\end{equation}
Assume that
\begin{equation}\label{fbounds}
 f(L) \m>\m 1+\frac{n-1}{L}\qquad{\rm and}\qquad -\frac{n-1}{L} f(L)-1+f^2(L) \m>\m f(L)
\end{equation}
for some $L>\widetilde L$ (if no such $L$ exists, we are done). Then clearly $f'(L)>f(L)$. We claim that $f'(r)>f(r)$ for each $r\in(L,\infty)$. Suppose this is false. Then there exists a $\hat L\in(L,\infty)$ such that $f'(r)>f(r)\geq0$ for $r\in(L,\hat L)$ and $f'(\hat L)=f(\hat L)$. From \eqref{fbounds} we have
$$ \Big(1+\frac{n-1}{L}\Big)f(L) \m<\m -1+f^2(L).$$
Since $\hat L> L$ and $f(\hat L) >f(L)$, the above inequality implies that
$$ \Big(1+\frac{n-1}{\hat L}\Big)f(\hat L) \m<\m -1+f^2(\hat L)$$
which, using \eqref{case 1E} with $r=\hat L$, gives the contradiction $f'(\hat L)>f(\hat L)$. Hence
$$ f'(r) \m>\m f(r) \FORALL r\in(L,\infty) \m.$$
Therefore comparing \eqref{case 1E} with the differential equation
$$ y'(r) \m=\m y(r),\qquad y(L)\m=\m f(L),$$
we get that
\begin{equation}\label{est1_fr}
  f(r) \m\geq\m y(r) \m=\m f(L) e^{r-L} \FORALL r\in(L,\infty).
\end{equation}
From the estimates in \eqref{est_gen} we obtain
\begin{equation}\label{est2_fr}
  f(r) \m=\m -\frac{v'(r)}{v(r)} \m\leq\m \frac{C_1(\gamma)} {C_2(\gamma)} e^{(\sqrt{1+\gamma}-\sqrt{1-\gamma})r} \FORALL r\in(0,\infty).
\end{equation}
The estimates in \eqref{est1_fr} and \eqref{est2_fr} give the following
contradiction
$$1 \m=\m \lim_{r\to\infty}\frac{f(r)}{f(r)} \m\leq\m \lim_{r\to\infty} \frac{C_1(\gamma) e^{(\sqrt{1+\gamma}-\sqrt{1-\gamma})r}}{C_2(\gamma) f(L)e^{r-L}} \m=\m0\m.$$
Therefore no $L>\widetilde L$ exists such that \eqref{fbounds} holds. Hence  $f$ is a bounded function on $(0,\infty)$. This completes the proof of the lemma.
\end{proof}

Using the estimates in \eqref{eq:finv'estimate}, \eqref{lower_bound_poly},
\eqref{eq:finvsqr_est}, \eqref{eq:finvsqr_der_est}, \eqref{est1} and
\eqref{est2} we see that the term
$$\int\limits_{\R^n} U^2_m(x-z_m) \partial_{x_m} V(x)\,\mathrm{d}x$$
cannot be canceled by the remaining terms in \eqref{parallel_contr} when $m$
is sufficiently large. Therefore it follows that there are no solutions to
\eqref{comp_stat} such that the sequence $z_m$, and consequently $y_m$, has
an unbounded subsequence. Hence $y_m$ must be bounded. This completes the proof of the theorem in the compactness case (see discussion above \eqref{unb_y}).

\noindent
{\bf Splitting ($0<\mu<1$):} Since $ \psi_{E_m}=\|\psi_{E_k}\|_{L^2}\Psi_m$ and $\lim_{m\to\infty}\|\psi_{E_m}\|_{L^2}=a>0$, it follows from concentration compactness theory, see Proposition \ref{prop:comp}, that if $0<\mu<1$, then there exists a subsequence of $(E_m)_{m\in N}$, again denoted as $(E_m)_{m\in \N}$, on which we can decompose $\psi_{E_m}$ as follows:
\begin{equation}\label{split_soln}
  \psi_{E_m} = \sum_{j=1}^d T_{y^j_m} \psi_m^j + \psi_m^0  + \hat v_m\m.
\end{equation}
Here $\psi_m^j,\hat v_m\in H^1(\R^n)$ for $j\in\Jscr=\{0,1,2,\ldots d\}$ and there exist non-zero functions $\psi_*^j$ such that as $m \to \infty$
\begin{itemize}
\item[(i)] $\psi_m^j\stackrel{H^1}{\rightharpoonup} \psi^j_*$\ \ and\ \ $\psi_m^j\stackrel{L^q}{\rightarrow}\psi^j_*$\ \ for\ \ $2\leq q<2n/(n-2)$\ ($2\leq q<\infty$ if $n=2$,\ $2 \leq q\leq\infty$ if $n=1$),
\item[(ii)] $|y^j_m|\to\infty$\ \ \ and\ \ \ $|y^j_m-y^k_m|\to
  \infty$\ \ \ for\ \ \ $j, k\in \{1,2,\ldots d\}$ \ \  and\ \ $j \neq k$,
\item[(iii)] $\hat v_m\stackrel{H^1}{\rightharpoonup}0$\ \ \ and\ \ \
  $\hat v_m\stackrel{L^q}{\rightarrow}0$\ \ \ for\ \ \ $2<q<2n/(n-2)$
  \ ($2<q<\infty$ if $n=2$,\ $2<q\leq\infty$ if $n=1$).
\end{itemize}
We will show that splitting cannot occur. We will obtain the contradiction that if the $y_m^j$s satisfy the limits in (ii) above, then $(\psi_{E_m},E_m)$ cannot be a solution of \eqref{stationary}.

Concentration compactness theory does not imply that $\psi_*^0\neq 0$, but we assume it only to simplify our presentation. All the arguments in this section remain valid if $\psi_*^0=0$. Let $\psi=\psi_{E_m}$ and $E=E_m$ in \eqref{stationary} and take the $L^2$ scalar product of the resulting equation with $\phi\in C_0^{\infty}(\R^n)$ to obtain
$$ \left<\nabla\psi_{E_m},\nabla\phi\right> + \left<(V+E_m)\psi_{E_m},\phi
   \right> + \sigma\left<|\psi_{E_m}|^{2p}\psi_{E_m},\phi\right>=0.$$
Substituting for $\psi_{E_m}$ from \eqref{split_soln} and taking the limit
along $E_m \to E_*$ we get,
$$ \left<\nabla\psi_*^0,\nabla\phi\right> + \left<(V+E_*)\psi_*^0,
  \phi\right> + \sigma\left<|\psi_*^0|^{2p}\psi_*^0,\phi\right>=0.$$
Therefore $\psi_*^0$ is a weak solution (and hence a strong solution) of
\eqref{stationary} at $E=E_*$ and hence
$$  (-\Delta+V+E_*)\psi_{E_*}^0 + \sigma|\psi_{E_*}^0|^{2p}\psi_{E_*}^0=0. $$
Similarly for any $j\in\Jscr$, $j\neq0$, taking the $L^2$ scalar product of the equation
$$ (-\Delta+T_{-y^j_m}V +E_m)T_{-y^j_m}\psi_{E_m} + \sigma T_{-y^j_m} |\psi_{E_m}|^{2p}\psi_{E_m} = 0$$
(which is nothing but a translated version of \eqref{stationary} with $E=E_m$ and $\psi=\psi_{E_m}$) with $\phi\in C_0^{\infty}(\R^n)$ and taking the limit as $m\to\infty$ we can conclude that $\psi_*^j$ is a weak solution (and hence a strong solution) to
\begin{equation} \label{fallhorz}
 (-\Delta+E)\psi + \sigma|\psi|^{2p}\psi = 0
\end{equation}
at $E=E_*$. It follows from Proposition \ref{prop:exdecay} in the Appendix that there exists $C>0$ and $\gamma\in(0,E_*)$ such that
\begin{equation} \label{fireworks}
 |\psi_*^j(x)| \m\leq\m C e^{-\sqrt{E_*-\gamma}|x|} \FORALL x\in\R^n, \FORALL j\in\Jscr\m.
\end{equation}
Since $(\psi_{E_m},E_m)$ is a ground state solution for \eqref{stationary}, we have $\psi_{E_m}>0$ for all $m$. It now follows using \eqref{split_soln} and \eqref{fireworks} and the limits in (i), (ii) and (iii) above that for any nonnegative $\phi\in C_0^{\infty} (\R^n)$ and any $j\in\Jscr$, using the notation $y_m^0=0$, we have
$$ 0 \leq \liminf_{m\to\infty} \langle \psi_{E_m}, T_{y^j_m} \phi \rangle \m=\m \langle \psi_*^j,\phi\rangle + \liminf_{m\to\infty}\langle \hat v_m, T_{y_m}\phi\rangle \m=\m \langle \psi_*^j,\phi\rangle\m,$$
which implies that $\psi_*^j$ is a nonnegative function. Since any non-negative solution to \eqref{fallhorz} must be a translation of the radially symmetric nonnegative function $U_E$ introduced in \eqref{schr_no_pot}, it follows that for each $j\in\Jscr$, $j\neq0$, the function $\psi_*^j$ is a translation of $U_{E_*}$ and so $\psi_*^j=T_{y_*^j}U_{E_*}$. By redefining $\hat v_m$ to be $\hat v_m+\psi_m^0-\psi_*^0+\sum_{j=1}^d T_{y_m^j}\psi_m^j-\sum_{j=1}^d T_{y_m^j+y_*^j}U_m$ and then redefining ${y_m^j+y_*^j}$ to be $y_m^j$, we obtain the following decomposition instead of \eqref{split_soln} for $\psi_{E_m}$:
\begin{equation}\label{new_split_soln}
  \psi_{E_m} = \sum_{j=1}^d T_{y^j_m} U_m + \psi_*^0   + \hat v_m.
\end{equation}
Here we have denoted $U_{E_m}$ by $U_m$. Using $\lim_{m\to\infty}E_m=E_*$, \eqref{UEdecayest}, \eqref{diwali} and the fact that $\psi_m^j$s, $y_m^j$s and $\hat v_m$ satisfy the limits in (i), (ii) and (iii) above, it is easy to verify that the redefined $y_m^j$s and $\hat v_m$ also satisfy the limits in (ii) and (iii).

Since $(\psi_{E_m},E_m)$ solves \eqref{stationary}, by substituting
\eqref{new_split_soln} we obtain
\begin{align}
  &(-\Delta+V+E_m)\Big(\sum_{j=1}^d T_{y^j_m} U_m + \psi_*^0 + \hat v_m\Big) \nonumber\\
  & \hspace{10mm}+ \sigma\left|\sum_{j=1}^d T_{y^j_m} U_m + \psi_*^0   + \hat v_m\right|
  ^{2p}\sum_{j=1}^d T_{y^j_m} U_m + \psi_*^0   + \hat v_m)=0. \label{passIIT}
\end{align}
Recall the nonlinear operator $N(U,v)$ from \eqref{bound_N}:
$$ N(U,v) = \sigma|U+v|^{2p}(U+v) - \sigma|U|^{2p}U - \sigma(2p+1)|U|^{2p}v.$$
Using this, \eqref{passIIT} can be rewritten as
$$ \hat v_m=(-\Delta+E_m)^{-1}\Bigg[-V\sum_{j=1}^d T_{y^j_m} U_m- \Big(E_m-E_*+V+\sigma(2p+1) \Big|\sum_{j=1}^d T_{y^j_m}U_m +\psi_*^0\Big|^{2p} \Big)\hat v_m\Bigg. $$
$$-N\Big(\sum_{j=1}^d T_{y^j_m}U_m+\psi_*^0,\hat v_m\Big) -\sigma\Bigg. \Big|\sum_{j=1}^d T_{y^j_m} U_m+\psi_*^0\Big|^{2p} \Big(\sum_{j=1}^d T_{y^j_m}U_m+\psi_*^0\Big)+\sigma\sum_{j=1}^d T_{y^j_m}|U_m|^{2p} U_m+ \sigma|\psi_*^0|^{2p} \psi_*^0  \Bigg]. $$
We will now show that the term in the bracket in the above expression (to which $-(-\Delta+E_*)^{-1}$ is applied) converges to 0 in $L^2(\R^n)$ as $m\to\infty$. This will imply that
\begin{equation} \label{noconffail}
 \lim_{m\to\infty}\|\hat v_m\|_{H^2} \m=\m0.
\end{equation}
Recall the limits $|y^j_m|\to\infty$ and $|y^j_m-y^k_m|\to\infty$ for $j, k\in \{1,2,\ldots d\}$ and $j \neq k$ from (ii) and $\hat v_m\stackrel{L^q} {\rightarrow}0$ for $2< q<2n/(n-2)$ ($2< q<\infty\textrm{ if } n=2$, $2<q \leq\infty\textrm{ if } n=1$) as $m\to\infty$ from (iii). Note that from \eqref{UEdecayest}, \eqref{diwali}, $\lim_{m\to\infty} E_m=E_*$ and \eqref{fireworks} we get that $U_m, \psi_*^0\in L^q(\R^n)$ and $\|U_m\|_{L^q}$ can be bounded by a constant independent of $m$ for each $q\in[1,\infty]$. This, along with the limit for $\hat v_m$ in (iii), the limit $\lim_{m\to\infty} E_m=E_*$ and the fact that $\sup_{m\in\N}\|\hat v_m\|_{L^2}=M_v<\infty$, implies that
\begin{equation} \label{yuri1}
 \lim_{m\to\infty} \Big\|\Big(E_m-E_*+\sigma(2p+1) \Big|\sum_{j=1}^d T_{y^j_m}U_m +\psi_*^0\Big|^{2p} \Big)\hat v_m \Big\|_{L^2} \m=\m0.
\end{equation}
Also, using $\lim_{|x|\to\infty}V(x)=0$, the limit for $y_m^j$ from (ii), \eqref{UEdecayest}, \eqref{diwali} and $\lim_{m\to\infty} E_m=E_*$ we get that
\begin{equation} \label{yuri2}
 \lim_{m\to\infty} \Big\|V \sum_{j=1}^d T_{y^j_m} U_m\Big\|_{L^2} \leq \lim_{m\to\infty} \Big\|V \sum_{j=1}^d T_{y^j_m}(U_m - U_{E_*})\Big\|_{L^2} + \lim_{m\to\infty} \Big\|V \sum_{j=1}^d T^j_{y_m} U_{E_*} \Big\|_{L^2} =0.
\end{equation}
From the exponential decay of $U_m$ (uniform in $m$) and $\psi_*^0$ (see \eqref{UEdecayest}, \eqref{diwali} and \eqref{fireworks}) it follows that
\begin{equation} \label{yuri3}
 \lim_{m\to\infty} \Bigg\|\Big|\sum_{j=1}^d T_{y^j_m} U_m+\psi_*^0\Big|^{2p} \Big(\sum_{j=1}^d T_{y^j_m}U_m+\psi_*^0\Big)-\sum_{j=1}^d T_{y^j_m}|U_m|^{2p} U_m-|\psi_*^0|^{2p}\psi_*^0 \Bigg\|_{L^2} =0.
\end{equation}
Note that from \eqref{UEdecayest}, \eqref{diwali}, $\lim_{m\to\infty} E_m=E_*$ and \eqref{fireworks} we get that $U_m, \psi_*^0\in H^2(\R^n)$ and $\|U_m\|_{H^2}$ is bounded uniformly in $m$. Since $\|\psi_{E_m}\|_{H^2}$ is bounded in uniformly in $m$, it follows that $\|\hat v_m\|_{H^2}$ is bounded uniformly in $m$. It now follows from Lemma \ref{lm:nuv} and (iii) that
\begin{equation} \label{yuri4}
 \lim_{m\to\infty} \Big\|N\Big(\sum_{j=1}^d T_{y^j_m} U_m+\psi_*^0,\hat v_m\Big)\Big\|_{L^2} \m=\m0,
\end{equation}
To show that $\|V \hat v_m\|_{L^2}$ converges to 0, we use the following argument. Fix $\epsilon>0$ and choose $L_\e>0$ sufficiently large such that $\|V\|_{L^\infty (\{|x|\geq L_\e\})}<\epsilon/(2M).$ Here $M=\sup_{m\in\N}\|\hat v_m\|_{L^2}$. Using $V\in L^\infty(\R^n)$ and the limits for $\hat v_m$ in (iii) it follows via H\"older's inequality that $\|\hat v_m V\|_{L^2(\{|x|\leq L_\e\})}<\epsilon/2$ for all $m>m_\e$ and some $m_\e$. Hence for all $m>m_\e$, we have $\|V \hat v_m\|_{L^2}<\e$. Since $\e>0$ is arbitrary we get that
\begin{equation} \label{yuri5}
 \lim_{m\to\infty} \|V \hat v_m\|_{L^2}=0.
\end{equation}
It now follows from \eqref{yuri1}-\eqref{yuri5} that \eqref{noconffail} holds.

Like in the compactness case, the decomposition
$$\psi_{E_m}=\sum_{j=1}^d T_{y^j_m} U_m + \psi_*^0   + \hat v_m,$$
together with $\hat v_m\stackrel{H^2}{\rightarrow}0$ and the assumptions $|y_m^j|\rightarrow\infty,|y_m^j-y_m^k|\rightarrow\infty,\ E_m\rightarrow E_*$ imply that our solution bifurcates from a sum of $\psi_*^0$ and $d$ copies of $U_{E_*}$ located at $|x|=\infty$. While this bifurcation point is inaccessible in the Banach spaces we are working with, we will use a related Lyapunov-Schmidt like decomposition. Recall the linear operator $L_+(U,E)$ from \eqref{def:L+}:
$$ L_+(U,E)[v] = (-\Delta +V+E)v + \sigma(2p+1)|U|^{2p}v. $$
Note that $L_+(\sum_{j=1}^d T_{y^j_m}U_m+\psi_*^0,E_m)$, the linearization of
\eqref{passIIT} at the profiles $\sum_{j=1}^d T_{y^j_m}U_m+\psi_*^0$, can be  viewed as a Schr\"odinger operator with $d+1$ potentials,  $V+\sigma(2p+1)|\psi_*^0|^{2p}$ and $\sigma(2p+1)|T_{y_m^j}U_m|^{2p}$ for $j\in\{1,2,\ldots d\}$. Since $|y_m^j|\rightarrow\infty$ and $|y_m^j-y_m^k| \rightarrow\infty$ as $m\to\infty$, the distance between the potentials eventually becomes very large. From the spectral theory for operators with potentials separated by a large distance \cite{MS:slsp}, as $m \to \infty$ the projection operator associated with the kernel of $L_+(\sum_{j=1}^d T_{y^j_m}U_m+\psi_*^0, E_m)$ converges to the projection operator associated with the set containing the kernel of $(-\Delta+V+E_*+(2p+1)\sigma |\psi^0_{E_*}|^{2p})$ and the kernel of  $(-\Delta+E_m+\sigma |T_{y^j_m}U_m|^{2p})$ for all $j\in\{1,2,\ldots d\}$. Therefore we now move to a decomposition with respect to the kernels of the operators
$(-\Delta+V+E_*+(2p+1)\sigma |\psi^0_{E_*}|^{2p})$ and $(-\Delta+E_m+\sigma |T_{y^j_m}U_m|^{2p})$. Since the kernel of the latter operator is
$ {\rm span}\{T_{y^j_m}\partial_{x_k} U_m \m\big|\m i=1,2,\ldots n\}$, which are infinitesimal generators of translations of $U_m$ in $\R^n$, the projection onto this kernel can be absorbed into translations of the profile $U_m$. This is established in a general setting in Proposition \ref{lm:newdec_gen} in the Appendix. Let
$${\rm ker} \left(-\Delta+V+E_*+(2p+1)\sigma|\psi^0_*|^{2p}\right) = {\rm
  span}\{\phi^i \m\big|\m i=1,2,\ldots n_0\}.$$
Applying Proposition \ref{lm:newdec_gen} (also see Remark \ref{svn_ram}) with $u_i$, $\psi$, $\phi_i$ in the proposition being $U_m$ (independent of $i$), $\psi_*^0$ and $\phi^i$, respectively, it follows from the limits $\hat v_m\stackrel{H^2}{\rightarrow}0$, $|y_m^j|\rightarrow\infty$ and $|y_m^j-y_m^k|\rightarrow\infty$ as $m\to\infty$ that \eqref{new_split_soln} can be rewritten (for large $m$)
as
\begin{equation}\label{split_soln_new}
  \psi_{E_m} = \sum_{j=1}^d T_{y_m^j+s_m^j} U_m + \psi^0_* +
  \sum_{i=1}^{n_0}a^i_m \phi^i + v_m\m,
\end{equation}
where $v_m$ satisfies the following orthogonality relations:
$$ \big\langle v_m, T_{y_m^j+s_m^j} \partial_{x_k} U_m
   \big\rangle \m=\m 0, \FORALL j\in\{1,2,\ldots d\}, \FORALL k\in\{1,2,\ldots n\}\m,$$
$$ \left\langle v_m, \phi^i \right\rangle \m=\m 0 \FORALL i\in\{1,2,\ldots n_0\}\m.$$
Moreover, as $m\to \infty$ we have $|s_m^j|\to 0$ for $j\in\{1,2,\ldots d\}$ and $|a_m^i|\to 0$ for $i\in\{1,2,\ldots n_0\}$. Furthermore, since
$$ v_m=\hat v_m+\sum_{j=1}^d T_{y_m^j} U_m-\sum_{j=1}^d T_{y_m^j+s_m^j}U_m- \sum_{i=1}^{n_0}a^i_m \phi^i $$ and
\eqref{noconffail} holds, $a_m^i\to0$ and $\sum_{j=1}^d T_{y_m^j} U_m-\sum_{j=1}^d T_{y_m^j+s_m^j}U_m \stackrel {H^2} {\to} 0$ (because $s_m^j\to0$ and $E_m\to E_*$), we have
$$ \lim_{m\to\infty} \|v_m\|_{H^2} =0. $$
We assume, with no loss of generality, that the decomposition of $\psi_{E_m}$ in \eqref{split_soln_new} holds for all $m$.

Define
$$ z_m^j \m=\m y_m^j+s_m^j \FORALL j\in\{1,2,\ldots d\}\m, \qquad \widetilde U_m \m=\m \sum_{j=1}^d T_{z_m^j} U_m+\psi^0_*\m, \qquad \phi^a_m \m=\m \sum_{i=1}^{n_0} a^i_m \phi^i \m. $$
Using \eqref{split_soln_new}, we can rewrite \eqref{stationary} as
\begin{equation}\label{splteq}
  (-\Delta +V+E_m)(\widetilde U_m+\phi^a_m+v_m) + \sigma|\widetilde U_m
  +\phi_m^a+v_m|^{2p}(\widetilde U_m+\phi_m^a+v_m)=0,
\end{equation}
which can equivalently be written using the operators $L_+(U,E)$ and $N(U,v)$, defined in \eqref{def:L+} and \eqref{bound_N}, as follows:
\begin{align}
  &\ L_+(\widetilde U_m+\phi^a_m, E_m)[v_m] + N(\widetilde U_m+\phi^a_m, v_m) + V\sum_{j=1}^d T_{z_m^j} U_m + \sigma|\widetilde U_m+\phi^a_m|^{2p}(\widetilde U_m +\phi^a_m) \nonumber\\
  & -\sigma \sum_{j=1}^d T_{z_m^j}|U_m|^{2p}U_m - \sigma|\psi_*^0|^{2p}
  \psi_*^0  + (E_m-E_*)(\psi^0_*+\phi^a_m) - \sigma(2p+1) |\psi^0_*|^{2p} \phi^a_m \m=\m 0. \label{eq:stationary_split}
\end{align}
We will apply the Lyapunov-Schmidt procedure to the above equation. Note that using the exponential decays of $U_m$ and $\psi_*^0$ and the fact that $E_m\to E_*$ we can write the above equation as
\begin{align}
  &\ L_+(\widetilde U_m+\phi^a_m,E_m)[v_m] + N(\widetilde U_m+\phi^a_m,v_m)
  + V\sum_{j=1}^d T_{z_m^j}U_m + O(|a_m|) \nonumber \\
  & + O(E_m-E_*) + \sum_{j=1}^d O\left(e^{-\sqrt{E_*-\gamma}|z_m^j|}\right) +
  \sum_{i\neq j,\ i,j=1}^d O\left(e^{-\sqrt{E_*-\gamma}|z_m^i-z_m^j|}\right)=0,
  \label{eq:spord}
\end{align}
for a constant $\gamma>0$. Define the operator $P^\perp_m$ on $L^2(\R^n)$ to be the orthogonal projection onto its subspace
$$ \left\{\phi^i, T_{z_m^j}\partial_{x_k}U_m \m\big|\m i=1,2,\ldots n_0\m,\ k=1,2, \ldots n\m,\ j=1,2,\ldots d\right\}^{\perp}.$$
For simplicity of notation, let us write \eqref{eq:stationary_split} as
$$F(v_m, z_m, a_m)=0\m,$$
where $z_m$ is the set $\{z_m^1,z_m^2,\ldots z_m^d\}$ and $a_m$ is the set $\{a_m^1,a_m^2,\ldots a_m^{n_0}\}$. This equation is equivalent to the following set of equations:
\begin{equation}\label{eq:infcom}
  P^\perp_m F(v_m, z_m, a_m)=0,
\end{equation}
which is an infinite dimensional equation to be solved for $v_m$ as a
function of $z_m$ and $a_m$, and
\begin{align}
   \langle F(v_m,z_m,a_m),\phi^i\rangle &= 0, \FORALL i\in\{1,2,\ldots n_0\},
   \nonumber\\
   \langle F(v_m,z_m,a_m), T_{z_m^j}\partial_{x_k} U_m \rangle
   &= 0, \FORALL j\in\{1,2,\ldots d\}, \FORALL k\in\{1,2,\ldots n\}, \label{eq:fincom}
\end{align}
which is a set of finite dimensional equations to be solved for $z_m$ and
$a_m$ using the solution $v_m$ of \eqref{eq:infcom}.

We can write $L_+(\widetilde U_m+\phi^a_m, E_m)=\widetilde L_m+\widetilde W_m$, where
$$ \widetilde L_m = -\Delta+ V+ E_*+\sigma(2p+1)\Big(\sum_{j=1}^d \big|T_{z_m^j} U_{E_*}\big|^{2p} +|\psi_*^0|^{2p}\Big),$$
$$ \widetilde W_m = E_m-E_*+\sigma(2p+1)\Big(\Big|\sum_{j=1}^d T_{z_m^j} U_m+\psi_*^0+\phi^a_m\Big|^{2p}-\sum_{j=1}^d\big|T_{z_m^j} U_{E_*}\big|^{2p} -|\psi_*^0|^{2p}\Big).$$
Since $E_m\to E_*$, $a_m^i\to0$, $|y_m^j|\rightarrow\infty$ and $|y_m^j-y_m^k| \rightarrow\infty$ as $m\to\infty$, we get $\|\widetilde W_m\|_{H^2\mapsto L^2}\to0$ as $m\to\infty$. Clearly, $\widetilde L_m$ is a Schr\"odinger operator with $d+1$ potentials,
$$ V_j=\sigma(2p+1)|T_{y_m^j}U_{E_*}|^{2p} \FORALL j\in\{1,2,\ldots d\}, \qquad V_{d+1}=V+\sigma(2p+1)|\psi_*^0|^{2p},$$
which are separated by a large distance. Let $\widetilde P_m^\perp$ be the projection operator in $L^2(\R^n)$ onto the set
$$ \left\{\phi^i, T_{z_m^j}\partial_{x_k}U_{E_*} \m\big|\m i=1,2,\ldots n_0\m,\ k=1,2, \ldots n\m,\ j=1,2,\ldots d\right\}^{\perp}.$$
This is the set containing the kernel of the operators $(-\Delta+V+\sigma(2p+1) |\psi_*^0|^{2p}+E_*)$ and $(-\Delta+E_*+\sigma(2p+1) |T_{y_m^j}U_{E_*}|^{2p})$. Let $R_m=1/E_m$. We can apply Proposition \ref{prop:MS} in the Appendix to $\widetilde L_m$,  with $R=R_m$, $L_{R_m}=\widetilde L_m$, $E(R_m)=E_*$, $s_j(R_m)=z_m^j$ for $j\in\{1,2,\ldots d\}$, $s_{d+1}(R_m)=0$ and the potential $V_1, V_2, \ldots V_{d+1}$ as defined above, and then apply the spectral perturbation theory to $\widetilde L_m+\widetilde W_m$, by regarding $\widetilde W_m$ as a perturbation, to conclude that the operator $\widetilde P_{m}^{\perp} L_+(\widetilde U_m+\phi^a_m, E_m) \widetilde P_{m}^{\perp}:H^2(\R^n)\mapsto L^2(\R^n)$ has a bounded inverse and the norm of the inverse operator can be bounded uniformly in $m$, for large $m$. The same conclusion holds for the operator $ P_{m}^{\perp} L_+(\widetilde U_m+\phi^a_m, E_m) P_{m}^{\perp}$ since the definitions of $P_m^\perp$ and $\widetilde P_m^\perp$ and the limit  $\lim_{m\to\infty}E_m=E_*$ imply that $\|P_m^\perp-\widetilde P_m^\perp\|_{L^2\mapsto L^2}\to0$ and $\|P_m^\perp-\widetilde P_m^\perp\|_{H^2\mapsto H^2}\to0$ as $m\to\infty$. Hence for all $m$ large and some $K>0$,
$$ \|[P_m^{\perp}L_+(\widetilde U_m+\phi^a_m, E_m) P_m^{\perp}]^{-1}\|_{L^2\mapsto H^2} \leq K. $$
Hence for large $m$, using \eqref{eq:spord} and the relation $P^\perp_m v_m=v_m$, we can rewrite the infinite dimensional equation
\eqref{eq:infcom} as
\begin{align}
  v_m &= -[\m P_m^\perp L_+(\widetilde U_m+\phi^a_m,E_m) P_m^\perp\m]^{-1} \Big(N(\widetilde U_m+\phi^a_m, v_m) + V\sum_{j=1}^d T_{z_m^j} U_m\Big.\nonumber\\
  &\qquad\qquad\Big. + (E_m-E_*)\psi_*^0 + O(|a_m|) + \sum_{\stackrel{j,k
  =1}{j\neq k}}^d O\left(e^{-\sqrt{E_*-\gamma}|z_m^j-z_m^k|}\right)\Big).
  \label{vm_com}
\end{align}
Like in the case of compactness, existence of an unique solution $v_m$ to this equation can established by applying contraction mapping principle using \eqref{est_nl} and the smallness of $\|v_m\|_{H^2}$.

While analyzing the finite dimensional equations \eqref{eq:fincom} it would
be useful to have estimates for the decay of $\|v_m\|_{H^2}$, explicitly in terms of the $z_m^j$s, similar to \eqref{bound_vk} in the case of compactness.  But due to the presence of the non-drifting terms $\psi_*^0$ and $\phi_m^a$ in \eqref{split_soln_new}, that remain localized near the origin in $\R^n$, we will not be able to get such estimates. Indeed, by adopting the approach in the compactness case, we can derive from \eqref{vm_com} that for some $C>0$
$$ \|v_m\|_{H^2} \m\leq\m C \sum_{j=1}^d\| VT_{z_m^j}U_m\|_{L^2}+ C(E_m-E_*)\|\psi_*^0\|_{L^2} + O(|a_m|). $$
But without knowing the rate of convergence of the terms $(E_m-E_*)\|\psi_*^0\|_{L^2} + O(|a_m|)$ to 0, this estimate is not useful. Hence we will express $v_m$ as
$$ v_m \m=\m P_m^\perp v_m^0+v_m^1 \m.$$
Here $v_m^0$ will be associated with the non-drifting terms in \eqref{split_soln_new} while $v_m^1$ will be associated with the drifting terms. For $v_m^0$ we will derive certain exponential in space decay estimates (see \eqref{tumbboost}) so that its scalar product with the drifting terms $T_{z_m^j}\partial_{x_k}U_m$ in \eqref{eq:fincom}, which is a quantity of interest, decays at a rate which depends explicitly on the $z_m^j$s.
For $\|v_m^1\|_{H^2}$ we will derive an decay estimate explicitly in terms of the $z_m^j$s, like \eqref{bound_vk} in the compactness case (see \eqref{newnotconf}). Using these estimates we will analyze the equations
\eqref{eq:fincom} to obtain a contradiction implying that splitting cannot occur.

We first derive the estimate \eqref{tumbboost}. Let  $P_0$ be the projection operator in $L^2(\R^n)$ onto its subspace
$$ S \m=\m {\rm ker} \left(-\Delta+V+E_*+(2p+1)\sigma|\psi^0_*|^{2p}\right) \m=\m \{\phi^1,\phi^2,\ldots \phi^{n_0}\} \m.$$
Consider the equation
\begin{equation}\label{eq:v0}
  P_0^\perp \left[(-\Delta+V+E_m)(\psi^0_* + \phi^a_m+v_m^0)+\sigma
  |\psi^0_*+\phi^a_m+v_m^0|^{2p}(\psi^0_*+\phi^a_m+v_m^0)\right]=0,
\end{equation}
We wish to find a $v_m^0\in H^2(\R^n)$ that solves this equation. Clearly the operator
$$  F_L(\cdot)=P_0^\perp\Big[(-\Delta +V+E_*)(\cdot) + \sigma
    (2p+1)|\psi_*^0|^{2p}(\cdot)\Big], $$
which maps $S^\perp\cap H^2(\R^n)$ to $S^\perp \cap L^2(\R^n)$, has a bounded inverse.
Hence we can write \eqref{eq:v0} as
\begin{align*}
  v_m^0 =& F_L^{-1}P_0^\perp \Big[(E_*-E_m)(\psi^0_* + \phi^a_m+v_m^0) +\sigma(2p+1)(|\psi_0^*|^{2p}-|\psi_0^*+\phi^a_m|^{2p})v_m^0 \\ &\hspace{25mm}+\sigma(2p+1)|\psi_*^0|^{2p}\phi^a_m +\sigma(|\psi_*^0|^{2p}\psi_*^0-|\psi_*^0+\phi_m^a|^{2p}(\psi_*^0+\phi_m^a)) \Big].
\end{align*}
By using the contraction mapping principle it follows from the above equation that for $m$ sufficiently large (so that $|E_m-E_*|$ and $|a_m|$ are sufficiently small) there exists a unique $v_m^0\in S^\perp\cap H^2(\R^n)$ that solves this equation with $\|v_m^0\|_{H^2}=O(|E_m-E_*|) +O(|a_m|)$. For $n\leq3$, this implies directly that $v_m^0 \stackrel{L^\infty}{\to}0$ as $m\to\infty$. For $n\geq4$, the same conclusion can be arrived at by first using regularity to show that $v_m^0$ converges to 0 in $W^{2,q}$ for all $q\geq 2$, see the Appendix \ref{se:app}. We will now show that $v_m^0$ decays exponentially.

From \eqref{eq:v0} we get that there exist $\beta_m^i\in\R$, for each $i\in\{1,2,\ldots n_0\}$, such that
\begin{equation} \label{raman_model}
 \Delta(\psi_*^0+\phi^a_m+v_m^0) = (V + E_m + \sigma|\psi_*^0+\phi^a_m+ v_m^0|^{2p})(\psi_*^0+\phi^a_m+v_m^0) - \sum^{n_0}_{i=1} \beta_m^i\phi^i.
\end{equation}
Note that for each $i\in\{1,2,\ldots n_0\}$ we have
\begin{align*}
  \beta_m^i\langle\phi^i,\phi^i\rangle &= \left<\phi^i,(-\Delta+V+E_m)(\psi^0_* +\phi^a_m+v_m^0) + \sigma|\psi^0_*+\phi^a_m+v_m^0|^{2p}(\psi^0_*
  +\phi^a_m+v_m^0)\right>\\
  &=\left<\phi^i,(L_*+E_m-E_*)(\phi^a_m+v_m^0)+(E_m-E_*)\psi^0_*+N(\psi^0_*, \phi^a_m+v_m^0)\right>\\ &=\left<\phi^i,(L_*+E_m-E_*)(\phi^a_m+v_m^0)+(E_m-E_*)\psi^0_*\right> +
  O(\|\phi^a_m+v_m^0\|_{H^2}^2),
\end{align*}
where we have used \eqref{est_nl} to get the last equality and $L_*[v]=(-\Delta +V+E_*)v+(2p+1)\sigma |\psi_{E_*}^0|^{2p}v$. It now follows using $L_*\phi^i=0$ and the estimate for $\|v_m^0\|_{H^2}$ discussed earlier that
$$ |\beta_m^i| = O(|E_m-E_*|) + O(|a_m|) .$$
Define the function $W_m=\psi_*^0+\phi^a_m+v_m^0$. From the estimates on $\psi_*^0$, $\phi^a_m$ and $v_m^0$ we get that $W_m\in H^2(\R^n)$ and, again using the regularity result in the Appendix \ref{se:app}, $W_m\in L^\infty(\R^n)$. Consider the set
$$ K_+=\Big\{x\in\R^n \m\big| W_m(x)>\max_{i=1,2,\ldots n_0}
   |\phi^i(x)|\Big\}.$$
On $K_+$ we have
$$ \Delta W_m(x) \geq \Big(E_m+\sigma |W_m(x)|^{2p}-|V(x)|-\sum_{i=1}^{n_0}
   |\beta_m^i|\Big)W_m(x).$$
Fix $0<\gamma<E_*$. Let $M\in\N$ and $L>0$ be sufficiently large so that for all $m>M$
$$ \sum_{i=1}^{n_0}|\beta_m^i|\leq\frac{\gamma}{4},\quad |E_m-E_*|\leq\frac
   {\gamma}{4},\quad \|V\|_{L^\infty}\leq\frac{\gamma}{4},\quad |\sigma| |W_m(x)|^{2p}\leq  \frac{\gamma}{4} \FORALL x\in \R^n\setminus B(0,L).$$
Here $B(0,L)$ is the closed ball of radius $L$ centered at the origin in $\R^n$ and  the last inequality follows from the inequalities $\psi_*^0(x)\leq Ce^{-\sqrt{E_*-\gamma}|x|}$ and $\phi^a_m(x)\leq Ce^{-\sqrt{E_*-\gamma}|x|}$  for some $C>0$ and the limit $v_m^0\stackrel{L^\infty}{\to}0$ as $m\to\infty$. Therefore
\begin{equation}\label{w_ineq}
  \Delta W_m(x) \geq (E_*-\gamma)W_m(x) \FORALL x\in K_+\cap\R^n\setminus B(0,L).
\end{equation}
Consider the function $\Phi(x)=C_0 e^{-\sqrt{E_*-\gamma}|x|}$ with $C_0>0$ being a constant (to be chosen later). This function satisfies
\begin{equation}\label{p_ineq}
  \Delta\Phi(x)\leq(E_*-\gamma)\Phi(x) \FORALL x\in\R^n\setminus0.
\end{equation}
It follows from \eqref{w_ineq} and \eqref{p_ineq} that
$$\Delta(W_m-\Phi)(x) \geq (E_*-\gamma)(W_m-\Phi)(x) \FORALL x\in K_+\cap \R^n\setminus B(0,L).$$
Hence there exists no positive maxima for the function $W_m-\Phi$ on the set
$K_+\cap B(0,L)$. Choose $C_0$ sufficiently large so that
$$ \Phi(x)\geq W_m(x) \FORALL x\in \partial B(0,L), \qquad \Phi(x)\geq\max_{i=1,2,\ldots n_0} |\phi^i(x)| \FORALL x\in\R^n.$$
Here $\partial B(0,L)=\{x\in\R^n\m\big|\m |x|=L\}$. Also, on the boundary of the set $K_+$, we have
$$ W_m(x)=\max_{i=1,\ldots,n_0}|\phi^i(x)|\leq\Phi(x).$$
Moreover, since $v_m^0\in W^{2,q}$ with $q\geq2$, we have $v_m^0(x)\to 0$ as
$|x|\to\infty$ which implies that $W_m(x)\to 0$ as $|x|\to \infty$. It now follows from the above discussion that
$$W_m(x)\leq\Phi(x) \FORALL x\in K_+\cap \R^n\setminus B(0,L).$$
A similar argument applied to the function $-W_m$ on the set
$$ K_- \m=\m \Big\{ x\in\R^n \m\big|\m -W_m(x) > \max_{i=1,2\ldots n_0} |\phi^i(x)|\Big\}$$
gives
$$ -W_m(x)\leq\Phi(x) \FORALL x\in K_-\cap \R^n\setminus B(0,L).$$
Therefore we can conclude that for each $x\in\R^n\setminus B(0,L)$ and some $C>0$
\begin{align}
  & |\psi_*^0(x)+\phi^a_m(x)+v_m^0(x)| \m\leq\m \max\{\phi^1(x),\phi^2(x), \ldots\phi^{n_0}(x),\Phi(x)\} \m\leq\m C e^{-\sqrt{E_*-\gamma}|x|} \nonumber\\
  \implies&\ |v_m^0(x)| \leq C e^{-\sqrt{E_*-\gamma}|x|}. \label{tumbboost}
\end{align}
For $x\in B(0,L)$, we have $|v_m^0(x)|\leq\|v_m^0\|_{L^\infty}$ which tends to zero as $m\to\infty$.

We will next derive estimates for the $H^2$ norm of $v_m^1=v_m-P_m^\perp v_m^0$. This function is the part of $v_m$ associated with the drifting terms in \eqref{split_soln_new}. Since $v_m^0\in S^{\perp}$ we have
$$ P_m^\perp v_m^0 = v_m^0 - \sum_{k=1}^n\sum_{j=1}^d \big\langle T_{z_m^j} \partial_{x_k}U_m, v_m^0 \big\rangle T_{z_m^j} \partial_{x_k}U_m = v_m^0-P_z v_m^0.$$
Here $P_z$ is the projection operator in $L^2(\R^n)$ onto its subspace
$$\textrm{ span} \{T_{z_m^j} \partial_{x_k}U_m \m\big|\m k=1,2,\ldots n,\ j=1,2,\ldots d\}.$$
To simplify the notation, we do not explicitly indicate the dependence of $P_z$ on $m$. Define
$$ \overline U_m \m=\m \sum_{j=1}^d T_{z_m^j} U_m \m.$$
Then \eqref{splteq} can be rewritten as
$$ (-\Delta+V+E_m)(\overline U_m+W_m+v_m^1-P_z v_m^0) + \sigma|\overline U_m
  +W_m+v_m^1-P_z v_m^0|^{2p}(\overline U_m+W_m+v_m^1-P_z v_m^0)=0,$$
where $W_m=\psi_*^0+\phi^a_m+v_m^0$. Using \eqref{raman_model} it follows that the above equation is equivalent to
\begin{align*}
  0 =&\ (-\Delta+V+E_m)v_m^1 + V\overline U_m - \sigma\sum_{j=1}^d T_{z_m^j}|U_m|^{2p}U_m - \sigma|W_m|^{2p}W_m + \sum_{i=1}^{n_0}\beta_m^i\phi^i \\
  & - V\sum_{k=1}^n\sum_{j=1}^d\big\langle
  T_{z_m^j}\partial_{x_k}U_m,v_m^0\big\rangle T_{z_m^j}\partial_{x_k}U_m+ \sigma(2p+1)\sum_{k=1}^n\sum_{j=1}^d\big\langle T_{z_m^j}\partial_{x_k}U_m,
  v_m^0 \big\rangle T_{z_m^j}|U_m|^{2p}\partial_{x_k}U_m\\[2mm]
  &+ \sigma|\overline U_m+W_m+v_m^1-P_z v_m^0|^{2p}(\overline U_m+W_m+v_m^1-
  P_z v_m^0),
\end{align*}
which can be rewritten using the linear operator $L_+(\overline U_m+W_m-P_z
v_m^0,E_m)$ and the nonlinear operator $N(\overline U_m+W_m-P_z v_m^0,v_m^1)$ as follows:
\begin{align}
  0 =&\ L_+(\overline U_m+W_m-P_z v_m^0,E_m)v_m^1 + N(\overline U_m+W_m-P_z
  v_m^0,v_m^1) + V\overline U_m \nonumber\\
  & + \sigma|\overline U_m+W_m-P_z v_m^0|^{2p}(\overline U_m+W_m-P_z v_m^0) - \sigma\sum_{j=1}^d T_{z_m^j}|U_m|^{2p}U_m - \sigma|W_m|^{2p}W_m + \sum_{i=1}^{n_0}\beta_m^i\phi^i \nonumber\\
  & - V\sum_{k=1}^n\sum_{j=1}^d\left<T_{z_m^j}\partial_{x_k}U_m,v_m^0
  \right>T_{z_m^j}\partial_{x_k}U_m + \sigma(2p+1)\sum_{k=1}^n\sum_{j=1}^d\Big\langle T_{z_m^j}\partial_{x_k} U_m,v_m^0\Big\rangle T_{z_m^j}|U_m|^{2p}\partial_{x_k}U_m. \label{cappan}
\end{align}
The argument based on Proposition \ref{prop:MS} and spectral perturbation theory (see discussion above \eqref{vm_com}) used to show that $P_{m}^{\perp}L_+(\widetilde U_m+\phi^a_m, E_m) P_{m}^{\perp}$ is invertible, with a uniform in $m$ bound for its inverse, can be used to show that the same is true for $P_m^\perp L_+(\overline U_m+W_m-P_z v_m^0,E_m) P_m^\perp$. Indeed, let $L_+(\overline U_m+W_m-P_z v_m^0,E_m)=\widetilde L_m+\widetilde W_m$, where
$$ \widetilde L_m = -\Delta+ V+ E_*+\sigma(2p+1)\Big(\sum_{j=1}^d \big|T_{z_m^j} U_{E_*}\big|^{2p} +|\psi_*^0|^{2p}\Big),$$
$$ \widetilde W_m = E_m-E_*+\sigma(2p+1)\Big(\Big|\overline U_m+W_m-P_z v_m^0\Big|^{2p}-\sum_{j=1}^d\big|T_{z_m^j} U_{E_*}\big|^{2p} -|\psi_*^0|^{2p}\Big).$$
Since $E_m\to E_*$, $a_m^i\to0$, $|y_m^j|\rightarrow\infty$ and $|y_m^j-y_m^k| \rightarrow\infty$ as $m\to\infty$, we get $\|\widetilde W_m\|_{H^2\mapsto L^2}\to0$ as $m\to\infty$.  Applying the argument above \eqref{vm_com} to $\widetilde L_m$ and $\widetilde W_m$ defined here, we get that $L_m=P_m^\perp L_+(\overline U_m+W_m-P_z v_m^0,E_m) P_m^\perp$ is invertible for large $m$ and the norm of its inverse operator can be uniformly bounded. Therefore applying $P_m^\perp$ to \eqref{cappan} and rewriting it as  $v_m^1= L_m^{-1}$(remaining terms), and using $P_m^\perp\sum_{i=1}^{n_0}\beta_
m^i\phi^i=0$ and deriving estimates for terms containing product of exponentially decaying functions separated by large distances, we can deduce via contraction principle that
\begin{equation} \label{newnotconf}
  \|v_m^1\|_{H^2} \leq C\|V\overline U_m\|_{L^2} + \sum_{j=1}^d O\Big(e^
   {-\sqrt{E_*-\gamma}|z_m^j|}\Big) + \sum_{i\neq j,\ i,j=1}^d O\Big( e^{-\sqrt{E_*-\gamma}|z_m^i-z_m^j|}\Big).
\end{equation}
In summary, we have
$$ v_m=P_m^{\perp}v_m^0+v_m^1=v_m^0-P_z v_m^0+v_m^1$$
with $v_m^0(x)\leq Ce^{-\sqrt{E_*-\gamma}|x|}$ for all $x\in\R^n$ which, using the definition of $P_z$, gives
$$ \|P_z v_m^0\|_{H^2}+\|v_m^1\|_{H^2} \leq C\|V\overline U_m\|_{L^2}+\sum_
   {i\neq j,\ i,j=1}^{d} Ce^{-\sqrt{E_*-\gamma}|z_m^i-z_m^j|}\ .$$

We will now study the set of finite dimensional equations in \eqref{eq:fincom} obtained by projecting \eqref{eq:stationary_split} on the functions $T_{z_m^j}\partial_{x_k}U_m$. Explicitly, the set of equations are
\begin{align}
  0=\ &\Big\langle T_{z_m^j}\partial_{x_k}U_m, V\sum_{i=1}^d T_{z_m^i}U_m
  + N(\widetilde U_m+\phi^a_m,v_m)\Big\rangle+ (E_m-E_*)\Big\langle T_{z_m^j}\partial_{x_k}U_m, \psi_{E_*}^0+\phi^a_m\Big\rangle\nonumber\\
  & +\sigma\Big\langle T_{z_m^j} \partial_{x_k}U_m, |\widetilde U_m+\phi^a_m|^{2p}(\widetilde U_m+\phi^a_m) - \sum_{i=1}^d T_{z_m^i}|U_m|^{2p}U_m -|\psi_*^0|^{2p} \psi_*^0 - (2p+1)|\psi_*^0|^{2p}\phi^a_m \Big\rangle\nonumber\\
  & +\Big\langle \Big(V+\sigma(2p+1)\big[|\widetilde U_m+\phi^a_m|^{2p}-T_{z_m^j}|U_m|^{2p}\big]\Big)T_{z_m^j}\partial_{x_k}U_m, v_m   \Big\rangle \label{findim_splt}
\end{align}
for $j=1,2,\ldots d$ and $k=1,2,\ldots n$. We will show that the set of
equations in \eqref{findim_splt} has no solution. Our analysis will focus on
two types of terms in \eqref{findim_splt} - terms involving interaction
between $T_{z_m^j}U_m$ and $T_{z_m^k}U_m$ for $j\neq k$ (such terms
arise from the second line of \eqref{findim_splt}) and terms involving
interaction between the various $T_{z_m^i}U_m$s and the potential $V$.
We will consider two cases, one of which must occur. In the first case we assume that the former terms, which decay like $e^{-\sqrt{E_*-\gamma} |z^j_m-z^k_m|}$, dominate the latter terms. Then, by combining the equations in \eqref{findim_splt} appropriately we construct a new equation in which the former terms dominate all the other terms. This will imply that there is no solution to \eqref{findim_splt}. In the second case, we will combine the equations in \eqref{findim_splt} appropriately to obtain a new equation in which the former terms (at least the important ones) are eliminated. Under the assumption that the terms involving the potential dominate terms decaying like $e^{-2\sqrt{E_*-\gamma}|z^j_m-z^k_m|}$, we can apply the argument from the proof of the compactness case to conclude that \eqref{findim_splt} has no solution. We will now discuss the two cases rigorously.

Let $M_m=\min\{|z_m^j-z_m^k|\m\big|\m j,k=1,2,\ldots d,\ j\neq k\}$. We assume without loss of generality, by considering subsequences if necessary, that each of the bounded sequences $(M_m/|z_m^j-z_m^k|)_{m\in\N}$ is converging
and that $\|VT_{z^1_m}U_m\|_{L^2}$ is the maximum of the set $\{\|V
T_{z^j_m}U_m\|_{L^2} \m\big|\m j=1,2,\ldots d\}$ for each $m$ and also that each of the sequences $(|z_m^1-z_m^k|/|z_m^1|)_{m\in\N}$ converges either to a finite limit or to $\infty$. Partition $\{z_m^1,z_m^2,\ldots z_m^d\}$ into sets $Z_m^1$ and $Z_m^2$ such that for each $z_m^k \in Z_m^1$ and no $z_m^k \in Z_m^2$ we have
$$\lim_{m\to \infty} \frac{|z_m^k-z_m^1|}{|z_m^1|}=0.$$
Define
$$ \widehat U_m = \sum_{z_m^k\in Z_m^1}T_{z_m^k}U_m\quad\textrm{and}\quad
   \widehat W_m = \sum_{z_m^k\in Z_m^2}T_{z_m^k}U_m.$$
In Case (i) we suppose that
\begin{equation}\label{case1_condition_p_n}
  \sqrt{E_*} < \limsup_{m \to \infty} \frac{-\ln\|V T_{z_m^j} U_m\|_{L^2}}
  {M_m} \FORALL j\in\{1,2,\ldots d\}
\end{equation}
holds and in Case (ii) we suppose that
\begin{equation}\label{case2_condition_p_n}
  2\sqrt{E_*} > \liminf_{m\to \infty} \frac{-\ln \Big|\Big\langle\partial_{x_m}V, \widehat U_m^2\Big\rangle\Big|}{M_m},
\end{equation}
where $x_m$ is the direction along the vector $z_m^1$, holds. One of the above conditions must hold. Indeed if \eqref{case1_condition_p_n} does not hold for any $j\in\{1,2,\ldots d\}$ then, since we have assumed that $\|VT_{z^1_m}U_m \|_{L^2}$ is the maximum of the set $\{\|VT_{z^j_m}U_m\|_{L^2} \m\big|\m j=1,2,\ldots d\}$ for each $m$, it follows that
\begin{equation}\label{contr}
  \limsup_{m\to \infty} \frac{-\ln\|V T_{z_m^1} U_m\|_{L^2}}{M_m} \leq
  \sqrt{E_*}.
\end{equation}
Then
\begin{align*}
  &\liminf_{m\to \infty} \frac{-\ln\Big|\Big\langle \partial_{x_m}  V,
  \widehat U_m^2 \Big\rangle\Big|} {M_m} \\
  =\ & \liminf_{m\to \infty} \frac{-\ln\|V T_{z_m^1}U_m\|_{L^2}^2}
  {M_m} \frac{-\ln\Big|\Big\langle \partial_{x_m}  V, \widehat U_m^2 \Big\rangle\Big|}{-\ln\|V T_{z_m^1}U_m\|_{L^2}^2}\\
  <\ &2 \sqrt{E_*},
\end{align*}
where we have used \eqref{contr} and the fact that the second ratio in
the second line above is less than 1  because for large $m$
$$ 1 >\left|\left\langle\partial_{x_m} V, V_{R_m}\widehat U_m^2 \right\rangle\right| > \|VT_{z_m^1}U_m\|_{L^2}^2. $$
Hence if \eqref{case1_condition_p_n} does not hold, then \eqref{case2_condition_p_n} holds, i.e. either Case (i) or Case (ii) must occur.

We say that the sequences $(z_m^j)_{m\in\N}$ and $(z_m^k)_{m\in N}$ are  connected if and only if
$$\lim_{m\to \infty} \frac {M_m} {|z_m^j-z_m^k|} \m=\m 1.$$

\noindent
{\bf Case (i):} Assume that \eqref{case1_condition_p_n} holds.  We say that the sequences $(z_m^{a_1})_{m\in\N}, (z_m^{a_2})_{m\in\N}, \ldots (z_m^{a_r})_{m\in\N}$, where each $a_i\in\{1,2,\ldots d\}$, form a connected component if two conditions are satisfied: (a) For any $i,j\in \{a_1,a_2, \ldots a_r\}$ with $i\neq j$, there exists a chain $(b_1, b_2, \ldots b_t)$, where each $b_k\in\{a_1,a_2,\ldots a_r\}$, $b_1=i$ and $b_t=j$, such that the sequences $(z_m^{b_k})_{m\in\N}$ and $(z_m^{b_{k+1}})_{m\in\N}$ are connected for each $k\in\{1,2,\ldots r-1\}$ and (b) For any $i \notin \{a_1,a_2,\ldots a_r\}$, there exists no $j\in\{a_1,a_2,\ldots a_r\}$ such that $(z_m^i)_{m\in\N}$ is connected to $(z_m^j)_{m\in\N}$.

Consider a connected component with at least two sequences, say $(z_m^1)_{m\in\N}, (z_m^2)_{m\in\N}, \ldots (z_m^r)_{m\in\N},$ see Figure \ref{fig:spt(i)}. For each $m$, the baricenter of these $z_m^j$s is
$$ \bar z_m \m=\m \frac{z_m^1+z_m^2 \ldots +z_m^r}{r} \m.$$
By going to a subsequence if necessary, suppose that $q\in\{1,2,\ldots r\}$, is such that $|z_m^q-\bar z_m|\geq |z_m^j-\bar z_m|$ for all $j\in\{1,2,\ldots r\}$ and all $m$. For each $z_m^l$ connected to $z_m^q$ it follows that
$$ \cos(\alpha_m^l) \geq \frac{|z_m^l-z_m^q|}{2|\bar z_m-z_m^q|},$$
where $\alpha_m^l$ is the angle between the vectors $z_m^l-z_m^q$ and $\bar
z_m-z_m^q$. Also, we have
\begin{align*}
  \lim_{m\to\infty}\frac{|\bar z_m-z_m^q|}{M_m} &= \lim_{m\to\infty} \frac{|z_m^1-z_m^q+\ldots+z_m^{q-1}-z_m^q+z_m^{q+1}-z_m^q+\ldots+z_m^r
  -z_m^q|}{rM_m}\\
  &\leq\lim_{m\to\infty}\sum_{i=1}^{r-1}\frac{|z_m^i-z_m^q|}{rM_m}
  \leq\frac{(r-1)^2}{r} \leq\frac{(d-1)^2}{d}.
\end{align*}
which implies that for each $z_m^l$ connected to $z_m^q$
$$ \lim_{m\to \infty} \frac{|z_m^l-z_m^q|}{2|\bar z_m-z_m^q|} \geq \liminf_{m   \to \infty}\frac{r |z_m^l-z_m^q|}{2(r-1) \sum_{i=1}^{r-1}|z_m^i-z_m^{i+1}|}
   \geq \frac{r}{2 (r-1)^2} \geq \frac{d}{2 (d-1)^2}.$$
By letting $j=q$ in \eqref{findim_splt}, we get that for some $\epsilon>0$
\begin{align}\label{eq:fin_dim_eq_case1}
  0 =&\ \Big\langle T_{z_m^q}\partial_{x_k}U_m, \sigma(2p+1) T_{z_m^q}U_m^{2p} \mathop{\sum_{l:z_m^l,z_m^q\textrm{ is}}}_{\textrm{connected}} T_{z_m^l} U_m\Big\rangle\nonumber\\
  &+ C\sum_{j=1}^d\|VT_{z_m^j}U_m\|_{L^2} + \sum_{i\neq j,\ i,j=1}^d O
  \Big(e^{-(1+\epsilon)\sqrt{E_*-\gamma}|z_m^i-z_m^j|}\Big).
\end{align}
Here we have used \eqref{hypo_smallE2}, which implies that $|V(x)|>C |x|^{-\beta+1}$ for large $|x|$, to dominate some exponentially decaying terms (for instance we have used that $\|VT_{z_m^j}U_m\|_{L^2}$ dominates terms with decay $e^{-\sqrt{E_*-\gamma}|z_m^j|}$). By choosing $x_k \parallel \bar z_m-z_m^q$ in \eqref{eq:fin_dim_eq_case1} and by denoting the direction parallel to $(z_m^l-z_m^q)$ by $x_{lq}$, we get
\begin{align}
  &\Big|\big\langle T_{z_m^q}\partial_{x_k}U_m, \sigma (2p+1)T_{z_m^q} U_m^{2p}\mathop{\sum_{l:z_m^l,z_m^q\textrm{ is}}}_{\textrm{connected}} T_{z_m^l} U_m\big\rangle\Big|\label{eq:fin_dim_merm_case1}\\
  =&\ \Big|\sigma(2p+1)\mathop{\sum_{l:z_m^l,z_m^q\textrm{ is}}}_{\textrm{connected}} \cos\alpha_m^l\big\langle T_{z_m^q} \partial_{x_{lq}}U_m, T_{z_m^q}U_m^{2p}\, T_{z_m^l}U_m\big\rangle \Big|\nonumber\\
  \geq&\ \frac{C(\gamma)d}{2(d-1)^2}\exp^{-\sqrt{E_*+\gamma}M_m}, \nonumber
\end{align}
for any $\gamma>0$. While deriving the above expression, we have used the
following fact: if $x_i \perp (z_m^l-z_m^q$) and $z_m^l$ and $z_m^q$ are
connected, then
$$ \big\langle T_{z_m^q}\partial_{x_i}U_m ,\sigma(2p+1)T_{z_m^q}U_m^{2p}
   T_{z_m^l}U_m\big\rangle=0.$$
It now follows from \eqref{case1_condition_p_n} that as $m \to \infty$, for a
sufficiently small $\gamma>0$ and each $j\in\{1,2,\ldots d\}$, we have
(at least on a subsequence) that
\begin{align*}
  &M_m \left[\sqrt{E_*+\gamma} + \frac{\ln\|VT_{z_m^j}U_m\|_{L^2}}{M_m}
  \right] \to -\infty\\
  \iff&\ \sqrt{E_*+\gamma} M_m + \ln\|VT_{z_m^j}U_m\|_{L^2} \to -\infty,\\
  \iff&\ e^{\sqrt{E_*+\gamma}M_m} \|VT_{z_m^j}U_m\|_{L^2} \to 0.
\end{align*}
The last expression above implies that the term \eqref{eq:fin_dim_merm_case1}
cannot be canceled by the remaining terms in \eqref{eq:fin_dim_eq_case1}.
Therefore \eqref{eq:fin_dim_eq_case1} and consequently the set of equations
\eqref{findim_splt} has no solution in this case.

\begin{figure}[t]
\begin{center}
 \includegraphics[scale=0.3]{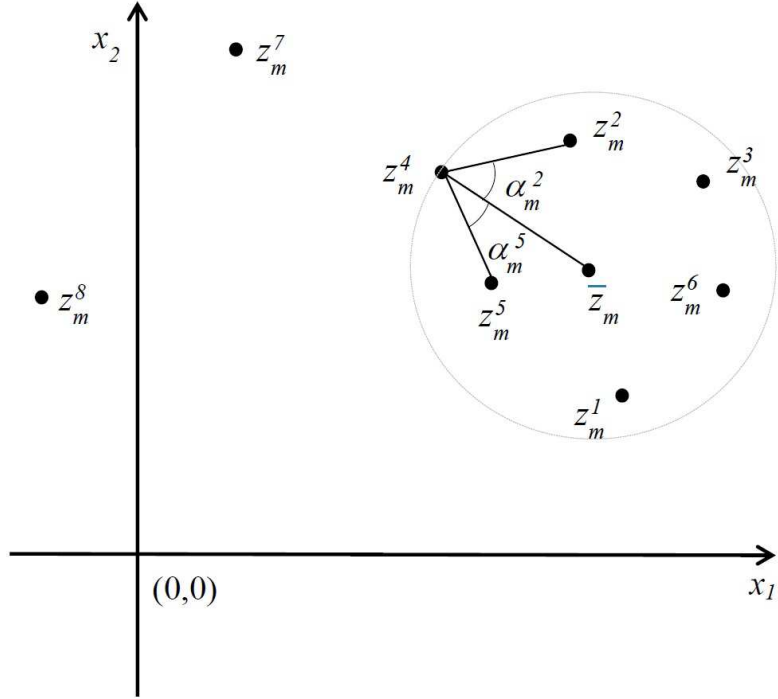}
\end{center}
\caption{An example of grouping of the $U_m$'s (each centered at a $z_m^j$)
  used while discussing splitting case (i). We assume
  that $n=2$ and $d=8$. The set $\{z_m^j:j=1,\ldots,6\}$ is a connected
  component whose baricenter is $\bar z_m$. $\bar z_m$ is also the center of
  the circle shown in the picture. Clearly $z_m^4$ is located at a maximal
  distance from $\bar z_m$ (i.e., $q=4$). We consider a finite dimensional
  equation \eqref{eq:fin_dim_eq_case1} with projection on to the function
  $\partial_{x_k}U_m(\cdot-z_m^4)$, where $x_k$ is parallel to $\bar z_m-z_m
  ^4$. We show that the terms containing $U_m$'s centered at $z_m^2$ and
  $z_m^5$, which are connected to $z_m^4$, cannot be canceled in
  \eqref{eq:fin_dim_eq_case1} by the remaining terms. This gives us the
  desired contradiction. The lower bound for $\cos\alpha_m^l$ used in Case
  (i) can be easily verified from the picture.}
\label{fig:spt(i)}
\end{figure}

\noindent
{\bf Case (ii):} Assume that \eqref{case2_condition_p_n} holds. By adding the
equations in \eqref{findim_splt} corresponding to each of the $z_m^i \in
Z_m^1,$ see Figure \ref{fig:Spt(ii)}, we obtain
\begin{align}\label{eq:fin_dim_eq_case2}
  0 =&\ \Big\langle \partial_{x_m}\widehat U_m, V(\widehat U_m + \widehat W_m) + \sigma (\widehat U_m + \widehat W_m)^{2p+1} - \sigma\widehat U_m^{2p+1} -
  \sigma \sum_{i=1}^d T_{z_m^i} U_m^{2p+1} \Big\rangle\nonumber\\
  & + C \sum_{j=1}^d\|VT_{z_m^j}U_m\|_{L^2}^2 + \sum_{i\neq j,\ i,j
  =1}^d O\Big(e^{-2\sqrt{E_*-\gamma}|z_m^i-z_m^j|}\Big).
\end{align}
\begin{figure}[h]
\begin{center}
 \includegraphics[scale=0.3]{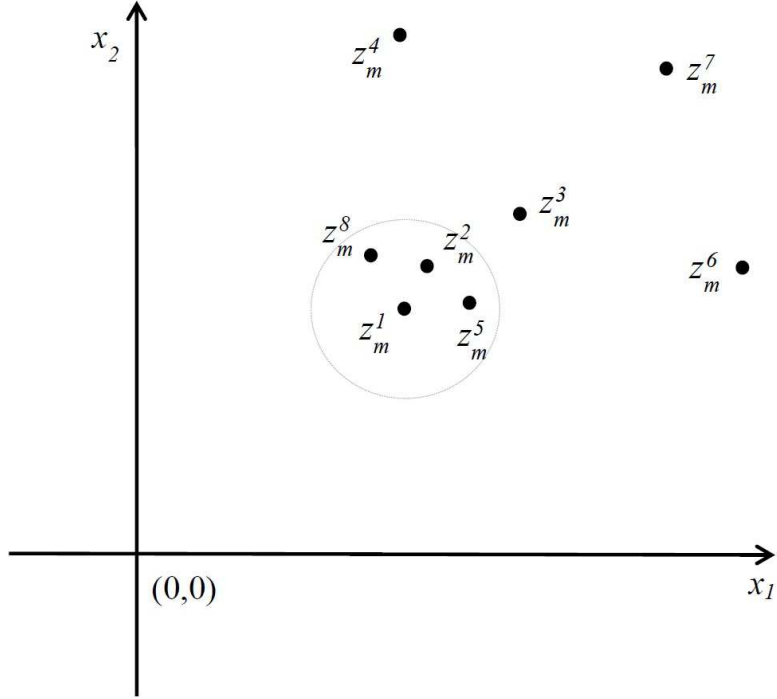}
\end{center}
\caption{An example of grouping of the $U_m$'s (each centered at a $z_m^j$)
  used while discussing splitting case (ii). We assume
  that $n=2$ and $d=8$. The norm $\|VT_{z_m^1}U_m\|_{L^2}$ is the maximum
  of the set $\{\|V T_{z_m^j}U_m\|_{L^2}\m\big|\m j=1,2,\ldots 8\}$ and $z_m^2$, $z_m^5$ and $z_m^8$ are near $z_m^1$. By adding the $U_m$'s centered at  $z_m^1$, $z_m^2$, $z_m^5$ and $z_m^8$ we get a function $\widehat U_m$.   By treating $\widehat U_m$ like $U_m$ in the case of compactness we get the desired contradiction.}
\label{fig:Spt(ii)}
\end{figure}
Since the functions $U_m$ are radially symmetric, one can check that for any
$i,j\in\{1,2,\ldots d\}$ and any direction $x_k$
$$ \big\langle T_{z_m^i}\partial_{x_k}U_m, T_{z_m^j}U_m^{2p+1}\big\rangle +
   \big\langle T_{z_m^j}\partial_{x_k}U_m, T_{z_m^i}U_m^{2p+1}\big\rangle=0.$$
In particular the above equality implies that
\begin{equation}\label{par_est}
  \bigg\langle \partial_{x_m}\widehat U_m, \sum_{i\in Z_m^1} T_{z_m^i} U_m^{2p+1} \bigg\rangle=0.
\end{equation}
Note that for some $\delta>0$
$$ \lim_{m\to\infty} \frac{|z_m^i-z_m^1|}{|z_m^1|}>\delta\quad\textrm{for
   each}\quad z_m^i\in Z_m^2.$$
Therefore
$$ \lim_{m\to\infty} \frac{|z_m^i-z_m^k|}{|z_m^1|}>\delta\quad\textrm{for
   each}\quad z_m^k\in Z_m^1,\ z_m^i\in Z_m^2.$$
From the above discussion we get the estimate
\begin{equation}\label{cross_est}
  \bigg\langle \partial_{x_m}\widehat U_m, V\widehat W_m + \sigma(\widehat U_m+\widehat W_m)^{2p+1} -\sigma\widehat U_m^{2p+1} - \sigma\sum_{l\in Z_m^2} T_{z_m^l}U_m^{2p+1} \bigg\rangle=O(e^{-\sqrt{E_*-\gamma}\delta |z_m^1|}).
\end{equation}
The hypothesis in \eqref{hypo_smallE2} implies that for some $C>0$
$$ \|V T_{z_m^1} U_m\|_{L^2}^2 \m>\m Ce^{-\sqrt{E_*-\gamma}\delta|z_m^1|} . $$ Using the above estimate, \eqref{par_est} and \eqref{cross_est}, we can rewrite \eqref{eq:fin_dim_eq_case2} as
\begin{equation}\label{case2_eqn_p_n}
  0=\Big\langle \partial_{x_m} \widehat U_m, V\widehat U_m \Big\rangle + C \sum_{j=1}^d \|V T_{z_m^j}U_m\|_{L^2}^2 + \sum_{i\neq j,\ i,j=1}^d O\Big(e^  {-2\sqrt{E_*-\gamma}|z_m^i-z_m^j|}\Big).
\end{equation}
We will next show that the first term in the above equation dominates the other two terms. Comparing the first term with the third term, it follows  from \eqref{case2_condition_p_n} that for a sufficiently small $\gamma>0$, as $m\to\infty$ (considering a subsequence if necessary)
\begin{align*}
  & M_m\left(\frac {-\ln\left|\left\langle\partial_{x_m}V, \widehat U_m^2
  \right\rangle\right|} {M_m}-2\sqrt{E_*- \gamma}\right) \to -\infty \\
  \iff&\ -\ln\left|\left\langle\partial_{x_m}V, \widehat U_m^2\right\rangle
  \right| - 2\sqrt{E_*-\gamma}M_m \to -\infty \\
  \iff&\ \frac{e^{-2\sqrt{E_*-\gamma}M_m}}{\left|\left\langle\partial_
  {x_m}V,\widehat U_m^2\right\rangle\right|} \to 0.
\end{align*}
We next compare the first term with the second term of \eqref{case2_eqn_p_n}. Using the hypothesis \eqref{hypo_smallE1} we get, as in the compactness case, that
$$ \lim_{m\to\infty}\frac{\left|\left\langle VT_{z_m^1}U_m, VT_{z_m^1}U_m
   \right\rangle\right|^2}{\left|\left\langle \widehat U^2_m,
   \partial_{x_m} V \right\rangle\right|} \leq
   \lim_{m\to\infty}\frac{\left|\left\langle\widehat U_m V, \widehat U_m V
   \right\rangle\right|^2}{\left|\left\langle\widehat U^2_m, \partial_
   {x_m} V\right\rangle\right|} = 0.$$
Hence the term $\langle\partial_{x_m}\widehat U_m ,V\widehat U_m\rangle$ in \eqref{case2_eqn_p_n} cannot be canceled by the remaining two terms. Thus we get that \eqref{eq:fin_dim_eq_case2}, and consequently the set of equations \eqref{findim_splt}, has no solution in this case.

Since either Case (i) or Case (ii) must occur, we get the contradiction that the set of equations \eqref{findim_splt} has no solution if splitting occurs. This completes the proof of this theorem.

\end{proof}

We currently do not have a similar result for arbitrary solutions of the time-independent Schr\"odinger equation \eqref{stationary} with focussing nonlinearity. The reason is that the delicate estimates in the nonlinear system of finite equations with finitely many unknowns obtained after a Lyapunov-Schmidt decomposition when attempting to show that the profiles cannot approach the limit $|x|=\infty$. Currently the estimates are using the positivity of the profiles, but we hope to remove this restriction in the future.

\section{Behavior of bound states at large $E$}\label{se:large}

In this section we analyze the limit points of zeroes of $F$ i.e., bound states of \eqref{GP}, at the $E=\infty$ boundary. Note that Corollary \ref{cor:endp} prevents branches of bound states from reaching this limit in the defocusing case $\sigma>0.$ However, in the focusing case $\sigma<0,$ we expect that all branches of bound states will end up at $E=\infty$ after maybe undergoing a finite number of bifurcations. Indeed, Theorem \ref{th:max} guarantees it for maximal branches along which e-values of $L_+$ do not approach zero. For any maximal ground state branches on finite interval $E\in(E_-,E_+),\ E_+<\infty$ we already know from Theorem \ref{th:comp} that it has (at least) a limit point $(\psi_{E_+},E_+)$ which is a zero of $F.$ In the analytic case we get that the limit point is unique and branch can be analytically continued past the bifurcation point $(\psi_{E_+},E_+),$ see \cite{BT:gbt}. By repeating this argument the branch extended analytically past the possible bifurcation points either forms a loop or ends up at the boundary of the Fredholm domain $H^2(\R^n,\C)\times [0,\infty).$ But Remark \ref{rm:endgb}, respectively Theorem \ref{th:boundedness}, exclude the $E=0,$ respectively the finite $E$ infinite $H^2$ norm portion of the boundary. Therefore, with the exception of the loops, all ground state branches will reach $E=\infty.$ We think that the restrictions to ground state branches and analytical maps will be both removed in the near future, hence all bound state branches will have an extension to $E=\infty$ boundary. Moreover, as we shall see in the examples of Section \ref{se:examples}, by analyzing the bound state branches near $E=\infty$ we will be able to discover all the non-loop branches in the interior of the Fredholm domain i.e., solutions $(\psi_E,E)$ of \eqref{stationary} with $E>0,$ together with their stability and bifurcation points.

The first result shows that the $L^2$ norm ${\cal N}(E),$ the energy ${\cal E}(E)$ and consequently the $H^1$ and $H^2$ norms of any $E\mapsto\psi_E,$  $C^1$ curve of solutions of \eqref{stationary} have a very precise behavior as $E\rightarrow\infty,$ see Theorem \ref{correct_scaling}. This leads to a unique choice for a renormalization (change of variables) which renders the curve bounded in $H^2$ as $E\rightarrow\infty.$ We then employ concentration compactness to show that such a curve, modulo the renormalization, must actually converge to a superposition of finitely many, nontrivial solutions of the limiting equation:
\begin{equation}\label{uinf}
  -\Delta u+u+\sigma|u|^{2p}u = 0.
\end{equation}
This is the well known and studied, constant coefficient, nonlinear Schr\"odinger equation, see for example \cite{Caz:nls}. Its set of solutions is invariant under the Euclidian group of transformations in $\R^n$ and is formed by smooth, exponentially decaying functions, see Appendix. In general we can show that the limiting point of our curve of bound states must be a finite superposition of solutions of \eqref{uinf}. However, since the limiting equation has a unique (up to translations in $\R^n$) positive solution $u_\infty$, see \cite{kwong}, we can be much more precise with the ground state branches, namely we can show that, up to rotations in the complex plane, they must converge in $H^2$ to a superposition of $u_\infty$ each centered at a critical point of the potential $V,$ see Theorem \ref{th:brelarge}.

Reciprocally, we show that solutions of our problem \eqref{stationary} in a sufficiently small $H^2$ neighborhood of any superposition of $\pm u_\infty$ each centered at a distinct critical point of the potential, form a unique, 2 dimensional $C^1$ manifold obtained by rotating a curve of real valued solutions, see Theorem \ref{th:bfelarge}. The orbital stability of these solutions is studied in the same theorem.

\begin{remark} \label{rmk:limitp}
Note that the two main theorems of this section determines almost all ground state branches at large $E.$ Indeed Theorem \ref{th:brelarge} shows that, modulo a rotation and renormalization, any such branch must converge in $H^2$ to a finite superposition of positive, $u_\infty,$ solutions of \eqref{uinf} each centered at a critical point of $V.$ Theorem \ref{th:bfelarge} confirms that near such a superposition, with at most one $u_\infty$ at each critical point of $V,$ there is a unique set of ground states and, moreover, the set can be organized as a $C^1$ manifold. We are missing the case when the superposition has multiple profiles $u_\infty$ approaching {\em the same} critical point. This situation occurs in related problems, see \cite{NY:mps}, and is allowed by Theorem \ref{th:brelarge} but not analyzed by Theorem \ref{th:bfelarge}. We will analyze it in a future paper, our current partial results show that multiple profiles at a minima of $V$ cannot occur but saddle points and maxima of $V$ allow for any number of profiles to approach it along the directions given by the e-vectors of the Hessian of $V$ at the saddle or maximal point corresponding to negative e-values.
\end{remark}

In the next theorem we use the following notation. For any $z\in\R^n$, let $\Tscr_z$ be the shift operator defined as follows: For any function $u$ defined on $\R^n$, $\Tscr_z u(x)=u(x-z)$ for all $x\in\R^n$. We define $\Iscr=\{1,2,\ldots n\}$.

\begin{theorem}\label{correct_scaling} Suppose that $\sigma<0$ and the potential $V$ is such that
\begin{equation}\label{potelarge}
  V\in L^{\infty}(\R^n), \qquad \lim_{|x|\rightarrow\infty}V(x)=0 \qquad
  \textrm{and} \qquad x.\nabla V\in L^\infty(\R^n).
\end{equation}
Consider a $C^1$ branch of real-valued solutions $(\psi_E,E)$ of \eqref{stationary} on the interval $[E_1,\infty)$, for some $E_1>E_0$. Then the following assertions hold:
\begin{enumerate}
 \item[(i)] There exists a $0<b<\infty$ such that
  \begin{align}
 &\lim_{E\rightarrow\infty}\frac{\|\psi_E\|_{L^{2p+2}}^{2p+2}}{E^{1-
 \frac{n}{2}+\frac{1}{p}}} = b, \label{2pnormelarge}\\
 &\lim_{E\rightarrow\infty}\frac{\|\psi_E\|_{L^2}^2}{E^{\frac{1}{p}-
 \frac{n}{2}}} = \frac{-\sigma}{2}\left(\frac{2p+2-np}{p+1}\right)b,
 \label{2normelarge}\\
 &\lim_{E\rightarrow\infty}\frac{\|\nabla\psi_E\|_{L^2}^2}{E^{1-\frac{n}
 {2}+\frac{1}{p}}} = \frac{-\sigma}{2}\left(\frac{np}{p+1}\right)b.
 \label{gradnormelarge}
  \end{align}
\ \ \ Moreover, the function $u_E$ defined as
\begin{equation}\label{uelarge}
  u_E(x)=E^{-\frac{1}{2p}}\psi_E(E^{-\frac{1}{2}}x) \qquad \forall x\in\R^n
\end{equation}
is such that
\begin{equation} \label{ue_limit}
 \lim_{E\rightarrow\infty}\|u_E\|_{H^1} = -\sigma b > 0
\end{equation}
and $u_E$ solves the equation
\begin{equation} \label{ue_equation}
 -\Delta u+E^{-1} V_{E^{-\frac{1}{2}}} u+u+\sigma|u|^{2p}u = 0.
\end{equation}
Here $V_{E^{-\frac{1}{2}}}(x)=V(E^{-\frac{1}{2}}x)$ for all $x\in\R^n$.

\item[(ii)] For any sequence $(E_m)_{m=1}^\infty$ in $[E_1,\infty)$ with $\lim_{m\to\infty} E_m =\infty$, there exists a subsequence $(E_{m_k})_{k=1}^\infty$, a positive integer $M$ and a subsequence $(x^i_k)_{k=1}^\infty$ in $\R^n$ for each $i\in\{1,2,\ldots M\}$ such that
\begin{equation}\label{ueh2_weak}
 \lim_{k\to\infty} \Big\|u_{E_{m_k}} -\sum_{i=1}^M u_i(\cdot-x_k^i) \Big\|_{H^2} \m=\m 0
\end{equation}
and for every $1\leq i,j\leq M$ with $i\not=j$
\begin{equation} \label{need_to_show_weak}
 \lim_{k\rightarrow\infty}|x_k^i-x_k^j|=\infty \quad\textrm{and}\quad \textrm{either}\quad\lim_{k\rightarrow\infty}\frac{x_k^i}{\sqrt{E_{m_k}}}=x^i\quad
 \textrm{or}\quad \lim_{k\rightarrow\infty}\frac{x_k^i}{\sqrt{E_{m_k}}}=\infty.
\end{equation}
Here $u_{E_{m_k}}$ is the scaled version of $\psi_{E_{m_k}}$ obtained by letting $E=E_{m_k}$ in \eqref{uelarge} and each $u_i$ is a solution of \eqref{uinf}.
\end{enumerate}

\end{theorem}

\begin{remark}
The above theorem applies to any $C^1$ branch of solutions of \eqref{stationary} and not just to $C^1$ branches of ground states. Also in this theorem $p$ can be smaller than $1/2$ i.e., $0< p<2/(n-2)$.
\end{remark}

\begin{proof} Consider a $C^1$ branch of real-valued solutions $(\psi_E,E)$ of
\eqref{stationary} on the interval $[E_1,\infty)$ for some $E_1>E_0$.
From \eqref{def:energy} recall the energy functional
$${\cal E}(E) = \int_{\R^n}|\nabla\psi_E(x)|^2\,\mathrm{d}x + \int_{\R^n}
  V(x)|\psi_E(x)|^2\,\mathrm{d}x + \frac{\sigma}{p+1}\int_{\R^n}|\psi_E(x)|^
  {2p+2}\,\mathrm{d}x.$$
Recall $\Nscr(E)=\|\psi_E\|_{L^2}^2,$ see \eqref{def:mass}. Then from \eqref{eq:denergy} we have
$$\frac{\mathrm{d}{\cal E}}{\mathrm{d}E} = -E\frac{\mathrm{d}\Nscr}
  {\mathrm{d}E}$$
and from \eqref{eq:spstat} we have
$${\cal E}(E)+\frac{\sigma p}{p+1}\|\psi_E\|_{L^{2p+2}}^{2p+2}=-E\Nscr(E).$$
Differentiating the above equation with respect to $E$ gives
\begin{equation}\label{eq:dnormp}
  \frac{\mathrm{d}}{\mathrm{d}E}\|\psi_E\|_{L^{2p+2}}^{2p+2} = \frac{p+1}
  {-\sigma p}\Nscr(E).
\end{equation}
Since $\psi_E$ is a weak solution of \eqref{stationary}, using
$x.\nabla\psi_E$ as a test function and integrating by parts we get the Pohozaev type identity
\begin{equation}\label{eq:poho}
  (n-2)\|\nabla\psi_E\|_{L^2}^2 + \int_{\R^n}(nV(x)+x.\nabla V(x)) |\psi_E(x)|^2\mathrm{d}x + \frac{n\sigma}{p+1}\|\psi_E\|_{L^{2p+2}}^{2p+2} =-En\|\psi_E\|_{L^2}^2.
\end{equation}
Adding \eqref{eq:poho} with \eqref{eq:spstat} multiplied by $-(n-2)$ gives
$$ \int_{\R^n}(2V(x)+x.\nabla V(x))|\psi_E(x)|^2\mathrm{d}x + \sigma  \left(\frac{n}{p+1}+2-n\right)\|\psi_E\|_{L^{2p+2}}^{2p+2}=-2E  \|\psi_E\|_{L^2}^2,$$
which can be written equivalently as
$$ 2E\Nscr(E) + \int_{\R^n}(2V(x)+x.\nabla V(x))|\psi_E(x)|^2\mathrm{d}x = -\sigma \left(2-\frac{np}{p+1}\right)\|\psi_E\|_{L^{2p+2}}^{2p+2}.$$
Since we have $2p<4/(n-2)$ for $n>2$, it follows that $2-np/(p+1)>0$ for any $n$. Bounding the integral term in the identity above via H\"older inequality we get
\begin{equation}\label{ineq:nnormp}
  2(E-C)\Nscr(E) \leq -\sigma\frac{2p+2-np}{p+1}\|\psi_E\|_{L^{2p+2}}^{2p+2} \leq 2(E+C)\Nscr(E),
\end{equation}
where $2C=2\|V\|_{L^\infty}+\|x.\nabla V\|_{L^\infty}$. Denoting $Q(E)= \|\psi_E\|_{L^{2p+2}}^{2p+2}$ and using \eqref{eq:dnormp} in \eqref{ineq:nnormp} it follows that
\begin{equation} \label{ineq:dnorm}
 \left(\frac{1}{p}+1-\frac{n}{2}\right)\frac{Q(E)}{E+C} \leq \frac{dQ(E)}{dE}
   \leq \left(\frac{1}{p}+1-\frac{n}{2}\right)\frac{Q(E)}{E-C}.
\end{equation}
Fix $E_2>\max\{C,E_1\}$ and $E_3>E_2$ and integrate the above expression on $[E_2,E_3]$ to get
$$  \frac{Q(E_2)}{(E_2+C)^{\left(\frac{1}{p}+1-\frac{n}{2}\right)}}(E_3+C)^
    {\left(\frac{1}{p}+1-\frac{n}{2}\right)} \leq Q(E_3) \leq \frac{Q(E_2)}
    {(E_2-C)^{\left(\frac{1}{p}+1-\frac{n}{2}\right)}}(E_3-C)^{\left(\frac{1}
    {p}+1-\frac{n}{2}\right)}$$
and therefore
$$  \frac{Q(E_2)}{(E_2+C)^{\left(\frac{1}{p}+1-\frac{n}{2}\right)}} \leq
    \frac{Q(E_3)}{E_3^{\left(\frac{1}{p}+1-\frac{n}{2}\right)}} \leq \frac
    {Q(E_2)}{(E_2-C)^{\left(\frac{1}{p}+1-\frac{n}{2}\right)}}.$$
This implies \eqref{2pnormelarge} since
$$  \lim_{E_2\rightarrow\infty}\frac{\frac{Q(E_2)}{(E_2-C)^{\left(\frac{1}{p}
    +1-\frac{n}{2}\right)}}}{\frac{Q(E_2)}{(E_2+C)^{\left(\frac{1}{p}+1-\frac
    {n}{2}\right)}}}=1.$$
Now \eqref{2normelarge} follows from dividing \eqref{ineq:nnormp} by $E^{1-n/2+1/p}$ and passing to the limit as $E\rightarrow\infty$, while \eqref{gradnormelarge} follows from dividing \eqref{eq:spstat} by $E^{1-n/2+1/p}$ and passing to the limit as $E\rightarrow\infty$.

Using the fact that $\psi_E$ solves \eqref{stationary}, it is easy to verify that $u_E$ solves \eqref{ue_equation}. Moreover \eqref{uelarge} combined with \eqref{2pnormelarge}-\eqref{gradnormelarge} gives
\begin{align}
  &\|u_E\|_{L^{2p+2}}^{2p+2} = E^{-1-\frac{1}{p}+\frac{n}{2}} \|\psi_E\|_{L^{2p+2}}^{2p+2}  \stackrel{E\to\infty}{\longrightarrow}b,\nonumber\\
  &\|u_E\|_{L^2}^2 = E^{-\frac{1}{p}+\frac{n}{2}} \|\psi_E\|_{L^2}^2\stackrel{E\to\infty} {\longrightarrow} \frac{-\sigma}{2}\frac{2p+2-np}{p+1}b, \label{repeated_ref}\\
  &\|\nabla u_E\|_{L^2}^2 = E^{-1-\frac{1}{p}+\frac{n}{2}}\|\nabla\psi_E\|_{L^2}^2
  \stackrel{E\to\infty}{\longrightarrow}\frac{-\sigma}{2}\frac{np}{p+1}b. \nonumber
\end{align}
Adding the last two equations we get \eqref{ue_limit}. This completes the proof of part (i) of the theorem.

To prove part (ii) of the theorem, consider a sequence $(E_m)_{m=1}^\infty$ in $[E_1,\infty)$ with $\lim_{m\to\infty}E_m=\infty$.
Define
$$ u_{E_m}(x)=E_m^{-\frac{1}{2p}}\psi_{E_m}(E_m^{-\frac{1}{2}}x) \FORALL x\in\R^n. $$
From part (i) we know that the sequence $(u_{E_m})_{m=1}^\infty$ is bounded in $H^1$ and the sequence $(\|u_{E_m}\|_{L^2})_{m=1}^\infty$ converges to a nonzero constant, see \eqref{repeated_ref}. Define
$$  \Psi_m = \frac{u_{E_m}}{\|u_{E_m}\|_{L^2}}\m.$$
Like in Section \ref{se:bif} (see discussion below \eqref{def:Psi}), we will apply concentration compactness theory to $(\Psi_m)_{m\in\N}$ by defining the concentration function $\rho$ and the variable $\mu$ for the sequence $(\Psi_m)_{m\in\N}$. Out of three possible cases in the concentration compactness theory we will show that vanishing cannot occur and compactness and splitting imply \eqref{ueh2_weak} and \eqref{need_to_show_weak} for a subsequence of $u_{E_m}$.

\noindent
{\bf Vanishing ($\mu=0$):} In this case there is a subsequence of $(\Psi_m)_{m\in\N}$, again denoted by $(\Psi_m)_{m\in\N}$, such that along this subsequence
\begin{equation} \label{swimming}
 \lim_{m\to\infty} \|\Psi_m\|_{L^q}=0 \quad \textrm{for all} \quad
 2<q<\frac{2n}{n-2} \qquad (2<q<\infty \textrm{ if } n=2,\ 2<q\leq\infty \textrm{ if } n=1).
\end{equation}
Since $u_{E_m}=\|u_{E_m}\|_{L^2}\Psi_m$ and part (i) gives
\begin{equation} \label{heavy_rain}
 \lim_{E_m\to\infty} \|u_{E_m}\|_{L^2} =a>0,
\end{equation}
it follows from \eqref{swimming} that there is a subsequence of $(u_{E_m})_{m\in\N}$, again denoted by $(u_{E_m})_{m\in\N}$, such that along this subsequence
\begin{equation}\label{conc_van}
 \lim_{m\to\infty} \|u_{E_m}\|_{L^q}=0 \quad \textrm{for all} \quad
 2<q<\frac{2n}{n-2} \qquad (2<q<\infty \textrm{ if } n=2,\ 2<q\leq\infty \textrm{ if } n=1).
\end{equation}
Note that $u_{E_m}$ solves \eqref{ue_equation} (with $E=E_m$ of course). Therefore
\begin{equation}\label{eq:uefp}
  u_{E_m}=(-\Delta+1)^{-1}\left[-R_m^2 V_{R_m} u_{E_m}-\sigma
  |u_{E_m}|^{2p}u_{E_m}\right],
\end{equation}
where $R_m=(E_m)^{-1/2}$ and $V_{R_m}(x)=V(R_m x)$ for all $x\in\R^n$. Clearly $\lim_{m\to\infty}R_m=0$. Since $L^2\hookrightarrow H^{-1}$ and $\|u_{E_m}\|_{L^2}$ is uniformly bounded, we get that for some $C>0$
\begin{equation} \label{gulmohar1}
 \lim_{m\to\infty}\|R_m^2 V_{R_m} u_{E_m}\|_{H^{-1}} \leq C\lim_{m\to\infty}\|R_m^2 V_{R_m} u_{E_m}\|_{L^2} \leq C \lim_{m\to\infty} R_m^2\|V\|_{L^\infty} \|u_{E_m}\|_{L^2}=0.
\end{equation}
Fix $r$ such that if $n>2$, then $2<r<2n/(n-2)$ and $2n/(n+2)<r/(2p+1)<2$ and if $n\leq 2$, then $2<r<\infty$ and $1<r/(2p+1)<2$. For $n>2$, the fact that such an $r$ exists follows easily from the observation that $p<2/(n-2)$ implies $(2p+1)2n/(n+2)<2n/(n-2)$. Then it follows that $L^{\frac{r}{2p+1}} \hookrightarrow H^{-1}$ and so for some $C>0$
\begin{equation} \label{gulmohar2}
\lim_{m\to\infty}\|\,|u_{E_m}|^{2p}u_{E_m}\|_{H^{-1}} \leq C\lim_{m\to\infty} \||u_{E_m}|^{2p} u_{E_m}\|_{L^{\frac{r}{2p+1}}} = C\lim_{m\to\infty} \|u_{E_m}\|_{L^r}^{2p+1}=0.
\end{equation}
Here we have used \eqref{conc_van} to obtain the last equality. It now follows easily from \eqref{gulmohar1} and \eqref{gulmohar2} and the fact that the operator $(-\Delta+1)^{-1}:H^{-1}\mapsto H^1$ is unitary that
$$  \lim_{m\to\infty} \|u_{E_m}\|_{H^1} =0$$
which contradicts \eqref{heavy_rain}. Therefore vanishing cannot happen.

\noindent
{\bf Compactness ($\mu=1$):} Since $ u_{E_m}=\|u_{E_m}\|_{L^2}\Psi_m$ and \eqref{heavy_rain} holds, it follows from concentration compactness theory that if $\mu=1$, then there exists a subsequence of $(E_m)_{m\in\N}$, also denoted as $(E_m)_{m\in\N}$, a corresponding sequence $(y_m)_{m\in\N}$ in $\R^n$ and a function $\hat u_*\in H^1(\R^n)$ such that for each $m\in\N$, $u_{E_m}$ can be decomposed as follows:
\begin{equation}\label{comp-conv}
  u_{E_m} = \Tscr_{y_m}\hat u_m + \hat v_m,
\end{equation}
where $\hat u_m, \hat v_m \in H^1(\R^n)$ and as $m\to\infty$
\begin{itemize}
  \item[(i)] $\hat u_m\stackrel{H^1}{\rightharpoonup} \hat u_*$\ \ \
  and\ \ \ $\hat u_m\stackrel{L^q}{\rightarrow} \hat u_*$\ \ \ for \
  \ \ $2\leq q<2n/(n-2)$\ \ ($2\leq q<\infty\textrm{ if } n=2$,\ $2\leq q\leq\infty\textrm{ if } n=1$),
  \item[(ii)] $\hat v_m\stackrel{H^1}{\rightharpoonup} 0$\ \ \ and\ \ \ $\hat v_m\stackrel{L^q}{\rightarrow}0$\ \ \ for\ \ \ $2\leq q<2n/(n-2)$ \  \ ($2\leq q<\infty\textrm{ if } n=2$,\ $2\leq q \leq\infty\textrm{ if } n=1$).
\end{itemize}
 The pair $(u_{E_m},R_m)$ satisfies \eqref{ue_equation} and therefore, by translating all the functions in this equation by $-y_m$, we get
$$  (-\Delta + R_m^2\Tscr_{-y_m} V_{R_m}+1)\Tscr_{-y_m}u_{E_m}+\sigma  \Tscr_{-y_m}|u_{E_m}|^{2p}u_{E_m}=0.$$
Here $R_m=E_m^{-1/2}$ and $V_{R_m}(x)=V(R_m x)$ for all $x\in\R^n$. Taking the $L^2$ scalar product of the above equation with any $\phi\in C_0^{\infty}(\R^n)$ gives
$$ \left\langle\nabla \Tscr_{-y_m}u_{E_m},\nabla\phi\right\rangle + R_m^2 \left\langle \Tscr_{-y_m}V_{R_m}\Tscr_{-y_m}u_{E_m},\phi(x)\right\rangle + \left\langle \Tscr_{-y_m} u_{E_m},\phi\right\rangle + \sigma\left\langle \Tscr_{-y_m} |u_{E_m}|^{2p}u_{E_m},\phi\right\rangle =0. $$
Substituting for $u_{E_m}$ from \eqref{comp-conv} into the last equation and taking the limit as $m\to\infty$ we obtain, using (i) and (ii) above, that
$$\left\langle\nabla\hat u_*,\nabla\phi\right\rangle + \left\langle\hat u_*,
    \phi\right\rangle + \sigma\left\langle|\hat u_*|^{2p}\hat u_*,\phi
    \right\rangle=0.$$
Thus $\hat u_*$ is a weak solution (and therefore a strong, $H^2$ solution) of \eqref{uinf}. Redefining $\hat v_m$ to be $\hat v_m+\Tscr_{y_m}\hat u_m -\Tscr_{y_m}\hat u_*$ we obtain the following different decomposition for $u_{E_m}$ in place of \eqref{comp-conv}:
\begin{equation}\label{comp-conv1_weak}
  u_{E_m} = \Tscr_{y_m}\hat u_* + \hat v_m.
\end{equation}
The redefined functions $\hat v_m$ in \eqref{comp-conv1_weak} also have the convergence properties mentioned in (ii) above. Since $(u_{E_m},R_m)$ satisfies \eqref{ue_equation}, substituting the above decomposition for $u_{E_m}$ into \eqref{ue_equation} and using the fact that $\hat u_*$ satisfies \eqref{uinf}, we obtain
\begin{equation} \label{vmhat_exp}
 \hat v_m=-(-\Delta+1)^{-1}\left[R_m^2 V_{R_m} T_{y_m} \hat u_* + (R_m^2 V_{R_m}+\sigma(2p+1) |T_{y_m} \hat u_*|^{2p})\hat v_m+N(T_{y_m} \hat u_*,\hat v_m)\right].
\end{equation}
Here $N(U,v):H^2(\R^n)\times H^2(\R^n)\mapsto L^2(\R^n)$ is the nonlinear function from \eqref{bound_N},
$$ N(U,v)=\sigma|U+v|^{2p}(U+v)-\sigma|U|^{2p}U-\sigma(2p+1)|U|^{2p}v.$$
From the limit $\hat v_m\stackrel{L^2} {\rightarrow}0$ as $m\to\infty$ from (ii), $\lim_{m\to\infty}R_m=0$ and $u^*\in L^\infty$ (see Proposition \ref{prop:h1toh2}), it follows that
$$ \lim_{m\to\infty} \|R_m^2 V_{R_m} T_{y_m} \hat u_* + (R_m^2 V_{R_m}+\sigma(2p+1) |T_{y_m} \hat u_*|^{2p})\hat v_m\|_{L^2} \m=\m0$$
and from Lemma \ref{lm:nuv} and the limits for $\hat v_m$ in (ii) it follows that
\begin{equation} \label{mus_aud}
 \lim_{m\to\infty}  \|N(T_{y_m} \hat u_*,\hat v_m)\|_{L^2} \m=\m 0.
\end{equation}
Hence from \eqref{vmhat_exp} we get that
\begin{equation} \label{almost_done_weak}
 \lim_{m\to\infty}\|\hat v_m\|_{H^2} = 0.
\end{equation}

Suppose that $(R_m y_m)_{m\in\N}$ has a converging subsequence, denoted again by $(R_m y_m)_{m\in\N}$, converging to $x^1\in\R^n$. Then it follows from \eqref{comp-conv1_weak} and \eqref{almost_done_weak} that \eqref{ueh2_weak} and the second limit in \eqref{need_to_show_weak} hold with $M=1$, $u_1=\hat u_*$ and $x_k^1= y_k$. On the other hand, if $(R_m y_m)_{m\in\N}$ has no bounded subsequence, then it follows from \eqref{comp-conv1_weak} and \eqref{almost_done_weak} that \eqref{ueh2_weak} and the third limit in \eqref{need_to_show_weak} hold with $M=1$, $u_1=\hat u_*$ and $x_k^1= y_k$. In either case, the first limit in \eqref{need_to_show_weak} holds vacuously. This completes the proof of part (ii) in the compactness case.

\noindent
{\bf Splitting ($0<\mu<1$):} Since $ u_{E_m}=\|u_{E_m}\|_{L^2}\Psi_m$ and \eqref{heavy_rain} holds, it follows from concentration compactness theory that if $0<\mu<1$, then there exists a subsequence of $(E_m)_{m\in\N}$, again denoted as $(E_m)_{m\in\N}$, on which we can decompose $u_{E_m}$ as follows:
\begin{equation}\label{eq:split_soln}
  u_{E_m} = \sum_{j=1}^{d} \Tscr_{y_m^j}u_m^j+u_m^0+ \hat v_m,
\end{equation}
where $u_m^j, v_m\in H^1(\R^n)$ for all $j\in\Jscr=\{0,1,\ldots d\}$ and there exist non-zero functions $u_*^j$ such that as $m\to\infty$
\begin{itemize}
  \item[(i)] $u_m^j \stackrel{H^1}{\rightharpoonup} u^j_*$\ \ and \ \ $u_m^j \stackrel{L^q}{\rightarrow} u^j_*$\ \  for \ \ $2\leq q< 2n/(n-2)$\ \ ($2\leq q<\infty$ if $n=2$,\ $2\leq q\leq\infty$ if $n=1$),
  \item[(ii)] $|y^j_m|\to\infty$\ \ and \ \ $|y^j_m-y^k_m|\to\infty$\ \
  for\ \ $j, k\in \{1,2,\ldots d\}$ and $j\neq k$,
  \item[(iii)] $\hat v_m \stackrel{H^1}{\rightharpoonup} 0$\ \ and\ \ $\hat
  v_m \stackrel {L^q} {\rightarrow} 0$\ \ for\ \ $2< q< 2n/(n-2)$ ($2< q<\infty$ if $n=2$,\ $2 < q\leq\infty$ if $n=1$).
\end{itemize}
Concentration compactness theory does not imply that $u_*^0\neq 0$, but we assume it only to simplify our presentation. All the arguments in this proof remain valid if $u_*^0 = 0$. Letting $u=u_{E_m}$ and $E=E_m$ in \eqref{ue_equation} and taking the $L^2$ scalar product of the resulting equation with $\phi\in C_0^{\infty}(\R^n)$ we  get
$$  \big\langle \nabla u_{E_m},\nabla\phi \big\rangle + \big\langle
    (R_m^2V_{R_m}+1)u_{E_m},\phi \big\rangle + \sigma\big\langle
    |u_{E_m}|^{2p}u_{E_m},\phi \big\rangle = 0.$$
Here $R_m=E_m^{-1/2}$ and $V_{R_m}(x)=V(R_m x)$ for all $x\in\R^n$. Substituting for $u_{E_m}$ from \eqref{eq:split_soln} into the above equation and taking the limit as $m\to\infty$ gives
$$ \big\langle \nabla u_*^0, \nabla\phi\big\rangle + \big\langle u_*^0,
   \phi\big\rangle + \sigma\big\langle |u_*^0|^{2p} u_*^0,\phi \big\rangle = 0. $$
Therefore $u_*^0$ is a weak solution (and hence a strong, $H^2$ solution) of \eqref{uinf}. Similarly for any $j\in\{1,2,\ldots d\}$, taking the $L^2$ scalar product of the equation
$$  (-\Delta+1+R_m^2 \Tscr_{-y^j_m}V_{R_m})\Tscr_{-y^j_m}u_{E_m}+\sigma
    \Tscr_{-y^j_m}|u_{E_m}|^{2p}u_{E_m}=0$$
(which is nothing but a translated version of \eqref{ue_equation} with $E=E_m$ and $u=u_{E_m}$) with $\phi\in C_0^{\infty}(\R^n)$ and taking the limit as $m\to\infty$ we can conclude that $u_*^j$ is a weak solution (and hence a strong, $H^2$ solution) to \eqref{uinf}. Hence for each $j\in\Jscr$ the function $u_*^j$ is a solution of \eqref{uinf}. Redefining $\hat v_m$ to be $\hat v_m+u_m^0+\sum_{j=1}^d T_{y_m^j} u_m^j-u_*^0-\sum_{j=1}^d T_{y_m^j} u_*^j$ we obtain the following different decomposition for $u_{E_m}$:
\begin{equation}\label{new_split_soln_inf_weak}
  u_{E_m} \m=\m \sum_{j=1}^{d} \Tscr_{y^j_m}u_*^j+u_*^0+
  \hat v_m \m.
\end{equation}
The redefined $\hat v_m$ satisfy the limits in (iii) above and from Proposition \ref{prop:exdecay} in the Appendix we get that there exist $C>0$ and $\gamma \in(0,1)$ such that
\begin{equation} \label{Mmime}
 |u^j_*(x)| \m<\m Ce^{-\sqrt{1-\gamma}|x|} \FORALL x\in\R^n, \FORALL j\in\Jscr.
\end{equation}
Since $(u_{E_m},R_m)$ satisfies \eqref{ue_equation}, substituting the decomposition in \eqref{new_split_soln_inf_weak} for $u_{E_m}$ into \eqref{ue_equation} and using the fact that each $u^j_*$ satisfies \eqref{uinf}, we get
\begin{align}
  \hat v_m=-(-\Delta+1)^{-1}&\left[R_m^2 V_{R_m} \sum_{j=0}^d T_{y_m^j}u_*^j+ \Big(R_m^2 V_{R_m}+\sigma(2p+1) \Big|\sum_{j=0}^d T_{y_m^j}u_*^j\Big|^{2p}\Big)\hat v_m\right.\nonumber\\
  &\left.-\sigma\sum_{j=0}^d T_{y_m^j}|u_*^j|^{2p}u_*^j+\sigma\Big|\sum_{j=0}^d T_{y_m^j}u_*^j\Big|^{2p}\sum_{j=0}^d T_{y_m^j}u_*^j+N\Big(\sum_{j=0}^d T_{y_m^j}u_*^j,\hat v_m\Big)\right]. \label{10autom}
\end{align}
Here $y_m^0=0$ and $N(U,v):H^2(\R^n) \times H^2(\R^n)\mapsto L^2(\R^n)$ is the nonlinear function defined in \eqref{bound_N}. We will now show that the term in the bracket in the above expression (to which $-(-\Delta+1)^{-1}$ is applied) converges to 0 in $L^2(\R^n)$ as $m\to\infty$. Indeed, since $u_*^j\in L^2(\R^n)$ for all $j\in\Jscr$, $V\in L^\infty(\R^n)$, $\sup_{m\in\N}\|v_m\|_{L^2}<\infty$ and $\lim_{m\to\infty} R_m=0$, we have
\begin{equation} \label{8TAC}
 \lim_{m\to\infty}\Big\|R_m^2 V_{R_m} \sum_{j=0}^d T_{y_m^j}u_*^j+ R_m^2 V_{R_m} \hat v_m\Big\|_{L^2} = 0.
\end{equation}
From \eqref{Mmime} and the limits for $y_m^j$ in (ii) we get that
\begin{equation} \label{hostel14}
 \lim_{m\to\infty} \Bigg\|\Big|\sum_{j=0}^d T_{y_m^j}u_*^j\Big|^{2p}\sum_{j=0}^d T_{y_m^j}u_*^j - \sum_{j=0}^d T_{y_m^j}|u_*^j|^{2p} u_*^j \Bigg\|_{L^2}=0.
\end{equation}
From \eqref{Mmime} we have $u_*^j\in L^q(\R^n)$ for any $q\in[1,\infty]$. This, along with the limits for $\hat v_m$ in (iii), implies using H\"older's inequality that
\begin{equation} \label{Irabday}
 \lim_{m\to\infty} \Bigg\| \Big|\sum_{j=0}^d T_{y_m^j}u_*^j\Big|^{2p}\hat v_m \Bigg\|_{L^2} \m=\m 0\m.
\end{equation}
Since $\|u_{E_m}\|_{H^2}$ is bounded in uniformly in $m$ and each $u_*^j\in H^2(\R^n)$, it follows that $\|\hat v_m\|_{H^2}$ is bounded uniformly in $m$.
Hence from Lemma \ref{lm:nuv} and (iii) it follows that
\begin{equation} \label{randcall}
 \lim_{m\to\infty} \Big\|N\Big(\sum_{j=0}^d T_{y_m^j}u_*^j,\hat v_m\Big)\Big\|_{L^2} \m=\m 0.
\end{equation}
From \eqref{10autom}-\eqref{randcall} we get that
\begin{equation} \label{almost_done_split1}
 \lim_{m\to\infty}\|\hat v_m\|_{H^2} = 0.
\end{equation}
For each $j\in\Jscr$, if $(R_m y_m^j)_{m\in\N}$ has a converging subsequence, then denote it (again) as $(R_m y_m^j)_{m\in\N}$ and its limit point as $x^j$. It now follows from \eqref{new_split_soln_inf_weak}, \eqref{almost_done_split1} and the limits for $y_m^j$ in (ii) that \eqref{ueh2_weak} and \eqref{need_to_show_weak} hold with $M=d+1$, $u_i=u_*^{i-1}$ and $x_k^i= y_k^{i-1}$ for $i\in\{1,2,\ldots M\}$. This completes the proof of part (ii) in the splitting case.

Theorem \ref{correct_scaling} is now completely proven.
\end{proof}

Recall that $\Iscr=\{1,2,\ldots n\}$. For any $z\in\R^n$ and any function $u$ defined on $\R^n$, $\Tscr_z u(x)=u(x-z)$ for all $x\in\R^n$. In the next theorem, under additional hypothesis on $p$ and the potential $V$, we strengthen the results in part (ii) of Theorem \ref{correct_scaling} as follows: If $(\psi_E,E)$ is a branch of ground states on $[E_1,\infty)$, then each $u_i$ in \eqref{ueh2_weak} is the unique radially symmetric positive solution $u_\infty$ of \eqref{uinf}, the third limit in \eqref{need_to_show_weak} does not occur and each $x^i$ in \eqref{need_to_show_weak} is a critical point of the potential $V$.

\begin{theorem}\label{th:brelarge} Let $\sigma<0$ and $p\geq 1/2$. Assume that the potential $V$ satisfies the conditions in \eqref{potelarge} and $\nabla V\in W^{1,\infty}(\R^n)$. Furthermore, let the radial derivative $\partial_r V(x)$ of $V$ and a set of its tangential derivatives $\{\partial_{s_i}V(x),i=1,2,\ldots n-1\}$, where  $s_1,s_2, \ldots s_{n-1}$ are a set of mutually orthogonal directions which are also orthogonal to the radial direction, satisfy the following hypotheses: There exists a set $\cal{M}$ of measure zero such that
\begin{equation} \label{Hypo_largeE1}
 \limsup_{x\in \R^n\setminus{\cal M},\ |x|\to\infty} \frac{V^2(x)}{|\partial_r V(x)|} <\infty
\end{equation}
and there exist constants $\Cscr,\beta,\rho>0$ such that for each $x\in\R^n\setminus{\cal M}$ with $|x|>\rho$
\begin{align}
  \Cscr|\partial_r V(x)| &> e^{-\beta|x|}, \label{Hypo_largeE2}\\
  \Cscr|\partial_r V(x)| &> \Big(\sum\limits_{i=1}^{n-1}| \partial_{s_i}V(x)|^2 \Big)^\frac{1}{2}. \label{Hypo_largeE3}
\end{align}

Under the above assumptions, consider a $C^1$ curve $(\psi_E,E)$ of real-valued positive ground states of \eqref{stationary}, see Definition \ref{def:gs}, on the interval $[E_1,\infty)$ for some $E_1>0$. Then for any sequence $(E_m)_{m=1}^\infty$ in $[E_1,\infty)$ with $\lim_{m\to\infty}E_m =\infty$, there exists a subsequence $(E_{m_k})_{k=1}^\infty$, a positive integer $M$ and a subsequence $(x^i_k)_{k=1}^\infty$ in $\R^n$ for each $i\in\{1,2,\ldots M\}$ such that
\begin{equation}\label{ueh2}
    \lim_{k\rightarrow\infty}\Big\|u_{E_{m_k}}-\sum_{i=1}^M u_\infty (\cdot-x_k^i)\Big\|_{H^2}=0
\end{equation}
and for every $1\leq i,j\leq M$ with $i\not=j$
\begin{equation} \label{need_to_show}
 \lim_{k\rightarrow\infty}|x_k^i-x_k^j|=\infty \quad\textrm{and}\quad\lim_
 {k\rightarrow\infty}\frac{x_k^i}{\sqrt{E_{m_k}}}=x^i.
\end{equation}
Here $u_{E_{m_k}}$ is the scaled version of $\psi_{E_{m_k}}$ obtained by letting $E=E_{m_k}$ in \eqref{uelarge}, $u_\infty$ is the unique radially symmetric positive solution of \eqref{uinf}. Furthermore, if the potential $V$ is continuously differentiable in a neighborhood of $x^i$, then $x^i$ must be a critical point of $V$.
\end{theorem}

\begin{remark}
The hypotheses in \eqref{Hypo_largeE1}-\eqref{Hypo_largeE3} in the above theorem impose restrictions on the behavior of the potential near $x=\infty$ alone. In particular, near infinity the potential should not have any critical points and must either strictly decrease or strictly increase to zero in the radial direction. The radial derivative of the potential multiplied by a sufficiently large constant must, in magnitude, dominate the tangential ones. These assumptions hold for physical potentials which model a field that is homogeneous and isotropic at infinity. The conclusions of the above theorem implies that for large $E$ the ground states on a $C^1$ branch of real-valued positive ground states are essentially a sum of finitely many humps. Each of these humps is a re-scaling (via \eqref{uelarge}) of the radially symmetric profile $u_\infty$, and each of them concentrates at a critical point of the potential $V$. In particular, these humps cannot drift towards infinity. The question of whether such $C^1$ branches of ground states actually exist is addressed in Theorem \ref{th:bfelarge}.
\end{remark}

\begin{proof}
Consider a $C^1$ branch of real-valued positive ground states $(\psi_E,E)$ of
\eqref{stationary} on the interval $[E_1,\infty)$ for some $E_1>0$. Let $(E_m)_{m=1}^\infty$ be a sequence with $\lim_{m\to\infty}E_m=\infty$.
Define
$$ u_{E_m}(x)=E_m^{-\frac{1}{2p}}\psi_{E_m}(E_m^{-\frac{1}{2}}x) \qquad \forall x\in\R^n. $$
From part (i) of Theorem \ref{correct_scaling} we know that the sequence $(u_{E_m})_{m=1}^\infty$ is bounded in $H^1$ and the sequence $(\|u_{E_m}\|_{L^2})_{m=1}^\infty$ converges to a nonzero constant, see \eqref{repeated_ref}. Define
$$  \Psi_m = \frac{u_{E_m}}{\|u_{E_m}\|_{L^2}}\m.$$
Like in the proof of part (ii) of Theorem \ref{correct_scaling}, we will apply concentration compactness theory to $(\Psi_m)_{m\in\N}$. We have already established in that proof that vanishing ($\mu=0$) cannot occur. We will show that if either compactness $(\mu=1)$ or splitting $(0<\mu<1)$ occurs, then \eqref{ueh2} and \eqref{need_to_show} hold for a subsequence of $u_{E_m}$.

\noindent
{\bf Compactness ($\mu=1$):} In the proof of part (ii) of Theorem \ref{correct_scaling} we have established that along a subsequence of $(E_m)_{m\in\N}$, again denoted as $(E_m)_{m\in\N}$, $u_{E_m}$ can be written as follows:
\begin{equation}\label{comp-conv1}
  u_{E_m} = \Tscr_{y_m} \hat u_* + \hat v_m.
\end{equation}
Here $\hat u_*$ is a strong, $H^2$ solution of \eqref{uinf} and
\begin{equation} \label{almost_done}
 \lim_{m\to\infty}\|\hat v_m\|_{H^2} = 0.
\end{equation}
Since $(\psi_E,E)$ is a ground state branch, we have $u_{E_m}>0$ for all $m$. It now follows using \eqref{comp-conv1} and \eqref{almost_done} that for any nonnegative $\phi\in C_0^{\infty}(\R^n)$
$$ 0 \leq \liminf_{m\to\infty} \langle u_{E_m}, T_{y_m} \phi \rangle \m=\m \langle \hat u_*,\phi\rangle + \liminf_{m\to\infty}\langle \hat v_m, T_{y_m}\phi\rangle \m=\m \langle \hat u_*,\phi\rangle\m,$$
which implies that $\hat u_*$ is a nonnegative function. Since every nonnegative \cite{kwong} solution of \eqref{uinf} must be a translation of the function $u_\infty$, we have $\hat u_*= \Tscr_{y_*}u_\infty$ for some $y_*\in\R^n$. Redefining $y_m$ to be $y_m+y_*$ we get
\begin{equation}\label{comp-conv2}
  u_{E_m} = \Tscr_{y_m}u_\infty + \hat v_m.
\end{equation}
To complete the proof of this theorem in the compactness case, we will show that the sequence $(R_m y_m)_{m\in\N}$ remains bounded and one of its subsequences, again denoted as $(R_m y_m)_{m\in\N}$, converges to some $x^1\in\R^n$. It will then follow from \eqref{comp-conv2} and \eqref{almost_done} that \eqref{ueh2} and \eqref{need_to_show} hold with $M=1$ and $x_k^1=y_k$. Finally we will show that if the potential $V$ is continuously differentiable in a neighborhood of $x^1$, then $x^1$ must be a critical point of $V$.

To show that $(R_m y_m)_{m\in\N}$ is a bounded sequence, we will assume the contrary, i.e. there exists a subsequence of $(y_m)_{m\in\N}$, also denoted by $(y_m)_{m\in\N}$, such that
\begin{equation} \label{contrad_comp}
  \lim_{m\to\infty}R_m|y_m|=\infty.
\end{equation}

Define the linear operator $L_+(u,R):H^2(\R^n)\times\R\mapsto L^2(\R^n)$ as follows:
\begin{equation} \label{anpan_call}
 L_+(u,R)[v]=(-\Delta+R^2 V_R+1)v +\sigma(2p+1) |u|^{2p}v.
\end{equation}
Note that the definition of $L_+$ here is different from the definition in Section \ref{se:bif} due to the scaling \eqref{uelarge} used to obtain $u_{E_m}$ from $\psi_{E_m}$. Since for the decomposition of $u_{E_m}$ in \eqref{comp-conv2} we have $\lim_{m\to\infty}y_m=\infty$ by assumption \eqref{contrad_comp} and \eqref{almost_done} holds, for reasons similar to those expressed in the case of bifurcation at finite $E$ (see discussion in the paragraph below \eqref{orlov3}), we wish to decompose $u_{E_m}$ with respect to the kernel of $L_+(T_{y_m}u_\infty,R_m)$. Since the kernel of $-\Delta +1+ \sigma(2p+1) T_{y_m} |u_\infty|^{2p}$ approaches the kernel of $L_+(T_{y_m}u_\infty,R_m)$ for large $m$, we will work with the former. The kernel of $-\Delta +1+ \sigma(2p+1) T_{y_m} |u_\infty|^{2p}$ is ${\rm span} \left\{T_{y_m} \partial_{x_k} u_\infty| k=1,2,  \ldots n\right\}$, which is the set of infinitesimal generators of translations of $T_{y_m} u_\infty$ in $\R^n$. Hence the projection onto the kernel can be absorbed into translations of the profile $u_\infty$. To be precise, we apply Lemma \ref{lm_dec_finite_c} to \eqref{comp-conv2} to obtain the following decomposition for sufficiently large $m$:
\begin{equation}\label{compact_soln_newE}
  u_{E_m} = \Tscr_{y_m+s_m}u_{\infty}+v_m = \Tscr_{z_m}u_{\infty}+v_m,
\end{equation}
where $z_m=y_m+s_m$ and $v_m$ satisfies
$$  \langle v_m,\partial_{x_k} T_{y_m+s_m}u_{\infty}\rangle=0\quad\textrm{for
    all} \quad k=1,2,\ldots n.$$
Moreover, as $m\to\infty$ we have $R_m\to 0$, $|s_m|\to 0$ and $R_m|z_m|\to\infty$. Furthermore, since $v_m=\hat v_m+T_{y_m}u_\infty- T_{y_m+s_m}u_\infty$ and \eqref{almost_done} holds and $T_{y_m}u_\infty-T_{y_m+s_m}u_\infty \stackrel {H^2} {\to} 0$ (because $s_m\to0$), we have
$$ \lim_{m\to\infty} \|v_m\|_{H^2} =0.  $$
We assume, with no loss of generality, that the decomposition of $u_{E_m}$ in \eqref{compact_soln_newE} holds for all $m$.

Since $u_{E_m}$ is a solution for \eqref{ue_equation}, using \eqref{compact_soln_newE} we obtain
\begin{equation} \label{eq:ue_comp}
  (-\Delta+R_m^2V_{R_m}+1)(\Tscr_{z_m}u_{\infty}+v_m)+\sigma|\Tscr_{z_m} u_{\infty}+v_m|^{2p}(\Tscr_{z_m}u_{\infty}+v_m)= 0.
\end{equation}
We will analyze the above equation using a Lyapunov-Schmidt decomposition. Our approach here will be identical to the one used in Section \ref{se:bif} and is summarized next. First we obtain a set of equations - an infinite dimensional equation and a set of finite dimensional equations - which is equivalent to \eqref{eq:ue_comp}. We show the existence of a unique solution $v_m$, as a function of $z_m$, to the infinite dimensional equation via contraction mapping. Using this solution we study the set of finite dimensional equations to obtain the contradiction that $u_{E_m}=\Tscr_{y_m} u_\infty + \hat v_m$ does not solve \eqref{ue_equation} if \eqref{contrad_comp} holds. Hence \eqref{contrad_comp} cannot hold.

Recall the linear operator $L_+(u,R)$ from \eqref{anpan_call} and the
nonlinear function $N(U,v):H^2(\R^n)\times H^2(\R^n)\mapsto L^2(\R^n)$ from \eqref{bound_N}:
$$ N(U,v)=\sigma|U+v|^{2p}(U+v)-\sigma|U|^{2p}U-\sigma(2p+1)|U|^{2p}v.$$
Using $L_+$ and $N$, we can rewrite \eqref{eq:ue_comp} as
\begin{equation}\label{eq:bif_Elarge}
  L_+(\Tscr_{z_m}u_\infty,R_m)[v_m] + R_m^2 V_{R_m} \Tscr_{z_m}u_\infty+ N(\Tscr_{z_m}u_\infty, v_m)=0.
\end{equation}
Consider the following projection operator $P^\perp_m$ which projects onto the orthogonal complement of the subspace
$$\textrm{ker}(L_+(\Tscr_{z_m}u_\infty,R_m)-R_m^2 V_{R_m})=\textrm{span}\left\{\Tscr_{z_m} \partial_{x_k} u_\infty\big| k\in\Iscr \right\}$$
in  $L^2(\R^n)$:
$$P^\perp_m \phi = \phi - \sum_{k=1}^n \langle \phi, \Tscr_{z_m} \partial_{x_k} u_\infty\rangle \frac{\Tscr_{z_m}\partial_{x_k} u_\infty} {\|\Tscr_{z_m} \partial_{x_k} u_\infty\|^2_{L^2}} \qquad \forall \phi \in L^2(\R^n).$$
Let $P^\perp_m L^2(\R^n)$ denote the range of $P^\perp_m$. Clearly \eqref{eq:bif_Elarge} is equivalent to the set of equations
\begin{equation}\label{perp_eq_Elarge}
  P^\perp_m L_+(\Tscr_{z_m}u_\infty,R_m)[v_m] + P^\perp_m R_m^2 V_{R_m}\Tscr_{z_m}u_\infty + P^\perp_m N(\Tscr_{z_m}u_\infty,v_m)=0,
\end{equation}
which is an infinite dimensional equation and
\begin{equation} \label{parallel_eq_Elarge}
  \langle\Tscr_{z_m}\partial_{x_k}u_\infty,R_m^2 V_{R_m}\Tscr_{z_m} u_\infty +R_m^2 V_{R_m} v_m+N(\Tscr_{z_m}u_\infty,v_m)\rangle=\ 0\qquad \forall
  k\in\Iscr,
\end{equation}
which is a set of finite dimensional equations. Clearly the operator
$$ P^\perp_m(-\Delta +1+\sigma|\Tscr_{z_m}u_\infty|^{2p})P^\perp_m: H^2(\R^n)\cap P^\perp_m L^2(\R^n)\mapsto P^\perp_m L^2(\R^n),$$
has a bounded inverse and the norm of the inverse operator is independent of $m$. The operators $P^\perp_m L_+(\Tscr_{z_m}u_\infty, R_m)P^\perp_m$ and $P^\perp_m (-\Delta+1+\sigma|\Tscr_{z_m}u_\infty|^{2p})P^\perp_m$ differ by $P^\perp_m R_m^2V(R_m x) P^\perp_m$, and $\|P^\perp_m R_m^2V(R_m x) P^\perp_m\|_{H^2\mapsto L^2}\to0$ as $m\to\infty$. Hence for sufficiently large $m$, $P^\perp_m
L_+(\Tscr_{z_m}u_\infty,R_m)P^\perp_m$ is boundedly invertible with a bound independent of $m$, i.e.
$$\|[P_m^{\perp}L_+(\Tscr_{z_m}u_\infty,R_m) P_m^{\perp}]^{-1}\|_{L^2\mapsto H^2} \leq K.$$
Since $P^\perp_m v_m=v_m$, \eqref{perp_eq_Elarge} can be written as
\begin{equation}\label{perp_eq_cont_Elarge}
  v_m=[P^\perp_m L_+(\Tscr_{z_m}u_\infty,R_m)P^\perp_m]^{-1}P^\perp_m(R_m^2 V_{R_m}\Tscr_{z_m} u_\infty+N(\Tscr_{z_m}u_\infty,v_m)).
\end{equation}
Fix $\delta>0$ and $M\in\N$ such for all $m>M$
\begin{equation} \label{sunday_lunch}
 \|v_m\|_{H^2}\leq\delta \quad\textrm{and}\quad KC(\|u_\infty\|_{H^2})
 2\delta\leq\frac{1}{2}.
\end{equation}
Here $C(\|u_\infty\|_{H^2})$ is the constant in \eqref{est_nl} with $L=\|u_\infty\|_{H^2}$. It now follows via the contraction mapping principle that there exists a unique solution $v_m$ to \eqref{perp_eq_cont_Elarge} for all $m>M$ such that
\begin{equation}\label{eq:estimate_drifting_vR}
  \|v_m\|_{H^2}\leq 2KR_m^2\|V_{R_m}T_{z_m}u_\infty\|_{L^2}.
\end{equation}
We will assume, with no loss of generality, that a $v_m$ satisfying \eqref{eq:estimate_drifting_vR} exists for all $m$.

Using the solution $v_m$ of \eqref{perp_eq_Elarge}, regarded as a function of $z_m$, we will now consider the set of finite dimensional equations \eqref{parallel_eq_Elarge} to be equations in the unknown $z_m$ and show that under our assumption \eqref{contrad_comp}, these equations have no solution.

Let $x_m$ denote the radial direction at $z_m$. For each $m\in\N$, consider the following finite dimensional equation obtained as a linear combination of the equations in \eqref{parallel_eq_Elarge}:
\begin{equation}\label{eq:drifting}
  \langle \Tscr_{z_m}\partial_{x_m}u_{\infty}, R_m^2 V_{R_m}\Tscr_{z_m} u_{\infty}\rangle + \langle \Tscr_{z_m}\partial_{x_m}u_{\infty}, R_m^2 V_{R_m}v_m+N(\Tscr_{z_m}u_{\infty},v_m)\rangle = 0.
\end{equation}
We will show that, under our hypothesis on the behavior of the potential $V(x)$ as $|x|\to\infty$, the dominant term in \eqref{eq:drifting} as $R_m z_m\to\infty$ is  $ \langle \Tscr_{z_m}\partial_{x_m}u_{\infty},R_m^2 V_{R_m}\Tscr_{z_m} u_{\infty}\rangle$ and by dividing the left hand side of \eqref{eq:drifting} with it and taking the limit as $m\to\infty$, we get the contradiction that $1=0$.

Fix $\theta\in(0,\pi/2)$ such that $2\Cscr\tan\theta<1$, where $\Cscr$ is the constant in the hypothesis \eqref{Hypo_largeE3}. Let $B_m$ be the closed ball in $\R^n$ with center $z_m$ and radius $|z_m|\sin\theta$. Then for the first term in \eqref{eq:drifting}
$$ -2\langle\Tscr_{z_m}\partial_{x_m}u_\infty,R_m^2 V_{R_m} \Tscr_{z_m} u_\infty\rangle = R_m^3\int\limits_{B_m\cup\, \R^n\setminus B_m}\!\!\!\! u^2_\infty(x-z_m) \partial_{x_m} V(R_m x) \mathrm{d}x.$$
Since there exist constants  $C>0$ and $\gamma\in(0,1/2)$ such that $|u_\infty(x)| <Ce^{-\sqrt{1-\gamma}|x|}$ and $|\partial_{x_k}u_\infty(x)| <Ce^{-\sqrt{1-\gamma}|x|}$ for all $x\in\R^n$, it follows that for sufficiently large $m$
\begin{align}
  \Bigg|R_m^3\int\limits_{\R^n\setminus B_m} u^2_\infty(x-z_m)\partial_{x_m}V(R_m x)
  \mathrm{d}x\Bigg| &\leq CR_m^3\|\nabla V\|_{L_{\infty}}\int\limits
  _{\R^n\setminus B_m} e^{-2\sqrt{1-\gamma}|x-z_m|}\mathrm{d}x \nonumber\\
  &\leq CR_m^3 e^{-2\sqrt{1-2\gamma}|z_m|\sin\theta} \label{eq:v'exp_decay}.
\end{align}
For each $x\in B_m$, let $\{\theta_i(x)\big| i\in\Iscr\}$ with  $|\theta_i(x)|<\theta$ be the set of angles using which the derivative of the potential along the direction $x_m$ can be written as a linear combination of the derivatives along the radial and tangential directions at $x$, i.e.
$$ \partial_{x_m}V(R_m x) = \partial_{r}V(R_m x)\cos\theta_n(x) +
   \sum\limits_{i=1}^{n-1}\partial_{s_i}V(R_m x)\sin\theta_i (x),$$
where $\partial_r V$ and $\partial_{s_i} V$ are the radial and tangential
derivatives introduced in the theorem. The function $\partial_{x_m} V(R_m x)$ has the same sign a.e. in $B_m$ for large $m$. Indeed, since $\lim_{m\to\infty}R_m z_m=\infty$, for almost every $x\in B_m$ we can use the hypothesis in \eqref{Hypo_largeE2} and \eqref{Hypo_largeE3} to get
\begin{align}
  \left|\partial_{x_m}V(R_m x)\right| &= \Big|\partial_r V(R_m x)  \cos\theta_n(x) +\sum\limits_{i=1}^{n-1}\partial_{s_i}V(R_m x)\sin\theta_i
  (x)\Big| \nonumber\\
  &\geq \cos\theta\left|\partial_r V(R_m x)\right|-\Big(\sum\limits_{i=1}^{n-1} \left|\partial_{s_i}V(R_m x)\right|^2\Big)^\frac{1}{2} \Big(\sum\limits_{i=1}^{n-1} \sin^2\theta_i(x)\Big)^\frac{1}{2} \nonumber\\
  &\geq\cos\theta\left|\partial_r V(R_m x)\right|-\Cscr\sin\theta\left|\partial_r V(R_m x) \right|\label{est_dvxinf}\\
  \implies\ \left|\partial_{x_m} V(R_m x)\right| &> C e^{-\beta R_m|x|}.
  \nonumber
\end{align}
Using the above estimate it follows that for $m$ large
\begin{align}
  \Big|R_m^3\int\limits_{B_m} u^2_\infty(x-z_m)\partial_{x_m}V(R_m x)
  \mathrm{d}x\Big| &\geq CR_m^3 \int\limits_{B_m} u^2_\infty(x-z_m)
  e^{-\beta R_m|x|} \mathrm{d}x \\
  &\geq C R_m^3 e^{-2\beta R_m|z_m|}, \label{useinsplit}
\end{align}
which along with \eqref{eq:v'exp_decay} gives
\begin{equation}\label{eq:v'estimate}
  \lim_{m\to\infty}\frac{\Big|\int\limits_{\R^n\setminus B_m}u^2_\infty(x-z_m)
  \partial_{x_m}V(R_m x)\mathrm{d}x\Big|}{\Big|\int\limits_{B_m}
  u^2_\infty(x-z_m)\partial_{x_m}V(R_m x)\mathrm{d}x\Big|} = 0.
\end{equation}
For the remaining terms in \eqref{eq:drifting} we have the following estimates
using \eqref{est_nl} and \eqref{eq:estimate_drifting_vR}:
\begin{align}
  &\ |\langle\Tscr_{z_m}\partial_{x_m}u_\infty,R_m^2 V_{R_m}v_m+ N(\Tscr_{z_m}u_\infty,v_m)\rangle| \nonumber\\
  \leq&\ R_m^2\|\partial_{x_m} V_{R_m}\Tscr_{z_m} u_\infty\|_{L^2}\|v_m\|_{L^2} + R_m^2\|V_{R_m}\Tscr_{z_m} u_\infty\|_{L^2}\|v_m\|_{H^1} +
  \|\Tscr_{z_m}\partial_{x_m}u_\infty\|_{L^2}\|N(\Tscr_{z_m}u_\infty,v_m)\|_{L^2}\nonumber\\
  \leq&\ C R_m^4\big(\|\partial_{x_m} V_{R_m} \Tscr_{z_m}u_\infty\|_{L^2}^2 + \|V_{R_m} \Tscr_{z_m}u_\infty \|_{L^2}^2\big). \label{bimamvas}
\end{align}
Now as in \eqref{eq:v'exp_decay}, using the decay estimates for $|u_\infty|$ and $|\partial_{x_i} u_\infty|$, we get
\begin{equation}\label{eq:vsqr_exp1}
  \Big|\int\limits_{\R^n\setminus B_m} u^2_\infty(x-z_m)V^2(R_m x) \mathrm{d}x\Big|
  < Ce^{-2\sqrt{1-2\gamma}|z_m|\sin\theta},
\end{equation}
\begin{equation}\label{eq:vsqr_exp2}
  \Big|\int\limits_{\R^n\setminus B_m} u^2_\infty(x-z_m) (R_m \partial_{x_m} V(R_m x))^2 \mathrm{d}x\Big| < Ce^{-2\sqrt{1-2\gamma}|z_m|\sin\theta}.
\end{equation}
Since $\partial_{x_m}V(R_m x)$ does not change sign in $B_m$ (excluding a set of measure 0), using \eqref{Hypo_largeE1} along with \eqref{est_dvxinf} gives us the following estimates:
\begin{align}
  &\lim_{m\to\infty}\frac{R_m\Big|\int\limits_{B_m}u^2_\infty(x-z_m)V^2(R_m x)
  \mathrm{d}x\Big|}{\Big|\int\limits_{B_m} u^2_\infty(x-z_m) \partial_{x_m}
  V(R_m x)\mathrm{d}x\Big|}\nonumber\\
  =&\lim_{m\to\infty}\frac{\Big|\int\limits_{B_m} u^2_\infty(x-z_m)\partial
  _{x_m}V(R_m x) [ R_m \frac{V^2(R_m x)}{\partial_{x_m}V(R_m x)}]\mathrm{d}x\Big|}
  {\Big|\int\limits_{B_m} u^2_\infty(x-z_m)\partial_{x_m}V(R_m x)\mathrm{d}x
  \Big|}=0\label{eq:vsqr_v'1}
\end{align}
and
\begin{align}
  &\lim_{m\to\infty}\frac{R_m^3\Big|\int\limits_{B_m} u_\infty^2(x-z_m) (\partial_{x_m} V(R_m x))^2 \mathrm{d}x\Big|}{\Big|\int\limits_{B_m} u^2_
  \infty(x-z_m)\partial_{x_m}V(R_m x)\mathrm{d}x\Big|}\nonumber\\
  =&\lim_{m\to\infty}\frac{\Big|\int\limits_{B_m} u^2_\infty(x-z_m)\partial
  _{x_m} V(R_m x) [R_m^3 \partial_{x_m} V(R_m x)] \mathrm{d}x\Big|} {\Big|\int\limits_{B_m} u^2_\infty(x-z_m)\partial_{x_m} V(R_m x) \mathrm{d}x\Big|} = 0.
  \label{eq:vsqr_v'2}
\end{align}
In deriving the last limit we have used the assumption $\nabla V\in L^\infty(\R^n)$. Using the estimates in \eqref{eq:v'exp_decay}, \eqref{eq:v'estimate}, \eqref{eq:vsqr_exp1}, \eqref{eq:vsqr_exp2}, \eqref{eq:vsqr_v'1} and \eqref{eq:vsqr_v'2} we see that the first term in
\eqref{eq:drifting}
$$\langle\Tscr_{z_m}\partial_{x_m}u_\infty,R_m^2 V_{R_m}\Tscr_{z_m} u_\infty
  \rangle$$
cannot be canceled by the remaining terms in \eqref{eq:drifting} when $m$
is large. Therefore it follows that there are no solutions to
\eqref{eq:ue_comp} such that the sequence $(R_m z_m)_{m\in\N}$, and consequently $(R_m y_m)_{m\in\N}$, is unbounded. Hence \eqref{contrad_comp} cannot hold and so $(R_m y_m)_{m\in\N}$ must be a bounded sequence.

Consider a subsequence of $(R_m y_m)_{m\in\N}$, again denoted as $(R_m y_m)_{m\in\N}$, converging to $x^1\in\R^n$. From \eqref{compact_soln_newE} and \eqref{almost_done} it follows that \eqref{ueh2} and \eqref{need_to_show} holds with $M=1$ and $x_k^1=y_k$. Suppose that $V$ is continuously differentiable at $x^1$. We will now complete the proof of the theorem in the compactness case by showing $x^1$ is a critical point of the potential. Indeed, for each $k\in\Iscr$, \eqref{parallel_eq_Elarge} can be written using \eqref{est_nl} and \eqref{eq:estimate_drifting_vR} (via calculations similar to those used to derive \eqref{bimamvas}) as
$$ \int_{\R^n} \partial_{x_k} V(R_m x+R_m z_m) u^2_\infty(x)   \dd x = O\big(R_m \left\|V_{R_m} T_{z_m} u_\infty\right\|_{L^2}^2 + R_m \left\|V_{R_m} T_{z_m} \partial_{x_1}u_\infty\right\|_{L^2}^2\big). $$
For every $x\in\R^n$, we have $\partial_{x_k} V(R_m x+R_m z_m) u^2_\infty(x)\to \partial_{x_k} V(x^1) u^2_\infty(x)$ as $m\to\infty$ and $|\partial_{x_1} V(R_m x+R_m z_m) u^2_\infty(x)|\leq \|\nabla V\|_{L^\infty}u^2_\infty(x)$ for all $m$. Since $\|\nabla V\|_{L^\infty}u^2_\infty\in L^1(\R^n)$, it follows using the Lebesgue dominated convergence theorem that as $m\to\infty$ the left hand side of the above equation converges to $\partial_{x_k} V(x^1)\|u_\infty\|_{L^2}^2$.
Clearly the right hand side of the above equation converges to 0. Hence we have
$$ \partial_{x_k} V(x^1)\|u_\infty\|_{L^2}^2 \m=\m0 \FORALL k\in\Iscr \m,$$
i.e., $x^1$ is a critical point of the potential.

\noindent
{\bf Splitting ($0<\mu<1$):} In the proof of part (ii) of Theorem \ref{correct_scaling} we have established in the splitting case that along a subsequence of $(E_m)_{m\in\N}$, again denoted as $(E_m)_{m\in\N}$, $u_{E_m}$ can be written as follows:
\begin{equation}\label{PQWL_split}
  u_{E_m} \m=\m \sum_{j=1}^{d} \Tscr_{y^j_m}u_*^j+u_*^0+ \hat v_m \m.
\end{equation}
Here $u_*^j$ is a strong solution of \eqref{uinf} for each $j\in \Jscr= \{0,1,\ldots d\}$,
\begin{equation} \label{almost_done_split}
 \lim_{m\to\infty}\|\hat v_m\|_{H^2} = 0,
\end{equation}
and as $m\to\infty$
\begin{equation} \label{GNWL_split}
  |y^j_m|\to\infty\quad \textrm{and} \quad |y^j_m-y^k_m|\to\infty\ \FORALL j, k\in \{1,2,\ldots d\},\ j\neq k.
\end{equation}
Concentration compactness theory does not imply that $u_*^0\neq 0$, but we assume it only to simplify our presentation. All the arguments in this section remain valid if $u_*^0 = 0$. Since $(\psi_E,E)$ is a ground state branch, we have $u_{E_m}>0$ for all $m$. It now follows using \eqref{PQWL_split} and \eqref{almost_done_split} that for any nonnegative $\phi\in C_0^{\infty} (\R^n)$ and any $j\in\Jscr$, using the notation $y_m^0=0$, we have
$$ 0 \leq \liminf_{m\to\infty} \langle u_{E_m}, T_{y^j_m} \phi \rangle \m=\m \langle u_*^j,\phi\rangle + \liminf_{m\to\infty}\langle \hat v_m, T_{y_m}\phi\rangle \m=\m \langle u_*^j,\phi\rangle\m,$$
which implies that $\hat u_*^j$ is a nonnegative function. Since every nonnegative solution of \eqref{uinf} must be a translation of the function $u_\infty$, see \cite{kwong}, we have $u_*^j= \Tscr_{y_*^j}u_\infty$ for some $y_*^j\in\R^n$. Redefining $y_m^j$ to be $y_m^j+y_*^j$ for each $j\in\Jscr$ we get the following different decomposition for $u_{E_m}$:
\begin{equation}\label{new_split_soln_inf}
  u_{E_m} = \sum_{j=1}^{d} \Tscr_{y^j_m}u_\infty+\Tscr_{y_m^0}u_\infty+
  \hat v_m.
\end{equation}
Here $y_m^0\in\R^n$ is a constant and the $y^j_m$s satisfy \eqref{GNWL_split}.

To complete the proof of this theorem in the splitting case, we will show that the sequence $(R_m y_m^j)_{m\in\N}$ remains bounded for all $j\in\Jscr$ and one of its subsequences, again denoted as $(R_m y_m^j)_{m\in\N}$, converges to some $x^j\in\R^n$. It will then follow from \eqref{new_split_soln_inf} and \eqref{almost_done_split}, that \eqref{ueh2} and \eqref{need_to_show} hold with $M=d+1$ and $x_k^1=y_k$. Finally we will show that if the potential $V$ is continuously differentiable in a neighborhood of $x^j$, then $x^j$ must be a critical point of $V$.

To show that $(R_m y_m^j)_{m\in\N}$ is a bounded sequence for each $j\in\Jscr$, we will assume the contrary i.e., we suppose that there exists a $0<d_0<d$ such that $(R_m y_m^j)_{m\in\N}$ is a bounded sequence for each $j\in\Jscr_0= \{0,1,\ldots d_0\}$ and for each $j\in\Jscr\setminus\Jscr_0$ the sequence $(y_m^j)_{m\in\N}$ has a subsequence, also denoted by $(y_m^j)_{m\in\N}$, such that
\begin{equation} \label{contrad_split}
  \lim_{m\to\infty}R_m|y_m^j|=\infty \FORALL j\in\Jscr\setminus\Jscr_0.
\end{equation}

Recall the linear operator $L_+(u,R):H^2(\R^n)\times\R\mapsto L^2(\R^n)$ from \eqref{anpan_call}:
$$ L_+(u,R)[v]=(-\Delta+R^2 V_R+1)v +\sigma(2p+1) |u|^{2p}v. $$
Since for the decomposition of $u_{E_m}$ in \eqref{new_split_soln_inf} we have $\lim_{m\to\infty}y_m^j=\infty$ and $\lim_{m\to\infty}|y_m^j-y_m^k|=\infty$ for $j,k\in\{1,2,\ldots d\}$ and \eqref{almost_done_split} holds, for reasons similar to those expressed in the case of bifurcation at finite $E$ (see discussion in the paragraph below \eqref{yuri5}), we wish to decompose $u_{E_m}$ with respect to the kernel of $L_+(\sum_{j=0}^\infty T_{y_m}u_\infty,R_m)$. From the theory of Schr\"odinger operators with potentials separated by a large distance, the kernel of $L_+(\sum_{j=0}^\infty T_{y_m}u_\infty,R_m)$ approaches the union of the kernels of $-\Delta +1+ \sigma(2p+1) T_{y_m^j} |u_\infty|^{2p}$, $j\in\Jscr$, for large $m$. We will therefore work with the latter. The kernel of $-\Delta +1+ \sigma(2p+1) T_{y_m^j} |u_\infty|^{2p}$ is ${\rm span} \left\{T_{y_m^j} \partial_{x_k} u_\infty \m|\m k=1,2, \ldots n\right\}$, which is the set of infinitesimal generators of translations of $T_{y_m^j} u_\infty$ in $\R^n$. Hence the projection onto this kernel can be absorbed into translations of the profile $u_\infty$. To be precise, we apply Proposition \ref{lm:newdec_gen} (also see Remark \ref{decay_not needed}) with $\psi=0$, $\phi_i=0$ for all $i$, $u_j= u_\infty$ and $y_j=y_m^j$ for $j\in\{0,1,\ldots d\}$ to $u_{E_m}$ in \eqref{new_split_soln_inf} to conclude, using the limits $\hat v_m\stackrel{H^2}{\rightarrow}0$ and $|y_m^j-y_m^k| \rightarrow \infty$ as $m\to\infty$, that for sufficiently large $m$:
\begin{equation} \label{eq:split_soln_new}
  u_{E_m}=\sum_{j=0}^d \Tscr_{y_m^j+s_m^j} u_{\infty} + v_m=\sum_{j=0}^d \Tscr_{z_m^j} u_{\infty} + v_m,
\end{equation}
where $z_m^j=y_m^j+s_m^j$ and  $v_m$ satisfies
$$  \big\langle v_m, \Tscr_{z_m^j}\partial_{x_k} u_{\infty}\big\rangle= 0\qquad \forall k\in\Iscr,\quad \forall j\in\Jscr.$$
Moreover, as $m\to\infty$ we have $R_m\to 0$, $|s_m^j|\to 0$ for $j\in\Jscr$, $R_m|z_m^j|$ remains bounded for $j\in\Jscr_0$ and $R_m|z_m^j|\to\infty$ for $j\in\Jscr\setminus\Jscr_0$. Furthermore, since $v_m=\hat v_m+\sum_{j=0}^d T_{y_m^j}u_\infty-\sum_{j=0}^d T_{y_m^j+s_m^j}u_\infty$ and \eqref{almost_done_split} holds and $\sum_{j=0}^dT_{y_m^j}u_\infty-\sum_{j=0}^d T_{y_m^j+s_m^j}u_\infty \stackrel {H^2} {\to} 0$ (because $s_m^j\to0$), we have
$$ \lim_{m\to\infty} \|v_m\|_{H^2} =0. $$
We assume, with no loss of generality, that the decomposition of $u_{E_m}$ in \eqref{eq:split_soln_new} holds for all $m$.

The arguments presented next to establish that the assumption in \eqref{contrad_split} cannot hold are a complex extension of those used to show that $(R_m z_m)_{m\in\N}$ is bounded in the compactness case. Define
\begin{equation} \label{defn_functions}
 \tilde u_m=\sum_{j=0}^d \Tscr_{z_m^j}u_{\infty}, \qquad \tilde u_m^0=\sum_{j=0}^{d_0} \Tscr_{z_m^j}u_{\infty}, \qquad \tilde u_m^1=\sum_{j=d_0+1}^d \Tscr_{z_m^j}u_{\infty}.
\end{equation}
Since $u_{E_m}$ solves \eqref{ue_equation}, we substitute its decomposition in \eqref{eq:split_soln_new} into \eqref{ue_equation} to get
\begin{equation}\label{splteq_inf}
  (-\Delta +1+R_m^2 V_{R_m})(\tilde u_m+v_m) + \sigma\left|\tilde
   u_m+v_m \right|^{2p}(\tilde u_m+v_m)(x) = 0.
\end{equation}
As in the compactness case, we will analyze the above equation using the Lyapunov-Schmidt decomposition. We first derive an equivalent family of equations: an infinite dimensional equation and a set of finite dimensional equations. We then solve the infinite dimensional equation for $v_m$ as a function of the $z_m^j$s and derive some useful estimates for its decay. Using this solution and its decay estimates we study the set of finite dimensional equations to derive the contradiction that if \eqref{contrad_split} holds, then \eqref{eq:split_soln_new} does not solve \eqref{ue_equation}. Therefore $\{R_m z_m^j\}_{m\in\N}$ must be a bounded sequence for all $j\in\Jscr$.

Recall the linear operator $L_+(u,R):H^2(\R^n)\times\R\mapsto L^2(\R^n)$ from \eqref{anpan_call} and the nonlinear operator  $N(U,v):H^2(\R^n) \times H^2(\R^n)\mapsto L^2(\R^n)$ from \eqref{bound_N},
$$ N(U,v)=\sigma|U+v|^{2p}(U+v)-\sigma|U|^{2p}U-\sigma(2p+1)|U|^{2p}v,$$
which satisfies the bound in \eqref{est_nl}. Using these operators we can rewrite
\eqref{splteq_inf} as
\begin{equation}
  L_+(\tilde u_m,R_m)[v_m]+N(\tilde u_m,v_m)+R_m^2 V_{R_m} \sum _{j=0}^d
  \Tscr_{z_m^j}u_{\infty}+\sigma\left|\tilde u_m \right|^{2p}\tilde u_m-\sigma\sum_{j=0}^d
  \Tscr_{z_m^j}|u_{\infty}|^{2p}u_{\infty}=0.\label{eq:stationary_split_Elarge}
\end{equation}
Define the operator
$P^{\perp}_{m}$ on $L^2(\R^n)$ to be the orthogonal projection onto the
subspace
$$  \big\{\Tscr_{z_m^j}\partial_{x_k}u_\infty \big| k\in\Iscr, j\in\Jscr \big\}^{\perp}.$$
Note that the definitions of $P_m^\perp$ here is different from the compactness case. For simplicity of notation, let us rewrite \eqref{eq:stationary_split_Elarge} as
$$  F(v_m, z_m)=0,$$
where $z_m$ is the set $\{z_m^0,z_m^1,\ldots z_m^d\}$. This equation is equivalent to the following set of equations:
\begin{equation}\label{eq:infcom_inf}
  P_m^\perp F(v_m,z_m)=0,
\end{equation}
which is an infinite dimensional equation to be solved for $v_m$ as a
function of $z_m$, and
\begin{equation}\label{eq:fincom_inf}
  \langle F(v_m, z_m), \partial_{x_k}u_\infty(x-z_m^j)\rangle=0,
  \quad j\in\Jscr, \quad k\in\Iscr,
\end{equation}
which is a set of finite dimensional equations to be solved for $z_m$ using the solution $v_m$ of \eqref{eq:infcom_inf}.

We will first solve the infinite dimensional equation \eqref{eq:infcom_inf}. When $m$ is sufficiently large, $R_m$ is sufficiently small and $|z_m^j|$ and $|z_m^j-z_m^k|$ are sufficiently large for all $j,k\in\Jscr$. We can write $L_+(\tilde u_m,R_m)=\widetilde L_m+\widetilde W_m$, where
$$ \widetilde L_m=-\Delta+1 +\sigma(2p+1)\sum_{j=0}^d \big|\Tscr_{z_m^j} u_{\infty}\big|^{2p}, \qquad  \widetilde W_m= R_m^2 V_{R_m} +\sigma(2p+1) |\tilde u_m|^{2p}-\sigma(2p+1)\sum_{j=0}^d \big|\Tscr_{z_m^j} u_{\infty}\big|^{2p}.$$
Clearly $\|\widetilde W_m\|_{H^2\mapsto L^2}\to0$ as $m\to\infty$. Applying Proposition \ref{prop:MS} in the Appendix to $\widetilde L_m$, with $R=R_m$, $L_{R_m}=\widetilde L_m$, $E(R_m)=1$, $s_k(R_m)=z_m^k$ and $V_k=\sigma(2p+1)|u_\infty|^{2p}$, and then applying the spectral perturbation theory to $\widetilde L_m+\widetilde W_m$ by regarding $\widetilde W_m$ as a perturbation, it follows that the operator $P_{m}^{\perp}L_+(\tilde u_m,R_m)P_{m}^{\perp}:H^2(\R^n)\mapsto L^2(\R^n)$ has a bounded inverse and the norm of the inverse operator can be bounded uniformly in $m$, for large $m$. Hence for sufficiently large $m$, using \eqref{eq:stationary_split_Elarge} and $P^{\perp}_{m}v_m=v_m$, we can rewrite the infinite dimensional equation \eqref{eq:infcom_inf} as
\begin{equation}   \label{vm_com_inf}
  v_m = -(P_m^\perp L_+(\tilde u_m,R_m)P_m^\perp)^{-1}\Big(N(\tilde u_m,v_m)+R_m^2
  V_{R_m}\tilde u_m +\sigma \left|\tilde u_m \right|^{2p}\tilde u_m-\sigma
  \sum_{k=0}^d\Tscr_{z_m^k}|u_{\infty}|^{2p}u_{\infty}\Big).
\end{equation}
Like in the case of compactness, using \eqref{est_nl}, the smallness of $\|v_m\|_{H^2}$ and the estimates
\begin{equation}  \label{vsmurthy}
  |u_\infty(x)| \m<\m Ce^{-\sqrt{1-\gamma}|x|} \FORALL x\in\R^n,
\end{equation}
for some $C>0$ and $\gamma\in(0,1)$, the existence of a unique solution $v_m$ to \eqref{vm_com_inf} can be established by applying the contraction mapping principle.

While analyzing the finite dimensional equations \eqref{eq:fincom_inf} it
will be useful to have estimates for the decay of $\|v_m\|_{H^2}$ explicitly in terms of $z^j_m$s. Indeed, like in the compactness case (see \eqref{eq:estimate_drifting_vR}), we can easily derive that for some $C>0$
\begin{equation} \label{prabhat}
  \|v_m\|_{H^2}\leq CR_m^2\sum_{j=0}^d\|V_{R_m}T_{z_m^j}u_\infty\|_{L^2}.
\end{equation}
However this estimate is not useful since $\|V_{R_m}T_{z_m^j} u_\infty\|_{L^2}$ need not decay to 0 as $z_m^j\to\infty$ for $j\in\Jscr_0$. This is a consequence of $(R_mz_m^j)_{m\in\N}$ being bounded for $j\in\Jscr_0$. Hence we will express $v_m$ as
$$ v_m=P_m^\perp v_m^0 + v_m^1. $$
Here $v_m^0$ will be associated with  $z^j_m$'s for $j\in\Jscr_0$ and $v_m^1$ will be associated with  $z^j_m$'s for $j\in\Jscr\setminus\Jscr_0$. For $v_m^0$ we will derive certain exponential in space decay estimates (see \eqref{vm0est}) and for the $v_m^1$ we will derive an estimate like \eqref{prabhat} (see \eqref{eq:vr1_estimate}). Using these estimates we will analyze the equations \eqref{eq:fincom_inf} to show that our assumption \eqref{contrad_split} leads to a contradiction.

We first derive the estimate \eqref{vm0est}. Let $P_0^\perp$ be the projection in $L^2(\R^n)$ onto its subspace
\begin{equation} \label{twosigns}
  S_m=\left\{\Tscr_{z_m^j}\partial_{x_k} u_{\infty}\big|\,j\in\Jscr_0,\,
  k\in\Iscr\right\}^\perp.
\end{equation}
Recall $\tilde u_m^0$ from \eqref{defn_functions}. Consider the equation
\begin{equation}\label{eq:v0_Elarge}
  P_0^\perp \Big[(-\Delta +1 + R_m^2 V_{R_m})\big(\tilde u_m^0+v_m^0 \big) +\sigma |\tilde u_m^0+v_m^0|^{2p}\big(\tilde u_m^0 +v_m^0\big)\Big]=0.
\end{equation}
We wish to find a $v_m^0\in H^2(\R^n)$ that solves this equation. For sufficiently large $m$, the linear operator $F_L:S_m\cap H^2(\R^n)\mapsto S_m\cap L^2(\R^n)$ defined as
$$  F_L(v)=-P_0^\perp\Big[(-\Delta +1+R_m^2V_{R_m})v + \sigma
    (2p+1)|\tilde u_m^0|^{2p}v\Big]$$
is invertible with the norm of its inverse operator bounded uniformly in $m$. This can be established using Proposition \ref{prop:MS} in the Appendix and the spectral perturbation theory (see, for instance, the discussion above \eqref{vm_com_inf}). Hence we can write \eqref{eq:v0_Elarge} as
\begin{align}
  v_m^0 =&\ F_L^{-1}P_0^\perp \Big(R_m^2 V_{R_m}\tilde u_m^0 + \sigma
  |\tilde u_m^0|^{2p}\tilde u_m^0 - \sigma\sum_{j=0}^{d_0} \Tscr_{z_m^j} |u_\infty|^{2p}u_\infty + N(\tilde u_m^0,v_m^0)\Big) \label{FLeqnvm0}\\
  \equiv&\ F_N(v_m^0). \nonumber
\end{align}
For $m$ sufficiently large and $\delta$ sufficiently small, the map $F_N$ is a contraction on the space $S_m\cap H^2(\R^n)\cap B_\delta$, where $B_\delta$ is a closed ball of radius $\delta$ around the origin in $H^2(\R^n)$. Hence for each $m$ large, a $v_m^0\in S_m\cap H^2(\R^n)$ solving \eqref{eq:v0_Elarge} exists via the contraction mapping theorem and
\begin{equation} \label{Mramsmy}
 \|v_m^0\|_{H^2}=O(R_m^2)+\sum_{k\neq j,\ j,k=0}^{d_0}O(\exp^{-\sqrt{1-
 \gamma}|z_m^j-z_m^k|}).
\end{equation}
Our assumption $p\geq1/2$, along with the constraint $0<p<2/(n-2)$, implies that $n\leq5$. We will now show that $\lim_{m\to\infty}\|v_m^0\|_{W^{2,4}}=0$, which will imply that $v_m^0$ is a continuous function and
\begin{equation} \label{reg_vm0}
 \lim_{m\to\infty}\|v_m^0\|_{L^\infty} = 0.
\end{equation}
From \eqref{FLeqnvm0} we have
$$ v_m^0=-P_0^\perp(-\Delta+1)^{-1}P_0^\perp\Big[R_m^2 V_{R_m} \tilde u_m^0 + (R_m^2 V_{R_m}+\sigma(2p+1) |\tilde u_m^0|^{2p})v_m^0$$
$$\hspace{40mm} +\sigma  |\tilde u_m^0|^{2p}\tilde u_m^0 - \sigma\sum_{j=0}^{d_0} \Tscr_{z_m^j} |u_\infty|^{2p}u_\infty + N(\tilde u_m^0,v_m^0)\Big]. $$
Since $(-\Delta+1)^{-1}: L^4\mapsto W^{2,4}$ is a bounded operator, it follows from the above expression that to establish \eqref{reg_vm0}, it suffices to show that $\lim_{m\to\infty}\|R_m^2 V_{R_m} \tilde u_m^0 + (R_m^2 V_{R_m} + \sigma(2p+1) |\tilde u_m^0|^{2p})v_m^0\|_{L^4}=0$, $\lim_{m\to\infty}\||\tilde u_m^0|^{2p}\tilde u_m^0 - \sigma\sum_{j=0}^{d_0} \Tscr_{z_m^j} |u_\infty|^{2p} u_\infty\|_{L^4}=0$ and $\lim_{m\to\infty}\|N(\tilde u_m^0,v_m^0)\|_{L^4}=0$. The first limit follows easily using $\lim_{m\to\infty}R_m=0$, \eqref{Mramsmy} which implies $\lim_{m\to\infty}\|v_m^0\|_{L^4}=0$ and the fact $u_\infty\in L^2(\R^n)\cap L^\infty(\R^n)$ which implies that $\|\tilde u_m^0\|_{L^4}<C$ for all $m$. The second limit follows easily from the estimate in \eqref{vsmurthy} and the fact that $|z_m^j-z_m^k|\to\infty$ as $m\to\infty$. We only need to prove the third limit. As shown in Lemma \ref{lm:nuv}, there exists $C>0$ such that
$$ |N(U(x),v(x))| \leq C \left[|U(x)|^{2p-1} |v(x)|^2+|v(x)|^{2p+1}\right] \FORALL x\in\R^n .$$
If $U\in L^\infty(\R)^n$, then
$$ \|N(U,v)\|_{L^4} \m\leq\m C \|U\|_{L^\infty}^{2p-1} \|v\|_{L^8}^2 + C \|v\|_{L^{8p+4}}^{2p+1} \m. $$
Since $H^2(\R^n)\hookrightarrow L^8(\R^n)$ and $H^2(\R^n)\hookrightarrow L^{8p+4}(\R^n)$ for $n\leq 5$ and $0<p<2/(n-2)$, it follows that
$$ \|N(U,v)\|_{L^4} \m\leq\m C \|U\|_{L^\infty}^{2p-1} \|v\|_{H^2}^2 + C \|v\|_{H^2}^{2p+1} \m. $$
From this estimate, \eqref{Mramsmy} and $\tilde u_m^0\in L^2(\R^n)\cap L^\infty(\R^n)$, we get $\lim_{m\to\infty}\|N(\tilde u_m^0, v_m^0)\|_{L^4}=0$ as desired.

From \eqref{FLeqnvm0} we get that there exist $\beta_i^j\in\R$, for each
$i\in\Iscr$ and $j\in\Jscr_0$, such that
$$  (-\Delta + 1 + R_m^2 V_{R_m}+\sigma(2p+1) |\tilde u_m^0|^{2p})v_m^0 + R_m^2 V_{R_m} \tilde u_m^0$$
$$\hspace{40mm} +\sigma  |\tilde u_m^0|^{2p}\tilde u_m^0 - \sigma\sum_{j=0}^{d_0} \Tscr_{z_m^j} |u_\infty|^{2p}u_\infty + N(\tilde u_m^0,v_m^0)=\sum_{j=0}^{d_0} \sum_{i=1}^n \beta_i^j \Tscr_{z_m^j}\partial_{x_i}u_{\infty}. $$
Each of the $\beta_i^j$s depend on $m$. By taking the $L^2$ scalar product
of the above equation with the functions $\Tscr_{z_m^j}\partial_{x_i} u_{\infty}$, and using the estimate for $\|v_m^0\|_{H^2}$ in \eqref{Mramsmy} and the estimate for $N$ in \eqref{est_nl}, we can show that
$$  \beta_i^j = O(R_m^2) + \sum_{k\neq r,\ r, k=0}^{d_0}O(\exp^{-\sqrt{1-
    \gamma} |z_m^r-z_m^k|}). $$
Define $W_m=\tilde u_m^0+v_m^0$ and consider the set
$$  K_+ = \Big\{x\in \R^n \big| W_m(x) > \max_{i\in\Iscr,\ j\in\Jscr_0} |{\partial_{x_i} u_{\infty}(x-z_m^j)}|\Big\}.$$
On $K_+$ we have
$$  \Delta W_m(x)\geq\bigg(1+\sigma |W_m(x)|^{2p}-R_m^2|V_{R_m}(x)|-\sum_{j=0}^
    {d_0} \sum_{i=1}^n |\beta_i^j|\bigg)W_m(x).$$
Fix $0<\gamma<1$. Fix $M\in\N$ such that for all $m>M$
$$ \sum_{j=0}^{d_0}\sum_{i=1}^n |\beta_i^j|< \frac{\gamma}{4},\qquad R_m^2
   \|V\|_{L^{\infty}}<  \frac{\gamma}{4}, \qquad \|v_m^0\|_{L^\infty}<\frac{\gamma}{8}. $$
Define
$$ \Bscr = \bigcap\limits_{j=0}^{d_0} \R^n\setminus B(z_m^j,L),$$
where $B(z_m^j,L)$ denotes an open ball in $\R^n$ with center $z_m^j$ and radius $L$ and $L>0$ is chosen sufficiently large such that $|W_m(x)|^{2p}\leq \gamma/4$ for all $x\in \Bscr$ and $m>M$. This is possible due to the exponential decay of $u_{\infty}$. Therefore
\begin{equation} \label{w_ineq_inf}
  \Delta W_m(x)\geq(1-\gamma)W_m(x) \qquad \forall x\in \Bscr\cap K_+.
\end{equation}
Consider the functions $\Phi_j(x) = C_0 e^{-\sqrt{1-\gamma}|x-z_m^j|}$ for $j\in\Jscr_0$ with $C_0>0$ being a constant (to be chosen later). These functions satisfy
\begin{equation} \label{p_ineq_inf}
  \Delta \Phi_j(x) \leq (1-\gamma)\Phi_j(x) \qquad \forall x\in\R^n\setminus0.
\end{equation}
Define $\Phi=\sum_{j=0}^{d_0}\Phi_j$. It follows from \eqref{w_ineq_inf} and \eqref{p_ineq_inf} that
$$  \Delta(W_m-\Phi)(x)\geq(1-\gamma)(W_m-\Phi)(x) \qquad \forall x\in
    K_+\cap \Bscr.$$
Hence there exists no positive maxima for the function $W_m-\Phi$ in the set
$K_+\cap \Bscr$. Choose $C_0$ sufficiently large so that
$$ \Phi(x)\geq W_m(x) \quad\forall x\in\partial\Bscr,\qquad \Phi(x)\geq\max_{j\in\Jscr_0,\ i\in\Iscr} |\partial_{x_i} u_{\infty}(x-z_m^j)| \qquad \forall x\in\R^n,$$
where $\partial\Bscr$ denotes the boundary of $\Bscr$. Also, on the boundary
$\partial K_+$ of $K_+$ we have
$$  W_m(x)=\max_{j\in\Jscr_0,\ i\in\Iscr} |\partial_{x_i} u_{\infty}(x-z_m^j)|.$$
Moreover, since $v_m^0\in W^{2,4}$ and $n\leq5$, we have $v_m^0(x)\to 0$ as
$|x|\to\infty$ which implies that $W_m(x)\to 0$ as $|x|\to \infty$. It follows
from the above discussion that
$$  \Phi(x) \geq W_m(x) \qquad \forall x\in K_+\cap \Bscr. $$
A similar argument applied to the function $-W_m$ on the set
$$  K_- = \Big\{x\in \R^n \big|\ -W_m(x)>\max_{j\in\Jscr_0,\ i\in\Iscr}|{\partial_{x_i} u_{\infty}(x-z_m^j)}| \Big\}$$
gives
$$  \Phi(x) \geq -W_m(x) \quad \forall x\in K_-\cap \Bscr. $$
Therefore, we can conclude that for each $x\in\Bscr$ and some $C>0$
\begin{align}
 &|\tilde u_m^0(x)+v_m^0(x)|\leq \max_{j\in\Jscr_0,\ i\in\Iscr} \{| \partial_{x_i} u_{\infty}(x-z_m^j)|,\Phi(x)\}\nonumber\\
\implies &|v_m^0(x)|\leq C \sum_{j=0}^{d_0} e^{-\sqrt{1-\gamma}|x-z_m^j|}.
\label{vm0est}
\end{align}
For $x\in\R^n\setminus\Bscr$, we have $|v_m^0(x)|\leq\|v_m^0\|_{L^\infty}$ which tends to 0 as $m\to\infty$.

Next we derive estimates for the $H^2$ norm of $v_m^1=v_m-P_m^\perp v_m^0$. This function is the part of $v_m$ associated with the $z_m^j$s for which the sequence $(R_m z_m^j)_{m\in\N}$ is unbounded. Since $v_m^0 \in S_m$ ($S_m$ is defined in \eqref{twosigns}) we have
$$  P_m^\perp v_m^0 = v_m^0-\sum_{k=1}^{n}\sum_{j=d_0+1}^{d}\left\langle
    \Tscr_{z_m^j}\partial_{x_k} u_{\infty},v_m^0 \right\rangle \Tscr_{z_m^j}\partial_{x_k}u_{\infty} = v_m^0 - P_z v_m^0.$$
Here $P_z$ is the projection in $L^2$ onto its subspace spanned by
$$  \big\{\Tscr_{z_m^j}\partial_{x_k} u_{\infty},\ j\in\Jscr\setminus\Jscr_0,
    \ k\in\Iscr \big\}.$$
To keep the notation simple, we do not indicate explicitly the dependence of
$P_z$ on $m$. Recall $\tilde u_m^1$ from \eqref{defn_functions}. Then \eqref{splteq_inf} can be written as
\begin{align*}
  0 =& (-\Delta+1+R_m^2 V_{R_m})(\tilde u_m^1+W_m+v_m^1-P_z v_m^0) \\[2mm]
  &+ \sigma|\tilde u_m^1+W_m+v_m^1-P_z v_m^0|^{2p} (\tilde u_m^1+W_m+v_m^1-P_z v_m^0).
\end{align*}
This equation is equivalent to
\begin{align*}
  0=&(-\Delta+1+R_m^2 V_{R_m})v_m^1 - R_m^2 V_{R_m}\sum_{k=1}^{n}\sum_{j=
  d_0+1}^{d}\big\langle \Tscr_{z_m^j} \partial_{x_k}u_{\infty},v_m^0  \big\rangle \Tscr_{z_m^j}\partial_{x_k}u_{\infty} \\
  &+\sigma \left|\tilde u_m^1+W_m+v_m^1-P_z v_m^0\right|^{2p}\left(\tilde
  u_m^1+W_m+v_m^1-P_z v_m^0\right)-\sigma \sum_{j=d_0+1}^d \Tscr_{z_m^j}|u_{\infty}|^{2p}u_{\infty}\\
  & -\sigma |W_m|^{2p} W_m+\sum_{j=0}^{d_0}\sum^n_{i=1} \beta_i^j\Tscr_{z_m^j}
  \partial_{x_i} u_{\infty}+R_m^2 V_{R_m} \tilde u_m^1\\
  & +\sum_{k=1}^{n}\sum_{j=d_0+1}^d (2p+1)\sigma\big\langle \Tscr_{z_m^j} \partial_{x_k}u_{\infty},v_m^0 \big\rangle \Tscr_{z_m^j}|u_{\infty}|^{2p}
  \partial_{x_k} u_{\infty}
\end{align*}
which, using the operators $L_+\left(\tilde u_m^1+W_m- P_z v_m^0, R_m\right)$ and $N\left(\tilde u_m^1+ W_m - P_z v_m^0,v_m^1\right)$, can be rewritten as
\begin{align}
  0=&L_+\left(R_m,\tilde u_m^1+W_m- P_z v_m^0\right)v_m^1 + N\left(\tilde
  u_m^1+ W_m - P_z v_m^0,v_m^1\right)+\sum_{j=0}^{d_0}\sum^n_{i=1}
  \beta_i^j \Tscr_{z_m^j}\partial_{x_i}u_{\infty}\nonumber\\
  &+\sigma \left|\tilde u_m^1+ W_m- P_z v_m^0\right|^{2p}\left(\tilde
  u_m^1+ W_m-P_z v_m^0\right) - \sigma \sum_{j=d_0+1}^d \Tscr_{z_m^j} |u_{\infty}|^{2p}u_{\infty}-\sigma |W_m|^{2p} W_m\nonumber\\
  &+R_m^2 V_{R_m} \tilde u_m^1 - R_m^2 V_{R_m} \sum_{k=1}^{n}\sum_{j=
  d_0+1}^{d}\big\langle \Tscr_{z_m^j}\partial_{x_k}u_{\infty},v_m^0 \big
  \rangle\Tscr_{z_m^j}\partial_{x_k} u_{\infty}\nonumber\\
  &+\sum_{k=1}^{n}\sum_{j=d_0+1}^{d}(2p+1)\sigma\big\langle \Tscr_{z_m^j}\partial_{x_k}u_{\infty},v_m^0 \big\rangle \Tscr_{z_m^j}|u_{\infty}|^{2p}\partial_{x_k}u_{\infty}. \label{onesign}
\end{align}
The argument based on Proposition \ref{prop:MS} and spectral perturbation theory (see discussion above \eqref{vm_com_inf}) used to show that $P_{m}^{\perp}L_+(\tilde u_m,R_m)P_{m}^{\perp}$ is invertible, with a uniform in $m$ bound for its inverse, can be used to show that the same is true for $P_m^\perp L_+\left(\tilde u_m^1+W_m- P_z v_m^0,R_m\right)P_m^\perp$. Indeed, let $L_+\left(\tilde u_m^1+W_m- P_z v_m^0,R_m\right) =\widetilde L_m+\widetilde W_m$, where $\widetilde L_m$ is as defined above \eqref{vm_com_inf} and
$$ \widetilde W_m= R_m^2 V_{R_m} +\sigma(2p+1) \Big(|\tilde u_m+v_m^0-P_z v_m^0|^{2p}-\sum_{j=0}^d \big|\Tscr_{z_m^j} u_{\infty}\big|^{2p}\Big).$$
Using the fact that $\|v_m^0\|_{H^2}\to0$ and $|y_m^j-y_m^k|\to\infty$ as $m\to\infty$, we get $\|\widetilde W_m\|_{H^2\mapsto L^2}\to0$ as $m\to\infty$. Applying the argument above \eqref{vm_com_inf} to $\widetilde L_m$ and $\widetilde W_m$ defined here, we get $L_m=P_m^\perp L_+\left(\tilde u_m^1+W_m- P_z v_m^0,R_m\right)P_m^\perp$ is invertible for large $m$ and the norm of its inverse operator can be uniformly bounded. Therefore applying $P_m^\perp$ to \eqref{onesign} and rewriting it as $v_m=L_+^{-1}$[remaining terms] and using $P_m^\perp \sum_{j=0}^{d_0}\sum^n_{i=1}\beta_i^j \Tscr_{z_m^j}\partial_{x_i} u_{\infty}=0$, we can deduce via contraction mapping principle that
\begin{equation}\label{eq:vr1_estimate}
  \|v_m^1\|_{H^2}\leq C R_m^2 \|V_{R_m}\tilde u_m^1\|_{L^2} + \sum_{j=d_0
  +1}^{d}\ \sum_{k=0,\, j\neq k}^{d} O(e^{-\sqrt{1-\gamma}|z_m^j-
  z_m^k|}).
\end{equation}

In summary, we have
$$ v_m=P_m^\perp v_m^0+v_m^1=v_m^0-P_z v_m^0 +v_m^1$$
with $v_m^0(x)\leq \sum_{k=0}^{d_0} C e^{-\sqrt{1-\gamma}|x-z_m^k|}$ for all
$x\in\R^n$, which using the definition of $P_z$, gives
$$  \|P_z v_m^0\|_{H^2} + \|v_m^1\|_{H^2}\leq C R_m^2 \|V_{R_m}\tilde  u_m^1
    \|_{L^2} + \sum_{j=d_0+1}^{d}\ \sum_{k=0,\, j\neq k}^{d} O(e^{-
    \sqrt{1-\gamma}|z_m^j-z_m^k|}). $$

Next we will study the set of finite dimensional equations in
\eqref{eq:fincom_inf} obtained by projecting \eqref{eq:stationary_split_Elarge}
on the functions $\Tscr_{z_m^j}\partial_{x_k} u_{\infty}$. Explicitly, the set
of equations are
\begin{align*}
  \Big\langle \Tscr_{z_m^j}\partial_{x_k}u_{\infty}, R_m^2 V_{R_m}
  \sum_{i=0}^d \Tscr_{z_m^i}u_{\infty} + N(\tilde u_m,v_m)\Big\rangle
  +\Big\langle\Tscr_{z_m^j} \partial_{x_k}u_{\infty}, \sigma|\tilde u_m|^
  {2p}\tilde u_m - \sigma\sum_{i=0}^d \Tscr_{z_m^i}|u_{\infty}|^{2p} u_{\infty}\Big\rangle
\end{align*}
\begin{equation}\label{eq:case1_eqn_n}
  +\Big\langle \Tscr_{z_m^j}\partial_{x_k}u_{\infty} \big(R_m^2 V_{R_m}+
  (2p+1)\sigma\big[\tilde u_m^{2p}-\Tscr_{z_m^j}u_{\infty}^{2p}\big]
  \big), v_m \Big\rangle=0
\end{equation}
for $j\in\Jscr$ and $k\in\Iscr$. We will show that the set of equations in \eqref{eq:case1_eqn_n} has no solution under our assumption \eqref{contrad_split}. Our analysis will focus on two types of terms in \eqref{eq:case1_eqn_n} - terms involving interaction between $\Tscr_{z_m^j}u_\infty$ and $\Tscr_{z_m^k}u_\infty$
for $j\neq k$ and $j,k\in\Jscr\setminus\Jscr_0$ (such terms arise from
the second inner product in \eqref{eq:case1_eqn_n}) and terms involving interaction between the various $\Tscr_{z_m^i} u_\infty$ and the potential $V$ with $i\in\Jscr\setminus\Jscr_0$. We will consider two cases, one of which must occur. In the first case we assume that the former terms, which decay slower than $e^{-\sqrt {1+\gamma}|z_m^j-z_m^k|}$ dominate the latter terms. Then, by combining the equations in \eqref{eq:case1_eqn_n} appropriately we construct a new equation in which the former terms dominate all the other terms. This will
imply that there is no solution to \eqref{eq:case1_eqn_n}. In the second
case, we will combine the equations in \eqref{eq:case1_eqn_n} appropriately
to obtain a new equation in which the former terms (at least the important
ones) are eliminated. Under the assumption that the terms involving the
potential dominate terms decaying like $e^{-2\sqrt{1-\gamma}|z_m^j-z_m^k|}$,
we can apply the argument from the proof of the compactness case to conclude
that \eqref{eq:case1_eqn_n} has no solution. We will now discuss the two cases rigorously. In the proof below, we assume for simplicity that the set $\Jscr\setminus\Jscr_0$ has at least two elements. If it has only one element, then the proof of contradiction is similar to the proof in the compactness case, see remark below \eqref{vacucase}.

Let $M_m=\min\{|z_m^j-z_m^k|\big| j,k\in\Jscr\setminus\Jscr_0, j\neq k\}$.  We
assume without of loss of generality, by considering subsequences if necessary,
that each of the bounded sequences $(M_m/|z_m^j-z_m^k|)_{m\in\N}$ is converging and that $\|V_{R_m}\Tscr_{z_m^{d_0+1}}u_\infty\|_{L^2}$ is the maximum of the set $\{\|V_{R_m}\Tscr_{z_m^j}u_\infty\|_{L^2}, j\in\Jscr\setminus \Jscr_0\}$ for each $m$ and also that each of the sequences $|z_m^{d_0+1}-z_m^k|/|z_m^{d_0+1}|$ converges either to a finite limit or to $\infty$. Partition $\{z_m^0,z_m^1,\ldots z_m^d\}$ into sets $Z^1_m$ and $Z_m^2$ such that for each $z_m^k\in Z_m^1$ and no $z_m^k \in Z_m^2$ we have
\begin{equation} \label{defnzm1}
 \lim_{m\to\infty} \frac{|z_m^k-z_m^{d_0+1}|}{|z_m^{d_0+1}|}=0.
\end{equation}
Define
$$ \hat u_m = \sum_{z_m^k\in Z_m^1}\Tscr_{z_m^k} u_\infty \quad \textrm
    {and} \quad \hat w_m = \sum_{z_m^k\in Z_m^2} \Tscr_{z_m^k} u_\infty .$$
In Case (i) we suppose that
\begin{equation}\label{eq:case1_condition_new_n}
  \limsup_{m\to\infty} \frac{-\ln \|R_m^2 V_{R_m} \Tscr_{z_m^j} u_\infty\|
  _{L^2}}{M_m}>1 \qquad \forall j\in\Jscr\setminus\Jscr_0
\end{equation}
holds and in Case (ii) we suppose that
\begin{equation}\label{eq:case2_condition_new_n}
  \liminf_{m\to\infty} \frac{-\ln\left|R_m^2 \left\langle
  V_{R_m}, \partial_{x_m}\hat u_m^2 \right\rangle\right|} {M_m} < 2,
\end{equation}
where $x_m$ is the direction along the vector $z_m^{d_0+1}$, holds. One of the above two conditions must hold. Indeed, if \eqref{eq:case1_condition_new_n} does not hold for any $j\in\Jscr \setminus\Jscr_0$ then, since we have assumed that $\|V_{R_m} T_{z_m^{d_0+1}} u_\infty\|_{L^2}$ is the maximum of the set  $\{\|V_{R_m}T_{z_m^j} u_\infty\|_{L^2} \big| j\in\Jscr\setminus\Jscr_0\}$ for each $m$, it follows that
$$ \limsup_{m\to\infty} \frac{-\ln \|R_m^2 V_{R_m} T_{z_m^{d_0+1}} u_\infty
  \|_{L^2}}{M_m} \leq 1. $$
Using this and the estimate
\begin{equation} \label{shownbelow}
 \lim_{m\to\infty}\frac{\ln \left|R_m^2\left\langle\partial_{x_m}\hat u_m, V_{R_m}\hat u_m\right \rangle\right|}{\ln \|R_m^2 V_{R_m}T_{z_m^{d_0+1}} u_\infty\|_{L^2}^2} \m=\m 0\m,
\end{equation}
which is established below, it follows that
\begin{align*}
  &\liminf_{m\to\infty} \frac{-\ln \left|R_m^2\left\langle\partial_{x_m}\hat u_m, V_{R_m}\hat u_m\right\rangle\right|} {M_m} \\
  =&\ \liminf_{m\to\infty} \frac{-\ln \|R_m^2 V_{R_m} T_{z_m^{d_0+1}}u_\infty
  \|_{L^2}^2} {M_m} \frac{-\ln \left|R_m^2\left\langle\partial_{x_m}\hat u_m, V_{R_m}\hat u_m\right\rangle\right|}{-\ln \|R_m^2 V_{R_m}T_{z_m^{d_0+1}}u_\infty\|_{L^2}^2}\\
  =&\ 0,
\end{align*}
which implies that if \eqref{eq:case1_condition_new_n} does not hold, then \eqref{eq:case2_condition_new_n} holds, i.e. either Case (i) or Case (ii) must occur. We will now establish \eqref{shownbelow} by showing that
\begin{equation} \label{tearettea}
 \lim_{m\to\infty} \frac{\|R_m^2 V_{R_m}\hat u_m\|_{L^2}^2} {|R_m^2 \langle \partial_{x_m} \hat u_m, V_{R_m}\hat u_m \rangle|} \m=\m0\m.
\end{equation}
This limit and $\hat u_m>T_{z_m^{d_0+1}} u_\infty$ imply that $(\|R_m^2 V_{R_m}T_{z_m^{d_0+1}} u_\infty \|_{L^2}^2)/(|R_m^2 \langle \partial_{x_m} \hat u_m, V_{R_m}\hat u_m \rangle|)\to0$ as $m\to\infty$ which, along with $\lim_{m\to\infty} \|R_m^2 V_{R_m}T_{z_m^{d_0+1}} u_\infty\|_{L^2}=0$, implies \eqref{shownbelow}. To show \eqref{tearettea}, we fix $\theta\in(0,\pi/2)$ such that $2\Cscr\tan\theta<1$, where $\Cscr$ is the constant in the hypothesis \eqref{Hypo_largeE3}. Let $B_m$ be the closed ball in $\R^n$ with center $z_m^{d_0+1}$ and radius $|z_m^{d_0+1}| \sin\theta$. By adopting the approach used in the compactness case, see \eqref{eq:v'exp_decay}-\eqref{eq:v'estimate} and \eqref{eq:vsqr_exp1}-\eqref{eq:vsqr_v'2} and the associated discussion, we get using the exponential decay of $u_\infty$, \eqref{defnzm1}, $\lim_{m\to\infty} |R_mz_m^{d_0+1}| =\infty$ and the hypothesis in \eqref{Hypo_largeE1}-\eqref{Hypo_largeE3} that
\begin{align*}
 \lim_{m\to\infty} \frac{\|R_m^2 V_{R_m} \hat u_m\|_{L^2}^2} {|R_m^2\langle \partial_{x_m} \hat u_m, V_{R_m}\hat u_m\rangle|} \m\leq\m& C\lim_{m\to\infty} \frac{R_m\Big|\sum_{z_m^k\in Z_m^1}\int\limits_{B_m}u^2_\infty(x-z_m^k)V^2(R_m x) \mathrm{d}x\Big|}{\Big|\sum_{z_m^k\in Z_m^1} \int\limits_{B_m} u^2_\infty(x-z_m^k) \partial_{x_m}  V(R_m x)\mathrm{d}x\Big|} \\
 \m=\m& C  \lim_{m\to\infty} \frac{R_m\Big|\sum_{z_m^k\in Z_m^1} \int\limits_{B_m}u^2_\infty(x-z_m^k)\partial_{x_m}  V(R_m x)\frac{V^2(R_m x)}{\partial_{x_m} V(R_m x)} \mathrm{d}x\Big|}{\Big|\sum_{z_m^k\in Z_m^1} \int\limits_{B_m} u^2_\infty(x-z_m^k) \partial_{x_m}  V(R_m x)\mathrm{d}x\Big|}\\
 \m=\m&0.
\end{align*}
Hence \eqref{tearettea}, and therefore \eqref{shownbelow}, holds.

We will now consider the two cases. We say that the sequences $(z_m^j)_{m\in\N}$ and $(z_m^k)_{m\in\N}$ are  connected if and only if
$$  \lim_{m\to \infty} \frac{M_m}{|z_m^j-z_m^k|}=1. $$

\noindent
{\bf Case (i):} Assume that \eqref{eq:case1_condition_new_n} holds. We say that the sequences $(z_m^{a_1})_{m\in\N}, (z_m^{a_2})_{m\in\N}, \ldots (z_m^{a_r})_{m\in\N}$, where each $a_i\in\Jscr\setminus\Jscr_0$, form a connected component if two conditions are satisfied: (a) For any $i,j\in \{a_1,a_2, \ldots a_r\}$ with $i\neq j$, there exists a chain $(b_1, b_2, \ldots b_t)$, where each $b_k\in\{a_1,a_2,\ldots a_r\}$, $b_1=i$ and $b_t=j$, such that the sequences $(z_m^{b_k})_{m\in\N}$ and $(z_m^{b_{k+1}})_{m\in\N}$ are connected for each $k\in\{1,2,\ldots r-1\}$ and (b) For any $i \notin \{a_1,a_2,\ldots a_r\}$, there exists no $j\in\{a_1,a_2,\ldots a_r\}$ such that $(z_m^i)_{m\in\N}$ is connected to $(z_m^j)_{m\in\N}$.

Consider a connected component with at least two sequences, say $(z_m^{d_0+1})_{m\in\N}, (z_m^{d_0+2})_{m\in\N}, \ldots $ $(z_m^{d_0+r})_{m\in\N}$. For each $m$, the baricenter of these $z_m^j$s is
$$  \bar z_m = \frac {z_m^{d_0+1}+z_m^{d_0+2}+\ldots+z_m^{d_0+r}}{r}. $$
By going to a subsequence if necessary, suppose that $q\in\{d_0+1,d_0+2, \ldots d_0+r\}$ is such that $|z_m^q-\bar z_m|\geq |z_m^j-\bar z_m|$ for all $j\in\{d_0+1,d_0+2, \ldots d_0+r\}$ and all $m$. For each $z_m^l$ connected to $z_m^q$ (clearly $l>d_0$) it follows that
$$  \cos\alpha_m^l \geq \frac{|z_m^l-z_m^q|}{2|\bar z_m-z_m^q|},$$
where $\alpha_m^l$ is the angle between the vectors $z_m^l-z_m^q$ and
$\bar z_m - z_m^q$. Also, we have
\begin{align*}
  \limsup_{m\to\infty} \frac{|\bar z_m -z_m^q|}{M_m} =& \limsup_{m\to\infty} \frac{|z_m^{d_0+1}-z_m^q+\ldots+z_m^{q-1}-z_m^q+z_m^{q+1}-z_m^q +\ldots+
  z_m^{d_0+r}-z_m^q|}{r M_m} \\
  \leq& \limsup_{m\to\infty} \sum_{i=d_0+1}^{d_0+r}\frac{|z_m^i-z_m^q|}
  {r M_m} \leq \frac{(r-1)^2}{r} \leq \frac{(d-1)^2}{d},
\end{align*}
which implies that for each $z_m^l$ connected to $z_m^q$
$$  \liminf_{m\to\infty} \frac{|z_m^l-z_m^q|}{2|\bar z_m-z_m^q|} \geq
    \liminf_{m\to\infty} \frac{r|z_m^l-z_m^q|}{2(r-1) \sum_{i=d_0+1}^
    {d_0+r-1}|z_m^i-z_m^{i+1}|} \geq \frac{r}{2(r-1)^2} \geq \frac{d}
    {2(d-1)^2}.$$
By letting $j=q$ in \eqref{eq:case1_eqn_n}, we get that for some for a fixed $\epsilon>0$
\begin{align}
  0 = &\Big\langle \Tscr_{z_m^q}\partial_{x_k}u_\infty, \sigma(2p+1)\Tscr_{z_m^q} [u_\infty^{2p}]\sum_{\stackrel {l:z_m^l, z_m^q \textrm{ is}}
  {\textrm{connected}}}\Tscr_{z_m^l}u_\infty\Big\rangle\nonumber\\
  &+ C R_m^2 \sum_{j=d_0+1}^d \|V_{R_m}\Tscr_{z_m^j} u_\infty\|_{L^2} + O\Big(\sum_{i\neq j,\ i,j=d_0+1}^{d} e^{-(1+\epsilon)\sqrt{1-
  \gamma}|z_m^i-z_m^j|}\Big). \label{eq:fin_dim_eq_case1_largeE}
\end{align}
Here we have used the hypothesis in \eqref{Hypo_largeE2}, which implies that $|V(x)|>Ce^{-\beta|x|}$ for almost all large $|x|$, to dominate some exponentially decaying terms (such as the product $\Tscr_{z_m^q}u_\infty \Tscr_{z_m^j} u_\infty$ for $j\in\Jscr_0$). By choosing $x_k\parallel \bar z_m-
z_m^q$ in \eqref{eq:fin_dim_eq_case1_largeE} and by denoting the direction
parallel to $(z_m^l-z_m^q)$ by $x_{lq}$, we get for $m$ large
\begin{align}
  &\Big|\big\langle \Tscr_{z_m^q}\partial_{x_k}u_\infty, \sigma(2p+1)\Tscr_{z_m^q}[u_\infty^{2p}]\sum_{\stackrel {l:z_m^l, z_m^q \textrm{ is}} {\textrm{connected}}} \Tscr_{z_m^l} u_\infty\big\rangle\Big|
  \nonumber\\
  =\ &\Big|\sigma(2p+1)\sum_{\stackrel{l:z_m^l,z_m^q \textrm{ is}} {\textrm{connected}}}
  \cos\alpha_m^l \big\langle \Tscr_{z_m^q}\partial_{x_{lq}} u_\infty,   \Tscr_{z_m^q} [u_\infty^{2p}] \Tscr_{z_m^l}u_\infty\big\rangle  \Big|\nonumber\\
  \geq\ & \frac{C(\gamma)d}{2(d-1)^2} \exp^{-\sqrt{1+\gamma}M_m}, \label{eq:fin_dim_merm_case1_largeE}
\end{align}
for any $\gamma>0$. While deriving the above expression, we have used two estimates. The first is the easily verifiable estimate that if $x_i\perp(z_m^l-z_m^q)$, then
$$  \big\langle\Tscr_{z_m^q} \partial_{x_i}u_\infty, \sigma(2p+1)\Tscr_{z_m^q} [u_\infty^{2p}]\Tscr_{z_m^l}u_\infty \big\rangle=0. $$
This is used to obtain the equality in the second line of \eqref{eq:fin_dim_merm_case1_largeE}. The second estimate is if $z_m^l-z_m^q$ is along the direction of $x_i$ and $m$ is large, then for every $\gamma>0$, there exists $C(\gamma)>0$ such that
\begin{equation} \label{tranrun}
  -\big\langle\Tscr_{z_m^q} \partial_{x_i}(u_\infty)^{2p+1}, \Tscr_{z_m^l}u_\infty \big\rangle \m\geq\m C(\gamma) \exp^{-\sqrt{1+\gamma}|z_m^l-z_m^q|}.
\end{equation}
This is used to derive the inequality in the third line of \eqref{eq:fin_dim_merm_case1_largeE}. The estimate in \eqref{tranrun} can be established as follows. Clearly $z_m^l-z_m^q$ has all coordinates 0, except the $i^{th}$ coordinate which is positive. We denote it by $z_m^\delta$. Let $y=(x_1,x_2,\ldots x_{i-1},x_{i+1},\ldots x_n)$ and $\partial_i$ denote the derivative in the $x_i$-direction.  For any $y\in\R^{n-1}$, consider the integral
$$ \int_\R \partial_i(u_\infty)^{2p+1}(x-z_m^q) u_\infty(x-z_m^l) \dd x_i = \int_\R \partial_i(u_\infty)^{2p+1}(x_i,y) u_\infty (x_i-z_m^\delta,y) \dd x_i\m. $$
From the radial symmetry of $u_\infty$ and the monotone decay of $u_\infty$ along the radial direction, it follows that for any $a>0$,
$$ \partial_i(u_\infty)^{2p+1}(a,y) = -\partial_i (u_\infty)^{2p+1} (-a,y)\leq0\m, \qquad u_\infty(a-z_m^\delta,y) \geq u_\infty(-a-z_m^\delta,y)>0\m.$$
From these estimates it is easy to see that for any $y\in\R^{n-1}$
\begin{equation} \label{julchan}
 -\int_\R \partial_{x_i}(u_\infty)^{2p+1}(x_i,y) u_\infty (x_i-z_m^\delta,y) \dd x_i > - \int_{A_m} \partial_{x_i}(u_\infty)^{2p+1}(x_i,y) u_\infty (x_i-z_m^\delta,y) \dd x_i > 0\m.
\end{equation}
Here $A_m=\{a\m\big|\m |a\pm z_m^\delta|<1\}$. For each $\epsilon\in(0,1)$, there exists $C_1(\epsilon),C_2(\epsilon)>0$ such that for $|x|$ large
$$ C_1(\epsilon)e^{-\sqrt{1+\epsilon}|x|} \leq |u_\infty(x)| \leq C_2(\epsilon) e^{-\sqrt{1-\epsilon}|x|}\m, \qquad |\partial_r u_\infty(x)| \geq C_1(\epsilon)e^{-\sqrt{1+\epsilon}|x|}\m. $$
Here $\partial_r$ is the radial derivative. Also, since $z_m^\delta\to\infty$ as $m\to\infty$, it follows that there exists $C_3(\epsilon)>0$ for $m$ large such that if $x_i\in A_m$ and $|y|<1$, then
$$ |\partial_i u_\infty(x)| \geq C_3(\epsilon)e^{-\sqrt{1+\epsilon}|x|}\m. $$
Using the above estimates for $u_\infty$ and its derivative, we get that for $m$ large and $|y|<1$
\begin{align*}
 -\int_{A_m} \partial_{x_i}(u_\infty)^{2p+1}(x_i,y) u_\infty (x_i-z_m^\delta,y) \dd x_i >\ & - \frac{1}{2}\int_{-1+z_m^\delta}^{1+z_m^\delta} \partial_{x_i} (u_\infty)^{2p+1} (x_i,y) u_\infty (x_i-z_m^\delta,y) \dd x_i \\
 >\ &  C_4(\epsilon)e^{-\sqrt{1+2\epsilon}|z_m^\delta|}\m,
\end{align*}
for some $C_4(\epsilon)>0$. This, along with \eqref{julchan}, implies that for $m$ large
\begin{align*}
-\big\langle\Tscr_{z_m^q} \partial_{x_i}(u_\infty)^{2p+1}, \Tscr_{z_m^l}u_\infty \big\rangle &\m=\m -\int_{\R^{n-1}}\int_\R \partial_{x_i}(u_\infty)^{2p+1} (x_i,y) u_\infty (x_i-z_m^\delta,y) \dd x_i\m \dd y\\
&\m\geq\m -\frac{1}{2}\int_{|y|<1}\int_{-1+z_m^\delta}^{1+z_m^\delta} \partial_{x_i}(u_\infty)^{2p+1} (x_i,y) u_\infty (x_i-z_m^\delta,y) \dd x_i \dd y\\
&\m\geq\m C_4(\epsilon)e^{-\sqrt{1+2\epsilon}|z_m^\delta|}\m,
\end{align*}
for some $C_5(\epsilon)>0$, i.e. \eqref{tranrun} holds.

It follows from \eqref{eq:case1_condition_new_n} that as $m\to\infty$, for
a sufficiently small $\gamma>0$ and each $j\in\Jscr\setminus\Jscr_0$, we have
(at least on a subsequence of $m$) that
\begin{align*}
  &M_m \left[\sqrt{1+\gamma} + \frac{\ln\|R_m^2 V_{R_m}\Tscr_{z_m^j} u_\infty\|_{L^2}}{M_m}\right] \to -\infty\\
  \iff&\ \sqrt{1+\gamma} M_m + \ln\|R_m^2 V_{R_m} \Tscr_{z_m^j}u_\infty\|_{L^2}
  \to -\infty,\\
  \iff&\ e^{\sqrt{1+\gamma}M_m}\|R_m^2 V_{R_m}\Tscr_{z_m^j}u_\infty\|_{L^2}\to 0.
\end{align*}
The last expression above implies that the term considered in \eqref{eq:fin_dim_merm_case1_largeE}
cannot be canceled in \eqref{eq:fin_dim_eq_case1_largeE}. Therefore
\eqref{eq:fin_dim_eq_case1_largeE} and consequently the set of equations
\eqref{eq:case1_eqn_n} has no solution in this case.

\noindent
{\bf Case (ii):} Assume that \eqref{eq:case2_condition_new_n} holds. By adding the equations in \eqref{eq:case1_eqn_n} corresponding to each of the $z_m^i\in
Z_m^1$, we obtain
\begin{align}
  0 =& \Big\langle \partial_{x_m}\hat u_m, R_m^2 V_{R_m} (\hat u_m+
  \hat w_m) + \sigma(\hat u_m + \hat w_m)^{2p+1} - \sigma\hat u_m^{2p+1} -\sigma\sum_{i=0}^d T_{z_m^i}u_\infty^{2p+1}\Big\rangle \nonumber\\
  &+O\Big(\sum_{j=d_0+1}^d \|R_m^2 V_{R_m}T_{z_m^j}u_\infty\|_{L^2}^2
  \Big) + \sum_{i=d_0+1}^{d} \sum_{\stackrel{k=0}{i\neq k}}^d O\left(e^{-
  2\sqrt{1-\gamma}|z_m^i-z_m^k|}\right). \label{eq:fin_dim_eq_case2_largeE}
\end{align}
Since the functions $u_\infty$ are radially symmetric, one can check that for
any $i,j\in\Jscr\setminus\Jscr_0$ and any direction $x_k$
$$  \left\langle T_{z_m^i}\partial_{x_k}u_\infty, T_{z_m^j}u_\infty^{2p+1} \right\rangle + \left\langle T_{z_m^j}\partial_{x_k}u_\infty, T_{z_m^i} u_\infty^{2p+1}\right\rangle=0. $$
In particular the above inequality implies that
\begin{equation}\label{par_est_inf}
  \bigg\langle \partial_{x_m}\hat u_m, \sum_{i\in Z_1} T_{z_m^i}u_\infty^{2p+1}
  \bigg\rangle=0.
\end{equation}
Note that for some $\delta>0$
$$  \lim_{m\to\infty} \frac{|z_m^i-z_m^{d_0+1}|}{|z_m^{d_0+1}|}>\delta
    \quad \textrm{for each} \quad z_m^i\in Z_2. $$
Therefore
$$  \lim_{m\to\infty} \frac{|z_m^i-z_m^k|}{|z_m^{d_0+1}|}>\delta \quad
    \textrm{for each} \quad z_m^k\in Z_1, z_m^i\in Z_2. $$
From the above discussion we get the estimate
\begin{equation} \label{cross_est_inf}
  \Big\langle \partial_{x_m}\hat u_m, R_m^2V_{R_m}\hat w_m +
  \sigma(\hat u_m + \hat w_m)^{2p+1} - \sigma\hat u_m^{2p+1} -
  \sigma\sum_{l\in Z_2} T_{z_m^l} u_\infty^{2p+1} \Big\rangle
  = O\left(e^{-\sqrt{1-\gamma}\delta |z_m^{d_0+1}|}\right).
\end{equation}
The hypothesis in \eqref{Hypo_largeE2} implies that for some $C>0$
\begin{equation}\label{Turkish_icecream}
 \|R_m^2 V_{R_m} T_{z_m^{d_0+1}}u_\infty\|^2_{L^2}>C e^{-\sqrt{1-\gamma}\delta |z_m^{d_0+1}|}
\end{equation}
Using \eqref{par_est_inf}, \eqref{cross_est_inf} and \eqref{Turkish_icecream} we can rewrite \eqref{eq:fin_dim_eq_case2_largeE} as
\begin{equation} \label{eq:case2_eqn_n}
  0 = \left\langle \partial_{x_m}\hat u_m, R_m^2 V_{R_m}\hat u_m
  \right\rangle + O(\|R_m^2 V_{R_m} T_{z_m^{d_0+1}} u_\infty\|_{L^2}^2) +\sum_{i=d_0+1}^{d} \sum_{\stackrel{k=0}{i\neq k}}^d O\left(e^{-2
  \sqrt{1-\gamma}|z_m^i-z_m^k|}\right)  .
\end{equation}
We will next show that the first term in the above equation dominates the other two terms. Comparing the first term with the third term, it follows from \eqref{eq:case2_condition_new_n} that for a sufficiently small $\gamma>0$, as $m\to\infty$ (considering subsequences if necessary)
\begin{align*}
  & M_m\left(\frac {-\ln\left|R_m^2\left\langle\partial_{x_m}\hat u_m, V_{R_m}\hat u_m \right\rangle\right|}{M_m} - 2\sqrt{1-\gamma}\right) \to
  -\infty \\
  \iff&\ -\ln\left|R_m^2\left\langle\partial_{x_m}\hat u_m, V_{R_m} \hat u_m
  \right\rangle\right| - 2\sqrt{1-\gamma}M_m \to -\infty \\
  \iff&\ \frac{e^{-2\sqrt{1-\gamma}M_m}}{\left|R_m^2\left\langle\partial_
  {x_m}\hat u_m, V_{R_m}\hat u_m\right\rangle\right|} \to 0.
\end{align*}
We next compare the first term with the second term of \eqref{eq:case2_eqn_n}. Using the hypothesis \eqref{Hypo_largeE1} we get, as in the compactness case, that
\begin{equation} \label{vacucase}
 \lim_{m\to\infty} \frac{\left| R_m^4\left\langle T_{z_m^{d_0+1}}u_\infty
  V_{R_m}, T_{z_m^{d_0+1}}u_\infty V_{R_m} \right\rangle\right|} {\left|R_m^2\left\langle\partial_{x_m}\hat u_m, V_{R_m}\hat u_m \right\rangle\right|} \leq \lim_{m\to\infty}\frac{\left| R_m^4\left\langle \hat u_m V(R_mx), \hat u_m V(R_mx) \right\rangle\right|} {\left|R_m^2\left\langle\partial_{x_m}\hat u_m, V_{R_m}\hat u_m\right\rangle\right|} = 0.
\end{equation}
Hence the term $|R_m^2\left\langle\partial_{x_m}\hat u_m, V_{R_m}\hat u_m\right\rangle|$ in \eqref{eq:case2_eqn_n} cannot be canceled
by the remaining two terms. Thus we get that \eqref{eq:fin_dim_eq_case2_largeE}, and consequently the set of equations \eqref{eq:case1_eqn_n}, has no solution in this case.

We remark that if the set $\Jscr\setminus\Jscr_0$ has only one element, then we need not consider the two cases. Indeed, using \eqref{Hypo_largeE2} the first term in \eqref{eq:case2_eqn_n} can be shown to be larger than $C R_m^3 e^{-2\beta R_m |z_m^d|}$ (see \eqref{useinsplit} in compactness case). Therefore it dominates the third term which will be $O(e^{-2\sqrt{1-2\gamma} |z_m^d|})$. The first term also dominates the second term (see \eqref{vacucase}). Hence \eqref{eq:fin_dim_eq_case2_largeE}, and consequently the set of equations \eqref{eq:case1_eqn_n}, has no solution.

Since either Case (i) or Case (ii) must occur, we get the contradiction that the set of equations \eqref{eq:case1_eqn_n} has no solution if our assumption \eqref{contrad_comp} is true. Hence it must be false and $(R_m y_m^j)_{m\in\N}$ is a bounded sequence for all $j\in\Jscr$. Next we will show that each accumulation point of the sequence $(R_m z_m^j)_{m\in\N}$ at which $V$ is continuously differentiable must be a critical point of the potential. This, according to the discussion below \eqref{new_split_soln_inf}, will complete the proof of this theorem.

Suppose that, by considering subsequences if required, that for each $j\in\Jscr$, the sequence $(R_m z_m^j)_{m\in\N}$ is convergent. To show that if the potential is continuously differentiable at a limit of these sequences, then that limit must be a critical point of the potential, it is sufficient to assume that there exists (by relabeling the sequences if needed) a $\bar d$ such that $\lim_{m\to\infty} R_m z_m^j=x^0$ for each $j\in\{0,1,\ldots\bar d\}$ and $\lim_{m\to\infty} |x^0-R_m z_m^j|>0$ for each $j\in\{\bar d+1,\bar d+2,\ldots d\}$ and prove that $x^0$ is a critical point of the potential whenever $V$ is continuously differentiable at $x^0$. Such a proof follows.

We write $u_{E_m}$ as
$$  u_{E_m} = \sum_{j=0}^{d} T_{z_m^j} u_\infty + v_m = \tilde u_m + v_m =
    \bar u_m^0+\bar u_m^1+v_m, $$
where
$$ \tilde u_m = \sum_{j=0}^d T_{z_m^j} u_\infty, \qquad \bar u_m^0 = \sum_{j=\bar d+1}^d T_{z_m^j}u_\infty, \qquad \bar u_m^1 = \sum_{j=0}^{\bar d} T_{z_m^j}u_\infty.$$
Earlier, using the decomposition $\tilde u_m=\tilde u_m^0 + \tilde u_m^1$, we obtained a decomposition of $v_m$ in terms of the pair of functions $(v_m^0,v_m^1)$. Similarly, using the decomposition $\tilde u_m= \bar u_m^0+\bar u_m^1$, we can obtain another estimate for $v_m$ in terms of a different
pair of functions $(\bar v_m^0,\bar v_m^1)$. Then, defining  $P_{\bar z}$
analogous to $P_z$, we have $v_m=\bar v_m^0-P_{\bar z}\bar v_m^0 +
\bar v_m^1$, where $\bar v_m^0(x)\leq \sum_{k=\bar d+1}^d Ce^{-\sqrt{1-
\gamma}|x-z_m^k|}$ and
$$  \|P_{\bar z}\bar v_m^0\|_{H^2} + \|\bar v_m^1\|_{H^2} \leq CR_m^2
    \left\|V_{R_m} \bar u_m^1\right\|_{L^2} + \sum_{j=0}^{\bar d}
    \sum_{\stackrel{k=0}{j\neq k}}^d O(e^{-\sqrt{1-\gamma}|z_m^j-z_m^k|}). $$
We will now consider two distinct cases to show that $x^0$ must be a critical
point of the potential whenever $V$ is continuously differentiable at $x^0$. Since these cases are similar to Case (i) and Case (ii) discussed above, we will keep our presentation concise and avoid detailed explanations. Define $D_m=\min\{ |z_m^j-z_m^k|\big|j,k=0,1,\ldots \bar d, j\neq k\}$ and assume that (by considering subsequences if necessary) that each of the bounded sequences $D_m/|z_m^j-z_m^k|$ is converging. Our proof is by contradiction. Hence we assume that $V$ is continuously differentiable at $x^0$, but $x^0$ is not a critical point of the potential and
\begin{equation} \label{high_tide}
 |\partial_{x_1} V(x^0)|>0.
\end{equation}
In Case (a) we assume that
\begin{equation}\label{eq:crit_pt_case1_condition}
  1 < \limsup_{m\to\infty} \frac{-\ln \|R_m^2 V_{R_m}T_{z_m^j}u_\infty\|
  _{L^2}}{D_m} \quad \forall j\in\{0,1,\ldots \bar d\}
\end{equation}
holds and in Case (b) we assume that
\begin{equation}\label{eq:crit_pt_case2_condition}
  \liminf_{m\to\infty} \frac{-\ln\left|R_m^2 \left\langle V_{R_m}, \partial_{x_1} (\bar u_m^1)^2 \right\rangle\right|} {D_m} < 2
\end{equation}
holds. One of these two cases must occur. Indeed, if \eqref{eq:crit_pt_case1_condition} does not hold, then for some $r\in\{0,1, \ldots\bar d\}$ we have
\begin{equation}\label{contr_crit}
  \limsup_{m\to\infty} \frac{-\ln \|R_m^2 V_{R_m}T_{z_m^r} u_\infty\|
  _{L^2}}{D_m} \leq 1.
\end{equation}
Since $\lim_{m\to\infty}\|R_m^2 V_{R_m}T_{z_m^j} u_\infty\|_{L^2}=0$ for all $j\in\{0,1,\ldots\bar d\}$, we can assume with no loss of generality that $\|R_m^2 V_{R_m}T_{z_m^r} u_\infty\|_{L^2}\geq \|R_m^2 V_{R_m}T_{z_m^j} u_\infty\|_{L^2}$ for all $j\in\{0,1,\ldots\bar d\}$. Using \eqref{contr_crit} and the estimate
\begin{equation} \label{shownbelow_ab}
 \lim_{m\to\infty}\frac{\ln |R_m^2\langle V_{R_m}, \partial_{x_1}(\bar u_m^1)^2 \rangle|}{\ln \|R_m^2 V_{R_m}T_{z_m^r} u_\infty\|_{L^2}^2} \m=\m 0\m,
\end{equation}
which is established below, it follows that
\begin{align*}
  &\liminf_{m\to\infty} \frac{-\ln\left|R_m^2 \left\langle V_{R_m}, \partial_{x_1} (\bar u_m^1)^2 \right\rangle\right|} {D_m} \\
  =&\ \liminf_{m\to\infty} \frac{-\ln \|R_m^2 V_{R_m} T_{z_m^r}u_\infty\|_
  {L^2}^2}{D_m} \frac{-\ln\left|R_m^2 \left\langle V_{R_m}, \partial_{x_1} (\bar u_m^1)^2 \right\rangle\right|} {-\ln \|R_m^2 V_{R_m} T_{z_m^r} u_\infty\|_{L^2}^2}\\
  =&\ 0,
\end{align*}
which implies that if \eqref{eq:crit_pt_case1_condition} does not hold, then
\eqref{eq:crit_pt_case2_condition} holds, i.e. either Case (a) or Case (b) must occur. We will now establish \eqref{shownbelow_ab} by showing that
\begin{equation} \label{tearettea_ab}
 \lim_{m\to\infty} \frac{\|R_m^2 V_{R_m} T_{z_m^r} u_\infty\|_{L^2}^2} {|R_m^2 \langle \partial_{x_1} (\bar u_m^1)^2, V_{R_m} \rangle|} \m=\m0\m.
\end{equation}
This limit and $\lim_{m\to\infty} \|R_m^2 V_{R_m}T_{z_m^r} u_\infty\|_{L^2}=0$ imply \eqref{shownbelow_ab}. Note that
\begin{align}
R_m^2 \langle \partial_{x_1} (\bar u_m^1)^2, V_{R_m} \rangle \m=\m& R_m^3 \sum_{j=0}^{\bar d}\int_{\R^n}u_\infty^2(x) \partial_{x_1}V(R_mx+R_m z_m^j)\dd x \nonumber\\
&+R_m^3 \sum_{j=0}^{\bar d}\sum_{ k\neq j,\,k=0}^{\bar d} \int_{\R^n} u_\infty(x-z_m^j) u_\infty(x-z_m^k) \partial_{x_1}V(R_mx) \dd x \m. \label{aug15}
\end{align}
Since $|z_m^j-z_m^k|\to\infty$ as $m\to\infty$ for all $j,k\in\{0,1,\ldots \bar d\}$ with $j\neq k$, $\nabla V\in L^\infty(\R^n)$ and $|u_\infty(x)|< Ce^{-\sqrt{1-\gamma}|x|}$ for all $x\in\R^n$ and for some $C>0$ and $\gamma\in(0,1/2)$, it follows that the integrals in the second term on the right hand side of the above equation converge to 0 as $m\to\infty$. Next consider the integrals in the first term on the right hand side of the above equation. For every $x\in\R^n$ and $j\in\{0,1,\ldots \bar d\}$, we have $\partial_{x_1} V(R_m x+R_m z_m^j) u^2_\infty(x)\to \partial_{x_1} V(x^0) u^2_\infty(x)$ as $m\to\infty$ and $|\partial_{x_1} V(R_m x+R_m z_m^j) u^2_\infty(x)|\leq \|\nabla V\|_{L^\infty}u^2_\infty(x)$ for all $m$. Since $\|\nabla V\|_{L^\infty} u^2_\infty\in L^1(\R^n)$, it follows using the Lebesgue dominated convergence theorem that
$$ \lim_{m\to\infty} \int_{\R^n}u_\infty^2(x) \partial_{x_1}V(R_mx+R_m z_m^j)\dd x \m=\m \partial_{x_1} V(x^0)\|u_\infty\|_{L^2}^2 \m.$$
It therefore follows from \eqref{aug15} that
\begin{equation} \label{bittkin}
  \lim_{m\to\infty} \frac{\langle \partial_{x_1} (\bar u_m^1)^2, V_{R_m} \rangle}{R_m} \m=\m (\bar d+1) \partial_{x_1} V(x^0)\|u_\infty\|_{L^2}^2 \m.
\end{equation}
This along with $R_m\|V_{R_m} T_{z_m^r} u_\infty\|_{L^2}^2\to0$ as $m\to\infty$ implies \eqref{tearettea_ab}, which in turn implies \eqref{shownbelow_ab}. We will now consider the two cases.

\noindent
{\bf Case (a):} Suppose that \eqref{eq:crit_pt_case1_condition} holds.
Let $\bar z_m$ be the baricenter of the set $\{z_m^0,z_m^1,\ldots z_m^{\bar d}\}$. For some $q\in\{0,1,\ldots,\bar d\}$, let $z_m^q$ be located at a maximal distance from $\bar z_m$. For each $l\in\{0,1,\ldots,\bar d\}$ with
$l\neq q$ we have
$$  \cos \alpha_m^l \geq \frac{|z_m^l-z_m^q|}{2|\bar z_m-z_m^q|},$$
where $\alpha_m^l$ is the angle between the vectors $z_m^l-z_m^q$ and $\bar
z_m-z_m^q$. It can be shown (as in Case (i) earlier) that
$$  \liminf_{m\to\infty} \frac{|z_m^l-z_m^q|}{2|\bar z_m-z_m^q|} \geq
    \frac{d+1}{2d^2}. $$
Let $j=q$ in \eqref{eq:case1_eqn_n}. Then
\begin{align}
  0 = &\Big\langle T_{z_m^q}\partial_{x_k}u_\infty, \sigma(2p+1)T_{z_m^q} u_\infty^{2p} \sum_{l\neq q,\ l=0}^{\bar d} T_{z_m^l}u_\infty\Big\rangle \nonumber\\
  &+ C R_m^2 \sum_{j=0}^{\bar d} \left\|V_{R_m}T_{z_m^j}u_\infty\right\|_{L^2} + O\Big(\sum_{i=0}^{\bar d}\sum_{i\neq j,\, j=0}^d e^{-(1+\epsilon)
  \sqrt{1-\gamma}|z_m^i-z_m^j|}\Big) \label{eq:crit_pt_case1_eq}
\end{align}
for some fixed $\epsilon>0$. 
By choosing $x_k\parallel\bar z_m-z_m^q$ in \eqref{eq:crit_pt_case1_eq} we
get (by adopting the approach used in Case (i)) that for $m$ large
\begin{equation}\label{eq:crit_pt_case1_merm}
  \Big|\Big\langle T_{z_m^q} \partial_{x_k}u_\infty, \sigma(2p+1) T_{z_m^q}  u_\infty^{2p}\sum_{l\neq q,\ l=0}^{\bar d} T_{z_m^l}u_\infty\Big\rangle\Big| \geq \frac{d+1}{2d^2} C(\gamma)\exp^{-\sqrt{1+\gamma} D_m}.
\end{equation}
It now follows from \eqref{eq:crit_pt_case1_condition} that as $m\to\infty$,
for a sufficiently small $\gamma>0$ and each $j\in\{0,1,\ldots \bar d\}$, we
have (at least on a subsequence of $m$) that
$$ e^{\sqrt{1+\gamma}D_m} \|R_m^2 V_{R_m} T_{z_m^j} u_\infty\|_{L^2} \to 0. $$
This implies that the term on the left hand side of \eqref{eq:crit_pt_case1_merm} cannot be canceled in \eqref{eq:crit_pt_case1_eq}. But since $u_{E_m}=\tilde u_m+v_m$ is assumed to be a solution to the equations in \eqref{eq:case1_eqn_n} and so  \eqref{eq:crit_pt_case1_eq}, we can conclude that \eqref{eq:crit_pt_case1_condition} cannot hold.

\noindent
{\bf Case (b):}  Suppose that \eqref{eq:crit_pt_case2_condition} holds. By adding all the equations in \eqref{eq:case1_eqn_n} for which $x_k=x_1$ and
$j\in\{0,1,\ldots \bar d\}$, we obtain
\begin{align}
  0 =&\Big\langle \partial_{x_1}\bar u_m^1, R_m^2V_{R_m}(\bar u_m^0 + \bar u_m^1) + \sigma(\bar u_m^0 + \bar u_m^1)^{2p+1} - \sigma(\bar u_m^1)^{2p+1} - \sigma\sum_{j=0}^d T_{z_m^j}u_\infty^{2p+1}\Big\rangle\nonumber\\
  &+O\bigg(\sum_{j=0}^{\bar d} \left\|R_m^2 V_{R_m}T_{z_m^j} u_\infty\right\|
  _{L^2}^2\bigg)\nonumber\\
  &+O\bigg(\sum_{i=0}^{\bar d} \sum_{j=\bar d +1}^d e^{-\sqrt{1-\gamma}|z_m^i-z_m^j|}+\sum_{i=0}^{\bar d} \sum_{j\neq i,\ j=0}^d e^{-2\sqrt{1-\gamma}|z_m^i-z_m^j|}\bigg). \label{eq:crit_pt_case2_eq}
\end{align}
As in Case (ii) above, we have
$$  \Big\langle \partial_{x_1}\bar u_m^1, \sum_{j=0}^{\bar d} T_{z_m^j} u_\infty^{2p+1}\Big\rangle=0 $$
and it can be verified that
\begin{align*}
  &\Big\langle \partial_{x_1}\bar u_m^1, R_m^2V_{R_m}\bar u_m^0
  + (\bar u_m^0 + \bar u_m^1)^{2p+1}-(\bar u_m^1)^{2p+1}-
  \sum_{l=\bar d+1}^d T_{z_m^l}u_m^{2p+1} \Big\rangle\\
  =&\ O\bigg(\sum_{i=0}^{\bar d} \sum_{j=\bar d+1}^d e^{-\sqrt{1-\gamma}
  |z_m^i-z_m^j|}\bigg).
\end{align*}
We can now write \eqref{eq:crit_pt_case2_eq} as follows:
\begin{align}
  0 =& -\frac{1}{2}\int_{\R^n} (\bar u_m^1(x))^2 \partial_{x_1} V(R_m x)  dx
  + O\bigg(R_m\sum_{j=0}^{\bar d} \left\|V_{R_m} T_{z_m^j}
  u_\infty\right\|_{L^2}^2\bigg)\nonumber\\
  &+\frac{1}{R_m^3}\ O\bigg(\sum_{i=0}^{\bar d}\sum_{j=\bar d +1}^d e^
  {-\sqrt{1-\gamma}|z_m^i-z_m^j|}\bigg) + \frac{1}{R_m^3}\ O\bigg(\sum_{i=0}
  ^{\bar d} \sum_{j\neq i,\ j=0}^d e^{-2\sqrt{1-\gamma}|z_m^i-z_m^j|}\bigg).
  \label{eq:crit_pt_case2_simple_eq}
\end{align}
As $m\to\infty$, since \eqref{eq:crit_pt_case2_condition} holds, it follows
that the first term in \eqref{eq:crit_pt_case2_simple_eq} is much larger
than the last term. Moreover this first term converges to $(\bar d+1)\partial_{x_1} V(x^0)\|u_\infty\|_{L^2}^2/2$, see \eqref{bittkin}, and is bounded away from zero (since  $|\partial_{x_1} V(x^0)|>0$) and the second and third terms tend to zero as $m\to\infty$. Therefore we obtain the contradiction that $u_{E_m}=\tilde u_m+v_m$ cannot be a solution to \eqref{eq:case1_eqn_n}. Thus \eqref{eq:crit_pt_case2_condition} cannot hold.

If $V$ is continuously differentiable at $x^0$ and \eqref{high_tide} holds, then both Case (a) and Case (b) give the contradiction that $u_{E_m}=\tilde u_m+v_m$ does not solve \eqref{eq:case1_eqn_n}. Since one of these cases must occur it follows that \eqref{high_tide} cannot hold and $x^0$ is a critical point of the potential.

Theorem \ref{th:brelarge} is now completely proven.
\end{proof}

Again, the reason we cannot obtain the above result for arbitrary branches of bound states is the same as in finite $E$ case, see the discussion at the end of Section \ref{se:bif}.

The next theorem proves the existence of $C^1$ branches of ground states considered in the previous theorem. We show that given a finite set of $m$ critical points of the potential, there exists a unique $C^1$ branch of ground states consisting of $m$ humps. Each of these humps is a re-scaling of $u_\infty$ using \eqref{uelarge} and localizes at one of the $m$ critical points.

\begin{theorem}\label{th:bfelarge}
Let $\sigma<0$. For each $E>E_0$, consider the nonlinear equation
\begin{equation}\label{pal_nlsE}
 -\Delta u(x)+E^{-1}V(E^{-\frac{1}{2}}\m x)u(x)+u(x)+\sigma|u|^{2p}u(x)
 \m=\m 0 \FORALL x\in\R^n\m,
\end{equation}
where the potential $V:\R^n\to\R$ satisfies Hypothesis (H1) and (H2). Let $\{x_1,x_2,\ldots x_m\}$ be a finite subset of all the critical points of $V$ at which $V$ is twice differentiable and has a nonsingular Hessian. Fix either $\mu_i=1$ or $\mu_i=-1$ for each $i\in\{1,2,\ldots m\}$. Then there exists an $\e>0$ and an $E_\e>0$ such that for each $E>E_\e$, \eqref{pal_nlsE} has a unique real-valued solution $\Uscr_E\in H^2(\R^n)$ that satisfies the following estimate:
\begin{equation}\label{pal_solnest}
 \|\m\Uscr_E-\sum_{i=1}^m\mu_i u_\infty(\cdot-x_i\sqrt{E}\m)\|_{H^2(\R^n)} \m<\m\e\m.
\end{equation}
Here $u_\infty$ is the unique, radially symmetric and positive solution to \eqref{uinf}. Moreover, the map $E\mapsto\Uscr_E$ is a $C^1$ curve and along this curve
\begin{equation} \label{smarat}
 \lim_{E\to\infty} \|\m\Uscr_E-\sum_{i=1}^m\mu_i u_\infty(\cdot-x_i\sqrt{E}\m) \|_{H^2(\R^n)} \m=\m0\m.
\end{equation}
In addition, any solution $(U,E)$ of \eqref{pal_nlsE} satisfying $E>E_\e$ and $\inf_{\theta\in(0,2\pi]}\|U-e^{i\theta}\Uscr_E\|_{H^2}<\e$ is of the form $e^{i\theta}\Uscr_E$ for some $\theta\in(0,2\pi].$

Furthermore when $p\geq1/2$, for large $E$, one can count the number of negative eigenvalues of the linearized operator $L_+(\Uscr_E,E): H^2(\R^n)\mapsto L^2(\R^n)$:
$$ L_+(\Uscr_E,E)[v]=(-\Delta+E^{-1} V_{E^{-\frac{1}{2}}}+1)v +\sigma(2p+1) |\Uscr_E|^{2p}v, $$
where $V_{E^{-\frac{1}{2}}}(x)=V(E^{-\frac{1}{2}}x)$ for all $x\in\R^n$:
\begin{equation} \label{nege}
\#neg(L_+(\Uscr_E,E))=m+\sum_{i=1}^m n_i, 
\end{equation}
where $n_i$ is the number of negative eigenvalues of the Hessian of $V$ at $x_i$.

Consequently, if the nonlinearity is super-critical, $p>2/n,$ then all branches are orbitally unstable, while for critical and subcritical nonlinearities, $p\leq 2/n,$ the branches emerging from a single profile $u_\infty$ centered at a  local minima of the potential are the only stable ones.
\end{theorem}

\begin{proof}
Using Lemma \ref{lm:newdec_gen} and Remarks \ref{decay_not needed} and \ref{svn_ram} it can be shown that the following two claims are equivalent:
\begin{enumerate}[label={(\arabic*)}] \setlength\itemsep{0em}
\item there exists $\e>0$ and an $E_\e>0$ such that for each $E>E_\e$, \eqref{pal_nlsE} has a unique real-valued solution $\Uscr_E$ satisfying \eqref{pal_solnest},
\item there exists $\e>0$ and an $E_\e>\|V\|_{L^\infty}$ such that for each $E>E_\e$, \eqref{pal_nlsE} has a unique solution $\Uscr_E$ of the form
 \begin{equation} \label{pal_decom}
  \Uscr_E \m=\m \sum_{i=1}^m \mu_i T_{x_i\sqrt{E}+s_i}u_{i,E}+v_E\m,
 \end{equation}
 where $v_E$, $u_{i,E}$ and $s_i$ satisfy the following\m: $v_E\in
 H^2(\R^n)$ with $\|v_E\|_{H^2(\R^n)}<\e$ and for each $i\in\{1,2,
 \ldots m\}$, $|s_i|<\e$  and $u_{i,E}$ is the unique radially
 symmetric positive function in $H^2(\R^n)$ satisfying
 \begin{equation} \label{vis_rain}
  -\Delta u_{i,E} + u_{i,E} + \frac{V(x_i)}{E} u_{i,E} + \sigma|
  u_{i,E}|^{2p}u_{i,E} \m=\m 0\m,
 \end{equation}
 and furthermore the following orthogonality relations hold:
 \begin{equation} \label{pal_fd}
  \langle\m v_E, T_{x_i\sqrt{E}+s_i}\partial_{x_k} u_{i,E} \rangle
  \m=\m 0 \FORALL k\in\{1,2,\ldots n\},\quad\forall\;i\in\{1,2,\ldots
  m\} \m.
 \end{equation}
 (To simplify the notation, the dependence of $s_i$ on $E$ is not
 explicitly shown.)
\end{enumerate}

To show that claim (1) implies claim (2), note that: (a) the $x_i$'s are all distinct and so for $i\neq j$, $\sqrt{E}|x_i-x_j|\to \infty$ as $E\to\infty$, (b) each $u_{i,E}$ is a scaled version of $u_\infty$ i.e.,
\begin{equation} \label{uiscaling}
 u_i(x) \m=\m \bigg(1+\frac{V(x_i)}{E}\bigg)^{\frac{1}{2p}} u_\infty\bigg(\sqrt{1+\frac{V(x_i)}{E} }\m x\bigg) \FORALL x\in\R^n,
\end{equation}
and so
$$ \lim_{E\to\infty}\sum_{i=1}^m \|u_{i,E}-u_\infty\|_{H^2} \m=\m0, $$
and hence we can replace $T_{x_i\sqrt{E}}u_\infty$ with $T_{x_i\sqrt{E}} u_{i,E}$ in \eqref{pal_solnest} to get a new estimate for $\Uscr_E$ instead of \eqref{pal_solnest}, which will be valid for large $E$ and (c) Lemma \ref{lm:newdec_gen} can be applied with $u=\Uscr_E$ (satisfying the new estimate), $y_i=\sqrt{E} x_i$, $u_i=u_{i,E}$ and $\psi=\phi_i=0$ to get the
decomposition in \eqref{vis_rain} (which is similar to the decomposition
in \eqref{inf_decom}). Showing that claim (2) implies claim (1) is
simpler. We omit the details.

We will complete the proof of this theorem below by establishing claim (2). For this we will assume that a function $\Uscr_E$ that can be decomposed as in \eqref{pal_decom} and satisfying \eqref{pal_fd} solves \eqref{pal_nlsE} and then establish the existence of $s_i\m$s and $v_E$, with appropriate properties, whenever $E$ is sufficiently large. Since we are interested in finding $s_i\m$s in a small neighborhood of the origin in $\R^n$, while deriving all our estimates below we will assume that
\begin{equation} \label{dan_pro}
 \sup_{i=1,2,\ldots m} |s_i| \m\leq\m1\m.
\end{equation}
In the rest of this proof, to simplify our notation, the dependence of $u_{i,E}$ on $E$ will not be shown explicitly and $u_{i,E}$ will instead be written as $u_i$. We will use the notation $s=[\m s_1,s_2, \ldots s_m \m]^\top$.

Substituting $\Uscr_E$ from \eqref{pal_decom} into \eqref{pal_nlsE} and using the fact that $u_i$ solves \eqref{vis_rain}, we get that
\begin{align}
 &\hspace{30mm}-\Delta v_R + R^2\sum_{i=1}^m(V_R-V_i)\mu_i T_{x_i/R+s_i} u_i + R^2 V_R v_R + v_R  \nonumber\\[2pt]
 & + \sigma\bigg|\sum_{i=1}^m\mu_i T_{x_i/R+s_i} u_i +v_R\bigg|^{2p}
 \bigg(\sum_{i=1}^m\mu_i T_{x_i/R+s_i} u_i+v_R\bigg)-\sigma\sum_{i=1}^m
 \mu_i T_{x_i/R+s_i}|u_i|^{2p} u_i \m=\m 0\m.
 \label{pal_lunch}
\end{align}
Here we have introduced the notation
$$ V_i\m=\m V(x_i)\m, \qquad R\m=\m \frac{1}{\sqrt{E}}\m, \qquad
   V_R(x)\m=\m V(Rx)\m, \qquad v_R(x)\m=\m v_E(x) \FORALL x\in\R^n\m.$$
Recall the nonlinear map $N(U,v)$ introduced in \eqref{bound_N}:
$$ N(U,v)=\sigma|U+v|^{2p}(U+v)-\sigma|U|^{2p}U-\sigma(2p+1)|U|^{2p}v,$$
which satisfies the bound in \eqref{psmall} and \eqref{est_nl}. Like in the proof of Theorem \ref{th:brelarge} (see \eqref{eq:bif_Elarge} and \eqref{eq:stationary_split_Elarge}), we will rewrite \eqref{pal_lunch} using the linear operator $L_+(u,R)$, which is denoted differently in this proof for convenience, and the nonlinear operator $N$. Define the linear operator $L_\infty^{R,s}:H^2(\R^n)\to L^2(\R^n)$ as follows:
for all $v\in H^2(\R^n)$,
$$ L_\infty^{R,s} v \m=\m -\Delta v+R^2V_R\m v + v + \sigma(2p+1)
   \bigg|\sum_{i=1}^m\mu_i T_{x_i/R+s_i} u_i\bigg|^{2p} v \m. $$
Using $L_\infty^{R,s}$ and $N$, \eqref{pal_lunch} can be rewritten as
\begin{align}
 &L_\infty^{R,s}\m v_R + R^2\sum_{i=1}^m(V_R-V_i)\mu_i T_{x_i/R+s_i}u_i+ N\bigg(\sum_{i=1}^m\mu_i T_{x_i/R+s_i}u_i,v_R\bigg)
 \nonumber\\[2pt]
 & \hspace{-12mm}+\sigma\bigg|\sum_{i=1}^m\mu_i T_{x_i/R+s_i}u_i
 \bigg|^{2p} \bigg(\sum_{i=1}^m\mu_i T_{x_i/R+s_i}u_i\bigg)-\sigma
 \sum_{i=1}^m\mu_i T_{x_i/R+s_i}|u_i|^{2p}u_i \m=\m 0\m.
 \label{pal_topp}
\end{align}
We can write $L_\infty^{R,s}=L_+^{R,s}+L_\delta^{R,s}$, where $L_+^{
R,s}, L_\delta^{R,s}:H^2(\R^n)\to L^2(\R^n)$ are defined as follows:
for all $v\in H^2(\R^n)$
$$ L_+^{R,s} v \m=\m -\Delta v+ v + \sigma(2p+1) \sum_{i=1}^m |T_{x_i/R+s_i}
   u_\infty|^{2p} v \m, $$
$$ L_\delta^{R,s} v \m=\m R^2V_R\m v + \sigma(2p+1) \bigg|\sum_{i=1}^m
 \mu_i T_{x_i/R+s_i}  u_i\bigg|^{2p} v -\sigma(2p+1) \sum_{i=1}^m |T_{x_i/R+s_i} u_\infty|^{2p} v\m. $$
Using \eqref{uiscaling} we get $\|L_\delta^{R,s}\|_{H^2\to L^2}\to0$ as $R\to0$, uniformly in $s$ satisfying \eqref{dan_pro}. For each $i\in\{1,2,\ldots m\}$, we have
$$ \ker\big[-\Delta+1+\sigma(2p+1)|T_{x_i/R+s_i} u_\infty|^{2p}\m
   \big]\m=\m \m {\rm span}\{T_{x_i/R+s_i} \partial_{x_k} u_\infty\m\big|
   \m k\m=\m1,2,\ldots m\}\m.$$
Applying Proposition \ref{prop:MS} in the Appendix to $L_+^{R,s}$, with $E(R)$, $s_k(R)$ and $V_k$ in the proposition being $1$, $x_k/R+s_k$ and $\sigma(2p+1)|u_\infty|^{2p}$, respectively, and then applying the spectral perturbation theory to $L_+^{R,s} + L_\delta^{R,s}$ by regarding $L_\delta^{R,s}$ as a perturbation, it follows that there exists $C, R_1>0$ such that $\widetilde P_{R,s}^\perp L_\infty^{R,s} \widetilde P_{R,s}^\perp$ is invertible and
\begin{equation} \label{spare_tyre}
 \|(\widetilde P_{R,s}^\perp L_\infty^{R,s} \widetilde P_{R,s}^\perp)^{-1} \|_{L^2\mapsto H^2}<C
\end{equation}
for all $R\in(0,R_1)$ and all $s$ satisfying \eqref{dan_pro}. Here $\widetilde P_{R,s}$ is the orthogonal projection operator onto
$$ {\rm span}\{T_{x_i/R+s_i}\partial_{x_j} u_\infty\m \big|\m i\m=\m 1,2,\ldots m,\ j\m=\m1,2, \ldots n\}\m. $$
Let $P$ (by suppressing its dependence on $R$ and $s$) be the orthogonal projection operator onto
$$ {\rm span}\{T_{x_i/R+s_i} \partial_{x_k} u_i\m
   \big|\m i\m=\m 1,2,\ldots m,\ k\m=\m1,2,\ldots m\}\m.$$
Using \eqref{uiscaling}, it follows that $\|\widetilde P_{R,s}-P\|_{L^2\mapsto L^2}\to0$ and $\|\widetilde P_{R,s}-P\|_{H^2\mapsto H^2}\to0$ as $R\to0$. Hence from \eqref{spare_tyre} we get that there exists $R_2\in(0,R_1)$ such that $P^\perp L_\infty^{R,s}P^\perp:P^\perp L^2\cap H^2\to P^\perp L^2$ is boundedly invertible for all $R\in(0,R_2)$ and $s$ satisfying \eqref{dan_pro} and there exists $K>0$ such that
\begin{equation} \label{bose_st}
 \sup_{ \begin{array}{c} R\in(0,R_2),\\ s\ \textrm{satisfying}
 \ \eqref{dan_pro}\end{array}}\| (P^\perp L_\infty^{R,s} P^\perp)^{-1}
 \|_{P^\perp L_2\to P^\perp L^2 \cap H^2} \m<\m K\m.
\end{equation}

Recall from \eqref{pal_fd} that $P^\perp v_R=v_R$. Define
$$ D(R,s) \m=\m \sigma\bigg|\sum_{i=1}^m\mu_i T_{x_i/R+s_i} u_i\bigg|^{2p} \bigg(\sum_{i=1}^m\mu_i T_{x_i/R+s_i} u_i\bigg)-\sigma \sum_{i=1}^m \mu_i T_{x_i/R+s_i} |u_i|^{2p} u_i\m.$$
It is easy to see that \eqref{pal_topp} is equivalent to the
following set of equations:
\begin{align}
 &P^\perp L_\infty^{R,s}P^\perp v_R + P^\perp R^2\sum_{i=1}^m(V_R-V_i)
 \mu_i T_{x_i/R+s_i} u_i + P^\perp N\bigg(\sum_{i=1}^m\mu_i T_{x_i/R+s_i} u_i, P^\perp v_R\bigg)\nonumber\\
 &\hspace{50mm}+P^\perp D(R,s) \m=\m 0 \label{pal_topp_a}
\end{align}
and for $k=1,2,\ldots n$ and $i=1,2,\ldots m$
\begin{align}
 &\bigg\langle L_\infty^{R,s}\m v_R + R^2\sum_{j=1}^m(V_R-V_j)\mu_j
 T_{x_j/R+s_j} u_j + N\bigg(\sum_{j=1}^m\mu_j T_{x_j/R+s_j} u_j, v_R
 \bigg), T_{x_i/R+s_i}\partial_{x_k} u_i\bigg\rangle \nonumber\\[2pt]
 &\hspace{35mm} +\bigg\langle D(R,s), T_{x_i/R+s_i}\partial_{x_k} u_i
 \bigg\rangle \m=\m 0 \m. \label{pal_topp_b}
\end{align}
A tedious but elementary calculation gives that for some $\gamma\in(0,1)$,
$$ \|D(R,s)\|_{L^2} {\rm \ \ is \ \ } O(e^{-(1-\gamma)\min\{|x_i-x_j|/R\,|\m i,j\m=\m1,2,\ldots m\m,\ i\m\neq\m j\}})\m.$$
In particular, $\|D(R,s)\|_{L^2}$ is $O(R^m)$  for any integer $m>0$, uniformly in $s$ satisfying \eqref{dan_pro}.

We will complete the proof of claim (2) in two steps. In the first step we will solve \eqref{pal_topp_a} for $v_R$, for each $s$ satisfying \eqref{dan_pro} and all $R$ sufficiently small. We will see that the dependence of $v_R$ on $R$ and $s$ is $C^1$ and $v_R\to0$ as $R\to0$, uniformly in $s$ satisfying \eqref{dan_pro}. In the second step, using some estimates for $v_R$ derived in the first step, we will solve \eqref{pal_topp_b} for $s$, for all $R$ sufficiently small. We will see that the dependence of $s$ on $R$ is $C^1$ and $s\to0$ as $R\to0$. This will complete the proof of claim 2 and also establish the claim in the theorem that the map $E\mapsto\Uscr_E$ is $C^1$ and satisfies \eqref{smarat}.

We now proceed with the first step. So as to explicitly indicate the dependence of $v_R$ on $s$, below we use the notations $v_{R,s}$ instead of $v_R$. But we will continue to use the notation $P^\perp$ by suppressing its dependence on $R$ and $s$. We can now rewrite \eqref{pal_topp_a} as
\begin{align}
 v_{R,s} &\m=\m (P^\perp L_\infty^{R,s}P^\perp)^{-1} \bigg[P^\perp
 R^2\sum_{i=1}^m(V_R-V_i)\mu_i T_{x_i/R+s_i}u_i + P^\perp N\bigg(
 \sum_{i=1}^m \mu_i T_{x_i/R+s_i}u_i, P^\perp v_{R,s}\bigg) \bigg.
 \nonumber\\
 &\hspace{70mm}\bigg.+P^\perp D(R,s)\bigg] \nonumber\\
 &\m=\m v_{R,s}^0+(P^\perp L_\infty^{R,s}P^\perp)^{-1} P^\perp N\bigg(
 \sum_{i=1}^m \mu_i  T_{x_i/R+s_i}u_i, P^\perp v_{R,s}\bigg)\m.
 \label{room_st}
\end{align}
Here
$$ v_{R,s}^0 \m=\m (P^\perp L_\infty^{R,s}P^\perp)^{-1}\big[P^\perp
   R^2\sum_{i=1}^m(V_R-V_i)\mu_i  T_{x_i/R+s_i}u_i + P^\perp D(R,s)
   \big] \m. $$
Since $V\in L^\infty(\R^n)$, $V$ is twice differentiable at $x=x_i$
and $x_i$ is a critical point for $V$ for every $i\in\{1,2,\ldots m\}$,
there exists $c_1>0$ such that $|V(x)-V(x_i)| \m\leq\m c_1|x-x_i|^2$
for all $x\in\R^n$ and every $i$. Therefore if $R\in(0,R_2)$ and
\eqref{dan_pro} holds, then
\begin{align}
  \sum_{i=1}^m\|(V_R-V_i)T_{x_i/R+s_i}u_i\|_{L^2}^2 &\m=\m \sum_{i=1
  }^m \int_{\R^n}[(V(Rx)-V(x_i)) u_i(x-x_i/R-s_i)]^2\dd x\nonumber\\[2pt]
  &\m=\m \sum_{i=1}^m\int_{\R^n} [(V(Rx+x_i+Rs_i)-V(x_i)) u_i(x)]^2\dd
  x \nonumber\\[2pt]
  &\m\leq\m c_2 R^4 \label{modal_analysis}
\end{align}
for some constant $c_2>0$. The above estimate, the fact that $\|D(R,
s)\|_{L^2}$ is $O(R^m)$ and \eqref{bose_st} imply that $\|v_{R,s}^0
\|_{H^2}$ is $O(R^4)$. We will now apply the contraction mapping
principle to \eqref{room_st}. From Lemma \ref{lm:nuv} we get that there exists a constant $C>0$ (independent of $s_i$, since \eqref{dan_pro} holds) such that
\begin{equation} \label{one_year}
  \bigg\|\m N\bigg(\sum_{i=1}^m \mu_i T_{x_i/R+s_i} u_i, P^\perp v_1\bigg)- N\bigg(\sum_{i=1}^m \mu_i T_{x_i/R+s_i} u_i, P^\perp v_2\bigg) \bigg\|_{L^2}\m\leq \m C\|v_1-v_2\|^{\min\{2,2p+1\}}_{H^2}
\end{equation}
for all $v_1,v_2\in H^2(\R^n)$ with $\|v_1\|_{H^2},\|v_2\|_{H^2}<1$ and all $R\in(0,R_2)$. Fix $r\in(0,1)$ and $R_3\in(0,R_2)$ such that
$$ \|v_{R,s}^0\|_{H^2}\m<\m \frac{r}{2} \FORALL R\in(0,R_3)\quad {\rm and}
   \quad  4KCr^{\min\{1,2p\}}\m<\m1\m. $$
Let $B(0,r)$ denote the closed ball of radius $r$ centered at the
origin in $H^2(\R^n)$. Then for each $R\in(0,R_3)$ and $s$ satisfying
\eqref{dan_pro}, the map
$$ \Nscr_{R,s}(v) \m=\m v_{R,s}^0 + (P^\perp L_\infty^RP^\perp)^{-1}
   P^\perp N\bigg(\sum_{i=1}^m \mu_i T_{x_i/R+s_i} u_i, P^\perp
   v\bigg) $$
is a contraction on $B(0,r)$. It now follows from the contraction mapping principle that for each $R\in(0,R_3)$ and $s$ satisfying \eqref{dan_pro}, there exists a unique solution $v_{R,s}\in B(0,r)$, with $P^\perp v_{R,s}=v_{R,s}$, that solves \eqref{room_st}. Since the dependence of $\Nscr_{R,s}$ on $R$ and $s$ is $C^1$ and $\Nscr_{R,s}$ are uniform contractions on the set $\{(R,s)\m\big|\m R\in(0,R_3), s\ {\rm satisfies\ \eqref{dan_pro}}\}$, it follows that the dependence of $v_{R,s}$ on $R$ and $s$ is also $C^1$. Furthermore, for each $R\in(0,R_3)$ and $s$ satisfying \eqref{dan_pro}, $\Nscr_{R,s}$ is a contraction on $B(2\|v_{R,s}^0\|,0)$ as well, meaning that $\|v_{R,s}\|_{H^2}$ is also $O(R^4)$. This completes step 1. We now proceed with step 2.

Differentiating \eqref{vis_rain} with respect to $x_k$ we get that for each $i\in\{1,2,\ldots m\}$
$$ -\Delta\m\partial_{x_k}u_i + \partial_{x_k}u_i + R^2V_i\m
    \partial_{x_k}u_i + \sigma(2p+1)|u_i|^{2p} \partial_{x_k} u_i
    \m=\m 0\m.$$
Using the above expression and the definition of $L_\infty^{R,s}$, it
is easy to see that \eqref{pal_topp_b} can be rewritten as follows:
for $k=1,2,\ldots n$ and $i=1,2,\ldots m$
\begin{align}
 &\bigg\langle R^2 (V_R-V_i) T_{x_i/R+s_i}u_i, T_{x_i/R+s_i}\partial_{x_k}
 u_i\bigg\rangle + \bigg\langle R^2\sum_{j=1,\ j\neq i}^m
 (V_R-V_j)\mu_j T_{x_j/R+s_j}u_j,\m \mu_i T_{x_i/R+s_i} \partial_{x_k} u_i
 \bigg\rangle\nonumber\\
 &\hspace{5mm}+\bigg\langle R^2(V_R-V_i) v_{R,s} + N\bigg(\sum_{j=1}^m
 \mu_j T_{x_j/R+s_j} u_j,v_{R,s}\bigg)+D(R,s),\m\mu_iT_{x_i/R+s_i}
 \partial_{x_k} u_i\bigg\rangle \nonumber\\
 &-\bigg\langle \sigma(2p+1)|T_{x_i/R+s_i}u_i|^{2p} v_{R,s}-\sigma(2p
 +1) \bigg|\sum_{j=1}^m\mu_j T_{x_j/R+s_j} u_j\bigg|^{2p} v_{R,s},\m
 \mu_i T_{x_i/R+s_i}\partial_{x_k} u_i\bigg\rangle \m=\m0\m.
 \label{prez_ele}
\end{align}
We will now simplify the first term in \eqref{prez_ele}. Denote the
Hessian of $V$ at $x$ (if it exists) by $\nabla^2 V(x)$ and its $k^{\rm th}$
row by $\nabla^2_k V(x)$.  Using the fact that each $x_i$ is a critical point
of $V$ at which $V$ is twice differentiable, we get that
\begin{align}
 \bigg\langle R^2(V_R-V_i)T_{x_i/R+s_i} u_i,T_{x_i/R+s_i} \partial_{x_k} u_i\bigg\rangle &\m=\m \frac{R^3}{2}\int_{\R^n} \partial
 V_{x_k}(Rx+x_i+Rs_i) \m u_i^2(x) \m\dd x \nonumber\\
 &\m=\m \frac{R^4}{2}\int_{\R^n} \nabla^2_k V(x_i).(x+s_i) \m u_i^2(x) \m\dd
 x + R^4 g(R,s_i) \nonumber\\
 &\m=\m \frac{R^4}{2} (\nabla^2_k V(x_i).\m s_i) \|u_i\|_{L^2}^2 + R^4
 g(R,s_i)\m. \label{2knotes}
\end{align}
Here, using Taylor's theorem and the exponential decay of $u_i$, we
have $\lim_{R\to0}g(R,s_i)=0$ uniformly in $s_i$ satisfying
\eqref{dan_pro}. Next we will show that the rest of the terms in
\eqref{prez_ele} (except the first one) are $o(R^4)$. Using the fact
that the functions $u_i$ and $\partial_{x_k}u_i$ are in $L^\infty$ and
decay exponentially, it can be shown via elementary estimates that
$$ \bigg\|\bigg(T_{x_i/R+s_i}|u_i|^{2p}-\bigg|\sum_{j=1}^m\mu_j T_{x_j/R+s_j} u_j\bigg|^{2p}\bigg) T_{x_i/R+s_i}\partial_{x_k} u_i\m\bigg\|_{L^2} $$
is $O(R^m)$, uniformly in $s$ satisfying \eqref{dan_pro}, for any
integer $m$. Hence the same is true for
$$ \bigg|\bigg\langle |T_{x_i/R+s_i}u_i|^{2p} v_{R,s}- \bigg|\sum_{j=1}^m\mu_j T_{x_j/R+s_j}u_j\bigg|^{2p} v_{R,s}, T_{x_i/R+s_i} \partial_{x_k} u_i\bigg\rangle\bigg|\m. $$
Finally, using \eqref{2knotes}, the above estimate and the estimates
$\|v_{R,s}\|_{H^2}$ is $O(R^4)$, $\|D(R,s)\|_{L^2}$ is $O(R^m)$ for any
integer $m$ and $\|N\big(\sum_{j=1}^m\mu_j T_{x_j/R+s_j}u_j,v_{R,s}
\big)\|_{L^2}$ is $o(R^4)$ (which follows from \eqref{one_year}), all
three valid uniformly in $s$ satisfying \eqref{dan_pro}, we get
that \eqref{prez_ele} can be written as
\begin{equation} \label{sch_drum}
 (\nabla^2_k V(x_i).\m s_i) \|u_i\|_{L^2}^2 + \Gscr^{ik}_R(s) \m=\m0
 \FORALL k\in\{1,2,\ldots n\}\m,\quad \forall\, i\in\{1,2,\ldots m\},
\end{equation}
where $|\Gscr_R^{ik}(s)|$ is $o(1)$ in $R$, uniformly in $s$ satisfying
\eqref{dan_pro}. Furthermore, since $v_{R,s}$ is differentiable with
respect to $s$, we can show that $\Gscr_R^{ik}(s)$ is differentiable
with respect to $s$ and its derivative is also $o(1)$ in $R$, uniformly
in $s$ satisfying \eqref{dan_pro}.
For each $i$, denote $\Gscr_R^i=[\m \Gscr_R^{i1}, \Gscr_R^{i2}, \ldots
\Gscr_R^{in}\m]^\top$ and let
$$ \Gscr_R \m=\m \bigg[\m \frac{[\nabla^2 V(x_1)]^{-1}\Gscr_R^1}{\|u_1\|_{L^2
   }^2}\m, \frac{[\nabla^2 V(x_2)]^{-1}\Gscr_R^2}{\|u_2\|_{L^2}^2}\m, \ldots
   \frac{[\nabla^2 V(x_m)]^{-1}\Gscr_R^m}{\|u_m\|_{L^2}^2}\m\bigg]^\top. $$
Then we can rewrite \eqref{sch_drum} as
$$ \|u_i\|_{L^2}^2 \nabla^2 V(x_i)\m s_i + \Gscr_R^i(s) \m=\m0 \FORALL
   i\in\{1,2,\ldots m\}\m, $$
which in turn can be rewritten as follows:
\begin{equation} \label{fin_cont}
 s \m=\m \Gscr_R(s)\m,
\end{equation}
where $\Gscr_R(s)$ and its derivative with respect to $s$ are both
$o(1)$ in $R$, uniformly in $s$ satisfying \eqref{dan_pro}. This means
that there exists $R_4\in(0,R_3)$ such that for each $R\in(0,R_4)$,
the map $\Gscr_R$ is a contraction on the closed unit ball in $\R^{n
\times m}$, with a contraction constant independent of $R$. Hence for
each $R\in(0,R_4)$ there exists a unique $s\in\R^n$, with $|s|\leq1$,
that solves \eqref{fin_cont}. Furthermore, since the dependence of
$\Gscr_R$ on $R$ is $C^1$, $\Gscr_R$ are uniform contractions for
$R\in(0,R_4)$ and $\Gscr_R$ is $o(1)$ in $R$, uniformly in $s$
satisfying \eqref{dan_pro}, it follows that the dependence of $s$ on
$R$ is $C^1$ and $s$ is also $o(1)$ in $R$. This completes the proof
of step 2.

We have now shown that there exists $\e>0$ and an $E_\e>0$ such that for each $E>E_\e$, \eqref{pal_nlsE} has a unique real-valued solution $\Uscr_E$ satisfying \eqref{pal_solnest} and the map $E\mapsto \Uscr_E$ is a $C^1$ curve satisfying \eqref{smarat}. Next we will count the number of negative eigenvalues of the linearized operator $L_+(\Uscr_E,E)$ along this curve. This will complete the proof of this theorem. We use the notation $\Uscr_R=\Uscr_E$ and $v_R=v_E$ and write $L_+(\Uscr_E,E)$ in terms of $R$ i.e., for all $v\in H^2(\R^n)$
$$ L_+(\Uscr_R,R)[v] \m=\m -\Delta v + R^2V_R v +v+ \sigma(2p+1) |\Uscr_R|^{2p} v\m.$$

First we will discuss the spectral properties of $L_+(\Uscr_R,R)$. From \eqref{uiscaling} we have $\lim_{R\to0} \|u_i-u_\infty\|_{H^2}=0$. Using this and the fact that $\|v_R\|_{H^2}$ is $O(R^4)$, it can be verified that $L_+(\Uscr_R,R)$ can be expressed as $\widetilde L_R + \widetilde W_R$, where
$$ \widetilde L_R[v] \m=\m -\Delta v + v+ \sigma(2p+1) \sum_{i=1}^m |T_{x_i/R+s_i}u_\infty|^{2p} v\m \FORALL v\in H^2(\R^n) $$
and $\lim_{R\to0}\|\widetilde W_R\|_{H^2\mapsto L^2}=0$. The first eigenvalue $\gamma$ of $(-\Delta +1)+\sigma(2p+1)|u_\infty|^{2p}$ is simple and negative and its second eigenvalue is zero with multiplicity $n$. Applying the spectral theory for operators with potentials separated by large distances (see proof of Proposition \ref{prop:MS} in the Appendix) and Proposition \ref{prop:MS} to $\widetilde L_R$, with $E(R)$, $V_k$ and $s_k$ in the proposition being $1$, $\sigma(2p+1)|u_\infty|^{2p}$ and $x_k/R+s_k$, respectively, and then applying the spectral perturbation theory to $\widetilde L_R + \widetilde W_R$ by regarding $\widetilde W_R$ as a perturbation, it follows that for $R$ small there exist exactly $m$ eigenvalues $\gamma_1^R, \gamma_2^R, \ldots \gamma_m^R$ of $L_+(\Uscr_R,R)$ which converge to $\gamma<0$ as $R\to0$ and exactly $mn$ eigenvalues $\lambda_1^R, \lambda_2^R,\ldots \lambda_{mn}^R$ of $L_+(\Uscr_R,R)$ with corresponding eigenfunctions $\phi_1^R,\phi_2^R,\ldots \phi_{mn}^R$ (we can suppose that these functions form an orthonormal set in $L^2$ as $L_+(\Uscr_R,R)$ is self-adjoint) such that for each $k\in\{1,2,\ldots mn\}$
$$ \lim_{R\to0}\lambda_k^R=0, \qquad L_+(\Uscr_R,R) \phi_k^R=\lambda_k^R\phi_k^R, \qquad \phi_k^R = \sum_{i=1}^m\sum_{j=1}^n a^k_{ij}(R)T_{x_i/R+s_i} \partial_{x_j} u_i + \Phi_k^R, $$
where $\langle\Phi_k^R,T_{x_i/R+s_i}\partial_{x_j} u_i\rangle=0$ and $a_{ij}^k(R)$ is bounded for all $i,j,k$ and $R$ small. All other eigenvalues of $L_+(\Uscr_R,R)$ remain positive for all sufficiently small $R$. Furthermore, $\widetilde P_R^\perp L_+(\Uscr_R,R) \widetilde P_R^\perp$ is invertible for all $R$ sufficiently small and there exists $C>0$ independent of $R$ small such that
\begin{equation} \label{stravels}
 \|(\widetilde P_R^\perp L_+(\Uscr_R,R) \widetilde P_R^\perp)^{-1} \|_{L^2\mapsto H^2}<C.
\end{equation}
Here $\widetilde P_R$ is the orthogonal projection operator onto
$$ {\rm span}\{T_{x_i/R+s_i}\partial_{x_j} u_\infty\m \big|\m i\m=\m 1,2,\ldots m,\ j\m=\m1,2, \ldots n\}\m. $$
To complete the proof of this theorem we need to identify the signs of the eigenvalues $\lambda_k^R$ for $R$ small. This is done below.

Let $P_R$ be the orthogonal projection operator onto
$$ {\rm span}\{T_{x_i/R+s_i}\partial_{x_j} u_i\m \big|\m i\m=\m 1,2,\ldots m,\ j\m=\m1,2, \ldots n\}\m. $$
Using \eqref{uiscaling}, it follows that $\|\widetilde P_R-P_R\|_{L^2\mapsto L^2}\to0$ and $\|\widetilde P_R-P_R\|_{H^2\mapsto H^2}\to0$ as $R\to0$. Hence it follows from \eqref{stravels} that $P_R^\perp L_+(\Uscr_R,R) P_R^\perp$ is invertible for all $R$ sufficiently small and there exists $C>0$ independent of $R$ small such that
\begin{equation} \label{siddhitravels}
 \|(P_R^\perp L_+(\Uscr_R,R) P_R^\perp)^{-1} \|_{L^2\mapsto H^2}<C.
\end{equation}
Define the operator $L_R, W_R:H^2(\R^n)\mapsto L^2(\R^n)$ as follows: For each $v\in H^2(\R^n)$,
$$ L_R v \m=\m -\Delta v+ v + \sigma(2p+1)\Big|\sum_{i=1}^m \mu_i T_{x_i/R+s_i}u_i\Big|^{2p} v,$$
$$ W_R v= R^2V_R v+ \sigma(2p+1)\Big(\Big|\sum_{i=1}^m \mu_i T_{x_i/R+s_i} u_i + v_R\Big|^{2p} - \Big|\sum_{i=1}^m \mu_i T_{x_i/R+s_i}u_i\Big|^{2p}\Big)v .$$
Then clearly
$$ L_+(\Uscr_R,R)v = L_R v + W_R v.$$
Using the approach used to obtain \eqref{est_nl} in Section \ref{se:bif} we get
\begin{equation} \label{Bbakshi}
 \Big\|\Big|\sum_{i=1}^m \mu_i T_{x_i/R+s_i} u_i + v_R\Big|^{2p} - \Big|\sum_{i=1}^m \mu_i T_{x_i/R+s_i}u_i\Big|^{2p}\Big\|_{L^2}=O(\|v_R\|_{H^2})
\end{equation}
and since $\|v_R\|_{H^2}$ is $O(R^4)$ it follows that $\|W_R\|_{H^2\mapsto L^2}$ is $O(R^2)$. Using $L_R$ and $W_R$, \eqref{pal_topp_a} can be rewritten as
\begin{align}
 &P_R^\perp L_R P_R^\perp v_R  +P_R^\perp R^2\sum_{i=1}^m(V_R-V_i)
 \mu_i T_{x_i/R+s_i} u_i  \nonumber\\
 \m=\m& - P_R^\perp W_R P_R^\perp v_R-P_R^\perp N\bigg(\sum_{i=1}^m\mu_i T_{x_i/R+s_i} u_i, P_R^\perp v_R\bigg)-P_R^\perp D(R,s). \label{celerioE}
\end{align}
Using the estimate $\|v_R\|_{H^2}$ is $O(R^4)$ and \eqref{est_nl} it follows that the right hand side of the above equation is $O(R^6)$. From \eqref{siddhitravels} and $\lim_{R\to0}\|W_R\|_{H^2\mapsto L^2}=0$, we get that $P_R^\perp L_R P_R^\perp$ is invertible for all $R$ sufficiently small and there exists a $C>0$ independent of $R$ small such that $\|(P_R^\perp L_R P_R^\perp)^{-1} \|_{L^2\mapsto H^2}<C$. Therefore, for $R$ small, using \eqref{celerioE} we can write $v_R=v_R^0+v_R^1$ so that $v_R^0$ is the unique solution of
\begin{equation} \label{many_changes}
 P^\perp_R L_R v_R^0= -P^\perp_R R^2\sum_{i=1}^m (V_R-V_i)\mu_i T_{x_i/R+s_i}u_i
\end{equation}
and $\|v_R^1\|_{H^2}$ is $O(R^6)$.

Apply $P_R^\perp$ to $\lambda_k\phi_k^R=L_R\phi_k^R+W_R\phi_k^R$ to obtain
\begin{align*}
 &\big(P^\perp_R L_R P^\perp_R - \lambda_k^R + P^\perp_R W_R P^\perp_R)\Phi_k^R = P^\perp_R \sum_{i=1}^m\sum_{j=1}^n a^k_{ij}(R)(R^2V_i-W_R) T_{x_i/R+s_i} \partial_{x_j} u_i \\
 &-P^\perp_R \sigma(2p+1) \sum_{i=1}^m\sum_{j=1}^n a^k_{ij}(R) \Big(\Big|\sum_{t=1}^m \mu_t T_{x_t/R+s_t} u_t\Big|^{2p}- |T_{x_i/R+s_i}u_i|^{2p}\Big) T_{x_i/R+s_i}\partial_{x_j} u_i \m.
\end{align*}
The operator (from $H^2\to L^2$) on the left hand side of the above equation is invertible for $R$ small with uniform in $R$ bound for the inverse. This follows from our earlier discussion on the invertibility of $P^\perp_R L_R P^\perp_R$ and the limits $\lambda_k^R\to0$ and $\|W_R\|_{H^2\mapsto L^2}\to0$ as $R\to0$. Using \eqref{modal_analysis} and \eqref{Bbakshi} it follows that the term on the right hand side of the above equation is $O(R^4)$ in $L^2$. Hence we get that $\|\Phi_k^R\|_{H^2}$ is $O(R^4)$. Taking the innerproduct of $L_+(\Uscr_R,R)\phi_k^R=\lambda_k^R\phi_k^R$ with $T_{x_i/R+s_i}\partial_{x_j} u_i$ we get
\begin{align}
  &a_{ij}^k(R) \lambda_k^R\|\partial_{x_j}u_i\|_{L^2}^2 = \Big\langle \sum_{t=1}^m\sum_{q=1}^n a_{tq}^k(R)(W_R-R^2V_t) T_{x_t/R+s_t} \partial_{x_q} u_t, T_{x_i/R+s_i}\partial_{x_j} u_i \Big\rangle + o(R^4)\nonumber\\
  &= \Big\langle \sum_{q=1}^n a_{iq}^k(R)(W_R-R^2V_i)T_{x_i/R+s_i} \partial_{x_q} u_i, T_{x_i/R+s_i} \partial_{x_j} u_i \Big\rangle + o(R^4)\nonumber\\
  &= \Big\langle \sum_{q=1}^n a_{iq}^k(R)(W_R-R^2V_R)T_{x_i/R+s_i} \partial_{x_q} u_i, T_{x_i/R+s_i} \partial_{x_j} u_i \Bigr\rangle \nonumber\\
  &\hspace{10mm}+ \Big\langle \sum_{q=1}^n  a_{iq}^k(R) R^2[V_R-V_i]) T_{x_i/R+s_i}\partial_{x_q} u_i, T_{x_i/R+s_i} \partial_{x_j} u_i \Big\rangle + o(R^4).  \label{pasta}
\end{align}
We will now derive estimates for the terms containing $W_R-R^2V_R$ in the above expression.
Consider the sets
$$ S_+^R =\Big\{x\in\R^n \m\Big|\m \sum_{j=1}^m \mu_j u_j(x-x_j/R-s_j)-|v_R(x)|>0\Big\},$$
$$ S_-^R =\Big\{x\in\R^n \m\Big|\m \sum_{j=1}^m \mu_j u_j(x-x_j/R-s_j)+|v_R(x)|<0\Big\}$$
and define
\begin{align*}
  \HHH^R=&\ \Big(\Big|\sum_{t=1}^m \mu_t T_{x_t/R+s_t} u_t + v_R\Big|^{2p} - \Big|\sum_{t=1}^m \mu_t T_{x_t/R+s_t}u_t\Big|^{2p}\Big)\\
  & - 2p\ \textrm{sign}\Big(\sum_{t=1}^m \mu_t T_{x_t/R+s_t} u_t\Big)\Big|\sum_{t=1}^m \mu_t T_{^{(x_t/R+s_t)}} u_t\Big|^{2p-1} v_R.
\end{align*}
Adapting the approach used to derive \eqref{est_nl}, it can be shown that $\|\HHH^R\|_{L^2(S_+^R\cup S_-^R)}$ is $O(\|v_R\|_{H^2}^{\min\{2p,2\}})$ if $p>1/2$ and 0 if $p=1/2$. So $\|\HHH^R\|_{L^2(S_+^R\cup S_-^R)}$ is $o(R^4)$. Since $\|v_R\|_{H^2}$ is $O(R^4)$, it follows for $n=1,2,3$ that $v_R \stackrel {L^\infty}{\to}0$ as $R\to0$. For $n\geq 4$, the same conclusion can be arrived at by first using regularity to show that $v_R$ converges to $0$ in $W^{2,q}$ for any $q\geq2$ (see Appendix). Using this and the exponential decays of $u_i$ and its derivative, we have
$\|\partial_{x_q}T_{x_i/R+s_i}u_i\|_{L^2(\R^n\setminus\{S_+^R\cup S_-^R\})}$ is $O(R)$ and so
\begin{equation} \label{new_caterer}
\Bigg|\int_{\R^n\setminus\{S_+^R\cup S_-^R\}} \HHH^R(x)  \partial_{x_q} u_i(x-x_i/R-s_i) \dd x\Bigg| \leq C \|\partial_{x_q} T_{x_i/R+s_i}u_i \|_{L^2(\R^n\setminus\{S_+^R\cup S_-^R\})} \|v_R\|_{L^2}=o(R^4).
\end{equation}

Clearly there exist coefficients $c_{tq}(R)$ which are $O(R^4)$ such that \eqref{many_changes} can be written as
$$ L_R v_R^0= -R^2\sum_{t=1}^m (V_R-V_t)\mu_t T_{x_t/R+s_t}u_t + \sum_{t=1}^m\sum_{r=1}^n c_{tr}(R)T_{x_t/R+s_t}\partial_{x_r} u_t. $$
Taking the derivative of the above equation along the $x_q$ direction we get
\begin{align}
 &L_R \partial_{x_q} v_R^0 + \sigma(2p+1)2p\ \textrm{sign}\Big(\sum_{t=1}^m \mu_t T_{x_t/R+s_t} u_t\Big)\Big|\sum_{t=1}^m \mu_t T_{x_t/R+s_t} u_t\Big|^{2p-1} \Big(\sum_{t=1}^m \mu_t T_{x_t/R+s_t}  \partial_{x_q}u_t\Big) v_R^0 \nonumber\\
 =&\ -R^2 \partial_{x_q} V_R \sum_{t=1}^m \mu_t T_{x_t/R+s_t}u_t - R^2\sum_{t=1}^m (V_R-V_t)\mu_t T_{x_t/R+s_t} \partial_{x_q} u_t + \sum_{t=1}^m\sum_{r=1}^n c_{tr}(R) T_{x_t/R+s_t}\partial^2_{x_q x_r} u_t.
 \label{compr_ortho}
\end{align}
Taking the innerproduct of the above equation with $T_{x_i/R+s_i}\partial_{x_j} u_i$
and $v_R=v_R^0+v_R^1$ with $\|v_R^1\|_{H^2}$ is $O(R^6)$ we get that
\begin{align}
 &\Big\langle\sigma(2p+1)2p\ \textrm{sign}\Big(\sum_{t=1}^m \mu_t T_{x_t/R+s_t} u_t\Big)\Big|\sum_{t=1}^m \mu_t T_{x_t/R+s_t} u_t\Big|^{2p-1} v_R T_{x_i/R+s_i}\partial_{x_q} u_i, T_{x_i/R+s_i} \partial_{x_j} u_i \Big\rangle \nonumber\\
 &\hspace{40mm} +\big\langle R^2[V_R-V_i] T_{x_i/R+s_i}\partial_{x_q} u_i, T_{x_i/R+s_i} \partial_{x_j} u_i \big\rangle \nonumber\\
 =&\ - R^3 \int_{\R^n} \partial_{x_q} V(Rx) u_i(x-x_i/R-s_i) \partial_{x_j} u_i(x-x_i/R-s_i) \dd x + o(R^4)\nonumber\\
 =&\ \frac{R^4}{2} \partial^2_{x_q x_j} V(x_i)\|u_i\|_{L^2}^2 + o(R^4). \label{sitter}
\end{align}
Recall that
$$ W_R- R^2 V_R =   \sigma(2p+1) \Big(\Big|\sum_{t=1}^m \mu_t T_{x_t/R+s_t} u_t + v_R\Big|^{2p} - \Big|\sum_{t=1}^m \mu_tT_{x_t/R+s_t}u_t\Big|^{2p}\Big).$$
Using  $\|\HHH^R\|_{L^2(S_+^R\cup S_-^R)}$ is $o(R^4)$ and \eqref{new_caterer}, it follows from \eqref{pasta} and \eqref{sitter} that
$$ a_{ij}^k(R) \lambda_k^R\|\partial_{x_j}u_i\|^2 = \frac{R^4}{2} \sum_{q=1}^n a_{iq}^k \partial^2_{x_q x_j} V(x_i)\|u_i\|_{L^2}^2 + o(R^4).$$
Recall that $\lim_{R\to0}\|u_i-u_\infty\|_{H^2}=0$. Using this, the above equation yields
\begin{equation} \label{baskin_robbins}
 a_{ij}^k(R) \lambda_k^R\|\partial_{x_j}u_\infty\|^2 = \frac{R^4}{2} \sum_{q=1}^n a_{iq}^k \partial^2_{x_q x_j} V(x_i)\|u_\infty\|_{L^2}^2 + o(R^4).
\end{equation}

Define
$$ A_k(R)=\begin{bmatrix} A_k^1(R) &  A_k^2(R) & \ldots &  A_k^m(R) \end{bmatrix},$$
where
$$ A_k^i(R)=\begin{bmatrix} a_{i1}^k(R) & a_{i2}^k(R)&\ldots & \ldots & a_{in}^k(R) \end{bmatrix}.$$
Let $\nabla^2 {\bf V}$ be the block diagonal matrix with diagonal entries $\nabla^2 V(x_1), \nabla^2 V(x_2), \ldots \nabla^2 V(x_m)$. Here $\nabla^2 V(x_i)$ is the Hessian of $V$ at $x_i$.
Multiplying \eqref{baskin_robbins} by $a_{ij}^l(R)$ and summing the result over $i$ and $j$ gives
\begin{equation} \label{Hetu_GMO}
 \lambda_k^R\|\partial_{x_j}u_\infty\|_{L^2}^2 A_k(R) A_l^\top(R) = \frac{R^4}{2} \|u_\infty\|_{L^2}^2 A_k(R)\nabla^2{\bf V}A_l^\top(R) + o(R^4).
\end{equation}
Let ${\bf A}_R$ be the $nm\times nm$ matrix whose $k^{\rm th}$ row is $\|\partial_{x_j}u_\infty\|_{L^2} A_k(R)$ and $\Lambda_R$ be the $nm\times nm$ diagonal matrix with diagonal entries $\lambda_1^R,\lambda_2^R, \ldots \lambda_{nm}^R$. Then it follows from \eqref{Hetu_GMO} that
\begin{equation} \label{purohit}
 \Lambda_R {\bf A}_R {\bf A}^\top_R = \frac{R^4}{2} \frac{\|u_\infty\|_{L^2}^2}{\|\partial_{x_j}u_\infty\|_{L^2}} {\bf A}_R\nabla^2{\bf V}\m {\bf A}_R^\top + o(R^4).
\end{equation}
Since the $\phi_k^R$s form an orthonormal set in $L^2$ and each $\|\Phi_k^R \|_{L^2}$ is $O(R^4)$, it follows that $\lim_{R\to0}{\bf A}_R{\bf A}^\top_R-I_{2m\times 2m}=0$. Let ${\bf E}_R$ be the unitary matrix obtained by applying the Gram-Schmidt orthonormalization process to the rows of ${\bf A}_R$. Then $\lim_{R\to0}{\bf A}_R-{\bf E}_R=0$ and \eqref{purohit} can written as
\begin{equation} \label{chinta}
 \Lambda_R = \frac{R^4}{2} \frac{\|u_\infty\|_{L^2}^2}{\|\partial_{x_j}u_\infty\|_{L^2}} {\bf E}_R\nabla^2{\bf V}\m {\bf E}_R^\top + o(R^4).
\end{equation}
Hence the sign of the eigenvalues $\{\lambda_1,\lambda_2,\ldots \lambda_{nm}\}$ are determined by the sign of the eigenvalues of the Hessian at the critical points $\{x_1,x_2,\ldots x_m\}$. 

Finally we discuss the orbital stability with respect to the dynamical system \eqref{GP} of the branches $(\psi_E,E)$ of solutions of \eqref{stationary} obtained from $(\Uscr_E,E)$ via the renormalization \eqref{uelarge}:
$$ \psi_E(x)=E^{2p}\Uscr_E(x\sqrt{E}). $$
The linearized operators $L_+(\psi_E,E)$ and $L_+(\Uscr_E,E)$ have the same spectra. Consequently, if the number of profiles $m\geq2$ formula \eqref{nege} shows that $L_+$ has at least two negative eigenvalues. Since along this branches $L_-$ is nonnegative, by the result in \cite{Gr}, we get that they are orbitally unstable. If there is a single profile $m=1$ centered at a local maxima or saddle point the branch is orbitally unstable by the same argument. If the profile at a local minima we have that $L_+$ has exactly one negative eigenvalue. The result in \cite{GSS} now implies that the orbital stability is determined by $\frac{d}{dE}\Nscr(E)$. From the formula \eqref{2normelarge} we deduce that the slope is positive for subcritical nonlinearity $p<2/n$ and hence these branches are orbitally stable. But the slope is negative for $p>2/n$ and this branches are then orbitally unstable. One can actually show that for $p=2/n$ the slope remains positive and hence these branches are also stable.

This completes the proof of the theorem.
\end{proof}

\section{Symmetry Breaking Bifurcation}\label{se:sb}

In this section we show how our results on the maximal extension of branches of bound states and their behavior at the $E=\infty$ limit, see Theorems \ref{correct_scaling} and \ref{th:bfelarge}, can be put together to show that, in the presence of symmetry, there is a symmetry breaking bifurcation point along the ground state branch emerging from $(0,E_0)$ where $-E_0$ is the lowest eigenvalue of $-\Delta+V,$ see Proposition \ref{th:ex}. We start by showing that if the potential in the Schr\"{o}dinger equation \eqref{stationary} is invariant under a subgroup of Euclidian symmetries, then this branch and its maximal extension must be invariant under the same subgroup.

\begin{proposition}\label{pr:sym} Consider the time independent Schr\"{o}dinger equation \eqref{stationary} with potential $V$ satisfying (H1)-(H3) and
$$V(Rx)=V(x)\mbox{ a.e. }x\in\R^n,\qquad \forall\ R\in G\preceq E(n),$$
where $G$ is a subgroup of the Euclidian group of isometries on $\R^n,$ denoted here by $E(n).$ Then we have:
$$\psi_E(Rx)=\psi_E(x)\mbox{ a.e. }x\in\R^n,\qquad \forall\ R\in G\preceq E(n)\mbox{ and }\forall E\in\Iscr$$
where $\{(\psi_E,E)\ |\ E\in\Iscr\}$ is the maximal extension, see Corollary \ref{cor:max}, of the real valued curve of solutions of \eqref{stationary} given in Proposition \ref{th:ex}, which bifurcates from the lowest eigenvalue
$-E_0$ of $-\Delta+V.$ Moreover, for all $E\in\Iscr$ there is
$\epsilon,\ \delta >0$ such that any solution $(\psi,E_*)$ of
\eqref{stationary}
 satisfying $|E_*-E|<\delta$ and $\min_{0\leq\theta<2\pi}\|\psi-e^{i\theta}\psi_{E}\|_{H^2}<\epsilon,$ where $(\psi_E,E)$ is on the above maximal curve, is also invariant under the group $G.$
\end{proposition}

\begin{proof} Since the Laplacian is invariant under Euclidian transformations, and $V$ is invariant under the action of $G$ we deduce that, for all $R\in G,$ $(\psi(Rx) , E)$ is a solution of \eqref{stationary} provided $(\psi(x), E)$ is a solution. Consider now
$$\Gscr=\{E\in\Iscr\ |\ \psi_E\mbox{ is invariant under the action of }G\}\subseteq\Iscr .$$
We will show that $\Gscr$ is non-empty, open and closed in $\Iscr$ and, since the later is an interval, we conclude that $\Gscr=\Iscr .$ Then the last part of the theorem follows from part (iii) in Corollary \ref{cor:max}.

We first prove $\Gscr\not=\emptyset .$ We know from Corollary \ref{cor:max} that a small subinterval $\Iscr_\delta=(E_0,E_0+\delta)$ if $\sigma<0$ or $\Iscr_\delta=(E_0-\delta,E_0)$ if $\sigma>0$ of the maximal interval $\Iscr$ is obtained via the local bifurcation result in Proposition \ref{th:ex}. Using the fact that the Sobolev norms are invariant under Euclidian transformations, with the notations in that proposition, for each $E\in\Iscr_\delta,$ and all $R\in G$ we have:
$$\|\psi_E(Rx)\|_{H^2}=\|\psi_E(x)\|_{H^2}\leq \epsilon.$$
Since $\psi_E(Rx)$ is also a real valued solution of \eqref{stationary}, by part (i) of the proposition we infer that $\psi_E(Rx)=\pm\psi_E(x).$ Now the eigenvector $\psi_0$ solves \eqref{linear-mode} and, consequently is invariant under the action of $G.$ Indeed $\psi_0(Rx)$ also solves the same equation, and, since $-E_0$ is a simple eigenvalue we have $\psi_0(Rx)=a\psi_0(x)$ for some $a\in\C.$ But both have the same $L^2$ norm and are real valued, hence $a=\pm 1.$ Finally they are both positive, hence $a=1.$ Consequently, by the invariance of the $L^2$ scalar product under Euclidian transformations we have:
$$\langle \psi_0(x),\psi_E(x)\rangle=\langle \psi_0(Rx),\psi_E(Rx)\rangle=\langle \psi_0(x),\psi_E(Rx)\rangle.$$
Hence, both $\psi_E(x)$ and $\psi_E(Rx)$ have the same $\psi_0(x)$ projection. By part (i) of Proposition \ref{th:ex} we have  $\psi_E(Rx)=\psi_E(x).$ Since this is valid for all $R\in G$ and all $E\in\Iscr_\delta$ we deduce $\emptyset\not=\Iscr_\delta\subseteq\Gscr.$

We now show that $\Gscr$ is open in $\Iscr.$ Fix $E_1\in\Gscr.$ From Corollary \ref{cor:max} $L_+(\psi_{E_1},E_1)$ is non-singular, therefore we can apply Theorem \ref{th:cont} at $(\psi_{E_1},E_1).$ With its notations, for each $E\in (E_1-\delta, E_1+\delta)$ we have
$$\|\psi_E(Rx)-\psi_{E_1}(x)\|_{H^2}=\|\psi_E(Rx)-\psi_{E_1}(Rx)\|_{H^2}=\|\psi_E(x)-\psi_{E_1}(x)\|_{H^2}\leq \epsilon.$$
By the uniqueness part of Theorem \ref{th:cont} we must have $\psi_E(Rx)=\psi_{E}(x)$ for all $E\in (E_1-\delta, E_1+\delta),$ hence $\Gscr$ is open in $\Iscr .$

Finally, we show that $\Gscr$ is closed in $\Iscr.$ Consider $\{E_k\}_{k\in\N}\subset\Gscr$ such that $\lim_{k\rightarrow\infty}E_k=E\in\Iscr.$ By triangle inequality and the invariance of $H^2$ norm under Euclidian transformations we have, for any $R\in G:$
$$\|\psi_{E}(Rx)-\psi_{E}(x)\|_{H^2}\leq \|\psi_{E}(Rx)-\psi_{E_k}(Rx)\|_{H^2}+\|\psi_{E_k}(Rx)-\psi_{E}(x)\|_{H^2}=2\|\psi_{E_k}(x)-\psi_{E}(x)\|_{H^2}. $$
By passing to the limit $k\rightarrow\infty$ and using the continuity in $H^2$ of the maximal branch we get $\lim_{k\rightarrow\infty}\|\psi_{E_k}(x)-\psi_{E}(x)\|_{H^2}.$ Therefore $\psi_{E}(Rx)=\psi_{E}(x)$ for all $R\in G$ and $\Gscr$ is closed in $\Iscr .$

The proof is now complete.
\end{proof}

Note that this proposition combined with Theorem \ref{th:defocusing} shows that, in the case of defocusing nonlinearity $(\sigma>0)$, all ground states of \eqref{stationary} are symmetric. However, in the case of focusing nonlinearity $(\sigma<0)$, symmetric ground states are generally unstable for large values of the parameter $E,$ and the ground state branch emerging from $(0,E_0)$ undergoes bifurcations:

\begin{theorem} \label{th:sb}
Consider the time independent Schr\"{o}dinger equation \eqref{stationary} with potential $V$ satisfying (H1)-(H3), \eqref{Hypo_largeE1}-\eqref{Hypo_largeE3},  $$V(Rx)=V(x)\mbox{ a.e. }x\in\R^n,\qquad \forall\ R\in G\preceq E(n),$$
and assume that the critical points of $V$ invariant under the action of $G$ are non-degenerate and are not local minima. Then any $C^1$ curve of ground states $E\mapsto\psi_E,\ E\in(E_\epsilon,\infty)$ with no singularity points and invariant under the action of $G$ must be orbitally unstable. In particular the maximal extension of the ground state branch bifurcating from $(0,E_0)$ has a bifurcation point (singularity) at a finite $E=E_+,\ E_0<E_+<\infty$.
\end{theorem}

\begin{proof} From Theorem \ref{th:brelarge} we deduce that,
$$u_E(x)=\frac{1}{E^{2p}}\psi_E\left(\frac{x}{\sqrt{E}}\right)$$
has on a sequences $\{E_k\}_{k\in\N}\subset\R,\ \lim_{k\rightarrow 0}E_k=\infty$ the property:
$$\lim_{k\rightarrow\infty}\|u_{E_k}(x)-\sum_{j=1}^du_\infty(x-x_j\sqrt{E_k})\|_{H^2}=0,$$
for some integer $d\in\N$ and a subset of critical points of the potential $V,$ $\Jscr=\{x_j,\ j=\overline{1,d}\},$ where $u_\infty$ is  the unique, positive, radially symmetric solution of \eqref{uinf}. Using the invariance of $\psi_E$ under the action of $G$ we get for all $E\in(E_\epsilon,\infty)$ and all $R\in G:$
$$u_E(Rx)=\frac{1}{E^{2p}}\psi_E\left(\frac{Rx}{\sqrt{E}}\right)=\frac{1}{E^{2p}} \psi_E\left(\frac{x}{\sqrt{E}}\right)=u_E(x)\ a.e.$$
Consequently
\begin{align*}
 &\|\sum_{j=1}^du_\infty(Rx-Rx_j\sqrt{E_k})-\sum_{j=1}^du_\infty(x-x_j\sqrt{E_k}) \|_{H^2}\\
 \leq& \|\sum_{j=1}^du_\infty(Rx-Rx_j\sqrt{E_k})-u_{E_k}(Rx)\|_{H^2}+ \|u_{E_k}(Rx)-\sum_{j=1}^du_\infty(x-x_j\sqrt{E_k})\|_{H^2}\\
 =&2\|u_{E_k}(x)-\sum_{j=1}^du_\infty(x-x_j\sqrt{E_k})\|_{H^2}
\end{align*}
and, for large $k,$ we deduce that
$$\sum_{j=1}^du_\infty(Rx-Rx_j\sqrt{E_k})=\sum_{j=1}^du_\infty(x-x_j\sqrt{E_k}),\qquad \forall R\in G.$$
In particular, if $x_j\in\Jscr$ then $Rx_j\in\Jscr$ for all $R\in G.$

In conclusion ${\rm card}\Jscr\geq 2$ or $\Jscr $ has a single critical point of the potential which is invariant under the action of $G.$ In the latter case the critical point is either a saddle or a local maxima, by hypotheses. In both cases, Theorem \ref{th:bfelarge} implies that $L_+(\psi_E,E)$ has at least two negative eigenvalues and the branch is orbitally unstable.

Finally we prove the last part by contradiction. Assume that Case (b) in Corollary \ref{cor:max} does not hold. Then by the same corollary the positive ground state curve $(\psi_E,E)$ bifurcating from $(0,E_0)$ forms a $C^1$ curve in $E\in(E_0,\infty)$. By the proof above $L_+(\psi_E,E),\ E\in (E_0,\infty)$ has at least two negative eigenvalue in contradiction to Remark \ref{rm:lpme0}.

Consequently Case (b) in Corollary \ref{cor:max} occurs showing that the maximal interval for this curve is $(E_0,E_+)$, $E_0<E_+<\infty$ and an eigenvalue of $L_+(\psi_E,E)$ approaches zero as $E\rightarrow E_+$. Now, if we assume the stronger hypothesis on the potential \eqref{hypo_smallE1}-\eqref{hypo_smallE3}, then there is actually a limit point $(\psi_{E_+},E_+)$ of this curve which is still a ground state and at which $L_+(\psi_{E_+},E_+)$ has non-trivial but finite dimensional kernel, see Corollary \ref{cor:comp}.

The theorem is now completely proven.
\end{proof}

The result generalizes to higher dimensions the one in \cite{KKP:sb}. We emphasize that this generalization is not at all trivial and relies on the delicate arguments of Theorem \ref{th:brelarge}. For example, if we consider a potential symmetric under reflections with respect to a hyperplane, it could be possible that the ground states on the branch bifurcating from $(0,E_0)$ approach one profile $u_\infty$ sliding along the hyperplane towards infinity. More precisely, in Theorem \ref{correct_scaling} $M=1$ and $x_k^1$ is on the hyperplane with $\lim_{k\to\infty}\frac{x_k^1}{\sqrt{E_{m_k}}}=\infty$. In this case, $L_+$ will have only one negative eigenvalue and the contradiction cannot be obtained. However, our local bifurcation from $|x|=\infty$ argument in Theorem \ref{th:brelarge} excludes this phenomenon and leads to the symmetry breaking bifurcation above.

There are other results identifying such symmetry breaking bifurcations, see for example \cite{Kirr, KK:sb}, but they rely on perturbative techniques and require special potentials, such as double wells with very large separations, to show that the bifurcation happens at small amplitudes.

Our result is comparable to the ones obtained via variational techniques, see for example \cite{AFGST:02,GS:bec}, but has the advantage of actually showing, besides the change in shape of the ground states, the existence of a singularity point which can then be analyzed with local bifurcation techniques. Note that if one of the eigenvector in the kernel of $L_+(\psi_{E_+}, E_+)$ is not invariant under the action of $G$, then some branches emerging from this point are no longer symmetric, see \cite{GSS:sgb}. Moreover, our Theorem \ref{th:bfelarge}, shows that all ground states emerging from one single profile centered at a local minima of a potential are orbitally stable while they might not be global minimizers of the energy with respect to fixed mass constraint. For example, such global minima will not exist when nonlinearity is $L^2$ supercritical i.e., $p>2/n$. Finally, we believe that our results can be extended to other equations including the Hartree equation, see Section \ref{sse:conc}, while the results in \cite{AFGST:02} can only be adapted to our problem with $p>2/n$ by changing the minimization problem to functionals unrelated to the dynamics of the Hamiltonian structure of the equation, see for example \cite{RW:88}. In this case they will carry no information about the stability of the minimizer, see the discussion in \cite{Kir:dcs}.

In the following section we show that our results can be put together to obtain the global picture of all ground state branches besides showing the existence of bifurcation points along symmetric branches.

\section{Conclusions and Examples}\label{se:examples}

To summarize our paper, we have completely determined the ground states of the Schr\"{o}dinger equation \eqref{GP} for defocusing nonlinearity $\sigma>0,$ see Theorem \ref{th:defocusing} and Remark \ref{rmk:defocusing}. We also analyzed the behavior of all bound states i.e., solutions of the time independent Schr\"{o}dinger equation \eqref{stationary} at both small and large $H^1$ (or equivalently $H^2$) norms. More precisely Proposition \ref{pr:nosol} shows that there are no nontrivial $(\psi_E,E)$ bound states with $-E$ away from the spectrum of $-\Delta+V$ and $\|\psi_E\|_{H^1}$ sufficiently small, while for a simple eigenvalue $-E_*$ of $-\Delta+V$ the nearby, small bound states form a $C^1$ two dimensional manifold obtained by rotating a curve $E\mapsto\psi_E,\ |E-E_*|<\delta$ of real valued solutions, see Proposition \ref{th:ex}. We have also shown that the bound states group themselves in $C^1$ manifolds, see Theorem \ref{th:cont}, and there are no such manifolds formed from very large bound states, see Theorem \ref{th:boundedness}, except when $E\mapsto\infty.$ In the latter case, up to a re-normalization, the bound states converge to a finite superposition of solutions of a limiting equation, see Theorem \ref{correct_scaling}. Again, we can be more precise when it comes to ground states manifolds, which, modulo rotations, converge to finitely many translations of the unique, positive and radially symmetric solution of the limiting equation, with each translation centering the radially symmetric solution at a critical point of the potential, see Theorem \ref{th:brelarge}. Moreover, in a neighborhood of any such finite superposition of positive, radially symmetric solutions of the limiting equation, each translated to a different critical point of the potential, the bound states form a unique $C^1$ manifold obtained by rotating a curve $E\mapsto\psi_E,\ E\in (E_*,\infty)$ of positive solutions, hence ground states, see Theorem \ref{th:bfelarge}. We now show, on a few examples, how to combine these results with basic tools in global bifurcation theory in order to find information about all bound states of a given Schr\"{o}dinger equation, far beyond the symmetry breaking phenomenon described in the previous section.

\subsection{Defocusing Nonlinearity}\label{sse:exdef}

Consider the Schr\"{o}dinger equation \eqref{GP} with $\sigma>0$ and assume hypothesis (H1), (H2) hold. If the spectral condition (H3) is satisfied then our Theorem \ref{th:defocusing} states that all ground states form a unique, two dimensional, $C^1$ manifold generated by rotating a real valued curve which starts at $(\psi_{E_0}\equiv 0,E_0)$ and is defined on the whole interval $(0, E_0)$ i.e., it approaches the boundary $E=0,$ where the linearization fails to be Fredholm. Note that ours is a much stronger result compared to the ones presented in \cite{jls:gsd}. It is true that we are dealing with a simpler nonlinearity however we are hoping to extend our technique to a much larger class of nonlinearities, hence fully generalizing the results in \cite{jls:gsd}. Moreover, if (H3) does not hold i.e., $-\Delta+V\geq 0,$ then Proposition \ref{pr:endp} part (ii) implies there are no ground states of \eqref{stationary} for $E>0.$ For $E\leq 0,$ see 
\cite{BL:nbs} for non-existence of ground states or any other bound states provided they are required to decay sufficiently fast (but polynomially in $|x|$) as $|x|\rightarrow\infty .$

Proposition \ref{pr:endp} also implies that there are no excited states (solutions of \eqref{stationary} which are not ground states) that can be rotated to a real valued one if $-E_0$ is the only negative e-value of $-\Delta+V.$ Now let us assume that a second negative eigenvalue $0>-E_1>-E_0$ exists and is simple. By Proposition \ref{th:ex} there is a curve $E\mapsto\psi_E,\ E\in(E_1-\delta, E_1)$ of real valued excited states. This curve can undergo bifurcations with interesting implications about its dynamical stability, see \cite{KK:sb} for the particular case of a double well potential with large separation. However, if the wells are not well separated for the bifurcation to occur at small amplitude or, if we want to work with more general potentials we run into the problem of not knowing the structure of solutions of \eqref{stationary} at the $E=0$ boundary of the Fredholm domain. Indeed, assume that case (b-) in part (ii) of our maximal extension Theorem \ref{th:max} does not hold. Then it follows from the same theorem that the curve can be uniquely continued (no bifurcations) on the interval $(0, E_1).$ We would need to know what kind of limit points can the curve approach as $E\searrow 0$ and the number of negative eigenvalues of the linearized operator near these limits to determine if our assumption is contradicted and bifurcation points emerge at $E_*>0,$ see the previous section for a similar but rigorous argument. The problem is that at $E=0$ the linearization of our map $DF_\psi(\psi,0)$ ceases to be Fredholm as zero is at the edge of the continuous spectrum. Hence, any curve approaching the $E=0$ boundary bifurcates from the edge of the continuous spectrum of the linearization at its limit. Such bifurcations are notoriously difficult to handle, and, while some progress has been made in the case of periodic potentials, see \cite{DVW:bcs}, the problem remains open for our type of potentials.

Moving now to analyzing all excited states that bifurcate from the trivial solution at a simple eigenvalue of $-\Delta+V,$ see Proposition \ref{th:ex}, global bifurcation theory implies that a connected component of the emerging curve ends up at $E=0,$ or converges back to $(0,E_*),$ where $-E_*<0$ is a another eigenvalue of $-\Delta+V,$ or blows up in $H^2$ norm, see \cite{jls:bsd, Rab:gbt, BT:gbt}. With our techniques we can exclude the latter case with one exception. Indeed, let's assume that $p$ is an integer, hence the nonlinearity is real analytic. If we start with the branch bifurcating from $(0,E_+)$ via proposition \ref{th:ex} and extend it on its maximal interval $(E_,E_+)$ via Theorem \ref{th:max}, we have the following dichotomy: either $E_-=0$ and we are done or $E_->0$ and there exists a sequence $E_n\searrow E_-$ such that $L_+(\psi_{E_n},E_n)$ has an eigenvalue $\lambda_n$ with $\lim_{n\rightarrow\infty}\lambda_n=0.$ But $\psi_{E_n}$ is bounded in $H^2,$ see Theorem \ref{th:boundedness}, and, by using the compactness results in Corollary \ref{cor:defocuscomp} with a possible passing to a subsequence we get a limit point in $H^2\times\R,$ $(\psi_{E_-},E_-)$ which, by continuity, is still a solution of \eqref{stationary} and $0$ is an eigenvalue of $L_+(\psi_{E_-},E_-).$ Now the structure zeroes of the analytic $F$ at a singular point, see \cite{BT:gbt}, implies that the whole curve has the limit point $(\psi_{E_-},E_-)$ as $E\searrow E_-$ and it can be analytically continued past the singularity with $L_+$ also becoming non-degenerate past the singularity. It is possible that $E$ is now increasing along the continuation of the initial curve i.e., $(\psi_{E_-},E_-)$ is a turning point. In this case we again consider the maximal interval $(E_,E'_+)$ on which the curve can be extended without any singularity. Note that by Corollary \ref{cor:endp} we must have $E'_+\leq E_0.$ On such a maximal curve it is possible that the limit point is $(0,E'_+)$ in which case $-E'_+$ must be an eigenvalue of $-\Delta+V,$ and we are done. By repeating the above argument we either get to $E_-=0$ or to a limit point $(0,E_*)$ after finitely many steps, or we have a countable number of bifurcations. In the latter case we project the analytically continued curve on the $(E, \Nscr=\|\cdot\|_{L^2}^2)$ plane. If we get something like in the Figure \ref{fig:badex} i.e., infinite many loops around which the curve turns counterclockwise and the sum of the encompassed areas is infinite, then it is possible that the curve blows up in $H^2$ norm before reaching $E=0.$ However, if the sum of the area is finite then we are guaranteed to have $E\rightarrow 0.$ Indeed, in this case the $L^{2p+2}$ norm remains bounded and the argument in Theorem \ref{th:boundedness} can be reproduced as follows. First note that when the projection of the curve on the $(E, {\cal N}(E)=\|\psi_E\|_{L^2}^2)$ plane forms a loop the self intersection point is not a bifurcation point, instead the actual solutions projected at this point differ in $L^{2p+2}$ norm to the power $2p+2$ by the factor $\frac{p+1}{-\sigma p}$ multiplying the area encompassed by the loop, see \eqref{eq:dnormp1}. If the motion around the loop is counterclockwise this norm jumps up when the curve returns to the same projected point while if the motion around the loop is clockwise the norm jumps down. In between these loops we can estimate the change in  $L^{2p+2}$ norm by breaking the curve into $C^1$ parts parameterized by $E,$ and using the differential equation \eqref{eq:dnormp1}. If the up jumps add up to a finite number and there is $0<E_-\leq E$ for all $E$ along the curve (note that we also have $E\leq E_0$) the end result is a uniform bound of the $L^{2p+2}$ norm along the entire curve. As in the proof of Theorem \ref{th:boundedness} we then get uniform bounds of the $L^2$ norms and, via regularity theory, uniform bounds for the $H^2$ norm. This means we have infinitely many bifurcation points in a bounded set in $H^2\times (0,\infty)$ which is bounded away from zero in $E.$ By the compactness result Corollary \ref{cor:defocuscomp} these bifurcation points will have an accumulation point which is excluded by the structure of zeroes of $F$ at singularity points, see \cite{BT:gbt}. Hence our assumption that there exists $E_->0$ such that $E_-\leq E$ for all $E$ along the curve is false and the curve approaches the $E=0$ boundary.

With this result and provided we know the limit points of solutions of \eqref{stationary} at $E=0$ we can fully analyze the excited states. We will show how to do it for attractive nonlinearity because, in this case, the curves ``end up" at $E=\infty$ where our results in Section \ref{se:large} have already determined the limit points at least for ground state branches.

\begin{figure}
\begin{center}
\includegraphics[scale=0.2]{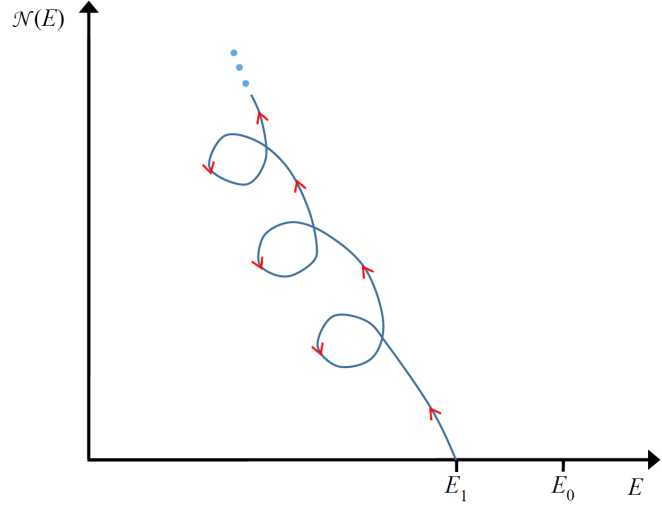}
\end{center}
\caption{The curve shown above is the projection of a branch of excited bound states $(E,\psi_E)$ of \eqref{stationary} with $\sigma>0$ to the $(E,\Nscr=\|\cdot\|^2_{L^2})$ plane. If the sum of the areas encompassed by the counterclockwise loops of this curve is finite, then the curve approaches $E=0$. If not, it is possible that $0<E_-\leq E$ for all $E$ on the curve. Note that $E\leq E_0$ for any $E$ on the curve by Proposition \ref{pr:endp}}
\label{fig:badex}
\end{figure}

\subsection{One Well Potentials with Focusing Nonlinearity}\label{sse:1well}

Consider the time independent Schr\"{o}dinger equation \eqref{stationary} with $\sigma<0$ and assume the potential $V$ satisfies hypothesis (H1), (H2) and, in addition it has only one critical point $x_0\in\R^n$ which is a non-degenerate minima i.e., $V$ is twice differentiable at $x_0$ with the gradient equal to zero and the Hessian strictly positive definite. From Theorem \ref{th:bfelarge} we know that here is a unique curve of ground states $E\mapsto \psi_E,\ E>E_\e$ which via the renormalization \eqref{uelarge} converges in $H^2,$ as $E\rightarrow\infty,$ to the unique, positive, radially symmetric solution of \eqref{uinf} centered at $x_0,$ $u_\infty(x-x_0).$ If we add the spectral hypothesis (H3) we also have the unique curve of ground states emerging from $(0,E_0),$ see Proposition \ref{th:ex}. Our {\em conjecture} is that these two curves can be smoothly extended until they connect to form a unique $C^1$ curve, which, together with its rotations give all the ground states for this problem. Indeed the conjecture has been proven in the particular case of one dimension, $n=1,$ for an even potential that is strictly increasing for $x>0,$ see \cite{js:ubs}. With the results developed in this paper we can prove the conjecture in any dimension and without any symmetry assumptions except for two obstacles: infinitely many bifurcations creating a similar phenomenon with the one described in Figure \ref{fig:badex}, see now Figure \ref{fig:badex2}, and canceling bifurcations, see Figure \ref{fig:1well}.

\begin{figure}
\begin{center}
\includegraphics[scale=0.2]{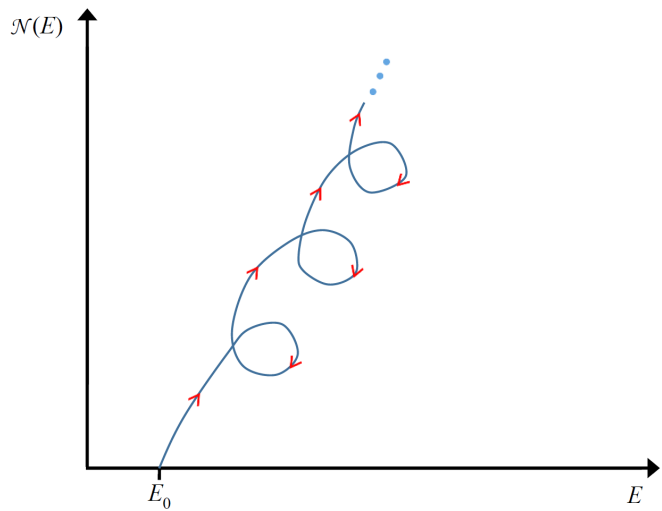}
\end{center}
\caption{The curve shown above is the projection of a branch of ground states $(E,\psi_E)$ of \eqref{stationary} with $\sigma<0$ to the $(E,\Nscr=\|\cdot\|^2_{L^2})$ plane. If the sum of the areas encompassed by the clockwise loops in this curve is finite, then the curve approaches $E=\infty$. If not, it is possible that $E\leq E_+<\infty$ for all $E$ on the curve. Note that $E\geq E_0$ for any $E$ on the curve by Proposition \ref{pr:endp}}
\label{fig:badex2}
\end{figure}

Indeed, let's start at the curve of positive ground states emerging from $(0,E_0).$ Use Theorem \ref{th:max} to extend it on its maximal interval $(E_0,E_+).$ If $E_+=\infty$ then Theorem \ref{th:brelarge} implies that the curve, modulo the renormalization \eqref{uelarge} converges in $H^2$ to $u_\infty(x-x_0),$ i.e. the two curves are one and the same and the conjecture is proven (we will discuss the fact that there are no other ground states except the rotations of this curve later). If $E_+<\infty$ then Corollary \ref{cor:comp} implies that the curve has a limit point $(\psi_{E_+},E_+)$ which is still a ground state, provided $V$ satisfies the additional hypotheses \eqref{hypo_smallE1}-\eqref{hypo_smallE3}. Assuming now that $p$ is an integer, hence $F$ is real analytic, we can use the structure of zeroes of an analytic map at a singular point to uniquely extend the curve past the singularity at $E_+.$ As in the previous subsection, it is possible that the extended curve moves to the left i.e., $(\psi_{E_+},E_+)$ is a turning point. In this case we again consider the maximal interval $(E_,E_+)$ on which the curve can be extended without any singularity. Note that by Corollary \ref{cor:endp} we must have $E_-\geq E_0.$ On such a maximal curve it is not possible that the limit point is $(0,E_-)$ because there are no eigenvalues of $-\Delta+V$ smaller than $-E_0$ and if $E_-=E_0$ then we get a contradiction with the fact that the only positive ground states near $(0, E_0)$ are the ones on the initial $(E_0, E_+)$ curve, see Proposition \ref{th:ex}. By repeating the above argument we either get to $E_+=\infty$ after finitely many steps or we have an infinite but countable number of bifurcations. In the latter case we  project the analytically continued curve on the $(E, {\cal N}(E)=\|\psi_E\|_{L^2}^2)$ plane. If we get something like in the Figure \ref{fig:badex2} i.e., infinite many loops around which the curve turns clockwise and the sum of the encompassed areas is infinite, then it is possible that the curve blows up in $H^2$ norm before reaching $E=\infty.$ However, if there are only finitely many such loops or the areas encompassed by them adds up to a finite number then, as in the previous subsection we first get uniform bounds of the $L^{2p+2}$ norms along the entire curve, then, using the arguments in Theorem \ref{th:boundedness} we get uniform bounds for the $L^2$ norms followed by uniform bounds in $H^1$ and $H^2.$ If we now assume that there exists $E_+<\infty$ such that for all $E$ on the curve $E\leq E_+$ then the curve lies in a bounded set of $H^2\times\R$ (recall that $E\geq E_0$ by Corollary \ref{cor:endp}) and, by Theorem \ref{th:comp}, the infinitely many bifurcation points along the curve will have an accumulation point. This is in contradiction with the structure of zeroes of $F$ at singularity points, see \cite{BT:gbt}. Hence, there is a sequence of ground states along this curve $(\psi_{E_k},E_k)$ with $\lim_{k\rightarrow\infty}E_k=\infty.$ In between these $E_k$ points we can estimate the change in  $L^{2p+2}$ norm by breaking the curve into $C^1$ parts parameterized by $E,$ and using the differential inequalities \eqref{ineq:dnorm}. If the up jumps in $L^{2p+2}$ norm (at the endpoints of the loops in Figure \ref{fig:badex2}) add up to a finite number we get the same results as for a single $C^1$ segment of the curve parameterized by $E,$ see Theorem \ref{correct_scaling}, namely that there is a number $b>0$ such that:
$$\lim_{k\rightarrow\infty}\frac{\|\psi_{E_k}\|_{L^{2p+2}}^{2p+2}}{E_k^{1-\frac{n}{2}+\frac{1}{p}}}=b,$$
from which, as in the proof of Theorem \ref{correct_scaling}, it follows that:
\begin{eqnarray}
\lim_{k\rightarrow\infty}\frac{\|\psi_{E_k}\|_{L^2}^2}{E_k^{\frac{1}{p}-\frac{n}{2}}}&=&\frac{-\sigma}{2}\left(\frac{2p+2-np}{p+1}\right)b,\nonumber\\
\lim_{k\rightarrow\infty}\frac{\|\nabla\psi_{E_k}\|_{L^2}^2}{E_k^{1-\frac{n}{2}+\frac{1}{p}}}&=&\frac{-\sigma}{2}\left(\frac{np}{p+1}\right)b.\nonumber
\end{eqnarray}
Hence, after the renormalization \eqref{uelarge} the sequence converges to a superposition of solutions of \eqref{uinf}. But, since we already assumed that $V$ satisfies the hypotheses \eqref{hypo_smallE1}-\eqref{hypo_smallE3} which are stronger than \eqref{Hypo_largeE1}-\eqref{Hypo_largeE3}, we deduce that the solutions of \eqref{uinf} must be centered at critical points of the potential, therefore they are centered at $x_0.$ By Theorem \ref{th:bfelarge} they are on the unique curve emerging from $u_\infty (x-x_0),$ therefore the two curves are connected.

\begin{figure}[t]
\begin{center}
\includegraphics[scale=0.2]{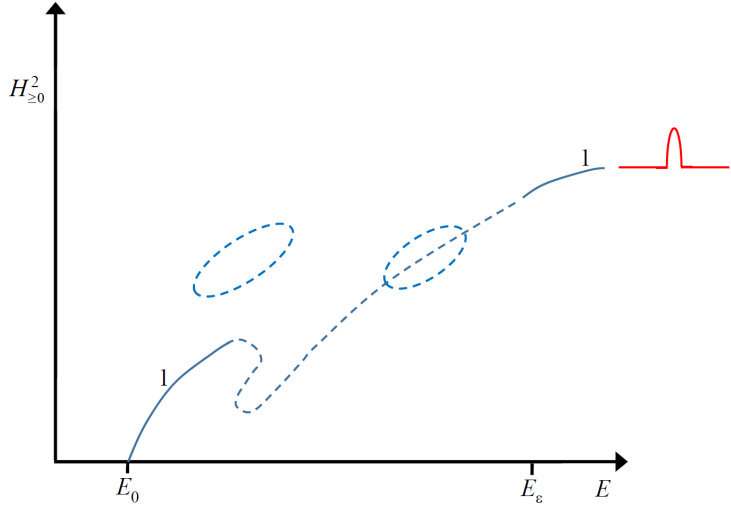}
\end{center}
\caption{A sketch of ground state branches in the case of one well potentials. The number on the branches indicate the number of negative eigenvalues of $L_+(\psi_E,E)$ along it. The positive ground state branch near $E=E_0$ can be continued to $E=E_\infty$ where it connects to the one profile branch bifurcating from the minima of the potential. We conjecture that modulo rotations, it contains all the ground states of this problem. However, at present, we cannot exclude the occurrence of loops. We also conjecture that the branch connecting $E=E_0$ to $E=E_\infty$ has no singularities along it. But at present we cannot exclude the represented canceling bifurcations.}
\label{fig:1well}
\end{figure}

We have finished proving a weaker form of our conjecture, namely that the small amplitude, positive, ground states for this problem form a curve which can be uniquely continued to $E=\infty,$ where it converges, modulo the renormalization \eqref{uelarge} to the unique, positive, radially symmetric solution of \eqref{uinf} centered at the minimum point $x_0.$ We now turn to the question of existence of other ground states outside the above curve. 
Suppose there is a positive ground state $(\psi_{E_*},E_*)$ not on the above curve. Then from Proposition \ref{pr:endp} we have $E_*\geq E_0$ and, depending on the kernel of $L_+(\psi_{E_*},E_*)$ we have two cases. If it is trivial, then Theorem \ref{th:cont} implies that nearby ground states from a two dimensional, $C^1$ manifold obtained by rotating a curve of real valued (and positive) solutions of \eqref{stationary}. By applying the above argument to this curve as opposed to the one emerging from $(0,E_0)$ we obtain that its unique analytic continuation either forms a closed loop or reaches $E=\infty$ where, modulo renormalization \eqref{uinf}, it converges in $H^2$ to $u_\infty(x-x_0).$ By the uniqueness of the curve emerging from $u_\infty(x-x_0),$ $(\psi_{E_*},E_*)$ must, in this situation, be on the curve connecting $(0,E_0)$ with $u_\infty(x-x_0)$ which contradicts the choice of $(\psi_{E_*},E_*).$ It remains that $(\psi_{E_*},E_*)$ and nearby solutions of \eqref{stationary} are part of a curve that forms a loop and its rotations. Now, if the kernel of $L_+(\psi_{E_*},E_*)$ is non-trivial then, by the structure of zeroes of an analytic $F$ near a singularity we have that either $(\psi_{E_*},E_*)$ is an isolated zero i.e., solution of \eqref{stationary}, in which case we can say it forms a degenerate loop, or there are finitely many, finite dimensional, analytical manifolds of zeroes of $F$ passing through $(\psi_{E_*},E_*).$ Now, if there is a sequence of points on these manifolds accumulating to $(\psi_{E_*},E_*)$ at which the kernel of $L_+$ is trivial, then there is a $C^1$ curve $E\mapsto\psi_E,\ E$ sufficiently close to $E_*$ which ends at $E_*.$ Choosing one point on this curve we are back to previous case and the first maximal extension ends up at $(\psi_{E_*},E_*).$ By repeating the above argument, we again get that $(\psi_{E_*},E_*)$ must be on a loop.

In conclusion, for a single well potential satisfying (H1)-(H3) and \eqref{hypo_smallE1}-\eqref{hypo_smallE3}, the ground state emerging from $(0,E_0)$ form, modulo rotations, a curve that can be continued analytically for all $E>0$ and at the $E=\infty$ limit it converges in $H^2,$ modulo renormalization \eqref{uelarge} to the unique, positive, radially symmetric solution of \eqref{uinf} centered at $x_0.$ Along this curve there can be only canceling bifurcations i.e., the curve experiences ``snaking'' or if curves of other solutions of \eqref{stationary} emerge at a bifurcation point they must form a closed loop, see Figure \ref{fig:1well}. All other ground states if they exist must form closed loops modulo rotations. If we assume that (H3) does not hold then the same argument, starting now with the curve emerging from $u_\infty (x-x_0)$ at $E=\infty,$ shows that it ends at $E=0.$ However, we can no longer say that there are no bifurcations along this curve generating new branches of solutions that also end at $E=0,$ and we can no longer exclude other branches, not connected to this one that instead of forming loops they end up at $E=0.$ All these are due to the fact that we do not yet know the limit points at the $E=0$ boundary.

\subsection{Double Well Potentials with Focusing Nonlinearity}\label{sse:2well}

As in the previous subsection we consider the time independent Schr\"{o}dinger equation \eqref{stationary} with $\sigma<0$ and assume the potential $V$ satisfies hypothesis (H1), (H2) and, in addition it has exactly three critical points, two minima and a maxima or saddle, all non-degenerate. Such potentials can be obtained by adding two one well potentials described in the previous subsection each centered at a different point. For simplicity we will start with (symmetric) potentials that are invariant under the reflection with respect to a hyperplane. Without loss of generality we can assume that the hyperplane is $\{x_1=0\}$ in which case we have $V(-x_1,x_2,\ldots, x_n)=V(x_1, x_2,\ldots, x_n).$ From Theorem \ref{th:bfelarge} we know that there are seven curves of non-trivial ground states $E\mapsto \psi_E,\ E>E_\infty$ which via the renormalization \eqref{uelarge} converge in $H^2,$ as $E\rightarrow\infty,$ to a superposition of positive, radially symmetric solution of \eqref{uinf} each centered at one critical point. By restricting the argument in the proof of Theorem \ref{th:bfelarge} to the Banach space of $H^2$ functions which are even in $x_1$ we deduce that three of these curves are formed by symmetric (even in $x_1$) ground states, namely the ones converging to $u_\infty$ centered at the maxima/saddle point, or centered at the two minima, or centered at all three critical points. The other four are asymmetric and converge to a superposition of $u'_\infty$s with one centered at a minima but none at the other minima.

If we add the spectral hypothesis (H3) we also have the unique curve of ground states emerging from $(0,E_0)$ which is also formed by symmetric (even in $x_1$) functions, see Proposition \ref{th:ex}. We can now redo the argument in the previous subsection but restricted to symmetric functions to show that this curve can be extended to $E=\infty,$ therefore connect to one of the three symmetric curves described above. Note that the global bifurcation results we used in the previous subsection are valid in any Banach spaces, hence in $H^2_{{\rm even}},$ but, in order to verify their hypotheses we need the compactness result in Theorem \ref{th:comp} and real analyticity of $F$ defined in \eqref{def:F}, hence we assume that the potential satisfies \eqref{hypo_smallE1}-\eqref{hypo_smallE3} and that $p$ is an integer. Now, numerical simulations suggest that in most cases the branch starting at $(0,E_0)$ connects to the one converging, as $E\rightarrow\infty,$ towards one profile ($u_\infty$) at each minima, see Figure \ref{fig:2well1}. We will discuss the exceptions later and remark that the next argument can be adapted to any of the three possible connections.

If the branch starting at $(0,E_0)$ connects to the profiles at each minima the number of negative eigenvalues of $L_+$ jumps from one to two, therefore there is at least one bifurcation point along it where the second eigenvalue of $L_+$ crosses zero. This is a particular case of Theorem \ref{th:sb} when the Euclidian subgroup of symmetries is the one generated by reflections with respect to the hyperplane $\{x_1=0\}.$ But now we can go further and analyze this bifurcation, see for example \cite{KK:sb} or, for the one dimensional version, \cite{KKP:sb}. At this point another branch of asymmetric ground states emerges along which $L_+$ has one negative eigenvalue and, depending on the direction along the branch, the ground states localize more and more in one of the two wells. By comparing with branches at $E=\infty,$ where this branch has to end, we see see that it will have each endpoint at the profile localized at one of the minima, see Figure \ref{fig:2well2}. Then we look at the curve of positive ground states emerging from one profile at the maxima/saddle. By the same global analytic argument enhanced by the fact that along it $E\geq E_0,$ see Corollary \ref{cor:endp}, and the fact that it cannot blow up in $H^2$ norm at a finite $E,$ we deduce that this branch ends up back at $E=\infty.$ Here it could connect to one of the negative branches obtained by rotating the curves described above by $e^{i\pi},$ however, since the branch started with positive ground states, this will require that it reaches a trivial ground state at some $E_*\geq E_0.$ By Proposition \ref{pr:nosol} $-E_*\leq -E_0$ must be an eigenvalue of $-\Delta+V,$ hence $E_*=E_0$ which contradicts the fact that the branch emerging from $(0,E_0)$ connects to the one converging, as $E\rightarrow\infty,$ towards one profile at each minima. Therefore, by the uniqueness of the branch emerging at $E=\infty$ from one profile at the maxima, this branch must connect to the one converging, as $E\rightarrow\infty,$ towards one profile at each critical point, see Figure \ref{fig:2well1}. Therefore, along this branch, the number of negative eigenvalues of $L_+$ jumps by two, the turning point accounts for one while the second one leads to a bifurcation from which another branch of asymmetric ground states emerges and has each of its endpoints at $E=\infty$ with two profiles, one localized at one of the minima the other one at the maxima/saddle, see Figure \ref{fig:2well2}. The branches and their connections in Figure \ref{fig:2well2} are fully explained and, modulo rotations, they should contain all ground states for this problem except maybe loops, as in the case of the previous subsection.

\begin{figure}[t]
\begin{center}
 \includegraphics[scale=0.13]{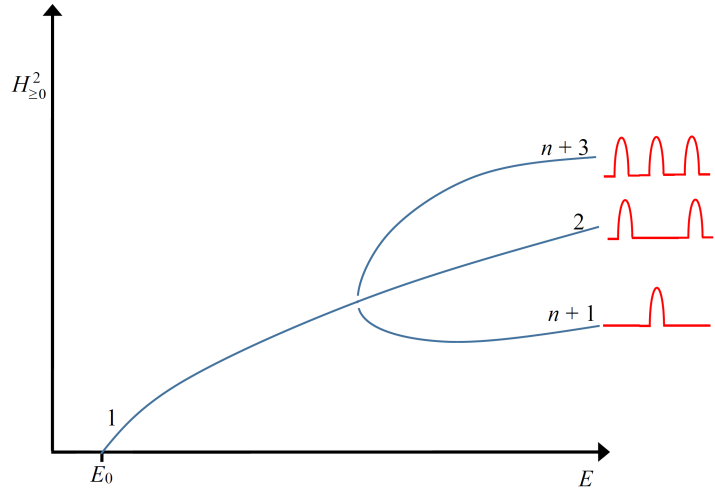}
\end{center}
\caption{A sketch of symmetric ground states in the case of double well potentials. The numbers on the branches indicate the number of negative eigenvalues of $L_+(\psi_E,E)$ along it. The only assumption made is that the branch near $E=E_0$ connects to the one near $E=\infty$ which has one profile at each minima. The rest of the connections including all the bifurcations along them and the asymmetric branches in Figure \ref{fig:2well2} are determined by our arguments. }
\label{fig:2well1}
\end{figure}

\begin{figure}
\begin{center}
\includegraphics[scale=0.13]{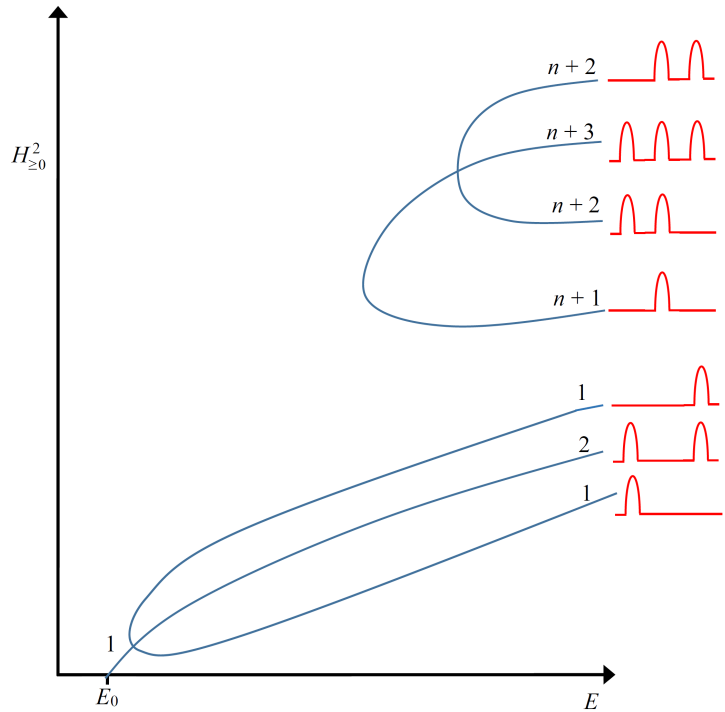}
\end{center}
\caption{A sketch of all ground states in the case of double well potentials. The numbers on the branches indicate the number of negative eigenvalues of $L_+(\psi_E,E)$ along it. The only assumption made is that the branch near $E=E_0$ connects to the one near $E=\infty$ which has one profile at each minima. The rest of the connections including all the bifurcations along them and the asymmetric branches are determined by our arguments.}
\label{fig:2well2}
\end{figure}

We now discuss the exceptions, when the branch starting at $(0,E_0)$ does not connect to the one converging, as $E\rightarrow\infty,$ towards one profile ($u_\infty$) at each minima. One dimensional $(n=1)$ simulations show that this branch connects to the one converging, as $E\rightarrow\infty,$ towards one profile at the maxima (there is no saddle point in one dimensions) in the ``small separation" case i.e. when all three critical points are very close to each other. However, the branch emerging at $E=\infty$ from two profiles, each localized at one of the minima, no longer connects to the one with three profiles, each localized at one critical point, but instead it connects to the one converging, as $E\rightarrow\infty,$ towards two profiles {\em each} approaching the maxima, see Figure \ref{fig:multilumps1}. 
\begin{figure}[h]
\begin{center}
\includegraphics[scale=0.13]{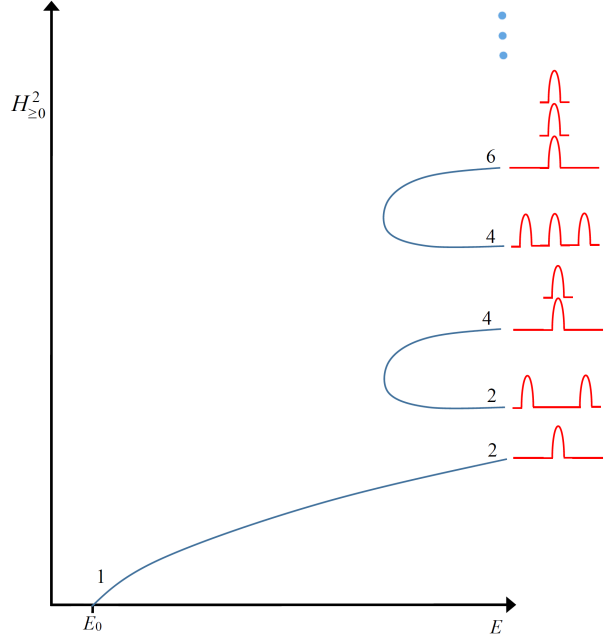}
\end{center}
\caption{In the case of double well potentials in {\em one space dimension} with small separation between the wells, numerical study indicates that the symmetric branches of positive ground states connect as shown above. The branch from $E=E_0$ connects to the one that localizes (as $E\to\infty$) to one profile at the maxima. The branch that localizes (as $E\to\infty$) to a profile at both minima connects to the one that localizes to two profiles at the maxima. The branch that localizes (as $E\to\infty$) to a profile at each critical point connects to the one that localizes to three profiles at the maxima. This is different from Figure \ref{fig:2well1} and the consequences are explained in Section \ref{sse:2well}. }
\label{fig:multilumps1}
\end{figure}
As discussed at the beginning of Section \ref{se:large} we have not yet analyzed these types of critical points. In the one dimensional case we do it in \cite{KKLN:sb} and explain why this connection happen and identify all the other ground state branches and their bifurcation points. We also explain what happens as we increase the the distance between the critical points: the bifurcation diagram is stable until the two curves above touch tangentially, then the one starting at $(0,E_0)$ switches and connects to the one converging, as $E\rightarrow\infty,$ towards one profile at each minima, while the one with an $E=\infty$ endpoint given by one profile at the maxima now connects to the one ending with two profiles at the maxima, see Figure \ref{fig:multilumps2}. By further increasing the distance between the wells the bifurcation diagram remains stable until this curve touches tangentially the one that connects one $E=\infty$ endpoint given by three profiles each at a critical point with the other $E=\infty$ endpoint given by three profiles {\em each} approaching the maxima. After that the  $E=\infty$ endpoint given by one profile connect to the $E=\infty$ endpoint given by three profiles each at a critical point, and, in this ``large separation" of wells in the potential, the bifurcation diagram settles into the one in Figure \ref{fig:2well1}, which we discussed above. Of course, the figure should be completed by adding  $E=\infty$ endpoints given by two or more profiles each approaching the maxima which will now be connected to $E=\infty$ endpoints having the same number of profiles at the maxima and two additional profiles each at a minima.

\begin{figure}[h]
\begin{center}
\includegraphics[scale=0.13]{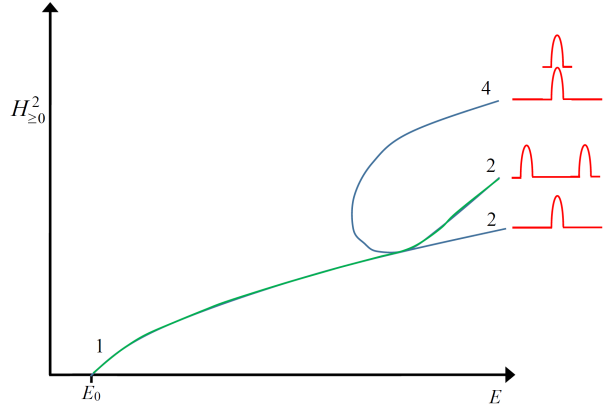}
\end{center}
\caption{The connection of positive symmetric ground state branches depicted in Figure \ref{fig:multilumps1} remains stable as the distance between the wells is increased until the branch from $E_0$ becomes tangential at some point to the branch that localizes (as $E\to\infty$) to two profiles. As the distance between the wells is increased further, the branch from $E_0$ connects to the branch that localizes (as $E\to\infty$) to a profile at both minima and  the branch that localizes (as $E\to\infty$) to one profile at the maxima connects to the one that localizes to two profiles at the maxima.
}
\label{fig:multilumps2}
\end{figure}

Finally, we discuss the case of non-symmetric potentials. The number and shape of the $E=\infty$ endpoints do not change but we can no longer separate them into symmetric and non-symmetric. Hence, in principle, the branch starting at $(0, E_0)$ can now connect to any of the seven branches at $E=\infty$ (or more if we add the ones which have multiple profiles approaching the maximum). Instead of analyzing each of these cases we could consider a continuous deformation connecting the current potential to a symmetric one. Starting now from the symmetric potential, the bifurcation diagram in Figure \ref{fig:2well2} remains unchanged due to the stability of (non-degenerate) bifurcations under perturbations. There are two ways in which a bifurcation and the connections it determines can disappear. One is when it is moved to $E=\infty$ but this will imply that we change the number of critical points of the potential. The second is when the bifurcation becomes degenerate i.e., the curves intersect at it tangentially then separate. This phenomenon was described in the previous paragraph albeit without loss of symmetry when the distance between wells is increased. In a future paper we will show that when the depth of one well is decreased, the bifurcation along the branch starting at $(0,E_0),$ see Figure \ref{fig:2well2}, eventually degenerates and this branch ends up connecting to $E=\infty$ endpoint with a single profile at the deeper minima while the endpoint with two profiles each at a minima will now connect with the one with one profile at the higher local minima.

\subsection{Conclusions}\label{sse:conc}

Besides describing all ground states for defocusing nonlinearity, we are on the verge of doing the same for focusing nonlinearity where their number and bifurcation points depends on the critical points of the potential. To finish the latter we still need to overcome three obstacles. The first one is illustrated in Figure \ref{fig:badex2} and consists of infinitely many bifurcations which, when the extended curve of bound states is projected onto the $(E,\Nscr =\|\cdot\|_{L^2}^2)$ plane, lead to infinitely many loops around which the curve turns clockwise and encompasses areas which add up to infinity. In this case we do not know yet how to control the $L^{2p+2}$ norm of the bound states on the curve and show, as in Theorem \ref{th:boundedness}, that the curve cannot blow up in $H^2$ norm if the parameter $E$ is in a compact subset of the interval $(0,\infty).$ The second obstacle is the analysis of $E=\infty$ limit points with profiles approaching the same critical point of the potential. Note that the partial results we already have show that they cannot occur at minima, hence the results presented in the examples above are essentially correct. The third obstacle only occurs when hypothesis (H3) does not hold or we are dealing with excited states of the defocusing problem. In this case, as seen in Subsections \ref{sse:exdef} and \ref{sse:1well}, branches can end up at $E=0$ and we will need a complete analysis of the limit points at this boundary to fully describe all solutions of our problem. As for excited states in the focusing case, we have not yet shown that they form relatively compact sets when bounded in $H^2$ with parameter $E$ in a compact subset of $(0,\infty).$ This result which we proved for ground states in Theorem \ref{th:comp} is essential in continuing manifolds of solutions past their possible singular (bifurcation) points, see the arguments in the subsections above. In fact, the proof of Theorem \ref{th:comp} goes through for excited states except for the case of splitting when profiles, while separating from each other, still stick together sufficiently close to have strong interacting terms. It is here where we used positivity of the profiles, valid for ground states, to get a contradiction. Nevertheless, we hope to surmount this technical detail and the obstacles discussed above in a future paper.

In conclusion, this paper shows that the algorithm described in \cite{Kir:dcs}, for identifying all zeroes of maps $F$ between Banach spaces is doable. Namely
\begin{enumerate}
\item Identify the Fredholm domain of $F$ i.e., the open set (which can be then broken in connected components) on which the Frechet derivative is a Fredholm operator. In our case it is $(\psi,E)\in H^2(\R ^n,\C)\times (0,\infty).$
\item Find the limit points of zeroes of $F$ at the boundary of the Fredholm domain. In our case, for focusing nonlinearity Corollary \ref{cor:endp} and Theorem \ref{th:boundedness} combine to show that there are no limit points at the $\{E=0\}$ or $\{E>0,\ \|\cdot\|_{H^2}=\infty\}$ boundaries. The limit points at $E=\infty$ boundary are a subset of the set described in Theorem \ref{correct_scaling} and are precisely identified for ground states in Theorems \ref{th:brelarge} and \ref{th:bfelarge}. For defocusing nonlinearity we are still missing the limit points at the $E=0$ boundary.
\item Use global bifurcation techniques to continue branches of zeroes of $F$ inside the Fredholm domain from their limit point at the boundary to their other limit point on the boundary. By comparing the discrete spectrum of the Fr\'echet derivative of $F$ at the two limit points identify the bifurcations along these branches and use local bifurcation techniques to determine the other branches emerging at these bifurcation points. In the examples above we used the analytic global bifurcation theory as it only requires relative compactness of the zeroes of $F$ in bounded domains away from the boundary of the Fredholm domain, which we proved in Theorem \ref{th:comp}. Degree based global bifurcation results are also very useful but we only applied them to the defocusing case, see Subsection \ref{sse:exdef}, because they require stronger compactness results.
\end{enumerate}
Note that the algorithm should, and it actually did in our examples, identify all zeroes in the Fredholm domain except the ones that form closed manifolds which do not reach its boundary. It adds to the classical global bifurcation theory the step where all possible limit points at the boundary of the Fredholm domain are identified. We think that this is possible in many other problems besides the nonlinear Schr\"{o}dinger one. In fact we only used its Hamiltonian structure (energy, mass and the differential equation \eqref{eq:diffen}) and Pohozaev identity to control the relevant norms of zeroes as they approach the boundary, find the appropriate renormalization which will keep these norms bounded and then show that the limit points are actual solutions of a simpler equation, see Theorem \ref{correct_scaling} and \ref{th:boundedness}. All these tools are available to a large class of dispersive wave equations, see for example \cite{Kir:dcs}, except maybe the Pohozaev identity. The problem then reduces to identifying similar identities in the other problems or showing, as we strongly believe, that they are not necessary for finding the limit points and implementing the above algorithm.

\section{Appendix}\label{se:app}

In this appendix we first recall regularity and decay properties of solutions of nonlinear Schr\"odinger equations. Then we prove certain compactness results for the set of solutions of families of such equations and spectral properties of their linearized operators which we use throughout this paper.

Slight adaptations of the regularity argument in \cite[Section 8.1]{Caz:nls} leads to:
\begin{proposition}\label{prop:h1toh2}
Let $\psi_E$ be a $H^1$ weak solution of
\begin{equation}\label{a:stationary}
 (-\Delta + V + E) \psi + \sigma|\psi|^{2p} \psi = f
\end{equation}
where $E\in\R$, $V$ is real-valued and satisfies Hypotheses (H1)-(H2), $\sigma\in\R,$ $0<p<2/(n-2)$ and $f\in L^q,$ for all $2\leq q<\infty.$ Then $\psi_E$ is in $W^{2,q}(\R^n)$ for all $2\leq q<\infty,$ in particular it is in $L^\infty(\R^n)$.
\end{proposition}

As we will mostly be interested in families of equations \eqref{a:stationary} indexed by $E,$ the same argument also shows:
\begin{remark}\label{rmk:h1toh2b} If $(\psi_E,E)$ is a bounded set in $H^1\times\R$ of solutions of \eqref{a:stationary}, then the set is also bounded in $W^{2,q}\times\R$ for all $q\geq2$, in particular is bounded in $L^\infty\times\R.$
\end{remark}

\begin{proposition}\label{prop:exdecay}
Under the hypothesis of Proposition \ref{prop:h1toh2} with $E>0$ and $f\equiv 0,$ we have that for any $\gamma\in(0,E)$ there exists a $C(\gamma)>0$ such that for all $x\in\R^n$ we have:
\begin{eqnarray} |\psi_E(x)| & \leq & C(\gamma)e^{-\sqrt{E-\gamma}|x|} \nonumber\\
|\nabla\psi_E(x)| & \leq & C(\gamma)e^{-\sqrt{E-\gamma}|x|} \nonumber
\end{eqnarray}
\end{proposition}

When applied to families of Schr\"odinger equations indexed by $E$ the above result leads to uniform decay on bounded sets for the {\em defocusing nonlinearity}. Consequently the first compactness tool in \cite[Section 1.7]{Caz:nls} implies relative compactness on bounded sets, see also \cite{jls:bsd}:

\begin{corollary} \label{cor:defocuscomp}
Assume that there exist constants $E_-,E_+,M>0$ and a sequence $\{\psi_{E_k},E_k\}_{k\in\N}$ of solutions for \eqref{a:stationary} with $\sigma>0$ and $f\equiv 0$ such that
$$ E_k\in[E_-,E_+]\m, \qquad \|\psi_{E_k}\|_{H^1} \leq M. $$
Then for any $0<\gamma<E_-$ there exists a constant $C(\gamma)$ independent of $k$ such that
$$ |\psi_{E_k}(x)| \m\leq\m C(\gamma)e^{-\sqrt{E_- -\gamma}|x|} \FORALL x\in\R^n, \FORALL k\in\N \m.$$
Consequently, there is a subsequence of $(\psi_{E_k},E_k)_{k\in\N}$, denoted again by  $(\psi_{E_k},E_k)_{k\in\N}$, and a solution $(\psi_{E_*},E_*)\in H^2(\R^n)\times[E_-,E_+]$ of \eqref{a:stationary} such that
$$ \lim_{k\to\infty} E_k \m=\m E_*\m, \qquad \lim_{k\to\infty} \|\psi_{E_k}-\psi_{E_*}\|_{H^2}=0 \m. $$
\end{corollary}

For attractive nonlinearity we have the much weaker result which already requires strong convergence in $L^q$ spaces:

\begin{proposition}\label{lm:lqtoh2} Assume that there exists $E_*\in(0,\infty)$, $\psi_*\in H^1$, $2\leq q<\frac{2n}{n-2}$, $\tilde\psi\in L^q(\R^n)$ and a sequence $\{\psi_{E_k},E_k\}_{k\in\N},\ E_k\in(0,\infty)$ of solutions of \eqref{a:stationary} with $\sigma<0$ such that
\begin{align}
  \lim_{k\to\infty} E_k &= E_*\m, \label{econv}\\
  \psi_{E_k} &\stackrel{H^1}{\rightharpoonup} \psi_*\ ,\label{weakconv}\\
  \psi_{E_k} &\stackrel{L^q}{\rightarrow} \tilde\psi\ \ \ \mbox{for some} \
  \ \ 2\leq q<\frac{2n}{n-2}\ (2\leq q<\infty\ \ \mbox{if}\ \ n=2,\ 2 \leq q\leq\infty  \ \ \mbox{if}\ \ n=1), \ \ \ \label{lqconv}
\end{align}
then $\psi_*=\tilde \psi\in H^2(\R^n)$ and $\psi_{E_k}\stackrel{H^2}{\rightarrow} \psi_*$. Moreover $ (\psi_*,E_*)$ is a solution of \eqref{a:stationary}.
\end{proposition}

\begin{proof} Note that by the standard Sobolev embedding theorem, weak convergence in $H^1$ implies weak convergence to the same limit in $L^q$ with $q$ given in \eqref{lqconv}. Since norm convergence implies weak convergence we get $\psi_*=\tilde\psi$ by uniqueness of the weak limit.

We are now going to employ the regularity result in Remark \ref{rmk:h1toh2b}. Since the sequence $\{\psi_{E_k}\}_{k\in\N}$ is weakly convergent in $H^1(\R^n)$, it is norm bounded in $H^1.$ From \eqref{econv} and $0<E_*<\infty$ we also have that, $\{E_k\}_{k\in\N}$ is bounded and bounded away from zero. Hence, by Remark \ref{rmk:h1toh2b} we have that the sequence $\{\psi_{E_k}\}_{k\in\N}$ is bounded in {\em each} $W^{2,q}, 2\leq q<\infty.$ In particular it is bounded in $L^2\bigcap L^\infty.$ Now, by the convergence \eqref{lqconv} and interpolation we get that the sequence $\{\psi_{E_k}\}_{k\in\N}$ converges to $\psi_*$ in {\em all} $L^q,\ q>2$ spaces.

Finally, by rewriting \eqref{a:stationary} we get:
\begin{equation}\label{h2stationary}
  \psi_{E_k} = (-\Delta+E_k)^{-1} \left[-V\psi_{E_k} - \sigma|\psi_{E_k}|^{2p}\psi_{E_k}+f\right].
\end{equation}
Now, from the convergence above, we have
$$ |\psi_{E_k}(x)|^{2p}\psi_{E_k}(x) \stackrel{L^2}{\rightarrow} |\psi_*(x)|
   ^{2p}\psi_*(x).$$
Since $V\in L^\infty$ we definitely have $V(x) \psi_{E_k}(x)$ converges in $L^2(\R^n)$ provided $q=2$ in hypothesis \eqref{lqconv}. If $q>2$ we use the
following compactness argument. Fix $\epsilon>0$ and choose $L>0$ sufficiently
large such that
$$\|V(x)\|_{L^\infty (\{|x|\geq L\})}<\frac{\epsilon}{4M}.$$
Then, it follows from \eqref{lqconv} that
$$\lim_{m\rightarrow\infty} \|\psi_{E_m} - \psi_*\|_{L^2(\{|x|<L\})} = 0,$$
and hence there exists $m_0$ such that
$$\|\psi_{E_m}-\psi_*\|_{L^2(\{|x|<L\})} < \frac{\epsilon}{2\|V\|_{L^\infty
  (\R^n)}} \qquad \mbox{for}\ \ \ m\geq m_0.$$
Consequently for each $m>m_0$
\begin{align*}
  \|V\psi_{E_m}-V\psi_*\|_{L^2(\R^n)} &\leq \|V\|_{L^\infty(\R^n)}\|\psi_
  {E_m}-\psi_*\|_{L^2(\{|x|<L\})}+\|V\|_{L^\infty(\{|x|\geq L\})}\|\psi_{E_m}
  -\psi_*\|_{L^2(\R)}\\
  &< \frac{\epsilon}{2}+\frac{\epsilon}{4M}2M\leq \epsilon,
\end{align*}
which proves that $V\psi_{E_k} \stackrel{L^2}{\rightarrow} V\psi_*.$ Now, as
$E_k \rightarrow E_*$, both terms in the square bracket on the right side of
\eqref{h2stationary} converge in $L^2(\R^n)$ and $(-\Delta+E_k)^{-1}$
converges to $(-\Delta+E_*)^{-1}$ in the space of bounded linear operators
from $L^2(\R^n)$ to $H^2(\R^n).$ Hence we can conclude that the sequence $\{\psi_{E_k}\}_{k\in\N}$
converges in $H^2(\R^n)$. Since $H^2(\R^n)$ is continuously embedded in
$L^q(\R^n)$ for $2\leq q<2n/(n-4)$, it follows from \eqref{lqconv} that
the sequence $\{\psi_{E_k}\}_{k\in\N}$  converges to $\psi_*$ in $H^2(\R^n)$. Moreover, by passing to
the limit in \eqref{h2stationary}, we get
$$ \psi_* = (-\Delta+E_*)^{-1}\left[-V(x)\psi_*(x) - \sigma|\psi_*(x)|^{2p}
   \psi_* + f\right].$$
Therefore $(\psi_*,E_*)$ is a solution of the
equation \eqref{a:stationary}.
\end{proof}

The above result will be combined with the following concentration compactness one, see \cite[Section 1.7]{Caz:nls}. For functions $\psi\in L^2(\R^n)$ and numbers $r>0,$ define the concentration function:
$$\rho(\psi,r)=\sup_{y\in\R^n}\int_{|x-y|\leq r}||\psi(x)|^2dx.$$
For bounded sequences in $H^1$ we have:
\begin{proposition}\label{prop:comp} Let $\{\psi_{k}\}_{k\in\N}\subset H^1(\R^n)$ such that there exist $a,\ M>0$ with the properties:
\begin{eqnarray}
\|\psi_{k}\|_{L^2}^2&=&a,\qquad\forall k\in\N\nonumber\\
\|\psi_{k}\|_{H^1}&\leq &M,\qquad\forall k\in\N\nonumber
\end{eqnarray}
Let
$$\mu=\lim_{r\rightarrow\infty}\liminf_{k\rightarrow\infty}\rho(\psi_k,r)\leq a.$$
Then one of the following holds.
\begin{itemize}
\item[(i)] If $\mu=a$ then there is a subsequence re-denoted by $\{\psi_{k}\}_{k\in\N}$ and a sequence $\{y_{k}\}_{k\in\N}\subset\R^n$ such that
$$\psi_k(x-y_k)\stackrel{L^q}{\rightarrow}\psi_*,\qquad\forall 2\leq q<\frac{2n}{n-2}\ (2\leq q<\infty\mbox{ if }n=2,\ 2\leq q\leq\infty\mbox{ if }n=1).$$
\item[(ii)] If $\mu=0$ then there is a subsequence re-denoted by $\{\psi_{k}\}_{k\in\N}$ such that
$$\psi_k\stackrel{L^q}{\rightarrow}0,\qquad\forall 2< q<\frac{2n}{n-2}\ (2< q<\infty\mbox{ if }n=2,\ 2< q\leq\infty\mbox{ if }n=1).$$
\item[(iii)] If $0<\mu <a$ then there is a subsequence re-denoted by $\{\psi_{k}\}_{k\in\N}$ and the sequences $\{v_{k}\}_{k\in\N},$ $\{w_{k}\}_{k\in\N},$ $\{z_{k}\}_{k\in\N}$ all bounded in $H^1$ such that
    \begin{eqnarray}
   \psi_k&=&v_k+w_k+z_k\quad\mbox{and}\quad |v_k|+|w_k|\leq |\psi_k|,\nonumber\\
   \lim_{k\rightarrow\infty}\|v_k\|_{L^2}^2&=&\mu,\mbox{ and }v_k(x-y_k)\stackrel{L^q}{\rightarrow}v_*,\ \forall 2\leq q<\frac{2n}{n-2}\ (2\leq q\leq\infty\mbox{ if }n=1),\label{vconv}\\
   \lim_{k\rightarrow\infty}\|w_k\|_{L^2}^2&=&a-\mu\quad\mbox{and}\quad\lim_{k\rightarrow\infty}{\rm dist}({\rm supp}\ w_k,\ {\rm supp}\ v_k)=\infty,\nonumber\\
   \lim_{k\rightarrow\infty}\|z_k\|_{L^q}&=&0\qquad \forall 2\leq q<\frac{2n}{n-2}\ (2\leq q\leq\infty\mbox{ if }n=1),\nonumber
   \end{eqnarray}
\end{itemize}
\end{proposition}
The proof follows from the one in \cite[Section 1.7]{Caz:nls} except for the splitting part (iii). The latter requires a slight adjustment in the choice of the cutoff function to insure large separation for the supports of $v_k$ and $w_k$ as $k\rightarrow\infty.$ The fact that $v_k$ can be chosen to satisfy the convergence properties \eqref{vconv} is very important in our analysis and follows from the following argument. With the notations in \cite{Caz:nls}, if $\lim_{r\rightarrow\infty}\liminf_{k\rightarrow\infty}\rho(v_k,r)=\mu$ then we can apply part (i) to the sequence $\{v_k\}_{k\in\N}$ instead of $\{\psi_k\}_{k\in\N},$ see \cite[Section 1.7]{Caz:nls} to get the convergence on a subsequence. If $\lim_{r\rightarrow\infty}\liminf_{k\rightarrow\infty}\rho(v_k,r)<\mu$ then one can show that $\lim_{r\rightarrow\infty}\liminf_{k\rightarrow\infty}\rho(w_k,r)=\mu.$ Of course, this leads to a contradiction if $a-\mu<\mu,$ hence $\{v_k\}_{k\in\N}$ must be in case (i), or leads to $\{w_k\}_{k\in\N}$ being in case (i) if $a-\mu=\mu.$ If the latter holds we switch the roles of $v_k$ and $w_k.$ Finally, if $a-\mu >\mu$ we repeat the splitting argument for $\{w_k\}_{k\in\N}$ instead of $\{\psi_k\}_{k\in\N}$ which now has $L^2$ norm convergent to $a-\mu<a$ but the same ``concentration" number $\mu.$  After no more then $m$ steps where $m\geq\frac{a}{\mu}$ the algorithm stops and a split with one function in case (i) is produced. Note that a similar result is obtained in \cite{Maris}.

We now present with full proof a result on uniform bounds for the resolvent of families of linear Schr\"{o}dinger operators with large separation potential. This is applied to obtain and control Lyapunov-Schmidt type decompositions when splitting, see (iii) above, occurs for sequences of ground states.

\begin{proposition} \label{prop:MS}
Let $V_1,V_2,\ldots V_m:\R^n\mapsto\R$ such that for each $k$, $V_k\in L^\infty(\R^n)$ and $\lim_{|x|\to\infty}V_k(x)=0$. Let $s_1,s_2,\ldots s_m$ be $\R^n$-valued functions of $R\in(0,1)$
such that $\lim_{R\to 0}|s_i(R)-s_j(R)|=\infty$ whenever $i\neq j$ and $E$ be a $\R$-valued function of $R\in(0,1)$ such that $\lim_{R\to 0}E(R)=E_*\in (0,\infty)$. Consider the family of operators
\begin{align}
 &L_R:H^2\mapsto L^2 \nonumber\\
 &L_R \m=\m-\Delta+E(R)+\sum_{k=1}^m T_{s_k(R)}V_k\m, \qquad R\in(0,1) \m ,
 \label{L_R_spec}
\end{align}
where we used the notation $T_yV(x)=V(x-y).$ Denote by $P^\perp_0(R)$ the projection operator in $L^2$ on to the
orthogonal complement of the set of all eigenvectors $\varphi$ that
satisfy $(-\Delta+E(R)+T_{s_k(R)}V_k)\varphi=0$ for some $k$. Then
there exists $0<r<1$ and $C>0$ such that for each $R\in(0,r)$, $P^\perp_0(R)L_R P^\perp_0(R)$ is invertible and
$$ \|(P^\perp_0(R)L_R P^\perp_0(R))^{-1}\|_{H^2\mapsto L^2}<C$$
\end{proposition}

\begin{proof}
Note that $L_R$ is a self-adjoint operator on $L^2(\R^n)$ with domain $H^2(\R^n)$ and
$$ K_1\m=\m\sup_{R\in(0,1)}\|L_R\|_{\Lscr(H^2,L^2)}<\infty.$$ Since each $V_k$ is a relatively compact operator with respect to $(-\Delta+E_*)$, it follows from Weyl's theorem that the essential spectrum of $(-\Delta +E_*+V_k)$ for all $k$ and of $L_R$ is $[1,\infty)$. Let the increasing sequence of numbers $\mu_1,\mu_2,\ldots \mu_p$, with $\mu_{q+1}=\mu_{q+2}=\ldots\mu_p=0$ for some $0\leq q\leq p$, be the  collection of all the non-positive discrete eigenvalues of the operators $(-\Delta+ E_*+V_k)$, $1\leq k\leq m$.

Using the spectral theory for operators with potentials separated by large distances in [Morgan and Simon, 1980] and the spectral perturbation theory (to handle the fact that $E(R)$ may not be a constant) we get that as $R\to0$, the discrete eigenvalues of $L_R$ approach precisely the union of the discrete eigenvalues of the family of operators $\{(-\Delta+E_*+V_k)\m\big|\m 1 \leq k\leq m\}$. Furthermore, the spectral projection operator $\widetilde P_\mu(R)$ associated with the set of all the eigenvalues of $L_R$ which converge to some $\mu$ in the above union, converges in $\Lscr(L^2)$ to $P_\mu (R)$  which is the projection operator on to the span of the set $\{\textrm{Range\m}(P_{k,\mu}(R))\m\big|\m 1\leq k\leq m\}$. Here $P_{k,\mu}(R)$ is the spectral projection operator corresponding to $\mu$ for the operator $(-\Delta+E_*+T_{s_k(R)}V_k)$. From all this we can conclude that there exist $r_1\in(0,1)$ such that for each $R\in(0,r_1)$, $\sigma(L_R)$ is the union of $\{\mu_1^R,\mu_2^R, \ldots\mu_p^R\}$ and $\sigma^c(L_R)$, where $\mu_1^R,\mu_2^R,\ldots
\mu_p^R$ are the $p$ smallest eigenvalues of $L_R$ counting multiplicities and $\sigma^c(L_R) \subset(a,\infty)$ for some $a>0$, and
\begin{equation} \label{ev_conv}
 \lim_{R\to 0}\mu^R_j \m=\m \mu_j \FORALL j\in\{1,2,\ldots p\}\m,
 \qquad \lim_{R\to 0}\|P_0^\perp(R)-\widetilde P_0^\perp(R)\|_{\Lscr
 (L^2)}\m=\m0\m.
\end{equation}
For each $j\in\{1,2,\ldots p\}$, let $\varphi_j^R$ be the eigenvector associated with the eigenvalue $\mu_j^R$ of $L_R$.

Using \cite[Theorem XII.5]{RS:bk4} and the fact that $L_R$ is a self-adjoint operator (which implies that the projection operator in that theorem is orthogonal), it follows from the first limit in \eqref{ev_conv} that there exists $r_2\in(0,r_1)$ such that $\widetilde
P_0^\perp(R)L_R\widetilde P_0^\perp(R): \widetilde P_0^\perp(R) L^2\cap H^2\to \widetilde P_0^\perp(R) L^2$ is a boundedly invertible operator for all $R\in(0,r_2)$. We claim that $r_2$ can be chosen such that $\|(\widetilde P_0^\perp(R) L_R\widetilde P_0^\perp(R))^{-1}\|$ has a uniform bound on $(0,r_2)$. Suppose that this is not true. Then there exists a sequence $(R_k)_{k=1}^\infty$ in $(0,r_2)$ converging to $0$ and a sequence of functions $(u_k)_{k=1}^\infty$ with $u_k\in \widetilde P_0^\perp(R_k) L^2\cap H^2$ and $\|u_k\|_{H^2}=1$ such that
$$ \lim_{k\to\infty} \widetilde P_0^\perp(R_k)L_{R_k}\widetilde
   P_0^\perp(R_k)u_k \m=\m 0\m, \qquad \langle u_k, \varphi_j^{R_k}
   \rangle \m=\m 0 \FORALL k\m\geq\m1\m, \FORALL j\in\{1,2, \ldots q\}
   \m. $$
Hence $u_k\in \widetilde P^\perp(R_k) L^2\cap H^2$ for each $k$, where $\widetilde P^\perp(R_k)$ is the projection operator in $L^2$ onto the orthogonal complement of the set
$$ {\rm span}\{\varphi_1^{R_k}, \varphi_2^{R_k}, \ldots \varphi_p^{R_k}\}
   \m. $$
Note that $\liminf\|u_k\|_{L^2}>0$. Indeed, if this were not true, then using $\lim_{k\to\infty}L_{R_k} u_k=0$ and $\lim_{k\to\infty} E(R_k)=E_*\in(0,\infty)$ it follows that along a subsequence $\lim_{k\to\infty}(-\Delta+E(R_k)) u_k= \lim_{k\to\infty}\sum_{j=1}^m T_{s_j(R)}V_j u_k=0$, which contradicts the fact that $\|u_k\|_{H^2}=1$ for all $k$. It now follows from the operator form of the min-max principle [Theorem XIII.1, Reed and Simon] that there exists an eigenvalue $\mu_{p+1}^{R_k}$ for $L_{R_k}$, with eigenvector $\varphi_{p+1}^{R_k}$ different from $\varphi_j^{R_k}$ for $1\leq j\leq p$, such that $\limsup_{{R_k}\to 0} \mu_{p+1}^{R_k}\leq0$. This contradicts our conclusion about the spectrum of $L_{R_k}$ (see discussion above \eqref{ev_conv}). Hence we can fix $r_2\in
(0,r_1)$ such that $(\widetilde P_0^\perp(R) L_R\widetilde P_0^\perp(R))^{-1}$ has a uniform bound on $(0,r_2)$. Define
$$ K_2=\sup_{R\in(0,r_2)}\| (\widetilde P_0^\perp(R) L_R \widetilde P_0^\perp(R))^{-1}\|_{L^2\mapsto H^2}.$$

Define the operator $\Delta P(R)=\m P_0^\perp(R)-\widetilde P_0^\perp(R)$ which maps $L^2$ to $L^2$ and $H^2$ to $H^2$. The second expression in \eqref{ev_conv} implies easily that the ranges of $P_0(R)$ and $\widetilde P_0(R)$ have the same dimension (for small $R$) and also that
\begin{equation} \label{nan_the}
 \lim_{R\to 0}\|P_0(R)-\widetilde P_0(R)\|_{\Lscr(H^2)} \m=\m0.
\end{equation}
To establish \eqref{nan_the} it is sufficient to show that given $\varepsilon>0$, there exists $r_\varepsilon\in(0,r_2)$ such that for each $R\in(0,r_\varepsilon)$ and every eigenvector $\psi\in \textrm{Range}(P_{k,0}(R))$ satisfying $\|\psi\|_{L^2}=1$, for some $k$, there exists a $\varphi\in\textrm{Range}(\widetilde P_0(R))$ such that $\|\psi-\varphi\|_{H^2}<\varepsilon$. Hence given $\varepsilon\in(0,1)$, fix $r_\varepsilon\in(0,r_2)$ such that for each $R\in(0,r_\varepsilon)$
\begin{equation} \label{change_smallE}
 |E(R)-E_*|<\frac{\varepsilon}{4+\varepsilon} \m.
\end{equation}
\begin{equation} \label{latte1}
 \sup_{1\leq k\leq m}\sup_{\begin{array}{c}\psi\in \textrm{Range}
 (P_{k,0}(R))\\ \|\psi\|_{L^2} =1\end{array}}\bigg\|\sum_{j=1,\m j\neq
 k}^m (T_{s_j(R)}V_j) \psi\bigg\|_{L^2} \m<\m \frac{\varepsilon}{4}\m, \vspace{-2mm}
\end{equation}
\begin{equation} \label{latte2}
 \|P_0(R)-\widetilde P_0(R)\|_{\Lscr(L^2)}\bigg(1+\sum_{j=1}^m \|V_j\|_{
 L^\infty}\bigg) \m<\m \frac{\varepsilon}{4}\m,
\end{equation}
\begin{equation} \label{latte3}
 \sup\big\{\m|\mu_j^R|\m\big|\m 1 \leq j\leq p, \lim_{R\to0} \mu_j^R=0
 \big\}<\frac{\varepsilon}{4+\varepsilon} \m.
\end{equation}
Let $R\in(0,r_\varepsilon)$. Consider a $\psi\in \textrm{Range}(P_{k,0}(R))$ satisfying $\|\psi\|_{L^2}=1$. From \eqref{latte2} it follows that there exists a $\varphi\in\textrm{Range}(\widetilde P_0(R))$ such that
\begin{equation} \label{latte4}
 \|\varphi-\psi\|_{L^2} \bigg(1+\sum_{j=1}^m \|V_j\|_{L^\infty}\bigg)
 \m<\m \frac{\varepsilon}{4}
\end{equation}
and $\varphi=\sum_{j=1}^{t} \alpha_j\varphi_j^R$, where $\{\varphi_1^R,
\varphi_2^R,\ldots \varphi_t^R\}$ is an orthonormal family of eigenvectors of $L_R$ corresponding to its eigenvalues that converge to $0$. Hence $\|\varphi\|_{L^2}<(4+\varepsilon)/4$ and we have
$$ (-\Delta +E_*+T_{s_k(R)}V_k)\psi\m=\m0\m, \qquad \Big(-\Delta +E(R)+\sum_{
   j=1}^m T_{s_j(R)}V_j\Big)\varphi\m=\m\sum_{j=1}^t \mu_j^R\alpha_j
   \varphi_j^R\m. $$
From these equations, using \eqref{change_smallE}, \eqref{latte1}, \eqref{latte3} and
\eqref{latte4}, we get
$$ \|(-\Delta +E_*)(\psi-\varphi)\|_{L^2}=\Big\|(E(R)-E_*)\varphi-\sum_{j=1}^m V_j^{s_j(R)}
   (\varphi-\psi)+ \sum_{j=1,\m j\neq k}^m V_j^{s_j(R)}\psi -\sum_{j=1}^t \mu_j\alpha_j\varphi_j^R\Big\|_{L^2}\m\leq\m \varepsilon\m.$$
Since the choice $\varepsilon\in(0,1)$ is arbitrary, \eqref{nan_the} follows.

From the second limit in \eqref{ev_conv} and \eqref{nan_the} it follows that there exists $r\in(0,r_2)$ such that
\begin{equation} \label{PDE_talk}
  \sup\{\|\Delta P(R)\|_{\Lscr(L^2,L^2)}, \|\Delta P(R)\|_{\Lscr(H^2,
  H^2)}\} \m<\m \inf\bigg\{\frac{1}{4 K_1K_2},\frac{1}{3}\bigg\}
  \FORALL R\in(0,r) \m.
\end{equation}
Consider the operator $P_0^\perp(R)L_R P_0^\perp(R):P_0^\perp(R)L^2\cap H^2
\to P_0^\perp(R)L^2$. For any $u\in H^2$, we have
\begin{equation} \label{new_shelf}
 P_0^\perp(R) L_R P_0^\perp(R) u \m=\m P_0^\perp(R) L_R\Delta P(R)u +
 \Delta P(R) L_R \widetilde P_0^\perp(R)u + \widetilde P_0^\perp(R) L_R
 \widetilde P_0^\perp(R) u\m.
\end{equation}
It follows from \eqref{new_shelf}, using the definitions of $K_1$ and $K_2$
and \eqref{PDE_talk}, that for each $R\in(0,r)$,
\begin{align}
 \|P_0^\perp(R) L_R P_0^\perp(R) u\|_{L^2} &\m\geq\m \|\widetilde P_0^\perp
 (R) L_R\widetilde P_0^\perp(R) u\|_{L^2} - \|P_0^\perp(R) L_R\Delta P(R)u
 \|_{L^2} - \|\Delta P(R) L_R \widetilde P_0^\perp(R)u\|_{L^2} \nonumber\\
 &\m\geq\m \frac{1}{2K_2}\|u\|_{H^2}\m. \label{carp_AC}
\end{align}
The range of $P_0^\perp(R)L_R P_0^\perp(R)$ is dense in $P_0^\perp(R)L^2$
for each $R\in(0,r)$. Indeed, suppose that for some $v\in L^2$
with $\|P_0^\perp(R)v\|_{L^2}=1$,
\begin{equation} \label{driving}
 \big\langle P_0^\perp(R)L_R P_0^\perp(R)u, P_0^\perp(R) v \big\rangle
 \m=\m 0 \FORALL  u\in H^2(\R^n)\m.
\end{equation}
Fix $u\in H^2$ such that $\widetilde P_0^\perp(R)L_R \widetilde P_0^\perp
(R) u=\widetilde P_0^\perp P_0^\perp(R) v$. Then, using \eqref{PDE_talk}
we get that
$$ \|u\|_{H^2} \m\leq\m K_2\|\widetilde P_0^\perp P_0^\perp(R) v\|_{L^2} \m
   \leq\m K_2 \|P_0^\perp(R) v-\Delta P(R) P_0^\perp(R) v\|_{L^2} \m\leq\m
   4K_2/3 \m. $$
Furthermore, using \eqref{PDE_talk} and \eqref{new_shelf}, we have
\begin{align*}
 & \big|\big\langle P_0^\perp(R)L_R P_0^\perp(R)u, P_0^\perp(R) v \big
 \rangle\big|\\
 \geq\m& \big|\big\langle \widetilde P_0^\perp(R)L_R \widetilde P_0^\perp
 (R)u, P_0^\perp(R) v \big\rangle\big|-\big|\big\langle P_0^\perp(R)L_R
 \Delta P(R)u, P_0^\perp(R) v \big \rangle \big|-\big|\big\langle \Delta
 P(R)L_R \widetilde P_0^\perp(R)u, P_0^\perp(R) v \big\rangle \big|\\
 >\m& \big|\big\langle \widetilde P_0^\perp(R) P_0^\perp(R)v,P_0^\perp(R)
 v \big\rangle\big| -\frac{2}{3}\m>\m 0
\end{align*}
which contradicts \eqref{driving}, thereby establishing that the range of
$P_0^\perp(R)L_R P_0^\perp(R)$ is dense in $P_0^\perp(R)L^2$ for each
$R\in(0,r)$. This, along with \eqref{carp_AC} and the fact that
$P_0^\perp(R)L_R P_0^\perp(R)\in\Lscr(H^2,L^2)$, implies that $P_0^\perp(R)
L_R P_0^\perp(R)$ is invertible for each $R\in(0,r)$ with $\|(P_0^\perp(R)
L_R P_0^\perp(R))^{-1}\|<2K_2$.
\end{proof}

Finally, we present with full proof the existence and properties of a Lyapunov-Schmidt type decompositions when splitting, see (iii) in Proposition \ref{prop:comp}, occurs for sequences of ground states.

\begin{proposition}\label{lm:newdec_gen}
Consider the two sets of real valued $C^1$ functions $\{u_1,u_2,\ldots u_m\}$
and $\{\phi_1,\phi_2,\ldots \phi_d\}$ in $H^2(\R^n)$ which satisfy
\begin{equation} \label{doc_vis}
 |\partial_{x_k} u_j(x)|\m\leq\m C e^{-\alpha|x|} \FORALL x\in\R^n\m,
 \FORALL j\in\{1,2,\ldots m\}\m, \FORALL k\in\{1,2,\ldots n\}\m,
\end{equation}
\begin{equation} \label{doc_vis_add}
 |\phi_i(x)|\m\leq\m C e^{-\alpha|x|} \qquad \FORALL x\in\R^n\m, \FORALL i\in\{1,2,\ldots d\},
\end{equation}
for some $C,\alpha>0$ and also satisfy
\begin{equation} \label{vistara}
 \langle \partial_{x_i} u_j, \partial_{x_k} u_j\rangle \m=\m0
 \FORALL j\in\{1,2,\ldots m\}\m, \FORALL i,k\in\{1,2,\ldots n\}\ \
 \textrm{with} \ \ i\neq k\m.
\end{equation}
and
\begin{equation} \label{vistara_add}
 \langle \phi_r, \phi_q \rangle \m=\m0 \FORALL r,q\in\{1,2,\ldots d\}\ \
 \textrm{with} \ \ r\neq q\m.
\end{equation}
Let $\psi\in H^2(\R^n)$ be given. With the notation $T_yu(x)=u(x-y)$ we have that there exist constants $L,\e>0$ and unique $C^1$ functions $\Sscr$ and $\Ascr$ defined on the set of all $y=[\m y_1,y_2,\ldots y_m\m]^\top$ in $\R^{nm}$ and all $u\in H^1(\R^n)$ which satisfy the conditions
$$ |y_j-y_k|\m>\m L \quad\forall\m j,k\in\{1,2,\ldots m\}\ \ with\ \ j\neq k\m, \qquad \quad \|u-\psi-\sum_{j=1}^m T_{y_j} u_j\|_{L^2}
   \m<\m\e\m, $$
\begin{equation} \label{not_necessary}
 |y_j|>L \FORALL j\in\{1,2,\ldots m\}\m,
\end{equation}
such that using $s=\Sscr(u,y)\in\R^{nm}$, written as $[\m s_1,s_2,\ldots
s_m \m]^\top$, and $a=\Ascr(u,y)\in\R^d$, written as $[\m a_1,a_2,\ldots
a_d \m]^\top$, $u$ can be decomposed as follows:
\begin{equation}\label{inf_decom}
  u \m=\m \psi+ \sum_{j=1}^m T_{y_j+s_j}u_j+v+\sum_{i=1}^d a_i\phi_i\m,
\end{equation}
where $v$ satisfies
$$ \langle v, T_{y_j+s_j} \partial_{x_k} u_j\rangle \m=0\m \FORALL
   k\in\{1,2,\ldots n\}\m, \FORALL j\in\{1,2,\ldots m\}\m,$$
$$ \langle v, \phi_i \rangle \m=0\m \FORALL i\in\{1,2,\ldots d\}\m.$$
In addition, if we let $A(y)=\inf\{|y_j|, |y_j-y_k|\m\big|\m j,k\in\{1,2,\ldots
m\}, j\neq k\}$ and $B(u,y)= \|u-\psi-\sum_{j=1}^m T_{y_j} u_j\|_{L^2}$,
then
\begin{equation} \label{nsttest}
  \lim_{A(y)\to\infty\m,\ B(u,y)\to0}\Sscr(u,y) \m=\m 0\m, \qquad  \lim_{A(y)\to\infty\m,\ B(u,y)\to0}\Ascr(u,y) \m=\m 0\m.
\end{equation}
\end{proposition}

\begin{proof}
Define the map $\FFF:L^2(\R^n)\times\R^{nm}\times\R^{nm}\times\R^d\mapsto \R^{nm}\times\R^d$ as follows:
$$ \FFF(u,y,s,a) \m=\m [\m\FFF_1(u,y,s,a),\FFF_2(u,y,s,a),\ldots\FFF_m(u,y,s,a), \GGG(u,y,s,a) \m]^\top\m,$$
where for each $j\in\{1,2,\ldots m\}$
$$ \FFF_j(u,y,s,a) \m=\m [\m F_{j1}(u,y,s,a),F_{j2}(u,y,s,a),\ldots F_{jn}
   (u,y,s,a)\m]^\top $$
with
$$ F_{jk}(u,y,s,a) \m=\m \bigg\langle\frac{T_{y_j+s_j}\partial_{x_k}u_j}
  {\|\partial_{x_k}u_j\|^2_{L^2}}, u-\psi-\sum_{t=1}^m T_{y_t+s_t} u_t- \sum_{i=1}^d a_i\phi_i\bigg \rangle_{L^2} \FORALL k\in\{1,2,\ldots n\} $$
and
$$ \GGG(u,y,s,a) \m=\m [\m G_1(u,y,s,a),G_2(u,y,s,a),\ldots G_d(u,y,s,a) \m]^\top $$
with
$$ G_r(u,y,s,a) \m=\m \bigg\langle\frac{\phi_r}{\|\phi_r\|_{L^2}^2}, u-\psi-\sum_{t=1}^m T_{y_t+s_t} u_t-\sum_{i=1}^d a_i\phi_i \bigg  \rangle_{L^2} \FORALL r\in\{1,2,\ldots d\}. $$
We will establish the proposition by solving
\begin{equation}\label{inf_eqdec}
  \FFF(u,y,s,a) \m=\m 0
\end{equation}
for $(s,a)\in\R^{nm+d}$, given $u$ and $y$. Note that an $(s,a)$ that solves \eqref{inf_eqdec} for a given $u$ and $y$ implies the decomposition in \eqref{inf_decom}. The function $\FFF$ is $C^1$ on $L^2(\R^n)\times\R^{nm} \times\R^{nm}\times\R^d$. Its $Fr\acute echet$ derivative with respect to $(s,a)$, denoted as $D_{(s,a)}\FFF(u,y,s,a)$, is a $\R^{nm+d}\times\R^{nm+d}$ matrix whose elements are as follows: For $j,l\in\{1,2,\ldots m\}$ and $i,k\in\{1,2,\ldots n\}$, the $((j-1)n+k, (l-1)n+i)^{th}$ element is
$$ \frac{\partial F_{jk}}{\partial s_{li}} \m=\m \bigg\langle\frac{
   T_{y_j+s_j}\partial_{x_k} u_j}{\|\partial_{x_k}u_j \|^2_{L^2}},
   T_{y_l+s_l}\partial_{x_i}u_l\bigg\rangle -\delta_{jl}\bigg\langle
   \frac{T_{y_j+s_j}\partial^2_{x_k x_i}u_j}{\|\partial_{x_k}u_j\|^
   2_{L^2}}, u-\psi-\sum_{t=1}^m T_{y_t+s_t}u_t-\sum_{t=1}^d a_t\phi_t\bigg\rangle. $$
Here $s_{li}$ is the $i^{th}$ element of $s_l$.  For $j\in\{1,2,\ldots m\}$, $k\in\{1,2,\ldots n\}$ and  $q\in\{1,2,\ldots d\}$, the $((j-1)n+k,nm+r)^{th}$ element of $D_{(s,a)}\FFF(u,y,s,a)$ is
$$ \frac{\partial F_{jk}}{\partial a_q} \m=\m \bigg\langle\frac{T_{y_j+s_j}
   \partial_{x_k} u_j}{\|\partial_{x_k}u_j \|^2_{L^2}},
   \phi_q\bigg\rangle. $$
For $l\in\{1,2,\ldots m\}$, $i\in\{1,2,\ldots n\}$ and $r\in\{1,2,\ldots d\}$,  the $(nm+r,(l-1)n+i)^{th}$ element is
$$ \frac{\partial G_r}{\partial s_{li}} \m=\m \bigg\langle\frac{\phi_r} {\|\phi_r\|_{L^2}^2}, T_{y_l+s_l} \partial_{x_i}u_l\bigg\rangle . $$
Finally for $r,q\in\{1,2,\ldots d\}$,
$$ \frac{\partial G_r}{\partial a_q} \m=\m \delta_{rq}. $$
It is easy to verify using the above expression and the decay estimates for $\partial_{x_k}u_j$ and $\phi_i$ that if $u$ and $y$ satisfy the conditions in the proposition for some $L>1$ and $\e$ and $|s_j|<L/4$ for each $j\in\{1,2,\ldots d\}$, then there exists a constant $C(n,\alpha)>0$ such that
\begin{align*}
 \left|\frac{\partial F_{jk}}{\partial s_{li}}\right| \m\leq\m&
 C(n,\alpha)L^n e^{-\alpha L/2} \qquad {\rm if} \quad j\m\neq\m l\m,\\[2pt]
 \left|\frac{\partial F_{jk}}{\partial s_{li}}\right| \m\leq\m&
 \delta_{ik}+\big(\e+|a|\sum\limits_{t=1}^d \|\phi_t\|_{L^2}+\sum\limits_{t=1}^m \|T_{y_t+s_t}u_t-T_{y_t}u_t\|_{L^2}\big) \frac{\|u_j\|_{H^2}}{\|\partial_{x_k}u_j\|^2_{L^2}}\qquad {\rm if}\quad j\m=\m l\m,
\end{align*}
$$ \left|\frac{\partial F_{jk}}{\partial a_q}\right| \m\leq\m C(n,\alpha)L^n e^{-\alpha L/2} \m,\qquad \left|\frac{\partial G_r}{\partial s_{li}}\right| \m\leq\m  C(n,\alpha)L^n e^{-\alpha L/2} \m. $$
For any $y$, define $u_y=\psi+\sum_{j=1}^m T_{y_j}u_j$. Choose a $\delta>0$.
Fix $L,\e,\eta>0$ and $\gamma\in(0,\delta)$ such that the following
estimates hold: for any $u$ and $y$ that satisfy the conditions in
the proposition, $s$ that satisfies $\sum_{j=1}^m |s_j|<\gamma$, $a$ that satisfies $|a|<\gamma$ and $\tilde y$ that satisfies $|\tilde y-y|<\eta$,
\begin{equation}
 |\FFF(u,\tilde y,0,0)|\m<\m\frac{\gamma}{4}\m, \qquad \|[D_{(s,a)}\FFF(u_y,
 y,0,0)]^{-1}\|<2\m, \label{inf_ineq1}
\end{equation}
\begin{equation}
 \|D_{(s,a)}\FFF(u_y,y,0,0)-D_{(s,a)}\FFF(u,\tilde y, s,a)\| \m<\m \frac{1}{4}\m. \label{inf_ineq2}
\end{equation}
The existence of constants $L,\e,\eta,\gamma$ follows from the definition of $\FFF$ and the estimates for the elements of $D_{(s,a)}\FFF$. Given $u$ and $y$ that satisfy the conditions in the proposition, we now apply the implicit function theorem to \eqref{inf_eqdec} at its solution $(u_y,y,0,0)$ and conclude, using
the estimates in \eqref{inf_ineq1} and \eqref{inf_ineq2}, that there exists a unique $s$ and $a$ satisfying $\sum_{j=1}^m |s_j|<\gamma$ and $|a|<\gamma$ such that $\FFF(u,y,s,a)=0$. The dependence of $s$ and $a$ on $(u,y)$ is $C^1$ since $\FFF$ is $C^1$. We define $\Sscr(u,y)=s$ and $\Ascr(u,y)=a$. The estimate in \eqref{nsttest} follows from the fact that we can choose $\delta>0$ to be arbitrarily small, in which case the corresponding $L$ will be sufficiently large and $\e, \eta,\gamma$ will be sufficiently small. The uniqueness of $s$ and $a$ satisfying $\sum_{j=1}^m |s_j|<\gamma$ and $|a|<\gamma$ (as mentioned above), along with the requirement that $\Sscr$ and $\Ascr$ be $C^1$ and satisfy the estimate in \eqref{nsttest}, implies the uniqueness of $\Sscr$ and $\Ascr$ we have constructed.
\end{proof}

We remark that in the above proof, by applying the implicit function theorem at $(u_y,y,0,0)$, we can conclude that there exists $L, \e_y>0$ ($\e_y$ depends on $y$) such that if $u$ and $y$ satisfy the conditions in the statement of the proposition with $\e=\e_y$, then there is a unique $s$ and a unique $a$ for which $\FFF(u,y,s,a)=0$. But to establish the existence of an $\e$ independent of $y$, we need to use the estimates in \eqref{inf_ineq1} and \eqref{inf_ineq2} which are used (in some form) in the proof of the implicit function theorem to construct closed balls for applying the contraction mapping principle.

\begin{remark} \label{decay_not needed}
When $\psi=0$ and $\phi_r=0$ for all $r\in\{1,2,\ldots d\}$, the conclusions of the above proposition continue to hold (with $\Ascr=0$ of course) with $A$ defined as $A(y)=\inf\{|y_j-y_k|\m\big|\m j,k\in\{1,2,\ldots m\}, j\neq k\}$ and without the restriction \eqref{not_necessary}. In particular one of the $y_j$s can be near 0.
\end{remark}

\begin{remark} \label{svn_ram}
It is evident from the proof of Proposition \ref{lm:newdec_gen} that the constants $L,\e>0$ can be chosen such that they are independent of the functions $\{u_1,u_2,\ldots u_m\}$ and the convergence of the maps $\Sscr$ and $\Ascr$ in \eqref{nsttest} is uniform in the same functions, as long as these functions satisfy the decay estimates \eqref{doc_vis}, the conditions \eqref{vistara} and the inequalities
$$ \|\partial_{x_k}u_j\|_{L^2}\m>\m c_1\m, \qquad \|u_j\|_{H^2}\m<\m
   c_2 \FORALL j\in\{1,2,\ldots m\}\m, \FORALL k\{1,2,\ldots n\}\m, $$
for some fixed $c_1,c_2,C,\alpha>0$. Furthermore, the dependence of $\Sscr$ and $\Ascr$ on $\{u_1,u_2,\ldots u_m\}$ is $C^1$. Indeed, consider the function $\FFF(u_1,u_2,\ldots u_m,u,y,s,a)$ defined in the proof of Proposition \ref{lm:newdec_gen}, where this function is denoted as $\FFF(u,y,s,s)$ (in Proposition \ref{lm:newdec_gen} $\{u_1,u_2,\ldots u_m\}$ are fixed). We can apply the implicit function theorem to this function, by repeating the steps
in the proof of Proposition \ref{lm:newdec_gen}, and use the fact that $\FFF$ is
$C^1$ to conclude that the dependence of $\Sscr$ and $\Ascr$ on $(u_1,u_2, \ldots u_d,u,y)$ is $C^1$.
\end{remark}

\end{document}